\documentclass[12pt, reqno]{amsart}
\usepackage{amsmath, amsthm, amscd, amsfonts, amssymb, graphicx, color, mathrsfs}
\usepackage[bookmarksnumbered, colorlinks, plainpages]{hyperref}
\usepackage[all]{xy}
\usepackage{slashed}
\usepackage{tikz-cd}
\usepackage{mathabx}
\usepackage{tipa}
\usepackage{soul}
\usepackage{cancel}
\usepackage{ulem}

\newtheorem{theorem}{Theorem}[section]
\newtheorem{lemma}[theorem]{Lemma}

\newtheorem{proposition}[theorem]{Proposition}
\newtheorem{corollary}[theorem]{Corollary}
\theoremstyle{definition}
\newtheorem{definition}[theorem]{Definition}
\newtheorem{example}[theorem]{Example}

\theoremstyle{remark}
\newtheorem{remark}[theorem]{Remark}
\numberwithin{equation}{section}

\begin{document}
\setcounter{page}{1}

\title[Operators on Homogeneous vector bundles  ]{ Pseudo-differential operators on Homogeneous vector bundles over  compact homogeneous manifolds}

\author[D. Cardona ]{Duv\'an Cardona}
\address{
  Duv\'an Cardona:
  \endgraf
  Department of Mathematics: Analysis, Logic and Discrete Mathematics
  \endgraf
  Ghent University, Belgium
  \endgraf
  {\it E-mail address} {\rm duvanc306@gmail.com, duvan.cardonasanchez@ugent.be}
  }
  
  \author[V. Kumar]{Vishvesh Kumar}
\address{
 Vishvesh Kumar:
  \endgraf
  Department of Mathematics: Analysis, Logic and Discrete Mathematics
  \endgraf
  Ghent University, Belgium
  \endgraf
  {\it E-mail address} {\rm vishveshmishra@gmail.com, Vishvesh.Kumar@UGent.be}
  }

\author[M. Ruzhansky]{Michael Ruzhansky}
\address{
  Michael Ruzhansky:
  \endgraf
  Department of Mathematics: Analysis, Logic and Discrete Mathematics
  \endgraf
  Ghent University, Belgium
  \endgraf
 and
  \endgraf
  School of Mathematical Sciences
  \endgraf
  Queen Mary University of London
  \endgraf
  United Kingdom
  \endgraf
  {\it E-mail address} {\rm michael.ruzhansky@ugent.be, m.ruzhansky@qmul.ac.uk}
  }

\dedicatory{ {To the memory of Isadore Singer, 1924--2021.}} 

 \allowdisplaybreaks



\thanks{The authors are supported  by the FWO  Odysseus  1  grant  G.0H94.18N:  Analysis  and  Partial Differential Equations and by the Methusalem programme of the Ghent University Special Research Fund (BOF)
(Grant number 01M01021). Michael Ruzhansky and Vishvesh Kumar are supported by the FWO Senior Research Grant G011522N. Duv\'an Cardona was supported by the Research Foundation-Flanders
(FWO) under the postdoctoral grant No 1204824N. The third author is also supported  by EPSRC grant EP/V005529/1.}

\begin{abstract} In this work, we introduce a global theory of subelliptic pseudo-differential operators on  arbitrary homogeneous vector bundles   over  orientable  compact homogeneous manifolds.   We will show that a global pseudo-differential calculus can be associated to the operators acting on any  pair of homogeneous vector-bundles with base space $M,$ if  the compact Lie group $G$ that acts on  $M=G/K$ is endowed with a (Riemannian or) sub-Riemannian structure. This is always possible if we choose on $G$ a sub-Laplacian associated to a H\"ormander system of vector-fields or  we fix the Laplace-Beltrami operator on $G$. We begin with developing a global subelliptic symbolic calculus for  vector-valued pseudo-differential operators on $G$ and then, we show that this vector-valued calculus  induces a pseudo-differential calculus on  homogeneous vector bundles, which, among other things, is stable under the action of the complex functional calculus. We prove  global versions of the Calder\'on-Vaillancourt theorem, Fefferman theorem and also, of the G\r{a}rding inequality. We present applications of the obtained G\r{a}rding inequality  to the wellposedness of evolution problems.  We characterise the classes of pseudo-differential operators on homogeneous vector bundles in the sense of H\"ormander (which are defined by using local coordinate systems) in terms of their  global symbols. Finally, using this formalism, we compute  the global symbol of the exterior derivative, its adjoint, and the symbol of the Dirac operator on the vector bundle of differential forms. We hope that this work will provide a solid foundation for further research using the global quantisation of operators on (vector-bundles over) compact homogeneous manifolds. 
\end{abstract} \maketitle

\tableofcontents
\section{Introduction}

Homogeneous vector bundles over compact homogeneous spaces  are geometrical structures of crucial interest in pure and applied mathematics, in particular, in the  index  theory of elliptic operators (see e.g. Atiyah and  Singer \cite[Page 494]{ASII}, Bott \cite{Bott1957} and Connes and Moscovici \cite{ConnesMoscovici}) and this work deals with the construction of a global/intrinsic theory of pseudo-differential operators on them. 

Our approach is free of local coordinate systems providing a new description e.g.  for the H\"ormander calculus of pseudo-differential operators on homogeneous vector-bundles  (see H\"ormander \cite{Hormander1985III}), and also for more general classes  of subelliptic pseudo-differential operators. The approach of working with the notion of a  global symbol, instead of the notion of a local symbol  defined by local coordinate systems, is of special interest in mathematical-physics due to the high computational costs and the complexity of the Fourier analysis under changes of coordinates (see Fulling \cite{Fullying96,Fullying98,FullyingKennedy88}).

The theory developed here shows, as it was conjectured in Subsection 1.2 of \cite{RuzhanskyCardona2020}, that a pseudo-differential calculus attached to a sub-Riemannian structure on compact manifolds with symmetries (more general than compact Lie groups) remains valid.     The pseudo-differential classes obtained in the present work absorb the ones developed by H\"ormander  \cite{Hormander1985III} in the case of homogeneous vector-bundles, and also the classes for compact Lie groups introduced by the third author and Turunen in \cite{Ruz} and by the first and third author in \cite{RuzhanskyCardona2020}.    The central notion of this theory is the notion of a global symbol and it will be used to characterise the classes of pseudo-differential operators on homogeneous vector bundles in the sense of H\"ormander, which are defined by local coordinate systems (for more details see the classical book of H\"ormander  \cite{Hormander1985III}). For the case of $(\rho,\delta)$-H\"ormander classes of pseudo-differential operators on compact Lie groups this characterisation (with the usual restriction $1-\rho\leq \delta$ and $\delta<\rho$ for the invariance of the classes under coordinate changes) was obtained in \cite{RWT}, with new classes in the case $1-\rho\leq \delta, $ allowing the construction of parametrices in the case $\rho=\delta,$ as we will discuss later. The standard restriction $1-\rho\leq \delta$ will also be removed in our context.

To explain the relevance of introducing a global calculus for pseudo-differential operators in our context, let us start by describing the following problem, which  has been extensively studied but remains open for several geometric structures, and arising also in several problems related with mathematical-physics (see e.g.  the central point of discussion in the works of  Fulling: \cite[Page 599]{Fullying96} and \cite{Fullying98}). 
\begin{itemize}
    \item  {(Q0): \it{How to provide on a $C^\infty$-manifold $M$ (or over two vector bundles over $M$) a quantisation process where the formalism for the symbolic calculus of pseudo-differential operators will  be geometrically covariant (or defined in an intrinsic way)?}  }
\end{itemize}
Indeed, from the physical point of view (and considering the case of the flat space $\mathbb{R}^n$), a pseudo-differential operator is nothing else than an object that, in quantum mechanics, is a function of position ${\bf x}$ and momentum, ${\bf P}:=-i\hbar \nabla_{{\bf x}}.$\footnote{ $\nabla_{{\bf x}}$ stands for the standard gradient on $\mathbb{R}^n.$ In SI units, the Planck constant is given by $\hbar=6.62607015\times 10^{-34}  $ $J\cdot s.$} In this formalism, the quantisation of the coordinate functions ${\bf x}_j$ and ${\bf p}_j,$ are the multiplication operator by ${\bf x}_j$ and the derivation operator ${\bf P}_j:=-i\hbar \partial_{{\bf x}_j}$ respectively.   A polynomial function of the  momentum operator admitting variable coefficients is just a partial differential operator, but, what operator corresponds to a general smooth function ${a({\bf x},{\bf p})}?$ The answer is given  by the conventional (Kohn-Nirenberg) quantisation\footnote{Let us remark that the first general quantisation procedure (in pseudo-differential operators theory)  was proposed by Weyl (see \cite{Hormander1985III}) not long after the invention of quantum mechanics.}
\begin{equation}
     a({\bf x},{\bf P})\psi(x):=(2\pi \hbar )^{-\frac{n}{2}}\int\limits_{\mathbb{R}^n}e^{i {\bf p} \cdot {\bf x}/ \hbar}a({\bf x},{\bf p})\widehat{\psi}({\bf p})d{\bf p},\,\,\,\,\,\,\,\,\,\,\,\,
\end{equation}where $\widehat{\psi}({\bf p})$ is the Fourier transform of $\psi$ defined by
\begin{equation*}
  \widehat{\psi}({\bf p}):=  (2\pi \hbar )^{-\frac{n}{2}}\int\limits_{\mathbb{R}^n}e^{-i {\bf p} \cdot {\bf x}/ \hbar}\psi({\bf x})d{\bf x}.
\end{equation*}In this case the function $a({\bf x},{\bf p})$ is defined on the phase space ${T}^{*}\mathbb{R}^n,$ that is the cotangent bundle of $\mathbb{R}^n.$ However, if we want to quantise on a Riemannian manifold $M,$ symbols are functions of   $({\bf x},{\bf p})$ in the phase space $T^{*}M,$ and quantisations are symplectically invariant only on the level of the so-called principal symbol.  This is not a problem for many applications of pseudo-differential operators, which are quite rough and qualitative (e.g. in  index theorems of Atiyah-Singer type \cite{AS} and other geometrical invariants \cite{BB}). However, for other kind of problems of interest in physics (like the computation of the resolvent/parametrix of elliptic operators \cite{FullyingKennedy87,FullyingKennedy88},  quantising symbols arising in elasticity theory \cite{elasticity}, and in computing parametrices/fundamental solutions of geometrical operators in the construction of certain quantum gravity theories \cite[Section III]{Christensen1,Christensen2}) this approach has  computational costs of high complexity (see \cite[Page 599]{Fullying96} for details). Indeed, the complexity is based on the Fourier analysis in local coordinate systems. The results that it begets are computed in terms of nontensorial functions and covariant derivatives. So, it is necessary to rewrite these results by using intrinsic objects in terms of covariant derivatives and tensors, or in other more easier objects (see e.g.  \cite{Christensen1,Christensen2}). So, for our context, we will present an alternative way of constructing a symbolic calculus over homogeneous vector bundles--which uses both-- the generality of these geometrical structures
and an intrinsic formalism, in this case given by the representation theory associated to them.  

Several answers for (Q0) have been obtained  on compact Riemanian manifolds in the presence of a connection (see e.g. Widom \cite{Widom1978,Widom1980}, Drager \cite{DragerThesis}, and  Safarov \cite{Safarov1997}), some of these works indeed include  replacements for the Weyl quantisation. However, the idea of a global quantisation procedure using the Fourier analysis on general locally compact Lie groups can be traced back to Taylor \cite{Taylorbook1986}. This is the approach under which the theory developed in \cite{Ruz} was constructed.

We will continue this introduction with the historical aspects of the theory of pseudo-differential operators  related with this work and we finish it with our main results and their applications. 

{\it This work is dedicated to the memory of Isadore Singer, who together with Atiyah, Calder\'on, Seeley,  Kohn, Nirenberg, H\"ormander and Mihlin were predecessors and/or fathers of the theory of pseudo-differential operators.}

\subsection{Historical perspective} On the Euclidean space $\mathbb{R}^n,$ every continuous linear operator $A:\mathscr{S}(\mathbb{R}^n)\rightarrow \mathscr{S}'(\mathbb{R}^n),$ admits an integral representation of the form (by adopting unities for ${\bf x}:=x$ and ${\bf p}:=\xi$ where $\hbar=\frac{1}{2\pi}$\footnote{and $D=(-i/2\pi) \nabla_{x}.$})
\begin{align}\label{pseudodifferentialoperator}
   \boxed{    Af(x)\equiv  a(x,D)f(x)=\int\limits_{\mathbb{R}^n}e^{2\pi i x\cdot\xi}a(x,\xi)\widehat{f}(\xi)d\xi,\quad f\in \mathscr{S}(\mathbb{R}^n)}
\end{align}in terms of the Fourier transform $\widehat{f}$ of $f,$ and of a distribution $a\in \mathscr{S}'(\mathbb{R}^n\times \mathbb{R}^n ), $ encoding the analytical and/or spectral  properties of $A.$ Let us remark that, in the sense of distributions, $a(x,\xi)=e_{-\xi}(x)(Ae_{\xi})(x),$\footnote{where the homomorphisms $e_{\xi}:\mathbb{R}^n\rightarrow\mathbb{T}^n\cong \mathbb{R}^n/\mathbb{Z}^n, $ are defined by $e_{\xi}(x):=e^{i2\pi x\xi}.$ These are the elements of the   unitary dual of $\mathbb{R}^n.$} and that  the relation
\begin{align*}
    \textnormal{properties  of the distribution }a\,\,\,\Longleftrightarrow \,\,\,\textnormal{properties  of the operator }A,
\end{align*}is the central problem in the theory of pseudo-differential operators. In the terminology used in quantum mechanics, $A$ is the quantisation of the function/distribution $a,$ and in the usual terminology of modern mathematics, $A$ is the so called, {\it{pseudo-differential operator}} associated with the {\it{symbol}} $a.$ The pioneering work of J. J. Kohn and L. Nirenberg \cite{KohnNirenberg1965} introduced an algebra of pseudo-differential operators $\Psi^{\infty}(\mathbb{R}^n\times \mathbb{R}^n ):=\bigcup_{m\in \mathbb{R}}\Psi^{m}(\mathbb{R}^n\times \mathbb{R}^n ), $ where every (Kohn-Nirenberg) class $\Psi^{m}(\mathbb{R}^n\times \mathbb{R}^n )$ is defined by the pseudo-differential operators associated with those symbols $a\in S^m(\mathbb{R}^n\times \mathbb{R}^n ),$ i.e. the family of distributions $a\in \mathscr{S}'(\mathbb{R}^n\times \mathbb{R}^n ),$ satisfying estimates of the kind,
\begin{equation}
    |\partial_\xi^\alpha\partial_x^\beta a(x,\xi)|\lesssim(1+|\xi|)^{m-|\alpha|},\,\,\,\,|\xi|\rightarrow\infty.
\end{equation}The resultant algebra $\Psi^{\infty}(\mathbb{R}^n\times \mathbb{R}^n )$  is stable under adjoints and  parametrices (inverses modulo a compact operator). They are very useful in the study of existence and uniqueness of elliptic differential equations. Few years after the Kohn-Nirenberg algebra of  pseudo-differential operators was known, a substantial generalisation of \cite{KohnNirenberg1965} was introduced by L. H\"ormander in \cite{Hormander1967} with the goal of studying hypoelliptic equations (and in particular, in constructing a parametrix for the heat operator). For this, the H\"ormander classes of pseudo-differential operators $\Psi^{\infty}_{\rho,\delta}(\mathbb{R}^n\times \mathbb{R}^n ):=\bigcup_{m\in \mathbb{R}}\Psi^{m}_{\rho,\delta}(\mathbb{R}^n\times \mathbb{R}^n ), $ $\Psi^{m}_{\rho,\delta}(\mathbb{R}^n\times \mathbb{R}^n )$ is defined by the pseudo-differential operators associated with  symbols $a\in S^m_{\rho,\delta}(\mathbb{R}^n\times \mathbb{R}^n ),$ determined by those distributions $a\in \mathscr{S}'(\mathbb{R}^n\times \mathbb{R}^n ),$ with (distributional) derivatives satisfying   estimates of the kind,
\begin{equation}\label{rhodeltadefiniti0n}
    |\partial_\xi^\alpha\partial_x^\beta a(x,\xi)|\lesssim (1+|\xi|)^{m-\rho|\alpha|+\delta|\beta|},\,\,\,\,|\xi|\rightarrow\infty,
\end{equation} with the condition $\delta<\rho,$ (and $0\leq \delta,\rho\leq 1$) allowing the algebra $\Psi^{\infty}_{\rho,\delta}(\mathbb{R}^n\times \mathbb{R}^n )$ to be closed under adjoints and parametrices. Observe that  $S^m(\mathbb{R}^n\times \mathbb{R}^n )=S^m_{1,0}(\mathbb{R}^n\times \mathbb{R}^n )$ and clearly $\Psi^{m}(\mathbb{R}^n\times \mathbb{R}^n )=\Psi^{m}_{1,0}(\mathbb{R}^n\times \mathbb{R}^n ).$  On every open subset of $\mathbb{R}^n,$ the class $S^m_{\rho,\delta}(T^{*}U),$  $T^{*}U\cong U\times \mathbb{R}^n,$ is determined by those symbols $a\in \mathscr{S}'(U\times \mathbb{R}^n ) ,$ satisfying estimates of the type \eqref{rhodeltadefiniti0n} for all $x$ in every compact subset $K$ of $U.$ These classes of symbols determine the class of pseudo-differential operators on $U,$ $\Psi^m_{\rho,\delta}(U\times \mathbb{R}^n), $ which also induces the H\"ormander classes on a closed manifold $M.$ Indeed, a continuous linear operator $A: C^\infty_0(M)\rightarrow\mathscr{D}'(M),$ belongs to the H\"ormander class $$\Psi^m_{\rho,\delta}(M;\textnormal{loc}), $$ 
if in every system of coordinates of $M,$ (in every chart $U$ of $M$), $A,$ admits a representation of the form \eqref{pseudodifferentialoperator} (with a symbol $a\in S^m_{\rho,\delta}(T^{*}U) $). The resultant algebra $$\Psi^{\infty}_{\rho,\delta}(M;\textnormal{loc}):=\bigcup_{m\in \mathbb{R}}\Psi^{m}_{\rho,\delta}(M;\textnormal{loc}), $$
is closed under parametrices and  adjoints, but, the conditions $\delta<\rho,$ and $\rho\geq 1-\delta,$ (which implies $\rho>1/2$) are necessary in order that the  classes $\Psi^{m}_{\rho,\delta}(M;\textnormal{loc})$ will be invariant under changes of coordinates. 

In this context (see H\"ormander \cite{Hormander1985III}), to every operator $A\in \Psi^{m}_{\rho,\delta}(M;\textnormal{loc}) $ one can associate a (principal) symbol $a\in S^m_{\rho,\delta}(T^*M),$ which is a section of the cotangent bundle $T^*M,$ which is uniquely determined, only as an element of the quotient algebra $S^m_{\rho,\delta}(T^*M)/S^{m-1}_{\rho,\delta}(T^*M).$ Similar constructions can be done on   vector bundles  on every closed manifold $M.$ The algebra of pseudo-differential operators in this case is denoted by $$\Psi^\infty_{\rho,\delta}(E,F;\textnormal{loc}):=\bigcup_{m\in \mathbb{R}}\Psi^m_{\rho,\delta}(E,F;\textnormal{loc})$$ and the corresponding class of symbols is denoted as $$S^{m}_{\rho,\delta}(E,F;\textnormal{loc}):=S^m_{\rho,\delta}(T^*M, \Gamma^\infty(E), \Gamma^\infty(F))\footnote{Here, for vector bundles $E$ and $F$  on $M,$  $\Gamma^\infty(E)$  and $ \Gamma^\infty(F)$ are the corresponding spaces of smooth sections.},$$ with subsequent works appearing also in the case of vector bundles on a compact manifold with boundary $\partial M,$ on manifolds with corners, grupoids, etc. (see \cite{Hormander1985III} and also the works \cite{CSS,Sch86,Sch88,Loya2001,Loya2003,Loya2003a}  and references therein).

Based on the local construction of the principal symbol in the case of vector bundles on $C^\infty$-manifolds, the following question arises in the theory:
\begin{itemize}
    \item {(Q1): \it{When is it possible to define a notion of a global symbol (this means, without using local coordinate systems) allowing a global quantisation formula for the H\"ormander class  $\Psi^m_{\rho,\delta}(E,F;\textnormal{loc}),$ whose uniqueness is determined by the algebra}} $ S^{m}_{\rho,\delta}(E,F;\textnormal{loc})$ {\it{ but, not  by the quotient algebra $$S^{m}_{\rho,\delta}(E,F;\textnormal{loc})/S^{m-1}_{\rho,\delta}(E,F;\textnormal{loc})?$$}} 
\end{itemize}
This question is equivalent to (Q0) and it could have, certainly, several answers, several of them depending of the kind of applications for which one is thinking about. For instance, in the presence of a linear connection $\Gamma$ of  a $C^\infty$-manifold $M$, it was shown by Safarov \cite{Safarov1997} that it is possible  to construct a calculus of pseudo-differential operators without using  local coordinate systems. Indeed, by using  the fact that a linear connection $\Gamma$ is a global object defined on $M,$ one can define the notion of a global symbol inducing a global calculus of pseudo-differential operators, but, depending of choice of  the connection $\Gamma$ if $\rho<1-\delta$.

As it was pointed out in \cite{Safarov1997},  this idea was put forward by H. Widom \cite{Widom1978,Widom1980}. The Safarov classes $S^m_{\rho,\delta}(\Gamma):=S^m_{\rho,\delta}(\Gamma,M)$ in  \cite{Safarov1997} induce classes of pseudo-differential operators denoted by $\Psi^m_{\rho,\delta}(\Gamma):=\Psi^m_{\rho,\delta}(\Gamma,M)$  which agree with the H\"ormander classes $\Psi^m_{\rho,\delta}(M,\textnormal{loc})$ in the case  $\rho\geq 1-\delta,$ and they are in this case, independent of the choice of the connection $\Gamma.$ In the case of the Levi-Civita connection $\Gamma_0,$  a global functional calculus for operators of the form $\Delta+\nu,$ where $\Delta$ is the positive Laplace-Beltrami operator on $M,$  and $\nu$ is a suitable first order differential operator, was developed  in \cite{Safarov1997}. Certainly, the Safarov calculus has potential applications to geometric invariants computed from zeta-functions of operators, Wodzicki residues, canonical traces, etc.

On the other hand, if one is thinking of answering (Q1) in the case of smooth structures with symmetries (Lie groups $M=G$ and homogeneous spaces $M=G/K,$ where $K$ is a closed subgroup of $G$), it was shown in \cite{Ruz,RuzhanskyTurunenIMRN} by the third author and V. Turunen that defining symbols on the phase space $G\times \widehat{G},$ with $\widehat{G}$ being the unitary dual of $G,$ and with providing to $\widehat{G}$ a (discrete differential structure) difference structure is enough for describing the H\"ormander classes $\Psi^m_{\rho,\delta}(G, \textnormal{loc}).$ 
The global H\"ormander classes in \cite{Ruz,RuzhanskyTurunenIMRN} denoted by $S^m_{\rho,\delta}(G\times \widehat{G}),$ $m\in \mathbb{R},$ induce classes of pseudo-differential operators denoted by  $\Psi^m_{\rho,\delta}(G\times \widehat{G}),$ and they satisfy
\begin{equation}
    \Psi^m_{\rho,\delta}(G\times \widehat{G})=\Psi^m_{\rho,\delta}(G,\textnormal{loc})=\Psi^m_{\rho,\delta}(\Gamma,G),
\end{equation}for any linear connection $\Gamma$ of $G$ for $\rho\geq 1-\delta$ and $\delta<\rho.$ It is a very special situation in this theory that the case $\rho=\delta$ is allowed and it implies that the theory begets the construction of parametrices of the heat operator, which with the usual H\"ormander classes on arbitrary manifolds  is impossible. Moreover, it also allows the construction of parametrices for vector fields of the form $X+c,$ $c\in i\mathbb{R},$ where the usual H\"ormander classes are inoperable.

We observe that the main tool  in the theory developed in \cite{Ruz,RuzhanskyTurunenIMRN} is the following quantisation formula for a distribution $a\in S^m_{\rho,\delta}(G\times \widehat{G}),$ given by
\begin{equation}\label{QuantisationofRuz-Tur}   \boxed{ 
    Af(x)=\sum_{[\xi]\in \widehat{G}}d_{\xi}\textnormal{Tr}[\xi(x)\sigma_{A}(x,\xi)    \widehat{f}(\xi)   ]}
\end{equation}where $a\equiv\sigma_A:   G \times \widehat{G} \rightarrow \bigcup \{\mathbb{C}^{\ell \times \ell}: \ell \in \mathbb{N}\}$ is the  (global) symbol of $A,$ and $ \widehat{f}$ is the Fourier transform of $f$ on the group $G.$ 

 The approach   in \cite{Ruz,RuzhanskyTurunenIMRN} uses both, the representation theory  and the Fourier analysis on  a compact Lie group $G,$ with effective applications to, for example, G\r{a}rding inequalities, global well posedness for elliptic and hypoelliptic problems, mapping properties in $L^p,$ Besov and Sobolev spaces, Schatten properties, nuclearity of pseudo-differential operators, global functional calculus, spectral asymptotics,  and more recently in \cite{RuzhanskyCardona2020} a global description of H\"ormander classes in the presence of a sub-Riemannian structure.  We also note that on spin groups, a global pseudo-differential calculus has been worked out explicitly by  Cerejeiras, Ferreira,  K\"ahler, and  Wirth in \cite{CerejeirasFerreiraKahlerWirth}. As for the explicit construction of differences operators for defining the global H\"ormander classes on $\textnormal{SU}(2),$ as well as the construction of these classes on the sphere we refer to \cite{RuzhanskyTurunenIMRN}. 

The aim of this paper is to give an answer for (Q1) in the case of two homogeneous vector bundles $p_1:E\rightarrow M,$ and $p_2:F\rightarrow M,$  over a compact homogeneous manifold $M=G/K,$ with applications to G\r{a}rding inequalities, wellposedness for subelliptic problems, $L^2$-estimates in the form of   Calder\'on-Vaillancourt type estimates, by constructing an algebra of pseudo-differential operators absorbing the H\"ormander classes $\Psi^m_{\rho,\delta}(E,F;\textnormal{loc})$ for the vector bundles considered  in this context.

Examples of diffeomorphisms between Lie groups and some smooth manifolds (for example $\textnormal{SU}(2)\cong \mathbb{S}^3$ and $M\cong \mathbb{S}^3$ for any connected and simply connected 3-dimensional smooth manifold in view of the Perelman Theorem for the Poincar\'e conjecture), motivates the fact of constructing a global quantisation procedure for homogeneous vector bundles over compact homogeneous manifolds. We refer the reader to Section \ref{vectcal}, where, in the context described in this paragraph,  (Q1) is answered successfully without using local coordinate systems. 

\subsection{Outline of the work and main results}
 The problem of finding  a pseudo-differential calculus for continuous linear operators  $A:C^{\infty}(M)\rightarrow \mathscr{D}'(M),$ where $M=G/K$ is an orientable compact homogeneous manifold, started in the book of the third author with Turunen \cite{Ruz}, where it was shown that an averaging process for continuous linear operators in $\Psi^m_{1,0}(G)$ begets pseudo-differential operators in the class $\Psi^m_{1,0}(G/K).$ We also refer the reader to the thesis of Turunen \cite{Turunen:Thesis}.  Also, further subsequent developments were done in the thesis of D. Connolly \cite{ThesisConnolly}.  
 
 In this work, we consider the problem of solving  (Q1) in the case of two homogeneous vector-bundles $E$ and $F$  on $M,$ by following the suggested approach  by  M. Wilson \cite{Wilson2019}, where it was pointed out that one can consider the classical  construction of homogeneous vector bundles due to Raoul Bott \cite{Bott1957,Bott1965} in order to provide a global quantisation procedure on homogeneous vector bundles. 
 
 Our main results can be summarised as follows. Here, $\mathcal{L}$ denotes the sub-Laplacian  on a compact Lie group  $G$ associated to a  system of vector fields $X=\{X_1,\cdots,X_k\},$ satisfying the H\"ormander condition of order $\kappa,$ and  $\mathcal{L}_G$ denotes the Laplace-Beltrami operator on $G.$ 
 \begin{itemize}
     \item Given two homogeneous vector bundles $p_1:E\rightarrow M,$ and $p_2:F\rightarrow M,$  over an orientable compact homogeneous manifold $M=G/K,$ in Section \ref{vectcal}, we introduce for every $m\in \mathbb{R},$ and for $0\leq \delta,\rho\leq 1,$ a  class of pseudo-differential operators $\Psi^{m,\mathcal{L}}_{\rho,\delta}(E,F),$ such that for $\delta<\rho,$ the class of operators $$\Psi^{\infty, \mathcal{L}}_{\rho,\delta}(E,F):=\bigcup_{m\in \mathbb{R}}\Psi^{m,\mathcal{L}}_{\rho,\delta}(E,F),$$ is stable under compositions (this means that $\Psi^{\infty, \mathcal{L}}_{\rho,\delta}(E,F)$ is an algebra), under adjoints, under parametrices and also stable under the action of the complex functional calculus. The main tool is the quantisation formula, obtained in Theorem \ref{mainqhomovec}, which associates to every continuous linear operator $\tilde{A}:\Gamma^\infty(E)\rightarrow\Gamma^\infty(F),$ a unique global symbol $\sigma_{\tilde{A}}$ such that
     \begin{equation}\label{QuantisationIntro}   \boxed{ 
    \tilde{A}u(gK)= \left(g, \,\sum_{[\xi]\in \widehat{G}}\sum_{i=1}^{d_\tau}\sum_{r=1}^{d_\omega}\textnormal{Tr}[\xi(g)\sigma_{\tilde{A}}(i,r,g,[\xi])  \widehat{\varkappa_\tau u}(i,[\xi])]e_{r,F_0}   \right)\cdot K}
\end{equation}for every section $u\in \Gamma^\infty(E),$ where  $d_\tau $ and $ d_\omega$ are dimensions associated to $E$ and $  F,$ respectively, and $\varkappa_\tau$ is the unitary isomorphism in \eqref{varkappatau}.        
     \item We prove the Calder\'on-Vaillancourt Theorem, which says that every operator in the class  $\Psi^{0,\mathcal{L}}_{\rho,\delta}(E,F),$ with $0\leq \delta< \rho\leq    1,$ (or $0\leq \delta\leq \rho\leq 1,$ $\delta<1/\kappa$) admits a bounded extension from $L^2(E)$ to $L^2(F).$ The $L^p$-boundedness of subelliptic pseudo-differential operators also will be established (see Theorem \ref{Lpthe:VB}) extending the classical $L^p$-estimate  of Fefferman \cite{Fefferman1973} to this setting.
     \item We prove the G\r{a}rding inequality, which establishes a lower bound of the form 
     \begin{align*}
    \textnormal{Re}(\tilde{A}u,u) \geqslant  C_1\Vert u\Vert^2_{{L}^{2,\mathcal{L}}_{\frac{m}{2}}(E)}-C_2\Vert u\Vert_{L^2(E)}^2,
\end{align*}  for $0\leq \delta<\rho\leq 1$, and for any section $u\in \Gamma^\infty(E),$ under the condition that the global symbol of $\tilde{A}$ is strongly elliptic. For details see Subsection \ref{Gardingsection} and also for the applications of this G\r{a}rding inequality to the global wellposedness of subelliptic pseudo-differential problems on homogeneous vector-bundles. 
     \item One special feature of our theory is the following. If $X$ is taken to be an orthonormal basis of the Lie algebra $\mathfrak{g}$ of $G,$ then $\mathcal{L}$ is replaced by the Laplace-Beltrami operator $\mathcal{L}_G$ and the class  $\Psi^{m}_{\rho,\delta}(E,F):=\Psi^{m,\mathcal{L}_G}_{\rho,\delta}(E,F)$ agrees with the H\"ormander class $\Psi^{m}_{\rho,\delta}(E,F;\textnormal{loc})$ defined by local coordinate systems if, additionally, we have the standard conditions $\rho\geq 1-\delta,$ $\delta<\rho,$ necessary for the invariance of coordinates. Our technique also works in the case when $\rho\leq 1-\delta$ and $\delta<\rho$ providing  new classes of pseudo-differential operators.
 \end{itemize}
 Now, we present some remarks where we briefly discuss several important special cases.
\begin{remark}  By considering the trivial vector bundles $E=F=M\times \mathbb{C},$ \eqref{QuantisationIntro} provides a quantisation formula in the context of continuous linear operators on $C^\infty(M).$ In this case the subelliptic H\"ormander  classes will be denoted by  $\Psi^{m,\mathcal{L}}_{\rho,\delta}(M):=\Psi^{m,\mathcal{L}}_{\rho,\delta}(E,F).$  If we choose   $X$ to be  an orthonormal basis of the Lie algebra $\mathfrak{g}$ of $G,$ and taking the Laplace-Beltrami operator $\mathcal{L}_G$ instead of $\mathcal{L},$  the class $\Psi^{m}_{\rho,\delta}(M):=\Psi^{m,\mathcal{L}_G}_{\rho,\delta}(M)$ agrees with the H\"ormander class $\Psi^{m}_{\rho,\delta}(M;\textnormal{loc})$ defined by localisations, by imposing the standard conditions $0\leq \delta<\rho\leq 1,$ and $\rho\geq 1-\delta$. 

\end{remark}
\begin{remark}
If we consider the trivial subgroup $K=\{e_G\}$ with $e_G$ being the identity element of $G,$ then $M=G$ and the subelliptic H\"ormander  classes   $\Psi^{m,\mathcal{L}}_{\rho,\delta}(M)=\Psi^{m,\mathcal{L}}_{\rho,\delta}(G),$ with $m\in \mathbb{R},$ are precisely those classes defined by the first and third author in \cite{RuzhanskyCardona2020}, which were introduced  with the goal of providing a sharp treatment for subelliptic problems. So, although the main problem in this paper is to solve (Q1), the developed classes here for homogeneous vector bundles are defined in the presence of a sub-Riemannian structure determined by  the sub-Laplacian $\mathcal{L}$  associated to a  system of vector fields $X=\{X_1,\cdots,X_k\},$ satisfying the H\"ormander condition. Again, if we replace $X$ by  an orthonormal basis of the Lie algebra $\mathfrak{g}$ of $G,$ and taking the Laplace-Beltrami operator $\mathcal{L}_G$ instead of $\mathcal{L},$  the global H\"ormander classes $\Psi^{m}_{\rho,\delta}(G):=\Psi^{m,\mathcal{L}_G}_{\rho,\delta}(M)$ are the ones introduced by the third author and Turunen in \cite{Ruz}.
\end{remark}
\begin{remark}
If we consider the case of the homogeneous vector bundles, $E=\Omega^{p}(M)\rightarrow M,$ and  $F=\Omega^{k}(M)\rightarrow M,$ that are, the vector bundle  of $p$-differential forms and  $k$-differential forms on $M,$ respectively, the classes $\Psi^{m,\mathcal{L}}_{\rho,\delta}(\Omega^{p}(M),\Omega^{k}(M))$ provide a global pseudo-differential calculus. In the case of $\Psi^{m,\mathcal{L}_G}_{\rho,\delta}(\Omega^{p}(M),\Omega^{k}(M))$ with $\rho\geq 1-\delta$ and $\delta<\rho$ we recover the H\"ormander classes defined by local coordinates. Using our further  formalism, we will compute the symbol of the exterior derivative, its adjoint, and the symbol of the Dirac operator (see Section  \ref{Section:Forms} for details). 
\end{remark}
We organise this work as follows:
\begin{itemize}
    \item In Section \ref{Preliminaries} we present some basics on homogeneous vector bundles, in particular, the construction of them due to Raoul Bott. We discuss some aspects of the Fourier analysis on homogeneous vector bundles and we record the basics on  pseudo-differential operators. We also introduce an auxiliary quantisation formula for vector-valued pseudo-differential operators on compact Lie groups. 
    \item In Section \ref{vectcal}, the pseudo-differential calculus for the vector-valued setting will be established. Then, after developing our quantisation formula, the vector-valued calculus will be used to obtain an equivalent calculus for homogeneous vector bundles. Consequently, we classify the H\"ormander classes of pseudo-differential operators on homogeneous vector bundles in terms of their global symbols. The classical Fefferman  $L^p$-theorem and the Calder\'on-Vaillancourt theorem will be extended to this setting.
    \item Section \ref{GFCSECT} is dedicated to construct the parametrices of the calculus, a global  functional calculus, to study the Fredholness of pseudo-differential operators  and to establish the global version of the G\r{a}rding inequality. Applications of the G\r{a}rding inequality on homogeneous vector-bundles to the wellposedness of  evolution problems associated with subelliptic operators will be considered.
    \item In Section \ref{examplecompactmanifolds} we establish the subelliptic calculus for pseudo-differential operators on compact homogeneous spaces when considering the case of the trivial vector-bundles, extending the theory in \cite{Ruz}.
    \item Finally, in Section \ref{Section:Forms}  we compute the matrix-valued symbol of the exterior derivative, the global symbol of its adjoint, and the global symbol of the Dirac operator when the homogeneous vector bundles are taken to be  vector-bundles of differential forms.
\end{itemize}

Throughout the paper, we shall use the notation $A \lesssim B$ to indicate $A\leq cB $ for a suitable constant $c >0$, whereas $A \asymp B$ if $A\leq cB$ and $B\leq d A$, for suitable $c, d >0.$

\section{Preliminaries}\label{Preliminaries} 
In this section, we present some basics on harmonic analysis on homogeneous vector bundles over compact homogeneous spaces. We will use the following notations:

- For  finite dimensional vector spaces $V$ and $W,$ we denote by $B_{V}$ the basis of $V,$ and by  $\mathbb{C}^{\ell\times \ell}(\textnormal{End}(V,W))$  the space of all matrices of size $\ell\times \ell$ with entries in $\textnormal{End}(V,W),$ the space of all endomorphisms from $V$ to $W,$ (the space of linear mappings from $V$ to $W$). 

- We will make use of the isomorphism $\textnormal{End}(V,W)\cong V^*\otimes W,$ and an orthonormal basis of $V$  or of $W$ will be denoted by $\{e_{i,V}\}_{i=1}^{\textnormal{dim}(V)}$ and  $\{e_{i,W}\}_{i=1}^{\textnormal{dim}(W)}$ respectively. For a vector space $V,$ endowed with an inner product, $\mathfrak{B}_{V}$ denotes the collection of orthonormal bases of $V.$

- For Banach spaces $X$ and $Y,$ $\mathscr{B}(X,Y)$ (or $\mathscr{B}(X)$ if $X=Y$), denotes the space of all bounded linear operators from $X$ into $Y.$ If $Z$ and $W$  have structure of Fr\'echet spaces, then $\mathscr{L}(Z,W)$ (or $\mathscr{L}(Z)$ if $Z=W$), denotes the family of continuous linear operators from $Z$ into $W.$

- The compact homogeneous manifold $M=G/K$ will be considered orientable.

\subsection{Homogeneous vector bundles over compact homogeneous manifolds}  Now we record some basics on vector bundles by following \cite[Chapter 1]{Wallach1973}.
Let $M$ be real (or complex) $C^\infty$-manifold. A real  (or complex) $C^\infty$-vector bundle  is a triple $(p,E,M),$ where $E$ is a $C^\infty$-manifold and $p:E\rightarrow M$ is a smooth surjective mapping from $E$ to $M$ satisfying:
\begin{itemize}
    \item[(i)] For every $x \in M,$ the set $E_x= p^{-1}(x)$ is a real (or complex) vector space.
    \item[(ii)] For each $x \in M$ there exist a neighbourhood  $U$ of $x$ and a homeomorphism $\phi:E_{U}:=p^{-1}(U)\rightarrow U\times \mathbb{R}^n$ (or $\phi:E_{U}:=p^{-1}(U)\rightarrow U\times \mathbb{C}^n$ in the complex case), such that $\phi(v)=(p(v),f(v)),$ and $f:E_{p(v)}\rightarrow \mathbb{R}^n$ (or $f:E_{p(v)}\rightarrow \mathbb{C}^n$) is a linear mapping. We summarise property $(\textnormal{ii})$ saying that  $E$ is trivial in $U.$
    \item[(iii)] There exists a trivializing covering $\mathfrak{U}=\{U_{i}\}_{i\in I}$ of $M.$ This means that for every $x\in M,$ there exists $U_{i},$ such that $x\in U_{i}\subset M,$ with $E$ being trivial in $U_{i}.$
\end{itemize} 
Following the standard terminology we say that $E$ is the total space, $M$ is the base space and  $p: E \rightarrow M$ is the projection from $E$ to $M$. Therefore, to define the $C^\infty$-vector bundle (in short, vector bundle) $(p, E,M)$, it is enough to specify the projection  $p:E\rightarrow M$ only. 

As usual, a section $s$ on $E$ is defined by the identity $p\circ s=\textnormal{id}_{M}.$ Throughout this paper we will use $\Gamma(E)$ for the space of all sections. For any integer $k\geq 1,$ the set $\Gamma^{k}(E)$ denotes the set of sections of class $C^{k},$ namely the family of differentiable sections of order $k,$ such that for any arbitrary choice of vector fields $X_1,\cdots X_n,$ $n=\dim (M),$ $$X_{1}^{\alpha_1}\cdots X_{n}^{\alpha_n} s $$ is a continuous section for any arbitray multi-index $\alpha=(\alpha_1,\cdots,\alpha_n)\in \mathbb{N}_0^n$ with $|\alpha|=k,$ where $|\alpha|=\alpha_1+\cdots +\alpha_n.$ Note that with $k=0,$ we define $\Gamma^{0}(E)=\Gamma(E)$ and  we use the notation $$\Gamma^\infty(E)=\bigcap_{k\geq 0}\Gamma^{k}(E)$$ for the space of all smooth sections $s:M\rightarrow E.$

We are interested in those vector bundles for which the base space is an orientable homogeneous manifold $M=G/K,$ where $G$ is a  Lie group and $K$ is a closed subgroup of $G.$ In particular, we will study homogeneous vector bundles. 
\begin{definition}\label{defi:vect.bund}
 We say that a vector bundle $p:E\rightarrow M$ is homogeneous (see \cite[page 114]{Wallach1973}), if the following two conditions are satisfied,
\begin{itemize}
    \item[(i)] $G$ acts from the left on $E$ and for every $g\in G,$ $g\cdot E_{x}=E_{g\cdot x},$ $x\in M.$ 
    \item[(ii)] For every  $g \in G,$ the induced map from $E_{g}$ into $E_{g\cdot x},$ is linear. 
\end{itemize}
\end{definition}

In this next example, we will provide an universal construction for homogeneous vector bundles (see \cite[page 115]{Wallach1973}). 
\begin{example}[{\bf Bott's construction}]\label{BottConstruction} Let  $p_{M}:G\rightarrow M=G/K$ be the natural projection. Let $\tau$ be a finite-dimensional unitary representation of $K$ with the representation space $E_0.$ Let us consider the vector bundle $p:E\rightarrow M=G/K,$ with $E=G\times_{\tau} E_{0},$ that is, the semi-direct product between $G$ and $E_{0}$ with respect to $\tau: K\rightarrow \textnormal{GL}(E_0),$ where  $E_0$ is the representation space of $K$ associated with $\tau.$ In view of the compactness of $K,$ we can assume that $E_0=\mathbb{C}^{d_\tau}$    is the complex vector space of finite dimension $d_\tau$. We recall that  the semi-direct product $G\times_{\tau}E_0$ is the set of all cosets $(g,v)\cdot K=:[g,v],$ $g\in G,$ $v\in E_{0},$ defined by the right action of $K$ on $G\times E_{0}, $ $$(g,v)\cdot k=(gk,\tau(k)^{-1}v).$$ Consequently,  $(g,v)\cdot K=\{(g,v)\cdot k:k\in K\}.$ The projection $p:E\rightarrow M,$ is given by $p((g,v)\cdot K)=gK,$ and $p:E\rightarrow M=G/K,$  has a natural structure of a homogeneous vector bundle.  
This construction makes commutative the following diagram (see \cite[Page 171]{Bott1957}):
\begin{eqnarray}
    \begin{tikzpicture}[every node/.style={midway}]
  \matrix[column sep={10em,between origins}, row sep={4em}] at (0,0) {
    \node(R) {$G$}  ; & \node(S) {$G\times E_0$};  & \node(S') {$E_0$}; \\
    \node(R/I) {$M$}; & \node (T) {$G\times_\tau E_0$};& \node (T') {$\star$};\\
  };
  \draw[<-] (R/I) -- (R) node[anchor=east]  {$p_{M}$};
  \draw[<-] (R) -- (S) node[anchor=south] {$\,$};
  \draw[->] (S) -- (S') node[anchor=south] {$\,$};
   \draw[->] (S') -- (T') node[anchor=south] {$\,$};
    \draw[->] (T) -- (T') node[anchor=south] {$\,$};
  \draw[->] (S) -- (T) node[anchor=west] {$\,$};
  \draw[<-] (R/I) -- (T) node[anchor=north] {$p$};
  \end{tikzpicture}
\end{eqnarray}\end{example} 
Now, let us record that every homogeneous vector bundle can be constructed (up to isomorphism) in the way described in Example \ref{BottConstruction} (see \cite[Lemma 5.2.3]{Wallach1973}).
\begin{theorem}Let $p:E\rightarrow M=G/K$ be a homogeneous vector bundle. Let $E_0=p^{-1}(K)$ be the fiber at the identity. Then, there exists a representation $\tau:K\rightarrow \textnormal{GL}(E_0)$ of $K,$ such that $E\cong G\times_{\tau}E_0.$ 
\end{theorem}
We assume that  $M=G/K$ is a  compact and orientable manifold and so $G$ is compact and consequently $K$ is also compact. Let us fix a $G$-invariant volume element $\Omega_M$ on $M$ such that, for $f \in C(G/K),$ (see \cite[Page 116]{Wallach1973})
$$ \int\limits_G f(gK)\, dg = \int\limits_{M} f \Omega_M.$$ 
The space $L^p(E)$ is defined by the sections $s:M\rightarrow E,$ satisfying
\begin{equation}
    \Vert s\Vert_{L^p(E)}:=\left(\,\int\limits_{M}\Vert s(x) \Vert_{E_{x}}^p\Omega_M\right)^{\frac{1}{p}}<\infty,
\end{equation} for all $1\leq p<\infty,$ and for $p=\infty,$
\begin{equation}
    \Vert s\Vert_{L^\infty(E)}:=\textnormal{ess}\sup_{x\in M}\Vert s(x)\Vert_{E_{x}}<\infty.
\end{equation} The left action of the group $G$ on the space of sections $\Gamma( E)$ is given by 
\begin{equation} \label{actionsect}
     G\times \Gamma^\infty(E)\rightarrow \Gamma^\infty(E),\textnormal{ by } (g,s)\mapsto g\cdot s,\,\, (g\cdot s)(x)=  g\cdot s(g^{-1}x),
\end{equation}   
where $g \in G, s \in \Gamma (E)$ and $x \in M.$
Observe that for any $g\in G,$ $g\cdot s$ is a section. Indeed, $ s(g^{-1}x)\in E_{g^{-1}x},$ and $g\cdot s(g^{-1}x)$ denotes the action of $g$ on the vector $s(g^{-1}x)$ which belongs to  $E_{x},$ in view of (ii) in Definition \ref{defi:vect.bund}, so that for any $x\in M,$ $(g\cdot s)(x)\in E_{x}$ and $p\circ (g\cdot s)(x)=x,$ where  $p: E \rightarrow M$ is the bundle projection. 
The action of $G$ on $\Gamma (E)$ extends to a unitary representation $(\pi, L^2(E))$ of $G$ defined by $g \mapsto \pi(g)s:=g\cdot s.$  

Now, let us define
\begin{equation}\label{tauinvariantfunction}
    C(G,E_0)^\tau=\{f\in C(G, E_0)\,\, |\forall g\in G,\,\forall k\in K,\,\,f(gk)=\tau(k)^{-1}f(g) \},
\end{equation}the space  of all  $\tau$-invariant continuous functions from $G$ to $E_0$ and write   $C^{\infty}(G,E_0)^\tau$ for its subset containing all smooth elements. In a similar fashion, we define the space $L^p(G,E_0)^\tau$ consisting of all $p$-integrable $\tau$-invariant functions $f: G \rightarrow E_0.$
As in \cite[Chapter 5]{Wallach1973} we  identify $\Gamma^\infty(E)\simeq C^\infty(G,E_{0})^{\tau}.$ Indeed, in the aforementioned chapter, it is shown that the mapping $\varkappa_\tau:\Gamma^\infty(E)\rightarrow C^\infty(G,E_{0})^{\tau}, $ given by \begin{equation}\label{varkappatau}
    \varkappa_\tau(s)(g):=g^{-1}\cdot s(gK)\equiv \pi(g^{-1})(s)(e_GK) ,
\end{equation} extends to a unitary map $\varkappa_\tau:L^2(E)\rightarrow L^2(G,E_{0})^{\tau}.$ The inverse of $\varkappa_{\tau},$ $\varkappa_{\tau}^{-1}: C^\infty(G,E_{0})^{\tau}\rightarrow \Gamma^\infty(E) $ is given by
\begin{align}\label{notationfunction}
  s(gK):=  \varkappa_{\tau}^{-1}(f)(gK)=(g,f(g))\cdot K=:[g,f(g)], 
\end{align}which, for instance, is a well defined section because $p((g,f(g))\cdot K)=g K,$ for all $g\in G.$
If $f\in C(G,E_0)^\tau,$ then define $\tilde{\pi}(g_0)f(g):=f(g_0^{-1}g).$ Then, $\tilde{\pi}$ extends to a unitary representation $(\tilde{\pi}, L^2(G,E_0)^\tau)$ of $G.$ Moreover, we have (see \cite[Page 117]{Wallach1973}):
\begin{theorem}The map $\varkappa_{\tau}:\Gamma^\infty(E)\rightarrow C(G,E_0)^{\tau}$ given by $\varkappa_{\tau}s:=\pi(\cdot^{-1})s(e_GK)$ extends to a unitary equivalence of $(\pi, \Gamma^\infty(E))$ with $(\tilde{\pi}, L^2(G,E_0)^\tau).$ This means that the following diagram
\begin{eqnarray}
    \begin{tikzpicture}[every node/.style={midway}]
  \matrix[column sep={10em,between origins}, row sep={4em}] at (0,0) {
    \node(R) {$L^2(E)$}  ; & \node(S) {$L^2(E)$}; \\
    \node(R/I) {$L^2(G,E_0)^\tau$}; & \node (T) {$L^2(G,E_0)^\tau$};\\
  };
  \draw[<-] (R/I) -- (R) node[anchor=east]  {$\varkappa_\tau$};
  \draw[->] (R) -- (S) node[anchor=south] {$\pi(g)$};
  \draw[->] (S) -- (T) node[anchor=west] {$\varkappa_\tau$};
  \draw[->] (R/I) -- (T) node[anchor=north] {$\tilde{\pi}(g)$};
\end{tikzpicture}
\end{eqnarray}
commutes for all $g\in G$. So, $\varkappa_{\tau}\circ \pi=\tilde{\pi}\circ\varkappa_\tau$ or $ \pi \circ \varkappa_\tau^{-1}=\varkappa_{\tau}^{-1}\circ \tilde{\pi}.$ 

\end{theorem}

\subsection{Fourier transform on  homogeneous vector bundles}
 This section is devoted to explaining the Peter-Weyl Theorem for the space of sections $\Gamma^\infty(E)=\Gamma(G\times_{\tau}E_0).$ As in Bott's construction (Example \ref{BottConstruction}), let $(\tau, E_0 (=\mathbb{C}^{d_\tau}))$ be a representation of the closed subgroup $K$ of a compact group $G.$   For $[\xi]\in \widehat{G},$ set $$\textnormal{Hom}_{K}(\mathbb{C}^{d_\xi},E_0)=\{L:\mathbb{C}^{d_\xi}\rightarrow E_0:\,L \textnormal{  is linear and  } \forall k\in K,\,\,L\xi(k)=\tau(k)L\}.$$
To every $[\xi]\in \widehat{G},$ we assign a map $$ \alpha_{[\xi]}: \mathbb{C}^{d_\xi}\otimes \textnormal{Hom}_K(\mathbb{C}^{d_\xi},E_0)\rightarrow C(G,E_0)^{\tau}, $$
by
\begin{equation}
    \alpha_{[\xi]}(v\otimes L)(g):=L(\xi(g^{-1})v),\,\,\, \textnormal{for}\, v\in \mathbb{C}^{d_\xi},\,g\in G.
\end{equation} The Peter Weyl-Theorem   for  homogeneous vector bundles can be formulated as follows (see \cite[Theorem 5.3.6, Page 118]{Wallach1973}).
\begin{theorem}\label{PeterWeylforS}Let $[\xi]\in \widehat{G}.$ Then  the following orthogonal decompositions in subspaces of finite dimensions
\begin{align}\label{PWT}
  L^2(G,E_{0})^{\tau}&=  \bigoplus_{[\xi]\in \widehat{G}}\alpha_{[\xi]}[\mathbb{C}^{d_\xi}\otimes \textnormal{Hom}_K(\mathbb{C}^{d_\xi},E_0)],\\
  L^2(E)&=  \bigoplus_{[\xi]\in \widehat{G}}\varkappa_{\tau}^{-1}\circ\alpha_{[\xi]}[\mathbb{C}^{d_\xi}\otimes \textnormal{Hom}_K(\mathbb{C}^{d_\xi},E_0)],
\end{align} are valid assuming that $G$ is a compact Lie group.

\end{theorem}
Motivated by the orthogonal decomposition for the space of sections in \eqref{PWT} and denoting $$\Gamma_{[\xi]}(E):=\varkappa_{\tau}^{-1}\circ\alpha_{[\xi]}[\mathbb{C}^{d_\xi}\otimes \textnormal{Hom}_K(\mathbb{C}^{d_\xi},E_0)] ,$$ for every $[\xi] \in \widehat{G},$ the Fourier transform of the section  $s\in L^2(E)$ at $[\xi]\in \widehat{G},$ is defined by (see \cite[Page 119]{Wallach1973})
\begin{equation}\label{FourierTransformOfSection:**}
    \widehat{s}(\xi):=P_{[\xi]}s\equiv d_{\xi}\pi(\overline{\chi}_{[\xi]})s, \,\,
\end{equation}where $\chi_{[\xi]}:=\textnormal{Tr}(\xi)$ is the character of $[\xi],$  $P_{[\xi]}:L^2(E)\rightarrow \Gamma_{[\xi]}(E)$ is the orthogonal projection of $L^2(E)$ onto $ \Gamma_{[\xi]}(E),$ and $\pi(\overline{\chi}_{[\xi]})$ is defined by the entries
\begin{align*}
   \langle  \pi(\overline{\chi}_{[\xi]})s,s'\rangle_{L^2(E)}:=\int\limits_{G}\overline{\chi}_{[\xi]}(g)\langle \pi(g)s,  s'\rangle_{L^2(E)}dg,\,\,\,\pi(g)s(xK):=g\cdot s(g^{-1}xK),
\end{align*} where $\,x,g\in G,$ for all $s,s'\in L^2(E), $ which in the sense of Bochner implies that
\begin{equation*}
    \pi(\overline{\chi}_{[\xi]})=\int\limits_{G}\overline{\chi}_{[\xi]}(g)\pi(g)dg,\,\quad g\in G.
\end{equation*}
(The identity $P_{[\xi]}s\equiv d_{\xi}\pi(\overline{\chi}_{[\xi]})s$ is indeed Proposition 5.3.7 in  \cite[Page 119]{Wallach1973}). Let us remark, that in view of \eqref{PWT}, the Fourier inversion formula in the space of sections takes the form
\begin{align}\label{invsectfour}
    s=\sum_{[\xi]\in \widehat{G}}\widehat{s}(\xi)=\sum_{[\xi]\in \widehat{G}}d_{\xi}\pi(\overline{\chi}_{[\xi]})s=\sum_{[\xi]\in \widehat{G}}d_{\xi}\int\limits_{G}\overline{\chi}_{[\xi]}(g)\pi(g)s\,dg.
\end{align}
\begin{example}[Fourier transform of typical sections in $\Gamma_{[\xi]}(E)$] A typical element  in $\mathbb{C}^{d_\xi}\otimes \textnormal{Hom}_{K}(\mathbb{C}^{d_\xi},E_0)$ is of the form $f=v\otimes L$ where $v\in \mathbb{C}^{d_\xi}$ and $L\in \textnormal{Hom}_{K}(\mathbb{C}^{d_\xi},E_0).$ Also, for all $g_0\in G,$
$  \alpha_{[\xi]}f(g_0)=L(\xi(g_0)^{-1}v).  $ So, let us compute the Fourier transform of the section  $s=\varkappa_{\tau}^{-1}\alpha_{[\xi]}(f)\in \Gamma_{[\xi]}(E)$ as follows,
\begin{align*}
    \widehat{s}(\xi)(g_0 K)&=d_\xi \int\limits_{G}\overline{\chi_{[\xi]}(g)}\pi(g)\varkappa_{\tau}^{-1}\alpha_{[\xi]}(v\otimes L)(g_{0}K)dg\\&=d_\xi\int\limits_{G}\overline{\chi_{[\xi]}(g)}\varkappa_{\tau}^{-1}\tilde{\pi}(g)\alpha_{[\xi]}(v\otimes L)(g_0K)dg\\
    &=d_\xi\int\limits_{G}\overline{\chi_{[\xi]}(g)}\varkappa_{\tau}^{-1}\alpha_{[\xi]}(\xi(g)v\otimes L)(g_0K)dg,
\end{align*}and using that $\alpha_{[\xi]}(\xi(g)v\otimes L)(g_0K)=L(\xi(g_{0}^{-1}g)v),$ we have
\begin{align*}
\widehat{s}(\xi)(g_0 K)&=   d_\xi\int\limits_{G}\overline{\chi_{[\xi]}(g)}\varkappa_{\tau}^{-1}L(\xi(g_{0}^{-1}g)v)dg=[g_0 ,d_\xi\int\limits_{G}\overline{\chi_{[\xi]}(g)}L(\xi(g_{0}^{-1}g)v)dg] \\
&=[g_0 ,d_\xi L\left( \xi(g_{0})^{-1}\int\limits_{G}\overline{\chi_{[\xi]}(g)}\xi(g)v\right) dg ] = d_\xi [g_0 , L( \xi(g_{0})^{-1}v) ],
\end{align*}
where we have used that \begin{align*}
    \int\limits_{G}\overline{\chi_{[\xi]}(g)}\xi(g)dg&=\int\limits_{G}\textnormal{Tr}(\xi(g)^*)\xi(g)dg=\sum_{j=1}^{d_\xi}\int\limits_{G}\xi(g)^{*}_{jj}\xi(g)dg\\
    &=\sum_{j=1}^{d_\xi}\frac{1}{d_\xi}I_{d_{\xi}\times d_\xi}\\
    &=I_{d_{\xi}\times d_\xi}.
\end{align*} In a similar way, if $s_{[\eta]}=\varkappa_{\tau}^{-1}\alpha_{[\eta]}(f),$ where $[\eta]\in \widehat{G},$ then (see the proof of Proposition 5.3.7 in  \cite[Page 119]{Wallach1973}),
\begin{equation*}
    \widehat{s}_{[\eta]}(\xi)(g_0K)= d_\xi [g_0 , L( \xi(g_{0})^{-1}v) ]\delta_{[\eta],[\xi]}.
\end{equation*}
\end{example}
Next we introduce the class of $G$-invariant continuous linear operators on homogeneous vector bundles. 

\begin{definition}[$G$-invariant continuous linear operator]
 Let $E=G\times_{\tau} E_0$ and $F=G\times_{\omega} F_0$ be homogeneous vector bundles. A continuous linear operator $\tilde{A}:\Gamma^\infty(E)\rightarrow \Gamma^\infty(F),$ is $G$-invariant if it is invariant under the action of $G$ on the space of sections $\Gamma^\infty(E)$, that is, for all $g \in G$
\begin{equation}\label{homogenesousdefinition}
    \tilde{A}(g\cdot s)=g\cdot (\tilde{A}s),\,\,s\in \Gamma^\infty(E).
\end{equation} 
\end{definition}
As it was observed in  \cite[Page 120]{Wallach1973}, for every continuous linear operator  $\tilde{A}:\Gamma^\infty(E)\rightarrow \Gamma^\infty(F)$  there exists a  continuous linear operator $A:C^\infty(G, E_0)^\tau \rightarrow C^\infty(G, F_0)^\omega$ such that the following diagram 
\begin{eqnarray}\label{maindiagram}
    \begin{tikzpicture}[every node/.style={midway}]
  \matrix[column sep={10em,between origins}, row sep={4em}] at (0,0) {
    \node(R) {$\Gamma^\infty(E)$}  ; & \node(S) {$\Gamma^\infty(F)$}; \\
    \node(R/I) {$C^\infty(G,E_0)^\tau$}; & \node (T) {$C^\infty(G,F_0)^\omega$};\\
  };
  \draw[<-] (R/I) -- (R) node[anchor=east]  {$\varkappa_{\tau}$};
  \draw[->] (R) -- (S) node[anchor=south] {$\tilde{A}$};
  \draw[->] (S) -- (T) node[anchor=west] {$\varkappa_\omega$};
  \draw[->] (R/I) -- (T) node[anchor=north] {$A$};
\end{tikzpicture}
\end{eqnarray}
commutes. Therefore, 
$A$ can be, intrinsically, defined as 
\begin{equation}\label{unitarilyequivaklent}
    A:=\varkappa_\omega\circ \tilde{A}\circ \varkappa_\tau^{-1}.
\end{equation}  
Here, we note  that if $\tilde{A}$ is a $G$-invariant operator then $A$ is a vector-valued Fourier multiplier  (see Lemma \ref{G-inv:Fourier:Mult} below).

One can construct differential operators mapping $C^\infty(G,E_0)^\tau$ into $C^\infty(G,F_0)^\omega$ as follows.  Let us denote by $\mathfrak{U}(\mathfrak{g})$  the universal enveloping algebra of $\mathfrak{g}:=\textnormal{Lie}(G)\sim T_eG.$ As in \cite[5.4.6, Page 121]{Wallach1973}, we define the action $\mu$ of $K$ on $\textnormal{End}(E_0, F_0) \otimes\mathfrak{U}(\mathfrak{g}) $ by $$\mu(k) (L \otimes X):= \tau(k) L \omega(k)^{-1} \otimes \text{Ad}(k)X$$ for $L\in \textnormal{End}(E_0, F_0) $  and $X \in\mathfrak{U}(\mathfrak{g}).$ Now, define the set
\begin{equation*}
    (\textnormal{End}(E_0,F_0)\otimes \mathfrak{U}(\mathfrak{g}))^{\tau,\omega}:=\{D\in \textnormal{End}(E_0,F_0)\otimes \mathfrak{U}(\mathfrak{g}):\mu(k)D=D,\,\,\,\forall\,\, k \in K \}.
\end{equation*} It is shown in \cite[Lemma 5.4.7]{Wallach1973} that for any $D \in (\textnormal{End}(E_0,F_0)\otimes \mathfrak{U}(\mathfrak{g}))^{\tau,\omega}$ we have ${ \omega(k)}D|_{C^\infty(G,E_0)^\tau}=D|_{C^\infty(G,E_0)^\tau}.$

\subsection{Pseudo-differential operators on homogeneous vector bundles} \label{qunt}
By the Schwartz kernel theorem, to any continuous linear operator $T:C^\infty(G,V)\rightarrow C^\infty(G,W),$ where $V$ and $W$ are finite dimensional vector spaces, one can  associate a distribution $K_{T}\in C^\infty(G)\otimes \mathscr{D}'(G,\textnormal{End}(V,W)),$ such that 
\begin{equation}
    Tf(x)=\int\limits_{G}K_{T}(x,y)f(y)dy,\,\,\,f\in C^\infty(G,V). 
\end{equation}
In this section we will analyse the case where $V=E_0$ and $W=F_{0}$ are representation spaces of $K.$ 
Let $\tilde{A}:\Gamma^\infty(E)\rightarrow \Gamma^\infty(F),$ be a  continuous linear operator where $E=G\times_{\tau}E_0$ and $F=G\times_{\omega}F_{0}$ are the total spaces of the homogeneous vector bundles $p_{E}:E\rightarrow M,$ and $p_{F}:F\rightarrow M.$ Under the identification $\Gamma^\infty(E)\simeq C^\infty(G,E_{0})^{\tau},$ and $\Gamma^\infty(F)\simeq C^\infty(G,F_{0})^{\omega},$ $\tilde{A}$ can be identified with a continuous linear operator $A:C^\infty(G,E_{0})^{\tau}\rightarrow  C^\infty(G,F_{0})^{\omega}.$  In the following theorem we characterise these continuous linear operators, $A:C^\infty(G,E_{0})^{\tau}\rightarrow  C^\infty(G,F_{0})^{\omega},$ in terms of  the Schwartz kernel $K_A$  of $A.$  
\begin{theorem}\label{TheoremCharK}
Let $A:C^\infty(G,E_{0})\rightarrow  C^\infty(G,F_{0}),$ be a continuous linear operator. Then, for every $f\in C^\infty(G,E_{0})^\tau$ we have $Af\in  C^\infty(G,F_{0})^{\omega}, $ if and only if, the identity
\begin{equation}\label{characterizationkernel}
 K_{A}(g,y)=   \omega(k_1)K_{A}(gk_1,yk_2)\tau(k_2)^{-1},
\end{equation}holds true for every $k_1,k_2\in K,$ and $g,y\in G.$
\end{theorem}
\begin{proof}
Let us assume that $A:C^\infty(G,E_{0})\rightarrow  C^\infty(G,F_{0}),$ is a continuous linear operator, and that for every $f\in C^\infty(G,E_{0})^\tau,$ we have
 $Af\in  C^\infty(G,F_{0})^{\omega}. $ Observe that
\begin{align*}
    Af(gk_1)=\int\limits_{G}K_{A}(gk_1,y)f(y)dy&=\int\limits_{G}K_{A}(gk_1,yk_2)f(yk_2)dy\\
    &=\int\limits_{G}K_{A}(gk_1,yk_2)\tau(k_2)^{-1}f(y)dy.
\end{align*}On the other hand,
\begin{align*}
    Af(gk_1)=\omega(k_1)^{-1}(Af)(g)&=\int\limits_{G}\omega(k_1)^{-1}K_{A}(g,y)f(y)dy.
\end{align*}
So, for every $k_1,k_2\in K,$ $$\omega(k_1)^{-1}K_{A}(g,y)=K_{A}(gk_1,yk_2)\tau(k_2)^{-1}.$$  The fact that \eqref{characterizationkernel} implies that  $Af\in  C^\infty(G,F_{0})^{\omega}, $  for every $f\in C^\infty(G,E_{0})^\tau,$ is a consequence of the algebraic identities proved in the first part of the theorem. Indeed, let us assume that $A$ satisfies \eqref{characterizationkernel}. So, for  $f\in C^\infty(G,E_{0})^\tau,$   we have
\begin{align*}
    Af(gk_1)&=\int\limits_{G}K_{A}(gk_1,y)f(y)dy =\int\limits_{G}K_{A}(gk_1,yk_2)f(yk_2)dy\\
    &=\int\limits_{G}K_{A}(gk_1,yk_2)\tau(k_2)^{-1}f(y)dy\\
    &=\int\limits_{G}\omega({k_1})^{-1}K_{A}(g,y)f(y)dy\\
     &=\omega({k_1})^{-1}Af(g).
\end{align*}So, we have proved that  $Af\in  C^\infty(G,F_{0})^{\omega}. $ 
Thus, we finish the proof. \end{proof}

Now,  we will use the Peter-Weyl theorem to introduce a suitable quantisation formula.
Let $B_{E_0}=\{e_{i,E_0}\}_{i=1}^{d_\tau},$ $d_\tau=\dim(E_0),$ be an orthonormal basis of $E_0.$ For every $f\in C^{\infty}(G,E_0),$ we can write
\begin{equation}
    f(x)=\sum_{i=1}^{d_\tau}e_{i,E_0}^{*}(f(x))\,e_{i,E_0},\,\,\,\,e_{i,E_0}^{*}(f(x))=(f(x),e_{i,E_0})_{E_0},
\end{equation}
and from the Peter-Weyl theorem we have
\begin{equation}
    f(x)=\sum_{i=1}^{d_\tau}\sum_{[\xi]\in \widehat{G}}d_{\xi}\textnormal{Tr}[\xi(x)\widehat{f}(i,\xi)   ]e_{i,E_0},\,\,\,\widehat{f}(i,\xi):=\mathscr{F}_{G}[x\mapsto e_{i,E_0}^{*}(f(x))].
\end{equation}So, motivated by the last expansion for $f$ in terms of the Fourier transform of coefficient functions $x\mapsto e_{i,E_0}^{*}(f(x))$, we define the $B_{E_0}$-Fourier transform of $f$ at $(i,[\xi]),$ by
\begin{equation}
  (\mathscr{F}_{G,B_{E_0}}f)(i,\xi) \equiv \widehat{f}(i,\xi)=\int\limits_{G}e_{i,E_0}^{*}(f(x))\xi(x)^{*}dx.
\end{equation}Observe that if $A:C^{\infty}(G, E_0)\rightarrow C^{\infty}(G, F_0),$ is a continuous linear operator and $B_{F_0}=\{e_{r,F_0}\}_{r=1}^{d_\omega},$ $d_\omega=\dim(F_0),$ is an orthonormal  basis for $F_0,$ we can first write
\begin{align*}
    Af(x)&=\int\limits_{G}K_A(x,y)f(y)dy=\int\limits_{G}\sum_{r,i=1}^{d_\tau}\sum_{[\xi]\in \widehat{G}}d_{\xi}\textnormal{Tr}[\xi(y)\widehat{f}(i,\xi)   ]K_{A}(x,y)e_{i,E_0}dy.
\end{align*}   
Then, writing $K_{A}(x,y)e_{i,E_0}\in F_0$ in terms of the elements of $B_{F_0},$ we have
\begin{align*}
    K_{A}(x,y)e_{i,E_0}=\sum_{r=1}^{d_\omega}e_{r,F_0}^{*}[ K_{A}(x,y)e_{i,E_0}]e_{r,F_0},
\end{align*}and we deduce
\begin{align*}
    Af(x)&=\int\limits_{G}\sum_{r=1}^{d_\omega}\sum_{i=1}^{d_\tau}\sum_{[\xi]\in \widehat{G}}d_{\xi}\textnormal{Tr}[\xi(y)\widehat{f}(i,\xi)   ]e_{r,F_0}^{*}[ K_{A}(x,y)e_{i,E_0}]e_{r,F_0}dy\\
    &=\int\limits_{G}\sum_{r=1}^{d_\omega}\sum_{i=1}^{d_\tau}\sum_{[\xi]\in \widehat{G}}d_{\xi}\textnormal{Tr}[\xi(y)\widehat{f}(i,\xi)   ]( K_{A}(x,y)e_{i,E_0}, e_{r,F_0})_{F_0}e_{r,F_0}dy.
\end{align*}Observe that, after changing integral and summations, we have
\begin{align*}
    Af(x)=\sum_{r=1}^{d_\omega}\sum_{i=1}^{d_\tau}\sum_{[\xi]\in \widehat{G}}d_{\xi}\textnormal{Tr}[\xi(x)\int\limits_{G} \xi(x)^{*}  \xi(y)e_{r,F_0}^{*}[ K_{A}(x,y)e_{i,E_0}] dy       \widehat{f}(i,\xi)   ]e_{r,F_0}.
\end{align*}So, using $\xi(x)^{*}\xi(y)=\xi(y^{-1}x)^*,$ and the change of variables $z=y^{-1}x,$ we have
\begin{align*}
    Af(x)&=\sum_{r=1}^{d_\omega}\sum_{i=1}^{d_\tau}\sum_{[\xi]\in \widehat{G}}d_{\xi}\textnormal{Tr}[\xi(x)\int\limits_{G} \xi(z)^*e_{r,F_0}^{*}[ K_{A}(x,xz^{-1})e_{i,E_0}] dz       \widehat{f}(i,\xi)   ]e_{r,F_0}\\
    &=\sum_{r=1}^{d_\omega}\sum_{i=1}^{d_\tau}\sum_{[\xi]\in \widehat{G}}d_{\xi}\textnormal{Tr}[\xi(x)\int\limits_{G} e_{r,F_0}^{*}[ R_{A}(x,z)e_{i,E_0}]\xi(z)^* dz       \widehat{f}(i,\xi)   ]e_{r,F_0},
\end{align*}where $R_{A}(x,z):=K_{A}(x,xz^{-1}).$ We refer to the distribution $R_A$ as the right convolution kernel of $A.$
Let us introduce the set of indices,
$${\boxed{ I_{d_\tau}:=\{i\in \mathbb{N}:1\leq i\leq d_\tau \},\quad I_{d_\omega}:=\{r\in \mathbb{N}:1\leq r\leq d_\omega \}}} $$
So, if we define
\begin{equation}\label{sigmair}
   \sigma_{A}(i,r,x,\xi):=\int\limits_{G} e_{r,F_0}^{*}[ R_{A}(x,z)e_{i,E_0}]\xi(z)^* dz,
\end{equation}
then 
\begin{equation}\label{irsymbol}
    \sigma_{A}:I_{d_\tau}\times I_{d_\omega}\times G\times \widehat{G}\rightarrow\bigcup\{\mathbb{C}^{d_\xi\times d_\xi}: {[\xi]\in \widehat{G}} \},
\end{equation} and we have the quantisation formula
\begin{equation}\label{quantization}
    Af(x)=\sum_{r=1}^{d_\omega}\sum_{i=1}^{d_\tau}\sum_{[\xi]\in \widehat{G}}d_{\xi}\textnormal{Tr}[\xi(x)\sigma_{A}(i,r,x,\xi)    \widehat{f}(i,\xi)   ]e_{r,F_0}.
\end{equation} We will say that $\sigma_A$ is the matrix-valued symbol of $A,$ associated with the bases $(B_{E_0},B_{F_0}),$ or for short, we say that $\sigma_A$ is the $(B_{E_0},B_{F_0})$-symbol of $A.$ We will simplify the notation in the quantisation formula \eqref{quantization} by using the  following one
\begin{equation}\label{Quantization}
    \boxed{  Af(x)=\sum_{i,r,[\xi]\in \widehat{G}}d_{\xi}\textnormal{Tr}[\xi(x)\sigma_{A}(i,r,x,\xi)    \widehat{f}(i,\xi)   ]e_{r,F_0}  }
\end{equation}Observe that we can recover the symbol $\sigma_A$ from the quantisation formula \eqref{Quantization}. 
Indeed, for every $g\in C^{\infty}(G),$ let $f=g\otimes e_{i_0,E},$   $1\leq i_0\leq d_\tau.$ The $B_{E_0}$-Fourier transform of $f$ at $(i,[\xi])$ is given by
\begin{equation}
 (\mathscr{F}_{B_{E_0},G}f) (i,[\xi])\equiv  \widehat{f}(i,[\xi])=\widehat{g}(\xi)\delta_{i,i_0}.
\end{equation}
Then, if we also fix $1\leq r_0\leq d_\omega,$ the linear operator $T_{\sigma_{A}(i_0,r_0,\cdot,\cdot)} $ defined by
\begin{align*}
 T_{\sigma_{A}(i_0,r_0,\cdot,\cdot)}  g(x)= e_{r_0,F}^{*}[A(g\otimes e_{i_0,E})(x)]&= \sum_{[\xi]\in \widehat{G}}d_{\xi}\textnormal{Tr}[\xi(x)\sigma_{A}(i_0,r_0,x,\xi)    \widehat{g}(\xi)   ],
\end{align*}is a continuous linear operator on $C^\infty(G),$ whose matrix-valued symbol is given by
\begin{align*}
    \sigma_{A}(i_0,r_0,x,\xi)= \xi(x)^{*}T_{\sigma_{A}(i_0,r_0,\cdot,\cdot)}\xi(x):= \xi(x)^{*}[ (T_{\sigma_{A}(i_0,r_0,\cdot,\cdot)}\xi_{uv})(x)]]_{u,v=1}^{d_\xi},
\end{align*}which implies,
\begin{align*}
    \sigma_{A}(i_0,r_0,x,\xi)=\xi(x)^{*} [ e_{r_0,F}^{*}[A(\xi_{uv}\otimes e_{i_0,E})(x)]]_{u,v=1}^{d_\xi}.
\end{align*}If we denote
\begin{equation}
 e_{r_0,F_0}^{*}[A(\xi\otimes e_{i_0,E_0})(x)]:=   [ e_{r_0,F_0}^{*}[A(\xi_{uv}\otimes e_{i_0,E_0})(x)]]_{u,v=1}^{d_\xi},
\end{equation}then we have
\begin{equation}\label{Homogeneoussymbol}{
     \sigma_{A}(i_0,r_0,x,\xi)= \xi(x)^{*}
     e_{r_0,F_0}^{*}[A(\xi\otimes e_{i_0,E_0})(x)]}
\end{equation}as expected from the matrix-valued formula in the case of the quantisation for compact Lie groups where,
\begin{equation*}
A(\xi\otimes e_{i_0,E}):=    (A(\xi_{uv}\otimes e_{i_0,E}))_{u,v=1}^{d_\xi}.
\end{equation*}
To express the dependence of the symbol 
$ \sigma_{A}(i_0,r_0,x,\xi)$ on the fixed bases, $B_{E_0}$ and $B_{F_0}$ of $E_0$ and $F_0,$ we can think that the symbol is a mapping 
\begin{equation}\label{basisirsymbol}
    \sigma_{A}:B_{E_0}\times B_{F_0}\times G\times \widehat{G}\rightarrow\bigcup\{\mathbb{C}^{d_\xi\times d_\xi}: {[\xi]\in \widehat{G}} \},
\end{equation} defined by
\begin{equation}
     \sigma_{A}(e_{i,E_0},e_{r,F_0},x,\xi):=\sigma_A(i,r,x,\xi),\,\,\,(x,[\xi])\in G\times \widehat{G}.
\end{equation} 
Now, if one wants to be precise, on the set
\begin{equation}
    \Sigma_A:=\{ \sigma_{A}:B_{E_0}\times B_{F_0}\times G\times \widehat{G}\rightarrow\bigcup_{[\xi]\in \widehat{G}}\mathbb{C}^{d_\xi\times d_\xi} :B_{E_0}\in \mathfrak{B}_{E_0},\,\,B_{F_0}\in \mathfrak{B}_{F_0}\},
\end{equation}
we define the equivalence relation $\sim$ by,
\begin{equation}\label{equivalence}
     \sigma_{A}(e_{i,E_0},e_{r,F_0},x,\xi)\sim \sigma_{A}(e'_{i,E_0},e'_{r,F_0},x',\eta)  
\end{equation}if and only if: 
\begin{itemize}
    \item $B_{E_0}=\{e_{i,E_0}:1\leq i\leq d_\tau\}$ and $B'_{E_0}=\{e'_{i,E_0}:1\leq i\leq d_\tau\}$ are orthonormal bases of $E_0,$
    \item $B_{F_0}=\{e_{i,F_0}:1\leq i\leq d_\omega\}$ and $B'_{F_0}=\{e'_{i,F_0}:1\leq i\leq d_\omega\}$ are orthonormal bases of $F_0,$
    \item $x=x'$ and $[\xi]=[\eta].$
\end{itemize}
The analysis above motivates the following  definitions.
\begin{definition}[{\bf Matrix-valued symbol as a set of equivalence classes}]\label{identificationofsymbols}
The global symbol of $A,$ is the set of equivalence classes
\begin{equation}
    \Sigma_A/\sim:=\{[\sigma_A]:\sigma_A\in \Sigma_A \}.
\end{equation}So, under the identification of every element $\sigma_A$ in \eqref{irsymbol}, (or as in \eqref{basisirsymbol}) with the equivalence class $[\sigma_A],$ we will not distinguish between the function $\sigma_A$ and its equivalence class  $[\sigma_A],$ unless that it is necessary. In that case, we will mention it separately. 
\end{definition}

To understand the identification in Definition \ref{identificationofsymbols}, we present the following proposition which says that for any linear operator $A$ the symbol $\sigma_A$ is invariant under the unitary change of basis. 
\begin{proposition}Let $B_{E_0}=\{e_{i,E_0}:1\leq i\leq d_\tau\}$ and $B'_{E_0}=\{e'_{i,E_0}:1\leq i\leq d_\tau\}$ be orthonormal bases of $E_0,$ and let $B_{F_0}=\{e_{i,F_0}:1\leq i\leq d_\omega\}$ and $B'_{F_0}=\{e'_{i,F_0}:1\leq i\leq d_\omega\}$ be  orthonormal bases of $F_0.$ Let $ \mathscr{U}\equiv U_{B_{E_0},B'_{E_0}}\in \textnormal{U}(E_0),$ and $ \tilde{\mathscr{U}}\equiv V_{B_{F_0},B'_{F_0}}\in \textnormal{U}(F_0),$  the unitary operators defined by
\begin{equation}
   \mathscr{U}(e_{i,E_0})\equiv U_{B_{E_0},B'_{E_0}}(e_{i,E_0})=e'_{i,E_0},\,\,\, \tilde{\mathscr{U}}(e'_{r,F_0})\equiv V_{B_{F_0},B'_{F_0}}(e_{r,F_0})=e'_{r,F_0}.
\end{equation}Then, for every continuous linear operator $A:C^\infty(G,E_0)\rightarrow C^{\infty}(G,F_0),$ we have
\begin{equation}\label{symbol}
   \sigma_A(e_{i,E_0},e_{r,F_0},x,\xi)= \sigma_{\tilde{\mathscr{U}}  A\mathscr{U}^{*}}(e'_{i,E_0},e'_{r,F_0},x,\xi).
\end{equation}
\end{proposition}
\begin{proof}
The proof is a straightforward computation. Indeed, 
\begin{align*}
    \sigma_A(e_{i,E_0},e_{r,F_0},x,\xi)&=\int\limits_{G} ( R_{A}(x,z)e_{i,E_0},e_{r,F_0})_{F_0}\xi(z)^* dz\\
    &=\int\limits_{G} ( R_{A}(x,z)U_{B_{E_0},B'_{E_0}}^{*}(e'_{i,E_0}),V^{*}_{B_{F_0},B'_{F_0}}(e'_{r,F_0}))_{F_0}\xi(z)^* dz\\
    &=\int\limits_{G} (V_{B_{F_0},B'_{F_0}} R_{A}(x,z)U_{B_{E_0},B'_{E_0}}^{*}(e'_{i,E_0}),e'_{r,F_0})_{F_0}\xi(z)^* dz\\
     &=\int\limits_{G} ( R_{V_{B_{F_0},B'_{F_0}}AU_{B_{E_0},B'_{E_0}}^{*}}(x,z)e'_{i,E_0},e'_{r,F_0})_{F_0}\xi(z)^* dz\\
    & =\sigma_{\tilde{\mathscr{U}}  A\mathscr{U}^{*}}(e'_{i,E_0},e'_{r,F_0},x,\xi),
\end{align*}proving \eqref{symbol}.
\end{proof}
In the following definition we summarise that, for every continuous linear operator $A:C^\infty(G,E_0)\rightarrow C^\infty(G,F_0),$ how one can associate to $A$ its matrix-valued symbol.
\begin{definition}[{\bf Matrix-valued symbols}] Let  ${A}:C^{\infty}(G,E_0)\rightarrow C^{\infty}(G,F_0)$ be a continuous linear operator. The global  symbol of $A,$ denoted by $\sigma_A,$ is defined by  the equivalence relation \eqref{equivalence}. In terms of two fixed basis $B_{E_0}$ and $B_{F_0},$ according with the notation above,  we have the following symbol formula 
\begin{equation}\label{Homogeneoussymbol'}\boxed{
     \sigma_{A}(i_0,r_0,x,\xi)= \xi(x)^{*}
     e_{r_0,F_0}^{*}[A(\xi\otimes e_{i_0,E_0})(x)]   }
\end{equation}  
Also, in terms of the right-convolution kernel $(x,z)\mapsto R_{A}(x,z):=K_{A}(x,xz^{-1})$ of $A,$
\begin{equation}\label{sigmairdefi}
   \sigma_{A}(i_0,r_0,x,\xi):=\int\limits_{G} e_{r_0,F_0}^{*}[ R_{A}(x,z)e_{i,E_0}]\xi(z)^* dz.
\end{equation}
The previous identity allows us to think of  $\sigma_A$ as a mapping
\begin{equation}\label{irsymboldefi}
    \sigma_{A}:I_{d_\tau}\times I_{d_\omega}\times G\times \widehat{G}\rightarrow\bigcup\{\mathbb{C}^{d_\xi\times d_\xi}: {[\xi]\in \widehat{G}} \}.
\end{equation}

\end{definition}
\begin{remark}[{\bf The identification $\tilde{A}\cong A$}]
Let $\tilde{A}:\Gamma^\infty(E)\rightarrow \Gamma^\infty(F)$ be a continuous linear operator where $E\cong G\times_{\tau}E_0$ and $F\cong G\times_{\omega}F_{0}$ are the total spaces of the homogeneous vector bundles $p_{E}:E\rightarrow M,$ and $p_{F}:F\rightarrow M.$ Under the identification $\Gamma^\infty(E)\simeq C^\infty(G,E_{0})^{\tau},$ and $\Gamma^\infty(F)\simeq C^\infty(G,F_{0})^{\omega},$ $\tilde{A}$ can be understood  as a continuous linear operator $A:C^\infty(G,E_{0})^{\tau}\rightarrow  C^\infty(G,F_{0})^{\omega}.$  In view of the quantisation formula \eqref{quantizationonhomogeneous2}, it will make sense to define $\sigma_A$  as the global symbol of $\tilde{A},$   see Definition \ref{symbolofAtilde}.
\end{remark} Next, we introduce a useful notation. 
\begin{remark}[The set $\mathbb{C}^{d_\xi\times d_\tau}(\textnormal{End}(E_0,F_0))$] We define the set $\mathbb{C}^{d_\xi\times d_\tau}(\textnormal{End}(E_0,F_0))$ to be the set of matrices $\{A_{i,j}\}_{1\leq i\leq d_\xi,\,1\leq j\leq d_\tau}$ where each entry $A_{i,j}$ is a linear mapping from $E_0$ to $F_0,$ namely, we have that  $A_{i,j}\in \textnormal{End}(E_0,F_0),$ for every $1\leq i\leq d_\xi,$ and all $1\leq j\leq d_\tau.$    
\end{remark}

\begin{definition}[{\bf$\textnormal{End}(E_0,F_0)$-valued symbols}]\label{equivalenceXXX}
In terms of the equivalence relation \eqref{equivalence}, $\sigma_A\sim [\sigma_A]$ can be realised as a mapping
\begin{equation}\label{mappings}
    \sigma_A\sim [\sigma_A]:G\times \widehat{G}\rightarrow \bigcup_{[\xi]\in \widehat{G}}\mathbb{C}^{d_\xi\times d_\xi}( \textnormal{End}(E_0,F_0)),
\end{equation} where $\textnormal{End}(E_0,F_0)$ is the set of linear operators from $E_0$ into $F_0,$ in such a way that for every $1\leq u,v\leq d_\xi,$ and  $(x,\xi)\in G\times \widehat{G},$ $\sigma_{A}(x,\xi)_{u,v}:E_{0}\rightarrow F_{0},$ $1\leq u,v\leq d_\xi,$ is an endomorphism. Indeed, using two bases $B_{E_0}=\{e_{i,E_0}\}\in \mathfrak{B}_{E_0} ,$  $   B_{F_0}=\{e_{r,F_0}\}\in \mathfrak{B}_{F_0}$ on the representation spaces $E_0$ and $F_0$, we recover the $(B_{E_0},B_{F_0})$-symbol for $A.$ To do this, define 
\begin{equation}\label{onetoone2}
    \sigma_{A}(x,\xi)_{u,v}:E_0\rightarrow F_0,
\end{equation}by its action on the bases $B_{E_0}$ and $B_{F_0}$ given by
\begin{equation}\label{onetoone}
  e_{r,F_0}^{*}(\sigma_A(x,\xi)_{u,v}e_{i,E_0}):=  \sigma_A(i,r,x,[\xi])_{u,v}.
\end{equation} So, we define the $\textnormal{End}(E_0,F_0)$-valued symbol of $A$ to be the endomorphism-valued matrix 
\begin{equation}
   \sigma_A(x,\xi):=(\sigma_A(x,\xi)_{u,v})_{1\leq v,u\leq d_\xi}\in \mathbb{C}^{d_\xi\times d_\xi}(\textnormal{End}(E_0,F_0)),
\end{equation}defining a symbol as in \eqref{mappings}. 
 The set of mappings in \eqref{mappings} where $A$ runs over all $\mathscr{L}(C^{\infty}(G, E_0),C^{\infty}(G, F_0))$ will be denoted by 
\begin{equation}
\Sigma    ((G\times \widehat{G})\otimes \textnormal{End}(E_0,F_0)).
\end{equation}
\end{definition}
\begin{remark}[{\bf Matrix-valued symbols vs $\textnormal{End}(E_0,F_0)$-valued symbols}]\label{remarkendvaluedmatrixvalued}
By the construction of our symbols, it is clear that {\it there exists an one-to-one correspondence } between the set $ \Sigma(I_{d_\tau}\times I_{d_\omega}\times G\times \widehat{G}):=$ 
\begin{equation}
   \{\sigma_{A}(i,r, x,[\xi]) :(i,r, x,[\xi])\in I_{d_\tau}\times  I_{d_\omega}\times G\times \widehat{G},\,A\in \mathscr{L}(C^{\infty}(G, E_0),C^{\infty}(G, F_0))   \}\footnote{where, for every $(i,r)\in I_{d_\tau}\times I_{d_\omega},$ $ \sigma_{A}(i,r, x,[\xi])\in \mathbb{C}^{d_\xi\times d_\xi}.$ }
\end{equation}    and $ \Sigma((G\times \widehat{G})\otimes \textnormal{End}(E_0,F_0))$ via \eqref{onetoone}.
\end{remark}

\begin{definition}\label{Product}
If $\Omega\in \Sigma(G\times \widehat{G})$ and $\sigma_A\in \Sigma((G\times \widehat{G})\otimes \textnormal{End}(E_0,F_0)), $ we will write
\begin{equation}
    \Omega(x,\xi)\otimes \sigma_A(x,\xi),
\end{equation} for the unique element in $ \Sigma((G\times \widehat{G})\otimes \textnormal{End}(E_0,F_0))$ induced (according to Definition \ref{equivalenceXXX}, via \eqref{onetoone}) by the matrix-symbol
\begin{equation*}
    \Omega(x,\xi)\sigma_A(i,r,x,\xi)\in  \Sigma(I_{d_\tau}\times I_{d_\omega}\times G\times \widehat{G}).
\end{equation*}Similarly we can define $  \sigma_A(x,\xi)\otimes \Omega(x,\xi).$
\end{definition}

Next, we characterise all the continuous linear operators $A:C^\infty(G,E_{0})^{\tau}\rightarrow  C^\infty(G,F_{0})^{\omega},$ in terms of their  global symbols  $\sigma_A.$
\begin{theorem}
Let $A:C^\infty(G,E_{0})\rightarrow  C^\infty(G,F_{0}),$ be a continuous linear operator. Then, for every $f\in C^\infty(G,E_{0})^\tau$ we have $Af\in  C^\infty(G,F_{0})^{\omega}, $ if and only if, the identity
\begin{equation}\label{characterizationsymbol}
 \sigma_A(e_{i,E_0},e_{r,F_0},x,\xi)=\xi(k_1)\sigma_A(\tau(k_2)^{-1}e_{i,E_0}, \omega(k_1)^{*}e_{r,F_0},xk_1,\xi)\xi(k_2)^{*},
\end{equation}holds true for every $k_1,k_2\in K,$ and $(x,[\xi])\in G\times \widehat{G}.$ In particular, if $k_2=e_{G},$ we have the identity
\begin{equation}\label{characterizationsymbol2}
 \sigma_A(e_{i,E_0},e_{r,F_0},x,\xi)=\xi(k_1)\sigma_A(e_{i,E_0},e_{r,F_0},xk_1,\xi),
\end{equation} for every $k_1\in K,$ and $x\in G.$
\end{theorem}
\begin{proof}Let us assume that for every $f\in C^\infty(G,E_{0})^\tau$ we have $Af\in  C^\infty(G,F_{0})^{\omega}. $ We have 
\begin{equation}\label{sigmair'}
   \sigma_{A}(e_{i,E_0},e_{r,F_0},x,\xi)=\int\limits_{G} ( R_{A}(x,z)e_{i,E_0},e_{r,F_0})_{F_0}\xi(z)^* dz,
\end{equation}where $R_A$ is the right-convolution kernel of $A.$ The definition $R_{A}(x,z):=K_{A}(x,xz^{-1}),$ implies that 
\begin{align*}
 \,\,\,\,\,\,\, \sigma_{A}(e_{i,E_0},e_{r,F_0},x,\xi)=\int\limits_{G} ( K_{A}(x,xz^{-1})e_{i,E_0},e_{r,F_0})_{F_0}\xi(z)^* dz  . 
\end{align*} From  \eqref{characterizationkernel} in Theorem \ref{TheoremCharK}, we deduce that for all $k_1, k_2 \in K$ we have
\begin{align*}
    \sigma_{A}(e_{i,E_0},e_{r,F_0},x,\xi)=\int\limits_{G} (\omega(k_1) K_{A}(xk_1,xz^{-1}k_2)\tau(k_2)^{-1}e_{i,E_0},e_{r,F_0})\xi(z)^* dz. 
\end{align*}Observe that
\begin{align*}
   & \int\limits_{G} (\omega(k_1) K_{A}(xk_1,xz^{-1}k_2)\tau(k_2)^{-1}e_{i,E_0},e_{r,F_0})_{F_0}\xi(z)^* dz\\
    &\hspace{1cm}=\int\limits_{G} (\omega(k_1) K_{A}(xk_1,xzk_2)\tau(k_2)^{-1}e_{i,E_0},e_{r,F_0})_{F_0}\xi(z) dz\\
    &\hspace{1cm}=\int\limits_{G} (\omega(k_1) K_{A}(xk_1,xw)\tau(k_2)^{-1}e_{i,E_0},e_{r,F_0})_{F_0}\xi(wk_2^{-1}) dw
    \\
    &\hspace{1cm}=\int\limits_{G} (\omega(k_1) K_{A}(xk_1,xk_1k_1^{-1}w)\tau(k_2)^{-1}e_{i,E_0},e_{r,F_0})_{F_0}\xi(w)\xi(k_2)^{-1}dw \\
    &\hspace{1cm}=\int\limits_{G} (\omega(k_1) K_{A}(xk_1,xk_1z)\tau(k_2)^{-1}e_{i,E_0},e_{r,F_0})_{F_0}\xi(k_1z)\xi(k_2)^{-1} dz \\
    &\hspace{1cm}=\int\limits_{G} (\omega(k_1) K_{A}(xk_1,xk_1z^{-1})\tau(k_2)^{-1}e_{i,E_0},e_{r,F_0})_{F_0}\xi(k_1z^{-1})\xi(k_2)^{-1} dz\\
    &\hspace{1cm}=\int\limits_{G} (\omega(k_1) K_{A}(xk_1,xk_1z^{-1})\tau(k_2)^{-1}e_{i,E_0},e_{r,F_0})_{F_0}\xi(k_1)\xi(z)^{*}\xi(k_2)^{-1} dz\\
    &\hspace{1cm}=\xi(k_1)\int\limits_{G} ( R_{A}(xk_1,z)\tau(k_2)^{-1}e_{i,E_0},\omega(k_1)^{*}e_{r,F_0})_{F_0}\xi(z)^{*}dz\xi(k_2)^{-1}\\
    &\hspace{1cm}=\xi(k_1)\sigma_{A}(\tau(k_2)^{-1}e_{i,E_0}, \omega(k_1)^{*}e_{r,F_0}, xk_1,\xi)\xi(k_2)^{*}.
\end{align*} Consequently, we have proved \eqref{characterizationsymbol}.
 Now, let us assume \eqref{characterizationsymbol}. We need to prove that for every $f\in C^\infty(G,E_{0})^\tau$ we have  $Af\in  C^\infty(G,F_{0})^{\omega}.$
From the first part of the proof, we have that 
\begin{align*}
    &\xi(k_1)\sigma_{A}(\tau(k_2)^{-1}e_{i,E_0}, \omega(k_1)^{*}e_{r,F_0}, xk_1,\xi)\xi(k_2)^{*}\\&= \int\limits_{G} (\omega(k_1) K_{A}(xk_1,xz^{-1}k_2)\tau(k_2)^{-1}e_{i,E_0},e_{r,F_0})_{F_0}\xi(z)^* dz.
\end{align*}  The uniqueness of the Fourier transform implies
\begin{align*}
    ( R_{A}(x,z)e_{i,E_0},e_{r,F_0})_{F_0}&=\mathscr{F}_{G}^{-1}[ \sigma_{A}(e_{i,E_0},e_{r,F_0},x,\xi)  ](z)\\
    &=\mathscr{F}_{G}^{-1}[\xi(k_1)\sigma_{A}(\tau(k_2)^{-1}e_{i,E_0}, \omega(k_1)^{*}e_{r,F_0}, xk_1,\xi)\xi(k_2)^{*}  ](z),
\end{align*}and using 
\begin{align*}
   &\mathscr{F}_{G}^{-1}[  \xi(k_1)\sigma_{A}(\tau(k_2)^{-1}e_{i,E_0}, \omega(k_1)^{*}e_{r,F_0}, xk_1,\xi)\xi(k_2)^{*} ]\\
   &=  \mathscr{F}_{G}^{-1}[ \int\limits_{G} (\omega(k_1) K_{A}(xk_1,xz^{-1}k_2)\tau(k_2)^{-1}e_{i,E_0},e_{r,F_0})_{F_0}\xi(z)^* dz  ]
\end{align*}
we deduce
\begin{align*}
    & ( K_{A}(x,xz^{-1})e_{i,E_0},e_{r,F_0})_{F_0}=( R_{A}(x,z)e_{i,E_0},e_{r,F_0})_{F_0}   \\
     &=(\omega(k_1) K_{A}(xk_1,xz^{-1}k_2)\tau(k_2)^{-1}e_{i,E_0},e_{r,F_0})_{F_0}.
\end{align*}Because $x,$ and $z$ are arbitrary elements of $G$ we deduce the identity
\begin{equation}
     K_{A}(x,y)=\omega(k_1) K_{A}(xk_1,yk_2)\tau(k_2)^{-1},
\end{equation}for every $x$ and $y$ in $G$ and $k_{1},k_2\in K.$ So, we conclude the proof by using  Theorem \ref{TheoremCharK}.
\end{proof}
Consider the Hilbert space $L^2(G,E_0).$ We record that the inner product in this space is given by 
\begin{equation}
    (f,g)_{L^2(G,E_0)}=\int\limits_{G}(f(x),g(x))_{E_0}dx,
\end{equation}where $(\cdot,\cdot)_{E_0}$ denotes the inner product on $E_0.$ 

As an application of the matrix-valued quantisation on $G,$ we will characterise, in the next theorem, the vector-valued linear operators of Hilbert-Schmidt type on $L^2(G,E_0).$ We recall that a bounded operator on a Hilbert space is of Hilbert-Schmidt type if it has a square-summable system of eigenvalues. 

In general, we recall that for any $r>0,$ a compact operator $T :L^2(G,E_0)\rightarrow L^2(G,E_0) $ belongs to the Schatten von-Neumann ideal $\mathscr{S}_{r}(L^2(G,E_0)),$ if the sequence of its singular values $\{s_{n}(T)\}_{n\in \mathbb{N}}$ (formed by the eigenvalues of the operator $\sqrt{T^*T}$) belongs to $\ell^r(\mathbb{N}),$ that is, if  $\sum_{n=1}^\infty s_{n}(T)^r<\infty.$ Here, under the identification $E_0\cong E_0^*, $ of finite-dimensional vector spaces we have considered the adjoint operator $T^*$ as a linear operator  $T^* :L^2(G,E_0)\rightarrow L^2(G,E_0). $

Our characterisation of Hilbert-Schmidt vector-valued operators is presented in the next theorem.
\begin{theorem}\label{HS:cHara:VB}
Let $A:C^\infty(G,E_{0})\rightarrow  C^\infty(G,F_{0}),$ be a continuous linear operator. Then, $A:L^2(G,E_{0})\rightarrow  L^2(G,E_{0}),$ is of Hilbert-Schmidt type, if and only if,
\begin{equation}\label{NSC:HS}
 \sum_{i,r=1}^{d_\tau}  \int\limits_{G} \sum_{[\xi]\in \widehat{G}}d_\xi\Vert \sigma_A(i,r,x,[\xi]) \Vert_{\textnormal{HS}}^2dx<\infty.
\end{equation}
\end{theorem}
\begin{proof}
Observe that $A:C^\infty(G,E_{0})\rightarrow  C^\infty(G,F_{0}),$ extends to a Hilbert-Schmidt operator, if and only if
\begin{equation}\label{Kernel:2}
   \int\limits_{G}\int\limits_{G}\Vert  K_{A}(x,y)\Vert_{\textnormal{HS}}^2dx\,dy=  \sum_{r,i=1}^{d_\tau}\int\limits_{G}\int\limits_{G}|e_{r,F_0}^{*}[ K_{A}(x,y)e_{i,E_0}]e_{r,F_0}|^2dx\,dy<\infty.
\end{equation}In view of the  change of variables $y=xz^{-1}$ and Fubini's theorem, to prove \eqref{Kernel:2} is equivalent to proving that
\begin{equation}\label{Kernel:2:2}
    \sum_{r,i=1}^{d_\tau}\int\limits_{G}\int\limits_{G}|e_{r,F_0}^{*}[ R_{A}(x,z)e_{i,E_0}]e_{r,F_0}|^2dz\,dx<\infty,
\end{equation}where $R_A$ is the right-convolution kernel of $A.$ Now, using the Plancherel theorem we obtain
\begin{align*}
     \sum_{r,i=1}^{d_\tau}\int\limits_{G}\int\limits_{G}|e_{r,F_0}^{*}[ R_{A}(x,z)e_{i,E_0}]e_{r,F_0}|^2dx\,dy=\sum_{i,r=1}^{d_\tau}\int\limits_{G} \sum_{[\xi]\in \widehat{G}}d_\xi\Vert \sigma_A(i,r,x,[\xi]) \Vert_{\textnormal{HS}}^2dx,
\end{align*} which shows that \eqref{NSC:HS} is a necessary and sufficient condition in order that $A$ will be of  Hilbert-Schmidt type on $L^2(G,E_0).$ This completes the proof. 
\end{proof}
Now, we consider the following application to the classification of Schatten-classes of vector-valued pseudo-differential operators by following the argument in \cite[Page 987]{DR:MRL:2017}. 
\begin{corollary}Let $r>0,$ and let $A:C^\infty(G,E_{0})\rightarrow  C^\infty(G,F_{0}),$ be a continuous linear operator. Then, $A:L^2(G,E_{0})\rightarrow  L^2(G,E_{0}),$ belongs to the Schatten class $\mathscr{S}_r(L^2(G,E_0))$, if and only if,
\begin{equation}\label{NSC:HS:2}
 \sum_{i,r=1}^{d_\tau}  \int\limits_{G} \sum_{[\xi]\in \widehat{G}}d_\xi\Vert \sigma_{|A|^{\frac{r}{2}}}(i,r,x,[\xi]) \Vert_{\textnormal{HS}}^2dx<\infty.
\end{equation}
    
\end{corollary}
\begin{proof}
    Note that $A\in \mathscr{S}_r(L^2(G,E_0)) $ if and only if $|A|^{\frac{r}{2}}\in  \mathscr{S}_2(L^2(G,E_0)) $ is a Hilbert-Schmidt operator. Then, the condition in \eqref{NSC:HS:2} is a necessary and sufficient condition for the Schatten property  $A\in \mathscr{S}_r(L^2(G,E_0)) $ in view of Theorem \ref{HS:cHara:VB}.
\end{proof}

\subsection{Classification of $\tau$-invariant functions}
 In the next theorem we classify the space of smooth $\tau$-invariant functions in terms of their Fourier coefficients.
\begin{theorem}\label{classifiucation:tau:f}
Let $f\in C^{\infty}(G,E_0).$ Then, $f\in C^{\infty}(G,E_0)^\tau$ if and only if
\begin{equation}\label{classitau}
    \widehat{f}(i,\xi)=\xi(k)^{*}\widehat{\tau(k)f}(i,\xi),
\end{equation}for all $[\xi]\in \widehat{G},$ $k\in K,$ and $1\leq i\leq d_\tau.$ 
\end{theorem}
\begin{proof} Let us denote by $\{\tau(k)^{*}_{ij}\}_{i,j=1}^{d_\tau}$  the matrix representation of $\tau(k)^{-1}$ with respect to the basis $\{e_{i,E_{0}}\}_{1\leq i\leq d_\tau}$ of the representation space $E_0.$
In view of the  Fourier inversion formula, for every $k\in K$ and $g\in G,$ we have
\begin{align*}
   f(gk)=\sum_{i=1}^{d_\tau}\sum_{[\xi]\in \widehat{G}}d_{\xi}\textnormal{Tr}[\xi(g)\xi(k)\widehat{f}(i,\xi)   ]e_{i,E_0}.
\end{align*}On the other hand, observe that
\begin{align*}
    \tau(k)^{-1}f(g)
    &=\sum_{j=1}^{d_\tau}\sum_{[\xi]\in \widehat{G}}d_{\xi}\textnormal{Tr}[\xi(g)\widehat{f}(j,\xi)   ] \tau(k)^{-1}e_{j,E_0}
    \\&=\sum_{j=1}^{d_\tau}\sum_{i=1}^{d_\tau}\sum_{[\xi]\in \widehat{G}}d_{\xi}\textnormal{Tr}[\xi(g)\widehat{f}(j,\xi)   ] (\tau(k)^{-1}e_{j,E_0},e_{i,E_{0}})_{E_0}e_{i,E_{0}}\\
    &=\sum_{i=1}^{d_\tau}\sum_{j=1}^{d_\tau}\sum_{[\xi]\in \widehat{G}}d_{\xi}\textnormal{Tr}[\xi(g)\widehat{f}(j,\xi)   ] (\tau(k)^{*}e_{j,E_0},e_{i,E_{0}})_{E_0}e_{i,E_{0}}\\
    &=\sum_{i=1}^{d_\tau}\sum_{[\xi]\in \widehat{G}}d_{\xi}\textnormal{Tr}\left[\xi(g) \left(  \sum_{j=1}^{d_\tau}       \widehat{f}(j,\xi)\tau(k)^{*}_{ij}\right)   \right] e_{i,E_0}.
\end{align*}Because of \eqref{tauinvariantfunction}, $f\in C^{\infty}(G,E_0)^\tau$ if and only if $f(gk)=\tau(k)^{-1}f(g),$ from the uniqueness of the Fourier coefficients we deduce that $f\in C^{\infty}(G,E_0)^\tau$ if and only if
\begin{equation*}
  \xi(k)\widehat{f}(i,\xi)=\sum_{j=1}^{d_\tau}\widehat{f}(j,\xi)\tau(k)^{*}_{ij},
\end{equation*}for all $[\xi]\in \widehat{G},$ $k\in K,$ and $1\leq i\leq d_\tau,$ 
which is equivalent to \eqref{classitau}. Indeed, the right hand side of the identity
\begin{equation*}
  \widehat{f}(i,\xi)=\sum_{j=1}^{d_\tau}\xi(k)^{*}\widehat{f}(j,\xi)\tau(k)^{*}_{ij},
\end{equation*} can be written as,
\begin{align*}
    \sum_{j=1}^{d_\tau}\xi(k)^{*}\widehat{f}(j,\xi)\tau(k)^{*}_{ij}&=\sum_{j=1}^{d_\tau}\xi(k)^{*}\Bigg(\int\limits_{G} ( f(x),e_{j,E_0}  )_{E_0}\xi(x)^{*}dx\Bigg)\tau(k)^{*}_{ij}\\
    &=\xi(k)^{*}\int\limits_{G} ( f(x),\sum_{j=1}^{d_\tau}\tau(k)^{*}_{ij}e_{j,E_0} )_{E_0}\xi(x)^{*}dx\\
     &=\xi(k)^{*}\int\limits_{G} ( f(x),\tau(k)^{*}e_{i,E_0} )_{E_0}\xi(x)^{*}dx.
\end{align*}Using that $\tau$ is a unitary representation we conclude that
$$  \widehat{f}(i,\xi)= \xi(k)^{*}\int\limits_{G}( \tau(k)f(x),e_{i,E_0} )_{E_0}\xi(x)^{*}dx=\xi(k)^{*}\widehat{\tau(k)f}(i,\xi).$$
Hence, the proof is complete.
\end{proof}

\subsection{Subelliptic H\"ormander classes on compact Lie groups }\label{subsectionofsub:C}
 In this section we record some definitions and results on the subelliptic calculus on compact Lie groups developed by the first and the last author \cite{RuzhanskyCardona2020}. 

Let $G$ be a compact Lie group  with Lie algebra $\mathfrak{g}.$ Under the identification $\mathfrak{g}\simeq T_{e_G}G,  $ where $e_{G}$ is the identity element of $G,$ let us consider  a system of $C^\infty$-vector fields $X=\{X_1,\cdots,X_k \}\in \mathfrak{g}$. For all $I=(i_1,\cdots,i_\omega)\in \{1,2,\cdots,k\}^{\omega},$ of length $\omega\geqslant   2,$ denote $$X_{I}:=[X_{i_1},[X_{i_2},\cdots [X_{i_{\omega-1}},X_{i_\omega}]\cdots]],$$ and for $\omega=1,$ $I=(i),$ $X_{I}:=X_{i}.$ Let $V_{\omega}$ be the subspace generated by the set $\{X_{I}:|I|\leqslant \omega\}.$ 

That the system $X$ satisfies the H\"ormander condition,  means that there exists $\kappa'\in \mathbb{N}$ such that $V_{\kappa'}=\mathfrak{g}.$ Certainly, we consider the smallest $\kappa'$ with this property and we denote it by $\kappa$ which will be  later called the step of the system $X.$ We also say that $X$ satisfies the H\"ormander condition of order $\kappa.$ 

The sum of squares
$$
    \mathcal{L}\equiv \mathcal{L}_{X}:=-(X_{1}^2+\cdots +X_{k}^2),
$$ by following the usual terminology is called the subelliptic Laplacian associated with the family $X.$ For short we refer to $\mathcal{L}$ as positive the sub-Laplacian.

 Here, $\{\widehat{ \mathcal{M}}(\xi)\}_{[\xi]\in \widehat{G}}$ is the matrix-valued symbol of the operator $\mathcal{M}:=(1+\mathcal{L})^{\frac{1}{2}},$ and  for every $[\xi]\in \widehat{G},$ $m\in \mathbb{R},$
   \begin{equation}\label{eigenvalues:hatM}
       \widehat{ \mathcal{M}}(\xi)^{m}:=\textnormal{diag}[(1+\nu_{ii}(\xi)^2)^{ \frac{m}{2} }]_{1\leq i\leq d_\xi}, 
   \end{equation} where $\widehat{\mathcal{L}}(\xi)=:\textnormal{diag}[\nu_{ii}(\xi)^2]_{1\leq i\leq d_\xi}$ is the symbol of the positive sub-Laplacian $\mathcal{L}$ at $[\xi].$ 
   
To classify symbols in the H\"ormander classes  developed in \cite{Ruz}, the notion of {\it{ difference operators}} on the unitary dual, by endowing $\widehat{G}$ with a difference structure, is an instrumental tool.  By following  \cite{Ruz},   a difference operator $Q_\xi$ of order $k,$  is defined by
\begin{equation}\label{taylordifferences}
    Q_\xi\widehat{f}(\xi)=\widehat{qf}(\xi),\,[\xi]\in \widehat{G}, 
\end{equation} for all $f\in C^\infty(G),$ for some function $q$ vanishing of order $k$ at the identity $e=e_G.$ We will denote by $\textnormal{diff}^k(\widehat{G})$  the set of all difference operators of order $k.$ For a  fixed smooth function $q,$ the associated difference operator will be denoted by $\Delta_q:=Q_\xi.$ We will choose an admissible collection of difference operators,
\begin{equation}\label{Difference.op}
  \Delta_{\xi}^\alpha:=\Delta_{q_{(1)}}^{\alpha_1}\cdots   \Delta_{q_{(i)}}^{\alpha_{i}},\,\,\alpha=(\alpha_j)_{1\leqslant j\leqslant i}, 
\end{equation}
where
\begin{equation*}
    \textnormal{rank}\{\nabla q_{(j)}(e):1\leqslant j\leqslant i \}=\textnormal{dim}(G), \textnormal{   and   }\Delta_{q_{(j)}}\in \textnormal{diff}^{1}(\widehat{G}).
\end{equation*}We say that this admissible collection is strongly admissible if 
\begin{equation*}
    \bigcap_{j=1}^{i}\{x\in G: q_{(j)}(x)=0\}=\{e_G\}.
\end{equation*}

\begin{remark}\label{remarkD} A special type of difference operators can be defined by using the unitary representations  of $G.$ Indeed, if $\xi_{0}$ is a fixed irreducible and unitary  representation of $G$, consider the matrix
\begin{equation}
 \xi_{0}(g)-I_{d_{\xi_{0}}}=[\xi_{0}(g)_{ij}-\delta_{ij}]_{i,j=1}^{d_\xi},\, \quad g\in G.   
\end{equation}Then, we associate  to the function 
$
    q_{ij}(g):=\xi_{0}(g)_{ij}-\delta_{ij},\quad g\in G,
$ a difference operator  via
\begin{equation}
    \mathbb{D}_{\xi_0,i,j}:=\mathscr{F}(\xi_{0}(g)_{ij}-\delta_{ij})\mathscr{F}^{-1}.
\end{equation}
If the representation is fixed we omit the index $\xi_0$ so that, from a sequence $\mathbb{D}_1=\mathbb{D}_{\xi_0,j_1,i_1},\cdots, \mathbb{D}_n=\mathbb{D}_{\xi_0,j_n,i_n}$ of operators of this type we define $\mathbb{D}^{\alpha}=\mathbb{D}_{1}^{\alpha_1}\cdots \mathbb{D}^{\alpha_n}_n$, where $\alpha\in\mathbb{N}^n$.
\end{remark}
\begin{remark}[Leibniz rule for difference operators]\label{Leibnizrule} The difference structure on the unitary dual $\widehat{G},$ induced by the difference operators acting on the momentum variable $[\xi]\in \widehat{G},$  implies the following Leibniz rule 
\begin{align*}
    \Delta_{q_\ell}(a_{1}a_{2})(x_0,\xi) =\sum_{ |\gamma|,|\varepsilon|\leqslant \ell\leqslant |\gamma|+|\varepsilon| }C_{\varepsilon,\gamma}(\Delta_{q_\gamma}a_{1})(x_0,\xi) (\Delta_{q_\varepsilon}a_{2})(x_0,\xi), \quad (x_{0},[\xi])\in G\times \widehat{G},
\end{align*} for $a_{1},a_{2}\in C^{\infty}(G, \mathscr{S}'(\widehat{G})).$ For details we refer the reader to  \cite{Ruz}.
\end{remark}   
   
\begin{lemma}[Global Taylor Series on compact Lie groups]\label{Taylorseries} Let $G$ be a compact Lie group of dimension $n.$ Let us consider a strongly admissible collection of difference operators $\mathfrak{D}=\{\Delta_{q_{(j)}}\}_{1\leqslant j\leqslant n}$, which means that 
\begin{equation*}
    \textnormal{rank}\{\nabla q_{(j)}(e):1\leqslant j\leqslant n \}=n, \,\,\,\bigcap_{j=1}^{n}\{x\in G: q_{(j)}(x)=0\}=\{e_G\}.
\end{equation*}Then there exists a basis $X_{\mathfrak{D}}=\{X_{1,\mathfrak{D}},\cdots ,X_{n,\mathfrak{D}}\}$ of $\mathfrak{g},$ such that $X_{j,\mathfrak{D}}q_{(k)}(\cdot^{-1})(e_G)=\delta_{jk}.$ Moreover, by using the multi-index notation $\partial_{X}^{(\beta)}=\partial_{X_{i,\mathfrak{D}}}^{\beta_1}\cdots \partial_{X_{n,\mathfrak{D}}}^{\beta_n}, $ $\beta\in\mathbb{N}_0^n,$
where $$\partial_{X_{i,\mathfrak{D}}}f(x):=  \frac{d}{dt}f(x\exp(tX_{i,\mathfrak{D}}) )|_{t=0},\,\,f\in C^{\infty}(G),$$ and denoting
for every $f\in C^{\infty}(G)$
\begin{equation*}
    R_{x,N}^{f}(y):=f(xy)-\sum_{|\alpha|<N}q_{(1)}^{\alpha_1}(y^{-1})\cdots q_{(n)}^{\alpha_n}(y^{-1})\partial_{X}^{(\alpha)}f(x),
\end{equation*} we have that 
\begin{equation*}
    | R_{x,N}^{f}(y)|\leqslant C|y|^{N}\max_{|\alpha|\leqslant N}\Vert \partial_{X}^{(\alpha)}f\Vert_{L^\infty(G)}.
\end{equation*}The constant $C>0$ is dependent on $N,$ $G$ and $\mathfrak{D},$ but not on $f\in C^\infty(G).$ Also, we have that $\partial_{X}^{(\beta)}|_{x_1=x}R_{x_1,N}^{f}=R_{x,N}^{\partial_{X}^{(\beta)}f}$ and 
\begin{equation*}
    | \partial_{X}^{(\beta)}|_{y_1=y}R_{x,N}^{f}(y_1)|\leqslant C|y|^{N-|\beta|}\max_{|\alpha|\leqslant N-|\beta|}\Vert \partial_{X}^{(\alpha+\beta)}f\Vert_{L^\infty(G)},
\end{equation*}provided that $|\beta|\leqslant N.$
 \end{lemma}
 \begin{remark}\label{derivative.cano}
      For any arbitrary choice of left-invariant vector fields $X_1,\cdots X_n,$ $n=\dim (G),$ we denote $$\partial_X^\alpha:=X_{1}^{\alpha_1}\cdots X_{n}^{\alpha_n}, $$ a canonical differential operator of order $|\alpha|:=\alpha_1+\cdots +\alpha_n.$
 \end{remark}
 
 Now with the notation in Lemma \ref{Taylorseries} above and the following one $\Delta_{\xi}^\alpha:=\Delta_{q_{(1)}}^{\alpha_1}\cdots   \Delta_{q_{(n)}}^{\alpha_{n}},$ we introduce the subelliptic H\"ormander class of symbols of order $m\in \mathbb{R},$ in the $(\rho,\delta)$-class.    

\begin{definition}[Subelliptic H\"ormander classes]\label{contracted''}
   Let $G$ be a compact Lie group and let $0\leq \delta,\rho\leq 1.$ Let us consider a sub-Laplacian $\mathcal{L}=-(X_1^2+\cdots +X_k^2)$ on $G,$ where the system of vector fields $X=\{X_i\}_{i=1}^{k}$ satisfies the H\"ormander condition of step $\kappa$. A symbol $\sigma$ belongs to the Subelliptic H\"ormander class ${S}^{m,\mathcal{L}}_{\rho,\delta}(G\times \widehat{G})$ 
   if $\sigma $ satisfies  the symbol inequalities
   \begin{equation}\label{InIC}
      p_{\alpha,\beta,\rho,\delta,m,\textnormal{left}}(\sigma)':= \sup_{(x,[\xi])\in G\times \widehat{G} }\Vert \widehat{ \mathcal{M}}(\xi)^{(\rho|\alpha|-\delta|\beta|-m)}\partial_{X}^{(\beta)} \Delta_{\xi}^{\alpha}\sigma(x,\xi)\Vert_{\textnormal{op}} <\infty,
   \end{equation}
   and 
   \begin{equation}\label{InIIC}
      p_{\alpha,\beta,\rho,\delta,m,\textnormal{right}}(\sigma)':= \sup_{(x,[\xi])\in G\times \widehat{G} }\Vert (\partial_{X}^{(\beta)} \Delta_{\xi}^{\alpha} \sigma(x,\xi) ) \widehat{ \mathcal{M}}(\xi)^{(\rho|\alpha|-\delta|\beta|-m)}\Vert_{\textnormal{op}} <\infty.
   \end{equation}
  \end{definition}

The following proposition states a version of the Corach-Porta-Recht inequality (see  Corach, Porta, and Recht \cite{CorachPortaRecht90} and Seddik \cite[Theorem 2.3]{Seddik} for \eqref{Recht1} and Andruchow, Corach, and  Stojanoff \cite[page 297]{Andruchow} for \eqref{Recht2}). This identity is useful to characterise  the subelliptic H\"ormander classes.

\begin{proposition}
Let $H$ be a complex Hilbert space and let $A,P,Q,X\in \mathscr{B}(H)$ be bounded operators on $H.$ Let us assume that $P$ and $Q$ are positive and invertible operators with $PQ=QP,$ and that $A$ is self-adjoint. Then we have the norm inequalities
\begin{equation}\label{Recht1}
    2\Vert X\Vert_{\textnormal{op}}\leqslant \max\{ \Vert PXP^{-1}+Q^{-1}XQ\Vert_{\textnormal{op}}, \Vert PX^*P^{-1}+Q^{-1}X^*Q\Vert_{\textnormal{op}}  \}, 
\end{equation}and 
\begin{equation}\label{Recht2}
    \Vert X \Vert_{\textnormal{op}}\leqslant \Vert AXA+(1+A^2)^{\frac{1}{2}}X(1+A^2)^{\frac{1}{2}} \Vert_{\textnormal{op}}.
\end{equation}
\end{proposition}

The next theorem characterises the subelliptic H\"ormander classes.

\begin{theorem}\label{gamma}
Let $G$ be a compact Lie group and let  $0\leqslant \delta,\rho\leqslant 1.$   The following conditions are equivalent.
\begin{itemize}
    \item[A.] For every $\alpha,\beta\in \mathbb{N}_0^n,$ \begin{equation}\label{InI2C}
      p_{\alpha,\beta,\rho,\delta,m,\textnormal{left}}(a)':= \sup_{(x,[\xi])\in G\times \widehat{G} }\Vert \widehat{ \mathcal{M}}(\xi)^{(\rho|\alpha|-\delta|\beta|-m)}\partial_{X}^{(\beta)} \Delta_{\xi}^{\alpha}a(x,\xi)\Vert_{\textnormal{op}} <\infty.
   \end{equation}
   \item[B.] For every $\alpha,\beta\in \mathbb{N}_0^n,$ \begin{equation}\label{InII2C}
      p_{\alpha,\beta,\rho,\delta,m,\textnormal{right}}(a)':= \sup_{(x,[\xi])\in G\times \widehat{G} }\Vert (\partial_{X}^{(\beta)} \Delta_{\xi}^{\alpha} a(x,\xi) ) \widehat{ \mathcal{M}}(\xi)^{(\rho|\alpha|-\delta|\beta|-m)}\Vert_{\textnormal{op}} <\infty.
   \end{equation}
   \item[C.] For all $r\in \mathbb{R},\alpha,\beta\in \mathbb{N}_0^n,$
    \begin{equation}\label{InI2XC}
      p_{\alpha,\beta,\rho,\delta,m,r}(a)':= \sup_{(x,[\xi])\in G\times \widehat{G} }\Vert \widehat{ \mathcal{M}}(\xi)^{(\rho|\alpha|-\delta|\beta|-m-r)}\partial_{X}^{(\beta)} \Delta_{\xi}^{\alpha}a(x,\xi)\widehat{ \mathcal{M}}(\xi)^{r}\Vert_{\textnormal{op}} <\infty.
   \end{equation}
   \item[D.] There exists $r_0\in \mathbb{R},$ such that for every $\alpha,\beta\in \mathbb{N}_0^n,$
    \begin{equation}\label{InI2X''C}
      p_{\alpha,\beta,\rho,\delta,m,r_0}(a)':= \sup_{(x,[\xi])\in G\times \widehat{G} }\Vert \widehat{ \mathcal{M}}(\xi)^{(\rho|\alpha|-\delta|\beta|-m-r_0)}\partial_{X}^{(\beta)} \Delta_{\xi}^{\alpha}a(x,\xi)\widehat{ \mathcal{M}}(\xi)^{r_0}\Vert_{\textnormal{op}} <\infty.
   \end{equation}
   \item[E.]  $a\in {S}^{m,\mathcal{L}}_{\rho,\delta}(G\times \widehat{G}).$
\end{itemize}
\end{theorem} 
Theorem \ref{gamma} is a useful tool to develop the next results building the subelliptic calculus associated to the sub-Laplacian.
\begin{theorem}\label{SubellipticcompositionC} Let  $0\leq \delta<\rho\leq 1.$ If $A_{i}\in \textnormal{Op}({S}^{m_i,\mathcal{L}}_{\rho,\delta}(G\times \widehat{G})), $ $A_{i}:C^\infty(G)\rightarrow C^\infty(G), $ $i=1,2,$ then the composition operator $A:=A_{1}\circ A_{2}:C^\infty(G)\rightarrow C^\infty(G)$ belongs to the subelliptic class $\textnormal{Op}({S}^{m_1+m_2,\mathcal{L}}_{\rho,\delta}(G\times \widehat{G})).$ The symbol of $A,$ $\widehat{A}(x,\xi),$ satisfies the asymptotic expansion,
\begin{equation*}
    \widehat{A}(x,\xi)\sim \sum_{|\alpha|= 0}^\infty(\Delta_{\xi}^\alpha\widehat{A}_{1}(x,\xi))(\partial_{X}^{(\alpha)} \widehat{A}_2(x,\xi)),
\end{equation*}this means that, for every $N\in \mathbb{N},$ and for all $\ell \in \mathbb{N},$ 
\begin{align*}
    &\Delta_{\xi}^{\alpha_\ell}\partial_{X}^{(\beta)}\left(\widehat{A}(x,\xi)-\sum_{|\alpha|\leq N}  (\Delta_{\xi}^\alpha\widehat{A}_{1}(x,\xi))(\partial_{X}^{(\alpha)} \widehat{A}_2(x,\xi))  \right)\\
    &\hspace{2cm}\in {S}^{m_1+m_2-(\rho-\delta)(N+1)-\rho\ell+\delta|\beta|,\mathcal{L}}_{\rho,\delta}(G\times \widehat{G}),
\end{align*} for every  $\alpha_\ell \in \mathbb{N}_0^n$ with $|\alpha_\ell|=\ell.$
\end{theorem}
The subelliptic calculus is stable under adjoints as we can see in the following theorem.
\begin{theorem}\label{AdjointC}
 Let $0\leq \delta<\rho\leq 1.$ If $A:C^\infty(G)\rightarrow C^\infty(G)$ is a continuous operator, $A\in \textnormal{Op}({S}^{m,\mathcal{L}}_{\rho,\delta}(G\times \widehat{G})),$ then $A^*\in \textnormal{Op}({S}^{m,\mathcal{L}}_{\rho,\delta}(G\times \widehat{G})).$ The  symbol of $A^*,$ $\widehat{A^{*}}(x,\xi)$ satisfies the asymptotic expansion,
 \begin{equation*}
    \widehat{A^{*}}(x,\xi)\sim \sum_{|\alpha|= 0}^\infty\Delta_{\xi}^\alpha\partial_{X}^{(\alpha)} (\widehat{A}(x,\xi)^{*}).
 \end{equation*} This means that, for every $N\in \mathbb{N},$ and all $\ell\in \mathbb{N},$
\begin{equation*}
   \Small{ \Delta_{\xi}^{\alpha_\ell}\partial_{X}^{(\beta)}\left(\widehat{A^{*}}(x,\xi)-\sum_{|\alpha|\leqslant N}\Delta_{\xi}^\alpha\partial_{X}^{(\alpha)} (\widehat{A}(x,\xi)^{*}) \right)\in {S}^{m-(\rho-\delta)(N+1)-\rho\ell+\delta|\beta|,\mathcal{L}}_{\rho,\delta}(G\times\widehat{G}) },
\end{equation*} where $|\alpha_\ell|=\ell.$
 \end{theorem}
In the following proposition we summarise the singularity order for integral kernels of subelliptic pseudo-differential operators (see Proposition 3.17 in \cite{RuzhanskyCardona2020}).  We denote by $Q,$ the Hausdorff dimension of $G$ associated with the control distance induced by the systems of vector fields $X=\{X_1,\cdots,X_k\}.$
\begin{proposition}\label{CalderonZygmund}
Let $G$ be a compact Lie group of dimension $n$ and let $0\leq \delta,\rho\leq 1.$  Let  $A:C^\infty(G)\rightarrow\mathscr{D}'(G)$ be a continuous linear operator with symbol $\sigma\in {S}^{m,\mathcal{L}}_{\rho,\delta}(G\times \widehat{G})$   in the subelliptic H\"ormander class of order $m$ and of type $(\rho,\delta)$. Then, the right-convolution kernel of $A,$ $x\mapsto k_{x}:G\rightarrow C^\infty(G\setminus\{e_G\}),$ defined by $k_{x}:=\mathscr{F}^{-1}\sigma(x,\cdot),$ satisfies the following estimates for $|y|<1$:
\begin{itemize}
    \item[(i)] if $m>-Q,$ there exists $\ell\in \mathbb{N},$ independent of $\sigma,$ such that  
        \begin{equation*}
            |k_{x}(y)|\lesssim_{m}  \Vert \sigma\Vert_{\ell, S^{m,\mathcal{L}}_{\rho,\delta}}|y|^{-\frac{Q+m}{\rho}}.
        \end{equation*}
         \item[(i)] If $m=-Q,$ there exists $\ell\in \mathbb{N},$ independent of $\sigma,$ such that  
        \begin{equation*}
            |k_{x}(y)|\lesssim_{m} \Vert \sigma\Vert_{\ell, S^{m,\mathcal{L}}_{\rho,\delta}}|\ln|y||.
        \end{equation*}
        \item[(i)] If $m<-Q,$ there exists $\ell\in \mathbb{N},$ independent of $\sigma,$ such that  
        \begin{equation*}
            |k_{x}(y)|\lesssim_{m} \Vert \sigma\Vert_{\ell, S^{m,\mathcal{L}}_{\rho,\delta}}.
        \end{equation*}
\end{itemize}
\end{proposition}
The good properties of the kernels of subelliptic operators imply some analogues of well known mapping properties for pseudo-differential operators. We will summarise it as follows. We start with the $L^2$-estimates by presenting the following subelliptic Calder\'on-Vaillancourt theorem.
\begin{theorem}\label{CVT}
Let $G$ be a compact Lie group and let us consider the sub-Laplacian $\mathcal{L}=\mathcal{L}_X,$ where  $X=\{X_{i}\}_{i=1}^{k}$ is a system of vector fields satisfying the H\"ormander condition of order $\kappa$.  For  $0\leq \delta< \rho\leq    1,$ (or $0\leq \delta\leq \rho\leq 1,$ $\delta<1/\kappa$) let us consider a continuous linear operator $A:C^\infty(G)\rightarrow\mathscr{D}'(G)$ with symbol  $\sigma\in {S}^{0,\mathcal{L}}_{\rho,\delta}(G\times \widehat{G})$. Then $A$ extends to a bounded operator from $L^2(G)$ to  $L^2(G),$ and 
\begin{equation}\label{eq1CVT}
    \Vert  A\Vert_{\mathscr{B}(L^2(G))}\leq C \Vert \sigma\Vert_{\ell, {S}^{0,\mathcal{L}}_{\rho,\delta} },
\end{equation} for $\ell \in \mathbb{N}$ large enough. 
\end{theorem} 

As usually, the $L^2$-boundedness of pseudo-diffenrential operators  implies its Sobolev continuity. 
\begin{corollary} \label{sobcont}
Let $A:C^{\infty}(G)\rightarrow \mathscr{D}'(G)$ be a continuous linear operator with symbol $a\in S^{m,\mathcal{L}}_{\rho,\delta}(G\times \widehat{G}),$ $0\leq \delta< \rho\leq    1,$ (or $0\leq \delta\leq \rho\leq 1,$ $\delta<1/\kappa$). Then $A:H^{s,\mathcal{L}}(G)\rightarrow H^{s-m,\mathcal{L}}(G) $ extends to a bounded operator for all $s\in \mathbb{R}.$
\end{corollary}
\begin{proof} In view of the closed graph Theorem, we only need to show that there exists $C>0,$ such that 
\begin{equation}\label{Sobolev}
    \Vert Au \Vert_{H^{s-m,\mathcal{L}}(G)}=\Vert \mathcal{M}_{s-m}Au\Vert_{L^2(G)}\leq C\Vert u \Vert_{H^{s,\mathcal{L}}(G)},\,\,u\in C^{\infty}(G).
\end{equation}By replacing $u$ by $\mathcal{M}_{-s}u,$ we can see that \eqref{Sobolev} is equivalent to the following estimate
\begin{equation}\label{Sobolev2}
    \Vert \mathcal{M}_{s-m}A\mathcal{M}_{-s}u\Vert_{L^2(G)}\leq C\Vert u \Vert_{L^2(G)},\,\,u\in C^{\infty}(G),
\end{equation}which, again, by the closed graph theorem is equivalent to show that $A_{s}:=\mathcal{M}_{s-m}A\mathcal{M}_{-s}$ admits a bounded extension from $C^{\infty}(G)$ to $L^2(G).$ By the subelliptic calculus $A_s\in S^{0,\mathcal{L}}_{\rho,\delta}(G\times \widehat{G}).$ So, the existence of a bounded extension of $A_{s}$ is a consequence of the subelliptic Calder\'on-Vaillancourt theorem. 
\end{proof}

For the $L^p$-estimates, let us recall that we denote by $Q,$ the Hausdorff dimension of $G$ associated with the control distance induced by the systems of vector fields $X=\{X_1,\cdots,X_k\}.$ First, we present the following $L^\infty$-$BMO$ pseudo-differential theorem.
\begin{theorem}\label{parta}
Let $G$ be a compact Lie group and let us denote by $Q$ the Hausdorff
dimension of $G$ associated to the control distance associated to the sub-Laplacian $\mathcal{L}=\mathcal{L}_X,$ where  $X=\{X_{i}\}_{i=1}^{k}$ is a system of vector fields satisfying the H\"ormander condition of order $\kappa$.  For  $0\leq \delta<\rho\leq    1,$  let us consider a continuous linear operator $A:C^\infty(G)\rightarrow\mathscr{D}'(G)$ with symbol  $\sigma\in {S}^{-\frac{Q(1-\rho) }{2},\mathcal{L}}_{\rho,\delta}(G\times \widehat{G})$. Then $A$ extends to a bounded operator from $L^\infty(G)$ to $BMO^{\mathcal{L}}(G),$ and from $H^{1,\mathcal{L}}(G)$ to $L^1(G).$ Moreover,
\begin{equation}\label{eq1}
    \Vert  A\Vert_{\mathscr{B}(L^\infty(G),BMO^{\mathcal{L}}(G))}\leq C \max\{\Vert \sigma\Vert_{\ell, {S}^{-\frac{Q(1-\rho) }{2},\mathcal{L}}_{\rho,\delta} },\Vert  \sigma(\cdot,\cdot)  \widehat{\mathcal{M}}(\xi)^{ Q(1-\rho)  }\Vert_{\ell, {S}^{0,\mathcal{L}}_{\rho,\delta} }\}
\end{equation} and 
\begin{equation}\label{eq2}
      \,\,\,\,\,\,\,\,\,\,\Vert  A\Vert_{\mathscr{B}(H^{1,\mathcal{L}}(G),L^1(G))}\leq C \max\{\Vert \sigma^*\Vert_{\ell, {S}^{-\frac{Q(1-\rho) }{2},\mathcal{L}}_{\rho,\delta} }, \Vert \sigma^*(\cdot,\cdot)  \widehat{\mathcal{M}}(\xi)^{ Q(1-\rho)  }\Vert_{\ell, {S}^{0,\mathcal{L}}_{\rho,\delta} }\}
\end{equation}
for some integer $\ell,$ where $\sigma^{*}$ denotes the matrix-valued symbol of the formal adjoint $A^*$ of $A.$ 
\end{theorem} The Fefferman-Stein interpolation Theorem implies the following subelliptic  $L^p$-estimate.
\begin{theorem}\label{parta2}
Let $G$ be a compact Lie group and let us denote by $Q$ the Hausdorff
dimension of $G$ associated to the control distance associated to the sub-Laplacian $\mathcal{L}=\mathcal{L}_X,$ where  $X=\{X_{i}\}_{i=1}^{k}$ is a system of vector fields satisfying the H\"ormander condition of order $\kappa$.  For  $0\leq \delta<\rho\leq 1,$ let us consider a continuous linear operator $A:C^\infty(G)\rightarrow\mathscr{D}'(G)$ with symbol  $\sigma\in {S}^{-m,\mathcal{L}}_{\rho,\delta}(G\times \widehat{G})$, $m\geq 0$. Then $A$ extends to a bounded operator on $L^p(G)$ provided that 
\begin{equation*}
    m\geq m_p:= Q(1-\rho)\left|\frac{1}{p}-\frac{1}{2}\right|.
\end{equation*}
\end{theorem}
For the proof of the results above we refer the reader to \cite{RuzhanskyCardona2020}.
\subsection{Subelliptic Sobolev spaces on vector bundles} In this subsection we define the subelliptic Sobolev spaces that we will consider in this work. In particular, it will be used in obtaining G\r{a}rding type inequalities. So, we start with the definition of vector-valued Sobolev spaces. 
\begin{definition} \label{Defivectorsob}
The Sobolev space $L^{2,\mathcal{L}}_{s}(G,E_0)$ or order $s\in \mathbb{R},$ modeled on $L^{2}(G,E_0)$ is defined by the completion of $C^{\infty}(G,E_0)$ with respect to the norm,
\begin{align*}
    \Vert u\Vert_{L^{2,\mathcal{L}}_{s}(G,E_0)}:=\left(\sum_{i=1}^{d_\tau}\sum_{[\xi]\in \widehat{G}}d_\xi\Vert \widehat{\mathcal{M}}(\xi)^{s} \widehat{u}(i,\xi)\Vert_{\textnormal{HS}}^2\right)^{\frac{1}{2}},
\end{align*}
where  $\{\widehat{ \mathcal{M}}(\xi)\}_{[\xi]\in \widehat{G}}$ is the matrix-valued symbol of the operator $\mathcal{M}:=(1+\mathcal{L})^{\frac{1}{2}}.$
\end{definition}
\begin{remark}\label{characterisationSobolevspaces}
 Let us define for every $\lambda\in \mathbb{R}$
 \begin{equation*}
     \mathcal{M}_{\lambda,E_0}:=\textnormal{Op}\left(\widehat{\mathcal{M}}(\xi)^{\lambda} \otimes I_{ \textnormal{End}(E_0)}\right).
 \end{equation*}Keeping in mind Remark \ref{remarkendvaluedmatrixvalued}, let us observe that the $\textnormal{End}(E_0)$-valued symbol  of ${\mathcal{M}}_{\lambda,E_0},$ is given by
 \begin{equation*}
     \widehat{\mathcal{M}}_{\lambda,E_0}(x,\xi)=(\widehat{\mathcal{M}}_{\lambda,E_0}(x,\xi)_{u,v=1}^{d_\xi}),\,\,\widehat{\mathcal{M}}_{\lambda,E_0}(x,\xi)_{u,v}:=\widehat{\mathcal{M}}(\xi)^{\lambda}_{uv}I_{\textnormal{End}(E_0)},
 \end{equation*} so that
 \begin{equation*}
   \widehat{\mathcal{M}}_{\lambda,E_0}(x,\xi)= \widehat{\mathcal{M}}(\xi)^{\lambda} \otimes I_{ \textnormal{End}(E_0)}  \in \mathbb{C}^{d_\xi\times d_\xi}( \textnormal{End}(E_0)).
 \end{equation*}Now, notice that the matrix-valued symbol of ${\mathcal{M}}_{\lambda,E_0}$ can be computed as follows: 
 \begin{align*}
     \widehat{\mathcal{M}}_{\lambda,E_0}(i,r,x,\xi)&= (\widehat{\mathcal{M}}_{\lambda,E_0}(i,r,x,\xi)_{u,v})_{u,v=1}^{d_\xi} =(e_{r,E_0}^{*}( \widehat{\mathcal{M}}_{\lambda,E_0}(x,\xi)_{uv}e_{i,E_0}) )_{u,v=1}^{d_\xi}\\
     &=(e_{r,E_0}^{*}( \widehat{\mathcal{M}}(\xi)^{\lambda}_{uv}e_{i,E_0}) )_{u,v=1}^{d_\xi}=(\widehat{\mathcal{M}}(\xi)^{\lambda}_{uv})_{u,v=1}^{d_\xi}\delta_{ir}\\
     &=\widehat{\mathcal{M}}(\xi)^{\lambda}\delta_{ir},
 \end{align*}which shows that
\begin{equation} \label{eq265}
      \Vert u\Vert_{L^{2,\mathcal{L}}_{t}(G,E_0)}=\Vert {\mathcal{M}}_{\lambda,E_0} u \Vert_{L^{2}(G,E_0)}. 
\end{equation}Indeed,
\begin{align*}
    \Vert {\mathcal{M}}_{\lambda,E_0} u \Vert_{L^{2}(G,E_0)}^2 &=\Vert \sum_{i,r,[\xi]\in \widehat{G}}d_{\xi}\textnormal{Tr}[\xi(x)\widehat{\mathcal{M}}(\xi)^{\lambda}\delta_{ir}    \widehat{u}(i,\xi)   ]e_{r,E_0} \Vert^2_{L^{2}(G,E_0)}\\
    &=\Vert \sum_{i;[\xi]\in \widehat{G}}d_{\xi}\textnormal{Tr}[\xi(x)\widehat{\mathcal{M}}(\xi)^{\lambda}    \widehat{u}(i,\xi)   ]e_{i,E_0} \Vert^2_{L^{2}(G,E_0)}\\
    &=\Vert\Vert \sum_{i;[\xi]\in \widehat{G}}d_{\xi}\textnormal{Tr}[\xi(x)\widehat{\mathcal{M}}(\xi)^{\lambda}    \widehat{u}(i,\xi)   ]e_{i,E_0}\Vert_{E_0} \Vert^2_{L^{2}(G)}\\
     &=\sum_{i=1}^{d_\tau}\Vert\Vert (1+\mathcal{L})^{\frac{\lambda}{2}}u_{i}(x)   ]e_{i,E_0}\Vert_{E_0} \Vert^2_{L^{2}(G)}=\sum_{i=1}^{d_\tau}\sum_{[\xi]\in \widehat{G}}d_\xi\Vert \widehat{\mathcal{M}}(\xi)^{\lambda} \widehat{u}(i,\xi)\Vert_{\textnormal{HS}}^2,
\end{align*}in view of the Plancherel theorem on $L^2(G),$ with $\widehat{u}_i(\xi)=\widehat{u}(i,\xi).$
 \end{remark}
 Let us define the space 
 $$L^{2, \mathcal{L}}_s(G, E_0)^\tau= \left\{f \in L^{2, \mathcal{L}}_s(G, E_0): f(gk)= \tau(k)^{-1} f(g),\,\,k \in K, g\in G  \right\}.$$
 The space $L^{2, \mathcal{L}}_s(G, E_0)^\tau$ is a closed subspace of $L^{2, \mathcal{L}}_s(G, E_0)$ and it agrees with the completion of $C^\infty(G, E_0)^\tau$ with respect to the norm $\Vert \cdot \Vert_{L^{2, \mathcal{L}}_s(G, E_0)}$ in $L^{2, \mathcal{L}}_s(G, E_0).$
\begin{definition} \label{subvectsobo}
The subelliptic Sobolev space $L^{2,\mathcal{L}}_{s_0}(E)$  of order $s_0\in \mathbb{R},$ is defined as the completion of $\Gamma^\infty(E)$ with respect to the norm
\begin{equation}
\Vert s \Vert_{L^{2,\mathcal{L}}_{s_0}(E)  }:= \Vert  \varkappa_\tau s \Vert_{L^{2,\mathcal{L}}_{s_0}(G,E_0)},
\end{equation}where $\varkappa_\tau:L^2(E)\rightarrow L^2(G,E_0)^{\tau}$ is the isomorphism defined in \eqref{notationfunction}.
\end{definition}
In the case where the sub-Laplacian $\mathcal{L}$ is replaced by the Laplacian $\mathcal{L}_G,$ the corresponding  Sobolev spaces will be denoted by $H^{s_0}(G,E_0)$ and $H^{s_0}(E).$
\begin{remark} \label{soboiso}
Let us observe that the isomorphism $ \varkappa_{\tau}:\Gamma^\infty(E)\rightarrow C(G,E_0)^{\tau}$ 
extends to an isometry of Hilbert spaces
\begin{equation}
  \varkappa_{\tau}:L^{2,\mathcal{L}}_{s_0}(E) \rightarrow L^{2,\mathcal{L}}_{s_0}(G,E_0), 
\end{equation} for any $s_0\in \mathbb{R}.$

\end{remark}
\begin{remark} 
Let us consider the vector-valued Sobolev spaces $L^{2}_{s}(G,E_0)$ associated with the Laplacian $\mathcal{L}_G$ on $G,$ and defined by the norm
\begin{align*}
    \Vert u\Vert_{L^{2}_{s}(G,E_0)}:=\left(\sum_{i=1}^{d_\tau}\sum_{[\xi]\in \widehat{G}}d_\xi\Vert \langle \xi\rangle^{s} \widehat{u}(i,\xi)\Vert_{\textnormal{HS}}^2\right)^{\frac{1}{2}},
\end{align*} where the elliptic weight $\langle \xi\rangle$ is defined in terms of the spectrum $\{\lambda_{[\xi]}\}_{[\xi]\in \widehat{G}}$ of the positive Laplacian $\mathcal{L}_G$ on $G$ as follows $$\langle\xi \rangle:=(1+\lambda_{[\xi]})^{1/2},\,\,[\xi]\in \widehat{G}.$$    For $E_0=\mathbb{C}$ we use the notation $L^{2}_{s}(G)=L^{2}_{s}(G,\mathbb{C}). $ The Sobolev space $L^{2}_{s_0}(E)$ of  order $s_0\in \mathbb{R},$ is defined as the completion of $\Gamma^\infty(E)$ with respect to the norm
\begin{equation}
\Vert s \Vert_{L^{2}_{s_0}(E)  }:= \Vert  \varkappa_\tau s \Vert_{L^{2}_{s_0}(G,E_0)}.
\end{equation}
For $s>0,$ we have the embeddings, see \cite{RuzhanskyCardona2020},
\begin{equation*}
    L^{2}_s(G)\hookrightarrow L^{2,\mathcal{L}}_s(G) \hookrightarrow L^{2}_{\frac{s}{\kappa}}(G) \hookrightarrow L^2(G)  \hookrightarrow  L^{2}_{-\frac{s}{\kappa}}(G)\hookrightarrow L^{2,\mathcal{L}}_{-s}(G) \hookrightarrow L^{2}_{-s}(G).
\end{equation*}So, we deduce the embeddings for the vector-valued setting
\begin{equation*}
    L^{2}_s(G,E_0)\hookrightarrow L^{2,\mathcal{L}}_s(G,E_0) \hookrightarrow L^{2}_{\frac{s}{\kappa}}(G,E_0) \hookrightarrow L^2(G,E_0)  ,
\end{equation*} and 
\begin{equation}
 L^2(G,E_0)  \hookrightarrow  L^{2}_{-\frac{s}{\kappa}}(G,E_0)\hookrightarrow L^{2,\mathcal{L}}_{-s}(G,E_0) \hookrightarrow L^{2}_{-s}(G,E_0).
\end{equation} 
These embeddings induce the following embeddings of Sobolev spaces defined on the homogeneous vector-bundle $E,$
\begin{equation*}
    L^{2}_s(E)\hookrightarrow L^{2,\mathcal{L}}_s(E) \hookrightarrow L^{2}_{\frac{s}{\kappa}}(E) \hookrightarrow L^2(E)  \hookrightarrow  L^{2}_{-\frac{s}{\kappa}}(E)\hookrightarrow L^{2,\mathcal{L}}_{-s}(E) \hookrightarrow L^{2}_{-s}(E).
\end{equation*} 
Consequently we have (see \cite[Page 130]{Wallach1973})

\begin{equation}\label{CinftyL}
   \bigcap_{s \in \mathbb{R}} L^{2}_s(G, E_0)=C^\infty(G, E_0)=\bigcap_{s \in \mathbb{R}} L^{2, \mathcal{L}}_s(G, E_0)
\end{equation}
as well as 
\begin{equation} \bigcap_{s \in \mathbb{R}} L^{2}_s(G, E_0)^{\tau}=C^\infty(G, E_0)^{\tau}=\bigcap_{s \in \mathbb{R}} L^{2, \mathcal{L}}_s(G, E_0)^{\tau},
\end{equation}
and consequently
\begin{equation}\label{CinftyL:2}
   \bigcap_{s \in \mathbb{R}} L^{2}_s(E)=\Gamma^\infty(E)=\bigcap_{s \in \mathbb{R}} L^{2, \mathcal{L}}_s(E).
\end{equation} These properties of Sobolev spaces will be useful for our further analysis of Fredholness of pseudo-differential operators in Theorem \ref{Vector:Fred}.
\end{remark}

\subsection{Derivatives and differences for $\textnormal{End}(E_0,F_0)$-valued symbols} In this subsection we fix the notation corresponding to the derivatives and differences of symbols of vector-valued pseudo-differential operators. Let us fix $\alpha,\beta\in \mathbb{N}_0^n.$ Then, every symbol of the form  $$(i,r,x,[\xi])\mapsto \partial_{X}^{\beta}\Delta_\xi^{\alpha}\sigma_A(i,r,x,[\xi]) \in  \Sigma(I_{d_\tau}\times I_{d_\omega}\times G\times \widehat{G}), $$ induces a symbol  $\partial_{X}^{\beta}\Delta_\xi^{\alpha}\sigma_A(x,\xi)$ in the class $ \Sigma((G\times \widehat{G})\otimes \textnormal{End}(E_0,F_0)),$ via Definition \ref{equivalenceXXX} as follows:
\begin{equation}\label{main:notation}
    \partial_{X}^{\beta}\Delta_\xi^{\alpha}\sigma_A(x,\xi):=[\partial_{X}^{\beta}\Delta_\xi^{\alpha}\sigma_A(x,\xi)_{u,v}]_{1\leq u,v\leq d_\xi}.
\end{equation}To do this, define 
\begin{equation}\label{onetoone2:32:2:2}
    \partial_{X}^{\beta}\Delta_\xi^{\alpha}\sigma_{A}(x,\xi)_{u,v}:E_0\rightarrow F_0,
\end{equation}by its action on the bases $B_{E_0}=\{e_{i,E_0}\}$ and  $B_{F_0}=\{e_{r,F_0}\}$ of $E_0$ and $F_0,$ respectively, as follows
\begin{equation}\label{onetoone:2:2:2}
  e_{r,F_0}^{*}( \partial_{X}^{\beta}\Delta_\xi^{\alpha}\sigma_A(x,\xi)_{u,v}e_{i,E_0}):=   \partial_{X}^{\beta}\Delta_\xi^{\alpha}\sigma_A(i,r,x,[\xi])_{u,v},
\end{equation} for any $1\leq u,v\leq d_\xi.$ For any $\alpha\in \mathbb{N}_0^n,$ the difference operator $\Delta_\xi^{\alpha}$ acting on the symbol $(x,[\xi])\mapsto \sigma_A(i,r,x,[\xi])$  is defined as in \eqref{Difference.op}, while for any $\beta\in \mathbb{N}_0^n,$ the notation $\partial_{X}^{\beta}$ as been introduced in Remark \ref{derivative.cano}.  The notation in \eqref{main:notation} will be crucial to define the subelliptic H\"ormander classes in the next section.

\section{Subelliptic pseudo-differential calculus  on homogeneous vector bundles} \label{vectcal}

In this section we will develop a subelliptic pseudo-differential calculus on homogeneous vector-bundles induced by the subelliptic calculus on compact Lie groups. First, we will define and develop a vector-valued calculus on compact Lie groups induced in a natural way by the subelliptic calculus. Later, we will use the Bott-construction on homogeneous vector bundles to make the construction of the pseudo-differential calculus for pseudo-differential operators on homogeneous vector-bundles. 
\subsection{Symbolic calculus on compact Lie groups}
We recall the  quantisation formula from Section \ref{qunt}. For any continuous operator $A:C^\infty(G, E_0) \rightarrow C^\infty(G, F_0)$ we have the following quantisation formula
$$Af(x)=\sum_{i,r,[\xi]\in \widehat{G}}d_{\xi}\textnormal{Tr}[\xi(x)\sigma_{A}(i,r,x,\xi)    \widehat{f}(i,\xi)   ]e_{r,F_0},$$
where $\sigma_A:  I_{d_\tau}\times I_{d_\omega}  \times G \times \widehat{G} \rightarrow \bigcup\{\mathbb{C}^{d_\xi \times d_\xi}: [\xi] \in \widehat{G}\}$ is the symbol of $A.$ 
We rewrite this quantisation formula to make it more convenient to use as follows:
\begin{align} \label{cutoff}
    Af(x)&=\sum_{i=1}^{d_\tau} \sum_{r=1}^{d_\omega}\sum_{[\xi]\in \widehat{G}}d_{\xi}\textnormal{Tr}[\xi(x)\sigma_{A}(i,r,x,\xi)    \widehat{f}(i,\xi)   ]e_{r,F_0} \nonumber \\ &= \sum_{i=1}^{d_\tau} \sum_{r=1}^{d_\omega} A_{ir}f_i(x) e_{r, F_0},
\end{align}
where $A_{ir}f_i$ defined on the complex valued function $f_i(x)= \langle f(x), e_{i, E_0}\rangle \equiv e_{i, E_0}^* (f(x)) $ is  given by the matrix-valued quantisation  formula
on  $G,$
\begin{equation}
    A_{ir}f_i(x)= \sum_{[\xi] \in \widehat{G}} d_\xi \textnormal{Tr}[\xi(x) \sigma_A(i, r, x ,\xi) \widehat{f}_i(\xi)],\,\,\,\widehat{f_i}(\xi):= \widehat{f}(i, \xi)= \mathscr{F}_G[e_{i, E_0}^* (f(x))].
\end{equation}

Having defined this quantisation, now it is time to introduce the subelliptic H\"ormander class of symbols of order $m \in \mathbb{R}$ in the $(\rho, \delta)$-class.
\begin{definition}[Vector-valued subelliptic H\"ormander classes]\label{symbolclass}
   Let $G$ be a compact Lie group and let  $0\leq \delta,\rho\leq 1.$ Let $E_0$ and $F_0$ be complex vector spaces with finite dimensions $d_\tau$ and $d_\omega,$ respectively. Let us consider a sub-Laplacian $\mathcal{L}=-(X_1^2+\cdots +X_k^2)$ on $G,$ where the system of vector fields $X=\{X_i\}_{i=1}^{k}$ satisfies the H\"ormander condition of step $\kappa$. If $m\in \mathbb{R},$ the H\"ormander symbol class $S^{m,\mathcal{L}}_{\rho,\delta}((G\times \widehat{G})\otimes \textnormal{End}(E_0,F_0))$ of subelliptic  order $m$ and of type $(\rho,\delta),$ consists of those functions $\sigma \in \Sigma(I_{d_\omega} \times I_{d_\tau} \times G\times \widehat{G}),$ satisfying the symbol inequalities
   \begin{equation}\label{InI}
      \tilde{p}_{\alpha,\beta,\rho,\delta,m, \gamma } (\sigma):= \sup_{\substack{1\leq i\leq d_\tau\\ 1\leq r\leq d_\omega \\ (x, [\xi])\in G\times \widehat{G} }
      } \Vert \widehat{ \mathcal{M}}(\xi)^{(\gamma+\rho|\alpha|-\delta|\beta|-m)}\partial_{X}^{(\beta)} \Delta_{\xi}^{\alpha}\sigma(i,r, x,\xi)\mathcal{M}(\xi)^{-\gamma} \Vert_{\textnormal{op}} <\infty,
   \end{equation} for any $\gamma\in \mathbb{R}.$
   We denote by $\Psi^{m,\mathcal{L}}_{\rho,\delta}((G\times \widehat{G})\otimes \textnormal{End}(E_0,F_0))$ the class of continuous linear operators $A:C^\infty(G,E_0)\rightarrow C^\infty(G,F_0)$ with symbols in the class $S^{m,\mathcal{L}}_{\rho,\delta}((G\times \widehat{G})\otimes \textnormal{End}(E_0,F_0)).$
  \end{definition}
\begin{remark}
If $E_0=F_0,$ we will use the simplified notation
 $$\Psi^{m,\mathcal{L}}_{\rho,\delta}((G\times \widehat{G})\otimes \textnormal{End}(E_0)):=\Psi^{m,\mathcal{L}}_{\rho,\delta}((G\times \widehat{G})\otimes \textnormal{End}(E_0,E_0)), $$ for the operator classes and also
 $$S^{m,\mathcal{L}}_{\rho,\delta}((G\times \widehat{G})\otimes \textnormal{End}(E_0)):=S^{m,\mathcal{L}}_{\rho,\delta}((G\times \widehat{G})\otimes \textnormal{End}(E_0,E_0)), $$ for the classes of symbols.
\end{remark}
\begin{remark}
In Theorem \ref{orders:theorems:VB} we prove that the real functional calculus of the vector-valued sub-Laplacian $\mathcal{L}_{E_0}$ (see \eqref{Vec:sub}) is included in the vector-valued subelliptic calculus. Of particular interest are the powers 
\begin{equation}
    \mathcal{M}_{s,E_0}:=(1+\mathcal{L}_{E_0})^{\frac{s}{2}}\in \Psi^{s,\mathcal{L}}_{1,0}((G\times \widehat{G})\otimes\textnormal{End}(E_0))
\end{equation}
for all $s\in \mathbb{R},$ (see \eqref{thepowersof:LE0} in Theorem \ref{orders:theorems:VB}). However, the stability of the  complex functional calculus for more general subelliptic  operators (with symbols that may depend of the spatial variable $x\in G$) will be proved in Section  \ref{GFCSECT}. In particular, from Theorem \ref{DunforRiesz} we have the following properties for complex powers
\begin{equation}
    \mathcal{M}_{z,E_0}:=(1+\mathcal{L}_{E_0})^{\frac{z}{2}}\in \Psi^{\textnormal{Re}(z),\mathcal{L}}_{1,0}((G\times \widehat{G})\otimes\textnormal{End}(E_0)),\,\,z\in \mathbb{C}.
\end{equation}
\end{remark}

\begin{remark}
We note here that if $E_0=\mathbb{C}=F_0$ we recover the usual subelliptic classes on compact Lie groups from \cite{RuzhanskyCardona2020}.
\end{remark} 

\begin{remark}\label{remark3.3}
With Definition \ref{Product} in mind, we define the following  seminorm on symbols which is  equivalent to $\tilde{p}_{\alpha,\beta,\rho,\delta,m, \gamma }, $  by 
\begin{equation*}
    p_{\alpha,\beta,\rho,\delta,m,\gamma}(\sigma)=: \sup_{(x,[\xi])\in G\times \widehat{G} }\Vert \widehat{ \mathcal{M}}(\xi)^{(\gamma+\rho|\alpha|-\delta|\beta|-m)} \otimes \partial_{X}^{(\beta)} \Delta_{\xi}^{\alpha}\sigma(x, \xi)\otimes\widehat{ \mathcal{M}}(\xi)^{{-\gamma}}\Vert_{\textnormal{op}} <\infty,
\end{equation*}where $\widehat{ \mathcal{M}}(\xi)^{(\gamma+\rho|\alpha|-\delta|\beta|-m)} \otimes \partial_{X}^{(\beta)} \Delta_{\xi}^{\alpha}\sigma(x, \xi)\otimes\widehat{ \mathcal{M}}(\xi)^{{-\gamma}}$ denotes  the unique element in $ \Sigma((G\times \widehat{G})\otimes \textnormal{End}(E_0,F_0))$ induced (according to Definition \ref{equivalenceXXX}, via \eqref{onetoone}) by the matrix-symbol
\begin{equation*}
    \widehat{ \mathcal{M}}(\xi)^{(\gamma+\rho|\alpha|-\delta|\beta|-m)}\sigma_A(i,r,x,[\xi])\widehat{ \mathcal{M}}(\xi)^{{-\gamma}}\in  \Sigma(I_{d_\tau}\times I_{d_\omega}\times G\times \widehat{G}).  
\end{equation*}    
\end{remark}
In the following theorem we provide a classification for the vector-valued subelliptic H\"ormander classes.

\begin{theorem}\label{characterisations}
 Let  $0\leq \delta,\rho\leq 1,$ and let $\alpha,\beta\in \mathbb{N}_0^n.$ The following conditions are equivalent.
\begin{itemize}
    \item[A.] For every $\alpha,\beta\in \mathbb{N}_0^n,$ \begin{equation}\label{InI2}
      p_{\alpha,\beta,\rho,\delta,m,\textnormal{left}}(\sigma):= \sup_{(x,[\xi])\in G\times \widehat{G} }\Vert \widehat{ \mathcal{M}}(\xi)^{(\rho|\alpha|-\delta|\beta|-m)} \otimes \partial_{X}^{(\beta)} \Delta_{\xi}^{\alpha}\sigma(x,\xi)\Vert_{\textnormal{op}} <\infty.
   \end{equation}
   \item[B.] For every $\alpha,\beta\in \mathbb{N}_0^n,$ \begin{equation}\label{InII2}
      p_{\alpha,\beta,\rho,\delta,m,\textnormal{right}}(\sigma):= \sup_{(x,[\xi])\in G\times \widehat{G} }\Vert (\partial_{X}^{(\beta)} \Delta_{\xi}^{\alpha} \sigma(x,\xi) ) \otimes  \widehat{ \mathcal{M}}(\xi)^{(\rho|\alpha|-\delta|\beta|-m)}\Vert_{\textnormal{op}} <\infty.
   \end{equation}
   \item[C.] For all $\gamma\in \mathbb{R},$ $\alpha,\beta\in \mathbb{N}_0^n,$
    \begin{equation}
      p_{\alpha,\beta,\rho,\delta,m,\gamma}(\sigma):= \sup_{(x,[\xi])\in G\times \widehat{G} }\Vert \widehat{ \mathcal{M}}(\xi)^{(\rho|\alpha|-\delta|\beta|-m-\gamma)} \otimes \partial_{X}^{(\beta)} \Delta_{\xi}^{\alpha}\sigma(x,\xi)\otimes\widehat{ \mathcal{M}}(\xi)^{{\gamma}}\Vert_{\textnormal{op}} <\infty.
   \end{equation}
   \item[D.] There exists $\gamma_0\in \mathbb{R},$ such that for every $\alpha,\beta\in \mathbb{N}_0^n,$
    \begin{equation}
      p_{\alpha,\beta,\rho,\delta,m,\gamma_0}(\sigma):= \sup_{(x,[\xi])\in G\times \widehat{G} }\Vert \widehat{ \mathcal{M}}(\xi)^{(\rho|\alpha|-\delta|\beta|-m-\gamma_0)} \otimes  \partial_{X}^{(\beta)} \Delta_{\xi}^{\alpha}\sigma(x,\xi) \otimes\widehat{ \mathcal{M}}(\xi)^{
      \gamma_0}\Vert_{\textnormal{op}} <\infty.
   \end{equation}
   \item[E.] $\sigma \in S^{m,\mathcal{L}}_{\rho,\delta}((G\times \widehat{G})\otimes \textnormal{End}(E_0,F_0)).$
\end{itemize}
\end{theorem} 
\begin{proof}
The idea is to use Remark \ref{remark3.3}. Indeed, for every $s,r\in \mathbb{R},$ the $\textnormal{End}(E_0,F_0)$-valued symbol 
\begin{equation*}
 (x,[\xi])\mapsto   \widehat{ \mathcal{M}}(\xi)^{s} \otimes  \partial_{X}^{(\beta)} \Delta_{\xi}^{\alpha}\sigma(x,\xi) \otimes\widehat{ \mathcal{M}}(\xi)^{r}: \mathbb{C}^{d_\xi\times d_\xi}(\textnormal{End}(E_{0},F_{0})),
\end{equation*} is, by definition, the unique symbol induced by the matrix-valued symbol
\begin{equation*}
 (i,r,x,[\xi])\mapsto   \widehat{ \mathcal{M}}(\xi)^{s}  \partial_{X}^{(\beta)} \Delta_{\xi}^{\alpha}\sigma(i,r,x,\xi) \widehat{ \mathcal{M}}(\xi)^{r}\in \mathbb{C}^{d_\xi\times d_\xi}.
\end{equation*}
So, by fixing $1 \leq i \leq d_\tau $ and $1 \leq r \leq d_\omega,$ the five conditions mentioned in the statement of the theorem are equivalent by Theorem \ref{gamma}. As a consequence of this fact, by taking the supremum over the finite family of indices $1 \leq i \leq d_\tau $ and $1 \leq r \leq d_\omega,$ again these estimates are equivalent.  
\end{proof}
In the following composition theorem $H\cong G\times_\nu H_0 \rightarrow M$ denotes a homogeneous vector bundle over $M=G/K.$

\begin{theorem}\label{VectorcompositionC} Let  $0\leq \delta<\rho\leq 1.$ If $A \in \Psi^{m_1,\mathcal{L}}_{\rho,\delta}((G\times \widehat{G})\otimes \textnormal{End}(E_0,F_0))$ and $B\in \Psi^{m_2,\mathcal{L}}_{\rho,\delta}((G\times \widehat{G})\otimes \textnormal{End}(F_0,H_0))$ then the composition operator $BA:=B\circ A:C^\infty(G, E_0)\rightarrow C^\infty(G, H_0)$ belongs to the subelliptic class $\Psi^{m_1+m_2,\mathcal{L}}_{\rho,\delta}((G\times \widehat{G})\otimes \textnormal{End}(E_0,H_0)).$ The symbol of $BA,$ $\sigma_{BA}\in\Sigma ((G\times \widehat{G})\otimes \textnormal{End}(E_0,H_0)),$ satisfies the asymptotic expansion,
\begin{equation}\label{compositionAandB}
    \sigma_{BA}(i, s, x,\xi) \sim \sum_{r=1}^{d_\omega}\sum_{|\alpha|= 0}^\infty(\Delta_{\xi}^\alpha\sigma_{B}(r, s, x,\xi))(\partial_{X}^{(\alpha)} \sigma_{A}(i, r, x,\xi)), 
\end{equation}this means that, for every $N\in \mathbb{N},$
\begin{align*}
    &\Delta_{\xi}^{\alpha_\ell}\partial_{X}^{(\beta)}\left(\sigma_{BA}(i, s, x,\xi) - \sum_{r=1}^{d_\omega}\sum_{|\alpha| \leq N }(\Delta_{\xi}^\alpha\sigma_{B}(r, s, x,\xi))(\partial_{X}^{(\alpha)} \sigma_{A}(i, r, x,\xi))  \right)\\
    &\hspace{2cm}\in {S}^{m_1+m_2-(\rho-\delta)(N+1)-\rho\ell+\delta|\beta|,\mathcal{L}}_{\rho,\delta}((G\times \widehat{G})\otimes \textnormal{End}(E_0,H_0)),
 \end{align*}for every difference operator $\Delta_{\xi}^{\alpha_\ell}$ of order $\ell\in\mathbb{N}_0.$ In particular, the same asymptotic expansion applies if additionally  $A:C^{\infty}(G,E_{0})^\tau\rightarrow C^{\infty}(G,F_{0})^\omega,$ and $B:C^{\infty}(G,F_{0})^\tau\rightarrow C^{\infty}(G,H_{0})^\nu$ are continuous linear operators. This means that \eqref{compositionAandB} remains valid for the continuous linear operator $B\circ A:C^{\infty}(G,E_{0})^\tau\rightarrow C^{\infty}(G,H_{0})^\nu.$
\end{theorem}
\begin{proof}
Since $B\in \Psi^{m_2,\mathcal{L}}_{\rho,\delta}((G\times \widehat{G})\otimes \textnormal{End}(F_0,H_0))$ we have 
 \begin{align*}
    (B\circ A) f(x):= B[Af](x)&= \sum_{r=1}^{d_\omega} \sum_{s=1}^{d_\nu} B_{rs}(Af)_r(x) e_{s, H_0} .
 \end{align*}
 Now, note that 
 \begin{align*}
     (Af)_r(x) := \langle Af(x), e_{r, F_0} \rangle &= \left\langle \sum_{i=1}^{d_\tau} \sum_{r'=1}^{d_\omega} A_{ir'} f_i(x) e_{r', F_0}, e_{r, F_0} \right\rangle \\&= \sum_{i=1}^{d_\tau}  A_{ir} f_i(x).
 \end{align*}
 Therefore, we obtain 
 \begin{align*}
     B[Af](x)&= \sum_{r=1}^{d_\omega} \sum_{s=1}^{d_\nu} B_{rs} \left[\sum_{i=1}^{d_\tau}  A_{ir} f_i(x) \right] e_{s, H_0} \\&= \sum_{r=1}^{d_\omega} \sum_{s=1}^{d_\nu} \sum_{i=1}^{d_\tau} B_{rs} A_{ir} f_i(x)  e_{s, H_0} \\&=\sum_{r=1}^{d_\omega} \sum_{s=1}^{d_\nu} \sum_{i=1}^{d_\tau} \sum_{[\xi] \in \widehat{G}} d_\xi \textnormal{Tr}[\xi(x) \sigma_{B_{rs} A_{ir}}(x, \xi) \widehat{f}_i(\xi)]   e_{s, H_0} \\&=  \sum_{s=1}^{d_\nu} \sum_{i=1}^{d_\tau} \sum_{[\xi] \in \widehat{G}} d_\xi \textnormal{Tr}[\xi(x) \sum_{r=1}^{d_\omega} \sigma_{B_{rs} A_{ir}}(x, \xi) \widehat{f}_i(\xi)]   e_{s, H_0}.
 \end{align*}
 By the uniqueness of the symbol, we obtain that 
 $$\sigma_{BA}(i,s, x, \xi)= \sum_{r=1}^{d_\omega} \sigma_{B_{rs} A_{ir}}(x, \xi).$$
 Now, $B_{rs} A_{ir}$ is the  composition of two pseudo-differential operators with $B_{rs} \in \textnormal{Op}({S}^{m_1,\mathcal{L}}_{\rho,\delta}(G\times \widehat{G}))$  and $A_{ir} \in \textnormal{Op}({S}^{m_1,\mathcal{L}}_{\rho,\delta}(G\times \widehat{G})),$ so by Theorem \ref{SubellipticcompositionC} we have $B_{rs} A_{ir} \in \textnormal{Op}({S}^{m_1+m_2,\mathcal{L}}_{\rho,\delta}(G\times \widehat{G}))$. Thus we also have $BA \in \Psi^{m_1+m_2,\mathcal{L}}_{\rho,\delta}((G\times \widehat{G})\otimes \textnormal{End}(E_0,H_0))  $ as symbols classes are closed under sums.
 Also, we have, by Theorem \ref{SubellipticcompositionC}, 
 $$ \sigma_{B_{rs} A_{ir}}(x, \xi) \sim \sum_{|\alpha|= 0}^\infty(\Delta_{\xi}^\alpha\sigma_{B_{rs}}( x,\xi))(\partial_{X}^{(\alpha)} \sigma_{A_{ir}}( x,\xi))$$
 and as a consequence 
 \begin{align} \label{Vish3.5e}
     \sigma_{BA}(i,s, x, \xi)&\sim \sum_{r=1}^{d_\omega} \sum_{|\alpha|= 0}^\infty(\Delta_{\xi}^\alpha\sigma_{B_{rs}}( x,\xi))(\partial_{X}^{(\alpha)} \sigma_{A_{ir}}( x,\xi)) \\&= \sum_{r=1}^{d_\omega}\sum_{|\alpha|= 0}^\infty(\Delta_{\xi}^\alpha\sigma_{B}(r, s, x,\xi))(\partial_{X}^{(\alpha)} \sigma_{A}(i, r, x,\xi)). \nonumber
 \end{align}
 The fact that, for every $N\in \mathbb{N},$
\begin{align*}
    &\Delta_{\xi}^{\alpha_\ell}\partial_{X}^{(\beta)}\left(\sigma_{BA}(i, s, x,\xi) - \sum_{r=1}^{d_\omega}\sum_{|\alpha| \leq N }(\Delta_{\xi}^\alpha\sigma_{B}(r, s, x,\xi))(\partial_{X}^{(\alpha)} \sigma_{A}(i, r, x,\xi))  \right)\\
    &\hspace{2cm}\in {S}^{m_1+m_2-(\rho-\delta)(N+1)-\rho\ell+\delta|\beta|,\mathcal{L}}_{\rho,\delta}((G\times \widehat{G})\otimes \textnormal{End}(E_0,H_0)),
 \end{align*}for every difference operator $\Delta_{\xi}^{\alpha_\ell}$ of order $\ell\in\mathbb{N}_0,$ follows immediately from Theorem \ref{SubellipticcompositionC} and  \eqref{Vish3.5e}. 
\end{proof}

\begin{definition}\label{Defi:adjointVect:Kernel}
Let $E_0^{*}$ and $F_{0}^{*}$ be the dual spaces of $E_0$ and $F_{0},$ respectively.   The formal adjoint of a vector-valued continuous linear operator $A:C^{\infty}(G,E_0)\rightarrow C^{\infty}(G,F_0)$ is the unique  vector-valued continuous linear operator $A^{*}:C^{\infty}(G,F_0^*)\rightarrow C^{\infty}(G,E_0^{*})$ satisfying the identity
 \begin{equation}
     ( h, Af ):=\int\limits_{G}( h(g), Af(g) ) dg=\int\limits_{G}( A^{*}h(g), f(g) ) dg=:( A^*h, f ),
 \end{equation} for any  $h\in C^{\infty}(G,F_0^*), $ $f\in C^{\infty}(G,E_0).$ 
\end{definition}
\begin{remark}[Kernel of the adjoint] Let $K_A$ be the Schwartz kernel of $A.$ Then, for any $g_1,g_2\in G,$ $K_A(g_1,g_2)\in\textnormal{ End}(E_0,F_0), $ and let us consider its adjoint $K_A(g_1,g_2)^{*}\in \textnormal{ End}(F_0^*,E_0^*),$ which is defined via
\begin{equation*}
    ( v_1,K_A(g_1,g_2)v_2 )=( K_A(g_1,g_2)^{*}v_1,v_2),\,\,v_1\in F_{0}^*,\,\,v_2\in E_0.
\end{equation*} By applying Fubini theorem  we have
\begin{align*}
     ( h, Af ) &:=\int\limits_{G}( h(g), Af(g) ) dg= \int\limits_{G}\int\limits_{G}\left( h(g), K_A(g,y)f(y)\right) dy \, dg\\
     &= \int\limits_{G}\int\limits_{G}\left( K_A(g,y)^{*}h(g), f(y) \right) dg\,dy\\
     &=\int\limits_{G}\left( \,\int\limits_{G} K_A(g,y)^{*}h(g)dg, f(y) \right) dy= ( A^*h, f ).
\end{align*}This shows that the Schwartz kernel  $K_{A^{*}}$ of $A^{*}$ is given by 
\begin{equation*}
    K_{A^*}(y,g)=K(g,y)^{*},\,\,g,y\in G.
\end{equation*}

\end{remark}

\begin{corollary}\label{Coro:Adjoint}  Let us consider the identifications $E_{0}^{*}\cong E_0$ and $F_{0}^{*}\cong F_0,$ and let $\tau$ and $\omega$ be two unitary representations of $K$ with representation spaces $E_0$ and $F_0,$ respectively. Then, every  continuous linear operator   $A:C^\infty(G,E_{0})\rightarrow  C^\infty(G,F_{0}),$ maps  $ C^\infty(G,E_{0})^\tau$ into $ C^\infty(G,F_{0})^{\omega}, $ if and only if,  $A^*:C^\infty(G,F_{0}^*)\rightarrow  C^\infty(G,E_{0}^*),$ maps  $ C^\infty(G,F_{0}^*)^\omega$ into $ C^\infty(G,E_{0}^*)^{\tau}. $

\end{corollary}

\begin{proof} In view of the identifications  $E_{0}^{*}\cong E_0$ and $F_{0}^{*}\cong F_0,$ 
 $E_0^{*}$ and $F_{0}^{*}$ can be thought as representation spaces of $K.$ Now, from Theorem \ref{TheoremCharK}, $A:C^\infty(G,E_{0})\rightarrow  C^\infty(G,F_{0}),$ maps  $ C^\infty(G,E_{0})^\tau$ into $ C^\infty(G,F_{0})^{\omega}, $ if and only if, \begin{equation}
 K_{A}(g,y)=   \omega(k_1)K_{A}(gk_1,yk_2)\tau(k_2)^{-1},\,k_1,k_2\in K,\,\,g,y\in G,
\end{equation}which is equivalent to the fact that
\begin{align*}
 K_{A^{*}}(y,g)&=K_{A}(g,y)^{*} =   \tau(k_2)K_{A}(gk_1,yk_2)^{*}\omega(k_1)^{-1}\\
 &= \tau(k_2)K_{A^{*}}(yk_2,gk_1)\omega(k_1)^{-1},\, \,k_1,k_2\in K,\,\,g,y\in G,
\end{align*}which again, is equivalent to the fact that $A^*:C^\infty(G,F_{0})\rightarrow  C^\infty(G,E_{0}),$ maps  $ C^\infty(G,F_{0})^\omega$ into $ C^\infty(G,E_{0})^{\tau}. $
\end{proof}

The following theorem says that the vector-valued subelliptic classes are closed under taking adjoint. For every finite dimensional vector space $V,$ we  denote by $V^*$ the dual space of $V,$ and for a basis $(e_{i})_{i=1}^{\dim V}$ of $V$ let us denote by $(e_i^{*})_{i=1}^{\dim V}$ the dual basis.
\begin{theorem}\label{VectorAdjoint}
 Let $0\leq \delta<\rho\leq 1.$ If $A:C^\infty(G, E_0)\rightarrow C^\infty(G, F_0)$ is a continuous operator, $A \in \Psi^{m,\mathcal{L}}_{\rho,\delta}((G\times \widehat{G})\otimes \textnormal{End}(E_0,F_0))$ then $A^* \in \Psi^{m,\mathcal{L}}_{\rho,\delta}((G\times \widehat{G})\otimes \textnormal{End}(F_0^*,E_0^*)).$ The  symbol of $A^*,$ $\sigma_{A^*}\in \Sigma((G\times \widehat{G})\otimes \textnormal{End}(F_0^*,E_0^*)),$ satisfies the asymptotic expansion,
 \begin{equation}\label{asymptoticofA*}
    \sigma_{A^*}(i,s, x,\xi)\sim  \sum_{|\alpha|= 0}^\infty\Delta_{\xi}^\alpha\partial_{X}^{(\alpha)} \sigma_{A}(i,s, x, \xi)^*.
 \end{equation} 
 This means that, for every $N \in \mathbb{N}$,
 \begin{equation*}
     \Delta_{\xi}^{\alpha_\ell} \partial_X^{(\beta)} \left(\sigma_{A^*}( i,s,x,\xi) -  \sum_{|\alpha|\leq N }\Delta_{\xi}^\alpha\partial_{X}^{(\alpha)} \sigma_{A}(i,s, x, \xi)^* \right) 
 \end{equation*} belongs to the symbol class ${S}^{m-(\rho-\delta)(N+1)-\rho\ell+\delta|\beta|,\mathcal{L}}_{\rho,\delta}((G\times \widehat{G}) \otimes \textnormal{End}(F_0^*,E_0^*)).$  In particular, under the identifications $E_{0}^{*}\cong E_0$ and $F_{0}^{*}\cong F_0,$ the same asymptotic expansion also applies if additionally  $A:C^{\infty}(G,E_{0})^\tau\rightarrow C^{\infty}(G,F_{0})^\omega$ is a continuous linear operator. This means that \eqref{asymptoticofA*} remains valid for the continuous linear operator $A^*:C^{\infty}(G,F_{0})^\omega\rightarrow C^{\infty}(G,E_{0})^\tau.$
 \end{theorem}
 \begin{proof}  Since $A: C^\infty(G, E_0) \rightarrow C^\infty(G, F_0)$  the adjoint $A^*$ will act from $C^\infty(G, F_0^*)$ into $C^\infty(G, E_0^*)$ (see Definition \ref{Defi:adjointVect:Kernel}). Then, for  $g \in C^\infty(G, F_0^*),\, x \in G$ and $v \in E_0$ we have
 \begin{align*}
     A^*g(x)(v) &= \left( \sum_{r=1}^{d_\omega} A^*_{ir} g_r(x) \right) \left(\sum_{i=1}^{d_\tau} v_i e_{i,E_0}\right) 
     \\&= \sum_{i=1}^{d_\tau} \sum_{r=1}^{d_\omega} \sum_{[\xi]\in \widehat{G}} d_\xi \textnormal{Tr}\left[\xi(x) \sigma_{A_{ir}^*}(x, \xi) \widehat{g}(r, \xi) \right] v_ie_{i,E_0} \\
     &= \sum_{i=1}^{d_\tau} \sum_{r=1}^{d_\omega} \sum_{[\xi]\in \widehat{G}} d_\xi \textnormal{Tr}\left[\xi(x) \sigma_{A_{ir}^*}(x, \xi) \widehat{g}(r, \xi) \right] e_{i, E_0}^{*}(v)\\
     &\equiv \sum_{i=1}^{d_\tau} \sum_{r=1}^{d_\omega} \sum_{[\xi]\in \widehat{G}} d_\xi \textnormal{Tr}\left[\xi(x) \sigma_{A_{ir}^*}(x, \xi) \widehat{g}(r, \xi) \right] e_{i, E_0^*}(v),\,\,e_{i, E_0}^{*}=: e_{i, E_0^*}.
 \end{align*}
 As a consequence, we obtain 
 $$A^*g(x)= \sum_{i=1}^{d_\tau} \sum_{r=1}^{d_\omega} \sum_{[\xi]\in \widehat{G}} d_\xi \textnormal{Tr}\left[\xi(x) \sigma_{A_{ir}^*}(x, \xi) \widehat{g}(r, \xi) \right] e_{i, E_0^*}.$$
 Now, the uniqueness of the symbol gives that 
 $$\sigma_{A^*}(i,r,x,  \xi)= \sigma_{A^*_{ir}}(x, \xi)$$ for all $i, r, x , \xi.$
 Now, note that $A_{ir}^*$ usual pseudo-differential operator with symbol $\sigma_{A^*_{ir}}(x, \xi)$ so by Theorem \ref{AdjointC}, we have 
 \begin{align} \label{visheq3.5}
     \sigma_{A^*}(i,r, x, \xi) = \sigma_{A^*_{ir}}(x, \xi) &\sim \sum_{|\alpha|= 0}^\infty\Delta_{\xi}^\alpha\partial_{X}^{(\alpha)} \sigma_{A_{ir}}(x, \xi)^* \\& = \sum_{|\alpha|= 0}^\infty\Delta_{\xi}^\alpha\partial_{X}^{(\alpha)} \sigma_{A}(i, r, x, \xi)^*. \nonumber
 \end{align}
  Now the asymptotic expansion \eqref{asymptoticofA*} follows from Theorem \ref{AdjointC} by using \eqref{visheq3.5}. Finally, if we make the identifications $E_{0}^{*}\cong E_0$ and $F_{0}^{*}\cong F_0,$ then $E_0^{*}$ and $F_{0}^{*}$ can be understood as representation spaces of $K,$ and \eqref{asymptoticofA*} remains valid for the continuous linear operator $A^*:C^{\infty}(G,F_{0})^\omega\rightarrow C^{\infty}(G,E_{0})^\tau,$ in view of Corollary \ref{Coro:Adjoint}. 
 \end{proof}
 
\subsection{Vector-valued functional calculus of the sub-Laplacian}

The aim of this subsection is to establish the stability of the functional calculus of the {\it vector-valued sub-Laplacian}
 \begin{equation}\label{Vec:sub}
     \mathcal{L}_{E_0}:=\textnormal{Op}\left(\widehat{\mathcal{L}}(\xi) \otimes I_{ \textnormal{End}(E_0)}\right):C^{\infty}(G,E_0)\rightarrow C^{\infty}(G,E_0),
 \end{equation} with respect to the vector-valued subelliptic H\"ormander classes. In \eqref{Vec:sub}, $\mathcal{L}$ denotes the matrix-valued symbol of the  sub-Laplacian $\mathcal{L}$  associated to a H\"ormander system of vector fields  $X=\{X_{i}\}_{i=1}^{k}$ at step $\kappa.$   
 For $m\in \mathbb{R},$ let us consider the Euclidean class of symbols $S^{m}(\mathbb{R}^{+}_0),$ defined by the countable family of seminorms 
\begin{equation}
    \Vert f \Vert_{d,\,S^{m}(\mathbb{R}^{+}_0)}:=\sup_{\ell\leq d}\sup_{\lambda\geq 0}(1+\lambda)^{-m+\ell}|\partial^\ell_{\lambda}f(\lambda)|,\,d\in \mathbb{N}_0.
\end{equation}
For any $f\in S^{m}(\mathbb{R}^{+}_0), $ we define
 \begin{equation}\label{Vec:sub:f}
    f( \mathcal{L}_{E_0}):=\textnormal{Op}\left(f(\widehat{\mathcal{L}}(\xi)) \otimes I_{ \textnormal{End}(E_0)}\right):C^{\infty}(G,E_0)\rightarrow C^{\infty}(G,E_0),
 \end{equation} with $f(\widehat{\mathcal{L}}(\xi)),$ $[\xi]\in \widehat{G},$ defined by the spectral calculus of matrices.
The main theorem of this subsection is the following.
\begin{theorem}\label{orders:theorems:VB} For $f\in S^{\frac{m}{2}}(\mathbb{R}^{+}_0), $ $m\in \mathbb{R},$ we have  
\begin{equation}
   f(\mathcal{L}_{E_0})\in \Psi^{m,\mathcal{L}}_{1,0}((G\times \widehat{G})\otimes \textnormal{End}(E_0)). 
\end{equation}
 In particular, for $f_s(\lambda):=(1+\lambda)^{\frac{s}{2}},$ the symbol of $\mathcal{M}_{s,E_0}=f_s(\mathcal{L}_{E_0})$ belongs to $ S^{s,\mathcal{L}}_{1,0}( (G\times \widehat{G})\otimes \textnormal{End}(E_0)) $ for all $s\in \mathbb{R}.$ In other words, we have 
\begin{equation}\label{thepowersof:LE0}
    \mathcal{M}_{s,E_0}:=(1+\mathcal{L}_{E_0})^{\frac{s}{2}}\in \Psi^{s,\mathcal{L}}_{1,0}((G\times \widehat{G})\otimes\textnormal{End}(E_0))
\end{equation}
for all $s\in \mathbb{R}.$
\end{theorem}
We first show the following result which is used to prove  Theorem \ref{orders:theorems:VB}.
\begin{lemma}\label{Symbol:Lemma:VB}The matrix-valued symbol of the vector-valued operator $ f(\mathcal{L}_{E_0})$ is given by
\begin{equation}
     \sigma_{f(\mathcal{L}_{E_0})}(i,r,[\xi])\equiv  f(\widehat{\mathcal{L}}_{E_0})(i,r,\xi)=f(\widehat{\mathcal{L}}(\xi))\delta_{i,r},\,\,1\leq i,r\leq d_\tau,\,\,[\xi]\in \widehat{G},
\end{equation} where $\delta_{i,r}$ is the Kronecker-Delta.
\end{lemma} 
\begin{proof}
Keeping in mind Remark \ref{remarkendvaluedmatrixvalued}, let us observe that the $\textnormal{End}(E_0)$-valued symbol  of $f({\mathcal{L}}_{E_0}),$ is given by
 \begin{equation*}
     f(\widehat{\mathcal{L}}_{E_0}(\xi))=(f(\widehat{\mathcal{L}}_{E_0}(\xi))_{u,v=1}^{d_\xi}),\,\,f(\widehat{\mathcal{L}}_{E_0}(\xi))_{u,v}:=f(\widehat{\mathcal{L}}(\xi))_{uv}I_{\textnormal{End}(E_0)},
 \end{equation*} so that
 \begin{equation*}
   f(\widehat{\mathcal{L}}_{E_0}(\xi))= f(\widehat{\mathcal{L}}(\xi)) \otimes I_{ \textnormal{End}(E_0)}  \in \mathbb{C}^{d_\xi\times d_\xi}( \textnormal{End}(E_0)).
 \end{equation*}Now, notice that the matrix-valued symbol of $f({\mathcal{L}}_{E_0})$ can be computed as follows: 
 \begin{align*}
     f(\widehat{\mathcal{L}}_{E_0})(i,r,\xi)&= (f(\widehat{\mathcal{L}}_{E_0})(i,r,\xi)_{u,v})_{u,v=1}^{d_\xi} =(e_{r,E_0}^{*}( f(\widehat{\mathcal{L}}_{E_0}(\xi))_{uv}e_{i,E_0}) )_{u,v=1}^{d_\xi}\\
     &=(e_{r,E_0}^{*}( f(\widehat{\mathcal{L}}(\xi))_{uv}e_{i,E_0}) )_{u,v=1}^{d_\xi}=(f(\widehat{\mathcal{L}}(\xi))_{uv})_{u,v=1}^{d_\xi}\delta_{ir}\\
     &=f(\widehat{\mathcal{L}}(\xi))\delta_{i,r}.
 \end{align*}Thus, we end the proof.
\end{proof} 
 \begin{proof}[Proof of Theorem \ref{orders:theorems:VB}] In view of Lemma \ref{Symbol:Lemma:VB}, in order to prove that  $f(\mathcal{L}_E)\in \Psi^{m,\mathcal{L}}_{1,0}((G\times \widehat{G})\otimes \textnormal{End}(E_0)),$ it is enough to check the validity of the following symbol inequalities
 \begin{equation}\label{InI:ex}
       \sup_{\substack{1\leq i,r\leq d_\tau\,  [\xi]\in  \widehat{G} }
      } \Vert \widehat{ \mathcal{M}}(\xi)^{\gamma+|\alpha|-m} \Delta_{\xi}^{\alpha}[f(\widehat{\mathcal{L}}(\xi))]\delta_{i,r}\mathcal{M}(\xi)^{-\gamma} \Vert_{\textnormal{op}} <\infty,
   \end{equation} for any $\gamma\in \mathbb{R}.$ But it was proved in Theorem 8.17 of  \cite{RuzhanskyCardona2020} that $f(\mathcal{L})\in S^{m,\mathcal{L}}_{1,0}(G\times \widehat{G}),$ from which the validity of \eqref{InI:ex} follows. The proof is complete.
 \end{proof}

\subsection{Mapping properties for subelliptic operators on vector bundles}

In this subsection we will discuss the mapping properties of subellipic pseudo-differential  operators. Let us recall that an operator $A:C^\infty(G, E_0) \rightarrow C^\infty(G, F_0)$ can be expressed, in view of \eqref{cutoff} as 
\begin{align}\label{sumofoperators}
    Af(x)= \sum_{i=1}^{d_\tau} \sum_{r=1}^{d_\omega} A_{ir}f_i(x) e_{r, F_0},
\end{align}
where $A_{ir}f_i$ is defined on the complex valued function $f_i(x)= \langle f(x), e_{i, E_0}\rangle$ for every $1 \leq i \leq d_\tau, 1 \leq r \leq d_\omega.$ 
\begin{lemma}\label{Lemmafor boundedness} Let $\mathcal{F}_1(G)$ and $\mathcal{F}_2(G)$ be two normed linear spaces containing $C^\infty(G)$ as a dense subspace. Let  $\mathcal{F}_1(G, E_0)=\mathcal{F}_1(G)\otimes E_0$ and  $\mathcal{F}_2(G, F_0)= \mathcal{F}_2(G)\otimes F_0,$ be defined by the norms
\begin{equation}
    \Vert  f_1\Vert_{\mathcal{F}_1(G, E_0)} = \Vert \Vert f_1(\cdot)\Vert_{E_0} \Vert_{\mathcal{F}_1(G)}\,\,\Vert  f_2\Vert_{\mathcal{F}_2(G, F_0)} = \Vert \Vert f_2(\cdot)\Vert_{F_0} \Vert_{\mathcal{F}_2(G)}.
\end{equation}
Then 
$A$ admits a  bounded extension from  $\mathcal{F}_1(G, E_0)$ to  $\mathcal{F}_2(G, F_0),$  if and only if, for all $1\leq i\leq d_\tau,$ and $1\leq r\leq d_\omega,$ $A_{ir}$ in \eqref{sumofoperators}, admits a bounded extension  from $\mathcal{F}_1(G)$ to $\mathcal{F}_2(G).$ 

\end{lemma}
\begin{proof} Let us first assume that for all $1 \leq i \leq d_\tau$ and $1 \leq r \leq d_\omega,$ $ A_{ir} $ admit a bounded extension for $\mathcal{F}_1(G)$ into $\mathcal{F}_2(G)$, that is, for all $g \in C^\infty(G),$ 
\begin{equation}\label{boundedenessir}
    \Vert A_{ir} g \Vert_{\mathcal{F}_2(G)} \leq  C \Vert g \Vert_{\mathcal{F}_1(G)}, \quad \textnormal{for some}\,\, C>0,\,\forall \,1\leq i\leq d_{\tau},\,1\leq r\leq d_\omega.
\end{equation}
Now, for $f \in C^\infty(G, E_0)$ we have 
\begin{align*}
    &\Vert A f\Vert_{\mathcal{F}_2(G, F_0)} = \Vert \Vert Af(\cdot)\Vert_{F_0} \Vert_{\mathcal{F}_2(G)} = \Vert \Vert \sum_{i=1}^{d_\tau} \sum_{r=1}^{d_\omega} A_{ir}f_i(\cdot) e_{r, F_0} \Vert_{F_0} \Vert_{\mathcal{F}_2(G)} \\& \leq \Vert \sum_{i=1}^{d_\tau} \sum_{r=1}^{d_\omega} \Vert  A_{ir}f_i(x) e_{r, F_0} \Vert_{F_0} \Vert_{\mathcal{F}_2(G)}
     =  \Vert \sum_{i=1}^{d_\tau} \sum_{r=1}^{d_\omega} |A_{ir}f_i(\cdot)| \Vert   e_{r, F_0} \Vert_{F_0} \Vert_{\mathcal{F}_2(G)}.
\end{align*}Using \eqref{boundedenessir}, we have
\begin{align*}
  \Vert A f\Vert_{\mathcal{F}_2(G, F_0)} &  \leq \sum_{i=1}^{d_\tau} \sum_{r=1}^{d_\omega} \Vert  A_{ir}f_i  \Vert_{\mathcal{F}_2(G)} \leq   \sum_{i=1}^{d_\tau} \sum_{r=1}^{d_\omega} C \Vert f_i  (\cdot)\Vert_{\mathcal{F}_1(G)}.
  \end{align*}Using the Cauchy-Schwarz inequality, we have
  \begin{align*}
       \Vert f_i  (\cdot)\Vert_{\mathcal{F}_1(G)}\leq  \Vert |\langle f(\cdot), e_{r, E_0} \rangle| \Vert_{\mathcal{F}_1(G)}\leq \Vert  \Vert f(\cdot) \Vert_{E_0} \Vert e_{r, E_0}\Vert_{E_0} \Vert_{\mathcal{F}_1(G)},
  \end{align*}and consequently,
  \begin{align*}
 \Vert A f\Vert_{\mathcal{F}_2(G, F_0)} &\leq  \sum_{i=1}^{d_\tau} \sum_{r=1}^{d_\omega} C \Vert  \Vert f(\cdot) \Vert_{E_0}  \Vert_{\mathcal{F}_1(G)}= \sum_{i=1}^{d_\tau} \sum_{r=1}^{d_\omega} C \Vert f  \Vert_{\mathcal{F}_1(G, E_0)}= C' \Vert f  \Vert_{\mathcal{F}_1(G, E_0)}.
\end{align*}
Therefore, by using the density of $C^\infty(G, E_0)$ we deduce that $A$ has a bounded extension from $\mathcal{F}_1(G, E_0)$ to $\mathcal{F}_2(G, F_0).$ For the converse part, let us assume that for $h\in C^\infty(G, E_0)$ we have $$\Vert Ah\Vert_{\mathcal{F}_2(G, F_0)} \leq C_1\Vert h \Vert_{\mathcal{F}_1(G, E_0)}.$$
Then for $f \in \mathcal{F}_1(G)$ and for $1 \leq j \leq d_\tau,$ consider the function $f_{e_{j, E_0}}:G \rightarrow E_0$ defined by $f_{e_{j, E_0}}(x):= f(x) e_{j, E_0}.$ Then it is easy to see that $f_{e_{j, E_0}} \in \mathcal{F}_1(G, E_0).$ Now, 
\begin{align*}
    A f_{e_{j, E_0}}(x)= \sum_{i=1}^{d_\tau} \sum_{r=1}^{d_\omega} A_{ir}(f_{e_{j, E_0}})_i(x) e_{r, F_0} = \sum_{r=1}^{d_\omega} A_{jr} f(x) e_{r, F_0},
\end{align*} 
and so, for $1 \leq s \leq d_\omega,$ we obtain $\langle A f_{e_{j, E_0}}(x), e_{s, F_0}   \rangle= A_{js}f(x),$ for $x \in G.$
Next, 
\begin{align*}
    \Vert A_{js}f \Vert_{\mathcal{F}_2(G)}&= \Vert |\langle A f_{e_{j, E_0}}(\cdot), e_{s, F_0}   \rangle| \Vert_{\mathcal{F}_2(G)} \\
    &\leq \Vert \Vert A f_{e_{j, E_0}}(\cdot) \Vert_{F_0} \Vert e_{s, F_0}\Vert_{F_0} \Vert_{\mathcal{F}_2(G)}  \\& = \Vert \Vert A f_{e_{j, E_0}}(\cdot) \Vert_{F_0}  \Vert_{\mathcal{F}_2(G)}\\
    &=\Vert A f_{e_{j, E_0}} \Vert_{\mathcal{F}_2(G, F_0)}\\
    &\leq C_1 \Vert  f_{e_{j, E_0}} \Vert_{\mathcal{F}_1(G, F_0)} \\&= C_1 \Vert  f e_{j, E_0} \Vert_{\mathcal{F}_1(G, E_0)} = C_1 \Vert \Vert f(\cdot) e_{j, E_0} \Vert_{E_0} \Vert_{\mathcal{F}_1(G)} = C_1 \Vert f \Vert_{\mathcal{F}_1(G)},   
\end{align*} for all $f \in C^\infty(G).$ Therefore, for each $1 \leq j \leq d_\tau$ and $1 \leq s \leq d_\omega,$ $A_{js}$ admits a bounded extension from $\mathcal{F}_1(G)$ into $F_{2}(G).$  \end{proof}

In view of Lemma \ref{Lemmafor boundedness} and Theorem \ref{CVT} we obtain the following subelliptic vector-valued Calder\'on-Vaillancourt theorem.

\begin{theorem}\label{CVT1}
Let $G$ be a compact Lie group and let us consider the sub-Laplacian $\mathcal{L}=\mathcal{L}_X,$ where  $X=\{X_{i}\}_{i=1}^{k}$ is a system of vector fields satisfying the H\"ormander condition of order $\kappa$.  For  $0\leq \delta< \rho\leq    1,$  (or $0\leq \delta\leq \rho\leq 1,$ and $\delta<1/\kappa$) let us consider a continuous linear operator $A:C^\infty(G, E_0)\rightarrow\mathscr{D}'(G, F_0)$ with symbol  $\sigma\in {S}^{0,\mathcal{L}}_{\rho,\delta}((G\times \widehat{G})\otimes \textnormal{End}(E_0, F_0))$. Then $A$ extends to a bounded operator from $L^2(G, E_0)$ to  $L^2(G, F_0),$ and 
\begin{equation}
    \Vert  A\Vert_{\mathscr{B}(L^2(G,E_0), L^2(G, F_0))}\leq C \Vert \sigma\Vert_{\ell, {S}^{0,\mathcal{L}}_{\rho,\delta}((G\times \widehat{G})\otimes \textnormal{End}(E_0, F_0)) },
\end{equation} for $\ell \in \mathbb{N}$ large enough.  In particular, if $A$ maps continuously $C^{\infty}(G,E_0)^\tau$ into $C^{\infty}(G,F_0)^\omega,$ then $A$ extends to a bounded operator from $L^2(G, E_0)^{\tau}$ to  $L^2(G, F_0)^{\omega}.$  
\end{theorem}

Again,  Lemma \ref{Lemmafor boundedness} implies the following  $L^p$-theorem of boundedness for vector-valued subelliptic pseudo-differential operators as a consequence of Theorem \ref{parta2}.
\begin{theorem}\label{parta22}
Let $G$ be a compact Lie group and let us denote by $Q$ the Hausdorff
dimension of $G$ associated to the control distance associated to the sub-Laplacian $\mathcal{L}=\mathcal{L}_X,$ where  $X=\{X_{i}\}_{i=1}^{k}$ is a system of vector fields satisfying the H\"ormander condition of order $\kappa$.  For  $0\leq \delta<\rho\leq 1,$ let us consider a continuous linear operator $A:C^\infty(G, E_0)\rightarrow\mathscr{D}'(G, F_0)$ with symbol  $\sigma\in {S}^{-m,\mathcal{L}}_{\rho,\delta}((G\times \widehat{G})\otimes \textnormal{End}(E_0, F_0))$, $m\geq 0$. Then $A$ extends to a bounded operator from $L^p(G, E_0)$ to $L^p(G, F_0)$ provided that 
\begin{equation*}
    m\geq m_p:= Q(1-\rho)\left|\frac{1}{p}-\frac{1}{2}\right|.
 \end{equation*} In particular,  if $A$ maps continuously $C^{\infty}(G,E_0)^\tau$ into $C^{\infty}(G,F_0)^\omega,$ then $A$ extends to a bounded operator from $L^p(G, E_0)^{\tau}$ to  $L^p(G, F_0)^{\omega}.$  
\end{theorem}

Now, Theorem \ref{CVT1} provides a vector-valued version of Corollary \ref{sobcont}. 

\begin{theorem} \label{VectSobosub}
Let $A:C^{\infty}(G, E_0)\rightarrow \mathscr{D}'(G, F_0)$ be a continuous linear operator with symbol $a\in S^{m,\mathcal{L}}_{\rho,\delta}((G\times \widehat{G})\otimes \textnormal{End}(E_0, F_0)),$ $0\leq \delta< \rho\leq    1,$ (or $0\leq \delta\leq \rho\leq 1,$ $\delta<1/\kappa$). Then $A:L_{s}^{2, \mathcal{L}}(G, E_0)\rightarrow L_{s-m}^{2,\mathcal{L}}(G, F_0) $ extends to a bounded operator for all $s\in \mathbb{R}.$ In particular,  $A:L_{s}^{2, \mathcal{L}}(G, E_0)^\tau \rightarrow L_{s-m}^{2,\mathcal{L}}(G, F_0)^\omega $ extends to a bounded operator for all $s\in \mathbb{R}.$
\end{theorem}
\begin{proof}
From Definition  \ref{Defivectorsob} and  \eqref{eq265} we have 
\begin{align}
    \Vert Af \Vert_{L_{s-m}^{2,\mathcal{L}}(G, F_0)} = \Vert \mathcal{M}_{s-m} A f\Vert_{L^2(G, F_0)},
\end{align} where $\mathcal{M}_{\lambda}:=\textnormal{Op}\left(\widehat{\mathcal{M}}(\xi)^{\lambda} \otimes I_{ \textnormal{End}(E_0)}\right),$ $\lambda\in \mathbb{R}.$

Now observing that $\mathcal{M}_{s-m} A \mathcal{M}_{-s} \in \Psi^{0,\mathcal{L}}_{\rho,\delta}((G\times \widehat{G})\otimes \textnormal{End}(E_0, F_0)) $ and  then using the vector-valued Calder\'on-Vaillancourt theorem (Theorem \ref{CVT1}) we obtain
\begin{align*}
    \Vert Af \Vert_{L_{s-m}^{2,\mathcal{L}}(G, F_0)}&= \Vert \mathcal{M}_{s-m} A \mathcal{M}_{-s} \mathcal{M}_s f\Vert_{L^2(G, F_0)}\\& \leq C \Vert \mathcal{M}_s f\Vert_{L^2(G, E_0)}=\Vert f \Vert_{L_{s-m}^{2,\mathcal{L}}(G, E_0)}. 
\end{align*}
This completes the proof of the theorem. \end{proof}

\subsection{Quantisation on homogeneous vector bundles} In this subsection we obtain a quantisation formula on homogeneous vector bundles. We will follow the notation in \eqref{notationfunction} for the inverse mapping $\varkappa_{\omega}^{-1}:C^\infty(G,F_0)^\omega\rightarrow\Gamma^\infty(F).$ In view that the following  diagram \begin{eqnarray}
    \begin{tikzpicture}[every node/.style={midway}]
  \matrix[column sep={10em,between origins}, row sep={4em}] at (0,0) {
    \node(R) {$\Gamma^\infty(E)$}  ; & \node(S) {$\Gamma^\infty(F)$}; \\
    \node(R/I) {$C^\infty(G,E_0)^\tau$}; & \node (T) {$C^\infty(G,F_0)^\omega$};\\
  };
  \draw[<-] (R/I) -- (R) node[anchor=east]  {$\varkappa_{\tau}$};
  \draw[->] (R) -- (S) node[anchor=south] {$\tilde{A}$};
  \draw[->] (S) -- (T) node[anchor=west] {$\varkappa_\omega$};
  \draw[->] (R/I) -- (T) node[anchor=north] {$A$};
\end{tikzpicture}
\end{eqnarray} commutes, we have the identity $\tilde{A}s=\varkappa_\omega^{-1} A\varkappa_\tau s,$ for any section $s\in \Gamma^\infty(E).$ Since
\begin{equation}   { 
  (A\varkappa_\tau s)(g)= \,\sum_{[\xi]\in \widehat{G}}\sum_{q=1}^{d_\tau}\sum_{r=1}^{d_\omega}\textnormal{Tr}[\xi(g)\sigma_A(q,r,g,[\xi])\widehat{\varkappa_\tau  s}(q,[\xi])]e_{r,F_0}  ,   }
\end{equation} when applying $\varkappa_\omega^{-1}$ in both sides of the previous identity, we obtain
\begin{align*}
   \tilde{A}s(gK)&= (\varkappa_\omega^{-1} A\varkappa_\tau s)(gK)=[g,(A\varkappa_\tau s)(g)]:=(g, (A\varkappa_\tau s)(g))\cdot K\\
   &=\left(g, \sum_{[\xi]\in \widehat{G}}\sum_{q=1}^{d_\tau}\sum_{r=1}^{d_\omega}\textnormal{Tr}[\xi(g)\sigma_A(q,r,g,[\xi])\widehat{\varkappa_\tau  s}(q,[\xi])]e_{r,F_0} \right)\cdot K.
\end{align*}
The previous analysis  proves the following quantisation theorem.

 

\begin{theorem}[Quantisation formula on homogeneous vector bundles]\label{mainqhomovec}  Let $\tilde{A}:\Gamma^\infty(E)\rightarrow\Gamma^\infty(F)$ be a continuous linear operator and let $A: C^\infty(G,E_{0})^\tau\rightarrow C^\infty(G,F_{0})^\omega$ be the induced operator by the identifications $\Gamma^\infty(E)\cong C^\infty(G,E_{0})^\tau$ and $\Gamma^\infty(F)\cong C^\infty(G,F_{0})^\omega.$ Then we have
\begin{equation}\label{quantizationonhomogeneous2}
  \boxed{     \tilde{A}s(gK)\equiv \left(g, \,\sum_{[\xi]\in \widehat{G}}\sum_{i=1}^{d_\tau}\sum_{r=1}^{d_\omega}\textnormal{Tr}[\xi(g)\sigma_A(i,r,g,[\xi])  \widehat{\varkappa_\tau  s}(i,[\xi])]   e_{r,F_0}\right)\cdot K}
\end{equation}for every section $s\in \Gamma^\infty(E),$ where $\sigma_{A}$ is the matrix-valued symbol of $A.$
\end{theorem} 

\begin{remark} Let us recall that the symbol $\sigma_A,$ can be obtained from $A=\varkappa_\omega\tilde{A}\varkappa_\tau^{-1}$ via
\begin{equation}\boxed{
     \sigma_{A}(i_0,r_0,x,\xi)= \xi(x)^{*}
     e_{r_0,F_0}^{*}[A(\xi\otimes e_{i_0,E_0})(x)]   }
\end{equation}  in view of \eqref{Homogeneoussymbol'}. Because of  Theorem \ref{mainqhomovec}, we will introduce in the following definition the global matrix-valued symbol associated to a continuous linear operator between the spaces of sections $\Gamma^\infty(E)$ and $\Gamma^\infty(F).$ 
\end{remark}

\begin{definition}[Matrix-valued symbols on homogeneous vector bundles]\label{symbolofAtilde} Let $\tilde{A}:\Gamma^\infty(E)\rightarrow\Gamma^\infty(F)$ be a continuous linear operator.
In view of Theorem \ref{mainqhomovec},  we will refer to $\sigma_{\tilde{A}}:=\sigma_A(\cdot,\cdot,\cdot,\cdot)$ as the matrix-valued symbol of $\tilde{A}.$ Keeping in mind the defined identity $$\sigma_{\tilde{A}}(i,r,g,[\xi]):=\sigma_A(i,r,g,[\xi]),$$ we will refer also to $\sigma_A$ the matrix-valued symbol of $\tilde{A}.$ 
\end{definition} 
\subsection{$L^p$-estimates for H\"ormander classes on homogeneous vector-bundles} 
Having defined a matrix-valued Fourier\,\, transform for sections and proved a quantisation formula on homogeneous vector bundles, 
let us now introduce the subelliptic H\"ormander classes.

\begin{definition}[Subelliptic H\"ormander classes on homogeneous vector bundles]\label{Hormanderclassesonvectorbunbdles}Let $G$ be a compact Lie group and let  $0\leq \delta,\rho\leq 1.$
  Let us consider the sub-Laplacian $\mathcal{L}=\mathcal{L}_X,$ where  $X=\{X_{i}\}_{i=1}^{k}$ is a system of vector fields satisfying the H\"ormander condition of order $\kappa$. 
  
  The subelliptic H\"ormander class of order $m,$ and of type $(\rho,\delta),$ $\Psi^{m,\mathcal{L}}_{\rho,\delta}(E,F),$ is defined by
  \begin{equation}\label{Subelliptichormanderclassesonvectorbundles}
     \Psi^{m,\mathcal{L}}_{\rho,\delta}(E,F):=\{\tilde{A}:\Gamma^\infty(E)\rightarrow\Gamma^\infty(F)|\,\,\,  \sigma_A\in {S}^{m,\mathcal{L}}_{\rho,\delta}((G\times \widehat{G})\otimes \textnormal{End}(E_0, F_0))  \} ,
  \end{equation}where the classes ${S}^{m,\mathcal{L}}_{\rho,\delta}((G\times \widehat{G})\otimes \textnormal{End}(E_0, F_0))$ were introduced in Definition \ref{symbolclass}.
\end{definition}
As a consequence of our analysis for vector-valued operators, and in view that the subelliptic calculus for vector-valued operators  is stable under compositions (Theorem \ref{VectorcompositionC}) and adjoints  (Theorem \ref{VectorAdjoint}), in the following theorem we establish the pseudo-differential calculus for subelliptic operators on homogeneous vector bundles. We will denote by $H$ a homogeneous vector bundle over $M=G/K,$ and $\Gamma^\infty(H)\cong C^\infty(G,H_0)^{\nu}.$

\begin{theorem}\label{pseudodifferentialcalculus} Let  $0\leq \delta<\rho\leq 1.$ If $\tilde{A} \in \Psi^{m_1,\mathcal{L}}_{\rho,\delta}(E,F)$ and $\tilde{B}\in \Psi^{m_2,\mathcal{L}}_{\rho,\delta}(F,H)$ then the composition operator $\tilde{B}\tilde{A}:=\tilde{B}\circ\tilde{ A}:\Gamma^\infty(E)\rightarrow \Gamma^\infty(H)$ belongs to the subelliptic class $\Psi^{m_1+m_2,\mathcal{L}}_{\rho,\delta}(E,H).$ The symbol of $\tilde{B}\tilde{A},$ $\sigma_{\tilde{B}\tilde{A}}\in\Sigma ((G\times \widehat{G})\otimes \textnormal{End}(E_0,H_0)),$ satisfies the asymptotic expansion,
\begin{equation*}
    \sigma_{\tilde{B}\tilde{A}}(i, s, x,\xi) \sim \sum_{r=1}^{d_\omega}\sum_{|\alpha|= 0}^\infty(\Delta_{\xi}^\alpha\sigma_{\tilde{B}}(r, s, x,\xi))(\partial_{X}^{(\alpha)} \sigma_{\tilde{A}}(i, r, x,\xi)), 
\end{equation*} this means that, for every $N\in \mathbb{N},$
\begin{align*}
    &\Delta_{\xi}^{\alpha_\ell}\partial_{X}^{(\beta)}\left(\sigma_{\tilde{B}\tilde{A}}(i, s, x,\xi) - \sum_{r=1}^{d_\omega}\sum_{|\alpha| \leq N }(\Delta_{\xi}^\alpha\sigma_{\tilde B}(r, s, x,\xi))(\partial_{X}^{(\alpha)} \sigma_{\tilde A}(i, r, x,\xi))  \right)\\
    &\hspace{2cm}\in {S}^{m_1+m_2-(\rho-\delta)(N+1)-\rho\ell+\delta|\beta|,\mathcal{L}}_{\rho,\delta}((G\times \widehat{G})\otimes \textnormal{End}(E_0,H_0)),
 \end{align*}for every difference operator $\Delta_{\xi}^{\alpha_\ell}$ of order $\ell\in\mathbb{N}_0.$ Moreover, under the indentifications $F_0^{*}\cong F_0$ and $E_0^{*}\cong E_0,$ the symbol of the adjoint $\tilde{A}^* \in \Psi^{m,\mathcal{L}}_{\rho,\delta}(F,E)$ of $\tilde{A},$   satisfies the asymptotic expansion,
 \begin{equation*}
    \sigma_{\tilde{A}^*}(i,s, x,\xi)\sim  \sum_{|\alpha|= 0}^\infty\Delta_{\xi}^\alpha\partial_{X}^{(\alpha)} \sigma_{\tilde A}(i,s, x, \xi)^*.
 \end{equation*} 
 This means that, for every $N \in \mathbb{N}$,
 \begin{equation*}
     \Delta_{\xi}^{\alpha_\ell} \partial_X^{(\beta)} \left(\sigma_{\tilde A^*}(i,s, x,\xi) -  \sum_{|\alpha|\leq N }\Delta_{\xi}^\alpha\partial_{X}^{(\alpha)} \sigma_{\tilde A}( i,s,x, \xi)^* \right) 
 \end{equation*} belongs to the symbol class ${S}^{m-(\rho-\delta)(N+1)-\rho\ell+\delta|\beta|,\mathcal{L}}_{\rho,\delta}((G\times \widehat{G}) \otimes \textnormal{End}(F_0,E_0)).$
\end{theorem}
\begin{proof}  The asymptotic expansions were proved in Theorems \ref{VectorcompositionC} and \ref{VectorAdjoint}.
In view of the following two commutative diagrams,
\begin{eqnarray}
    \begin{tikzpicture}[every node/.style={midway}]
  \matrix[column sep={10em,between origins}, row sep={4em}] at (0,0) {
    \node(R) {$\Gamma^\infty(E)$}  ; & \node(S) {$\Gamma^\infty(F)$};  & \node(S') {$\Gamma^\infty(H)$}; \\
    \node(R/I) {$C^\infty(G,E_0)^\tau$}; & \node (T) {$C^\infty(G,F_0)^\omega$};& \node (T') {$C^\infty(G,H_0)^\nu$};\\
  };
  \draw[<-] (R/I) -- (R) node[anchor=east]  {$\varkappa_\tau$};
  \draw[->] (R) -- (S) node[anchor=south] {$\tilde{A}$};
  \draw[->] (S) -- (S') node[anchor=south] {$\tilde{B}$};
   \draw[->] (S') -- (T') node[anchor=west] {$\varkappa_\nu$};
    \draw[->] (T) -- (T') node[anchor=south] {$B$};
  \draw[->] (S) -- (T) node[anchor=west] {$\varkappa_\omega$};
  \draw[->] (R/I) -- (T) node[anchor=south] {$A$};
  \end{tikzpicture} 
\end{eqnarray}    and 
\begin{eqnarray}
    \begin{tikzpicture}[every node/.style={midway}]
  \matrix[column sep={10em,between origins}, row sep={4em}] at (0,0) {
    \node(R) {$\Gamma(F^*)$}  ; & \node(S) {$\Gamma(E^*)$}; \\
    \node(R/I) {$C^\infty(G,F_0)^\omega$}; & \node (T) {$C^\infty(G,E_0)^\tau$};\\
  };
  \draw[<-] (R/I) -- (R) node[anchor=east]  {$\varkappa_{\tau}$};
  \draw[->] (R) -- (S) node[anchor=south] {$\tilde{A}^*$};
  \draw[->] (S) -- (T) node[anchor=west] {$\varkappa_\omega$};
  \draw[->] (R/I) -- (T) node[anchor=north] {$A^*$};
\end{tikzpicture}
\end{eqnarray}
we can conclude the proof because of the identities $\widetilde{AB}=\tilde{A}\tilde{B}$ and $\widetilde{A^*}=(\tilde{A})^{*},$ which implies that the induced vector-valued  operators by $\tilde{A}\tilde{B}$ and  $(\tilde{A})^{*}$ are indeed, the continuous linear operators $AB$ and $A^{*}.$ 
\end{proof}

Now, we present following $L^p$-estimates of subelliptic pseudo-differential operators on homogeneous vector bundles over compact homogeneous manifolds.
\begin{theorem}[$L^2$-Calder\'on-Vaillancourt Theorem and $L^p$-Fefferman Theorem]\label{Lpthe:VB}  For  $0\leq \delta< \rho\leq    1$ (or $0\leq \delta\leq \rho\leq 1,$ $\delta<1/\kappa$), any continuous linear operator $\tilde{A}\in  \Psi^{0,\mathcal{L}}_{\rho,\delta}(E,F)$ extends to a bounded operator from  $L^{2}(E)$ to  $L^{2}(F).$ Moreover, if $$m\geq Q(1-\rho)\left|\frac{1}{2}-\frac{1}{p}\right|,$$  with $1<p<\infty,$ then $\tilde{A}\in  \Psi^{-m,\mathcal{L}}_{\rho,\delta}(E,F)$ with  $0\leq \delta< \rho\leq    1,$ extends to a bounded operator from  $L^{p}(E)$ into $L^{p}(F).$  

\end{theorem}
\begin{proof} In view that the following  diagram \begin{eqnarray}
    \begin{tikzpicture}[every node/.style={midway}]
  \matrix[column sep={10em,between origins}, row sep={4em}] at (0,0) {
    \node(R) {$\Gamma^\infty(E)$}  ; & \node(S) {$\Gamma^\infty(F)$}; \\
    \node(R/I) {$C^\infty(G,E_0)^\tau$}; & \node (T) {$C^\infty(G,F_0)^\omega$};\\
  };
  \draw[<-] (R/I) -- (R) node[anchor=east]  {$\varkappa_{\tau}$};
  \draw[->] (R) -- (S) node[anchor=south] {$\tilde{A}$};
  \draw[->] (S) -- (T) node[anchor=west] {$\varkappa_\omega$};
  \draw[->] (R/I) -- (T) node[anchor=north] {$A$};
\end{tikzpicture}
 \end{eqnarray} commutes, and the fact that the arrows $\varkappa_\tau$ and $\varkappa_\omega$ admit  unitary extensions, $\tilde{A}$ admits a bounded extension from $L^{2}(E)$ to  $L^{2}(F),$ if and only if, $A$  extends to a bounded operator from $L^2(G, E_0)^\tau$ to  $L^2(G, F_0)^\omega.$ But, this is consequence of the vector-valued Calder\'on-Vaillancourt theorem (Theorem \ref{CVT1}). Now consider $m\geq Q(1-\rho)\left|\frac{1}{2}-\frac{1}{p}\right|,$   $1<p<\infty,$ and  $\tilde{A}\in  \Psi^{-m,\mathcal{L}}_{\rho,\delta}(E,F),$  $0\leq \delta< \rho\leq    1.$ If $s\in \Gamma^\infty(E),$ then there exists a unique $f\in C^{\infty}(G,E_0)^{\tau}$ such that 
 \begin{equation*}
     \Vert  s\Vert_{L^p(E)}=\Vert \varkappa_\tau^{-1}f \Vert_{L^{p}(G\times_\tau E_0)}=\Vert g \mapsto(g,f(g))\cdot K \Vert_{L^{p}(G\times_\tau E_0)}=\Vert f\Vert_{L^p(G,E_0)}.
 \end{equation*} Also, observe that
 \begin{align*}
     \Vert \tilde A s\Vert_{L^p(E)}
     &=\Vert \varkappa_\tau^{-1}\tilde{A} f \Vert_{L^{p}(G\times_\omega F_0)}\\
     &=\Vert g \mapsto(g,Af(g))\cdot K \Vert_{L^{p}(G\times_\omega F_0)}=\Vert Af\Vert_{L^p(G,F_0)}.
  \end{align*} In view of the Fefferman theorem (Theorem \ref{parta22}), we deduce the inequality $$\Vert Af\Vert_{L^p(G,F_0)}\leq C\Vert f\Vert_{L^p(G,E_0)},$$ for some $C>0,$ from which we deduce the $L^p$-boundeness of $\tilde A.$ The proof is complete.
\end{proof}
Now, we have the following Sobolev estimate.
\begin{theorem} For $0\leq \delta< \rho\leq    1$ (or $0\leq \delta\leq \rho\leq 1,$ $\delta<1/\kappa$), any continuous linear operator $\tilde{A} \in \Psi^{m,\mathcal{L}}_{\rho. \delta}(E, F),$ extends to a bounded operator from  $L_{s}^{2, \mathcal{L}}(E)$ into $ L_{s-m}^{2,\mathcal{L}}(F)$ for all $s\in \mathbb{R}.$
\end{theorem}
\begin{proof} By Definition \ref{subvectsobo} and noting that $\varkappa_\omega \tilde{A} \varkappa_{\tau}^{-1}=A$ and $\varkappa_\tau s=f,$  we have 
\begin{align*}
    \Vert \tilde{A} s\Vert_{L_{s-m}^{2,\mathcal{L}}(F)}&= \Vert \mathcal{M}_{s-m} \varkappa_\omega \tilde{A} \varkappa_{\tau}^{-1} \varkappa_\tau s \Vert_{L^2(G, F_0)}\\&= \Vert \mathcal{M}_{s-m} Af \Vert_{L^2(G, F_0)} =\Vert  Af \Vert_{L^{2, \mathcal{L}}_{s-m}(G, F_0)}.
\end{align*}
Now, as an application of Theorem \ref{VectSobosub}, we get 
\begin{align*}
    \Vert \tilde{A} s\Vert_{L_{s-m}^{2,\mathcal{L}}(F)} &= \Vert  Af \Vert_{L^{2, \mathcal{L}}_{s-m}(G, F_0)}\leq C \Vert f \Vert_{L^{2, \mathcal{L}}_{s}(G, E_0)}\\&= C \Vert \varkappa_\tau s \Vert_{L^{2, \mathcal{L}}_{s}(G, E_0)} = \Vert s\Vert_{L_{s-m}^{2,\mathcal{L}}(F)}.
\end{align*}
This complete the proof.
\end{proof}

\subsection{$G$-invariant operators  and quantisation formula}

We begin with the following known 
lemma (see e.g. \cite[Page 956]{Ricci18}). We present a proof for completeness.
\begin{lemma}\label{G-inv:Fourier:Mult} Let $\tilde{A}:\Gamma^\infty(E)\rightarrow\Gamma^\infty(F)$ be a continuous linear operator and let $A: C^\infty(G,E_{0})^\tau\rightarrow C^\infty(G,F_{0})^\omega$ be the induced operator by the identifications $\Gamma^\infty(E)\cong C^\infty(G,E_{0})^\tau$ and $\Gamma^\infty(F)\cong C^\infty(G,F_{0})^\omega.$   The following conditions are equivalent.
\begin{itemize}
    \item[(A).] $\tilde{A}$ is $G$-invariant. This means that $\tilde{A}$ is invariant under the action of $G$ on sections: $\tilde{A}{\pi} (g)s={\pi} (g) \tilde{A}s,$ $s\in \Gamma^\infty(E),$ where $\pi(g) s$ denotes the action of $g \in G$ on $s$ given by  \eqref{actionsect} (see also equality \eqref{homogenesousdefinition}). 
    \item[(B).]  ${A}$ is invariant under the regular action of $G$ on functions. This means that $A\tilde{\pi} (g)f= \tilde{\pi} (g)Af,$ $f\in C^\infty(G,E_0)^\tau,$ where $\tilde{\pi}$ is defined as  $\tilde{\pi}(g_0)f(g):=f(g_0^{-1}g),$ for all $ f\in C(G,E_0)^\tau$ and $ g, g_0 \in G.$ This also implies that $A$ is a left-invariant operator with right-convolution kernel given by $$k(z)=K_A(z,e_{G})=K_A(e_G,z^{-1})$$ where $K_{A}\in C^\infty(G)\otimes_{\pi}\mathscr{D}'(E_{0}^*\otimes F_{0})$ is the Schwartz kernel of $A.$
\end{itemize}

\end{lemma}
\begin{proof} Let us define $f=\varkappa_\tau s,$ $s\in \Gamma^\infty(E).$ Observe that, $\tilde{A}$ is $G$-invariant if and only if
\begin{align*}
    \tilde{A}(\pi(g)s)=\pi(g)(\tilde{A}s) &\Longleftrightarrow \tilde{A}(\varkappa_\tau^{-1} \varkappa_\tau\pi(g)s)=\pi(g)(\tilde{A}\varkappa_\tau^{-1} \varkappa_\tau s)\\
    &\Longleftrightarrow \varkappa_{\omega} \tilde{A}(\varkappa_\tau^{-1} \varkappa_\tau\pi(g)s)=\varkappa_{\omega}\pi(g)(\tilde{A}\varkappa_\tau^{-1} \varkappa_\tau s)\\
    &\Longleftrightarrow \varkappa_{\omega} \tilde{A}(\varkappa_\tau^{-1} \tilde{\pi}(g)\varkappa_\tau s)=\tilde{\pi}(g)\varkappa_{\omega}(\tilde{A}\varkappa_\tau^{-1} \varkappa_\tau s)\\
    &\Longleftrightarrow A\tilde{\pi}(g)f=\tilde{\pi}(g)A f,
\end{align*} which is equivalent to the fact   that ${A}$ is invariant under the regular action of $G$ on functions. Now, in order to prove that $A$ as in (B) is left-invariant, let us use the Schwartz kernel theorem. So, if $K_{A}\in C^\infty(G)\otimes_{\pi}\mathscr{D}'(E_{0}^*\otimes F_{0})$ is the Schwartz kernel of $A,$ we have that
\begin{align*}
 A\tilde{\pi}{(g)}f(x) &=   \int\limits_{G}K_A(x,y)\tilde{\pi}{(g)}f(y)dy= \int\limits_{G}K_A(x,y)f(g^{-1}y)dy\\
 &=\tilde{\pi}{(g)}Af(x)\\
 &=\int\limits_{G}K_A(g^{-1}x,y)f(y)dy\\
 &=\int\limits_{G}K_A(g^{-1}x,g^{-1}y)f(g^{-1}y)dy.
\end{align*}This implies that, for every $g\in G,$ and $x,y\in G,$ $K_A(g^{-1}x,g^{-1}y)=K_A(x,y).$ In particular, for $g=x$ we have $$K_A(e_{G},x^{-1}y)=K_A(x,y).$$ So, if $k(z):=K_A(e_{G},z^{-1}),$ then $k(y^{-1}x)=K_A(x,y)$ and   we have that
$$ Af(x)=\int\limits_{G}k({y^{-1}x}) f(y)dy=(f\ast k)(x),$$which proves that $A$ is left-invariant. On the other hand, if in the identity $K_A(g^{-1}x,g^{-1}y)=K_A(x,y)$ we take $g=y,$ we have $K_{A}(x,y)=K_A({y^{-1}x,e_{G}})$ and from the uniqueness of the right-convolution kernel we have $k(z)=K_A(z,e_{G}).$
\end{proof}
\begin{remark} Let $\tilde{A}:\Gamma^\infty(E)\rightarrow\Gamma^\infty(F)$ be a continuous linear operator and let $A: C^\infty(G,E_{0})^\tau\rightarrow C^\infty(G,F_{0})^\omega$ be the induced operator by the identifications $\Gamma^\infty(E)\cong C^\infty(G,E_{0})^\tau$ and $\Gamma^\infty(F)\cong C^\infty(G,F_{0})^\omega.$ If $\tilde{A}$ is invariant, then the matrix-valued symbol of $A$ is given by 
\begin{equation*}
    \sigma_A(i,r,[\xi]):=\mathscr{F}_{G}({k}_{ir})(\xi),\,(i,r)\in I_{d_\tau}\times I_{d_\omega}, \,[\xi]\in \widehat{G}.
\end{equation*}Here, $k(z):=R_{A}(e_{G},z)$ is the right-convolution kernel of $A$ and ${k}_{ir}\in \mathscr{D}'(G)$ are defined by the decomposition of $k(z)$ in the induced basis of $\textnormal{End}(E_0,F_0)=E_{0}^{*}\otimes F_{0}.$ Indeed,
\begin{align*}
    Af(x)&=\int\limits_{G}k(z)f(xz^{-1})dz=\sum_{i=1}^{d_\tau}\sum_{i'=1}^{d_\tau}\sum_{r=1}^{d_\omega}\int\limits_G k_{i'r}(z)f_{i}(xz^{-1})( e_{i',E_0}^{*}\otimes e_{r,F_0})(e_{i,E_0})\\
    &=\sum_{i=1}^{d_\tau}\sum_{r=1}^{d_\omega}\int\limits_G k_{ir}(z)f_{i}(xz^{-1}) e_{r,F_0}\\
    &=\sum_{i,r,[\xi]\in \widehat{G}}d_{\xi}\textnormal{Tr}[\xi(x)\mathscr{F}_{G}({k}_{ir})(\xi)    \widehat{f}(i,\xi)   ]e_{r,F_0}\\
    &=:\sum_{i,r,[\xi]\in \widehat{G}}d_{\xi}\textnormal{Tr}[\xi(x)\sigma_A(i,r,\xi)(\xi)    \widehat{f}(i,\xi)   ]e_{r,F_0},
\end{align*} for all $f\in C^{\infty}(G,E_0)^{\tau}.$

\end{remark}
In view of \eqref{quantizationonhomogeneous2}, we have

\begin{corollary}  Let $\tilde{A}:\Gamma^\infty(E)\rightarrow\Gamma^\infty(F)$ be a $G$-invariant  operator and let $A: C^\infty(G,E_{0})^\tau\rightarrow C^\infty(G,F_{0})^\omega$ be the induced operator by the identifications $\Gamma^\infty(E)\cong C^\infty(G,E_{0})^\tau$ and $\Gamma^\infty(F)\cong C^\infty(G,F_{0})^\omega.$ Then we have
\begin{equation}\label{quantizationonhomogeneous3}
    \tilde{A}s(gK)\equiv \left(g, \,\sum_{[\xi]\in \widehat{G}}\sum_{i=1}^{d_\tau}\sum_{r=1}^{d_\omega}\textnormal{Tr}[\xi(g)\sigma_A(i,r,[\xi])  \widehat{\varkappa_\tau s}(i,[\xi])] e_{r,F_0}  \right)\cdot K,
\end{equation}for every section $s\in \Gamma^\infty(E),$ where $\sigma_{A}$ is the matrix-valued symbol of $A.$
\end{corollary}

\subsection{Characterisation of H\"ormander classes}

In this subsection we characterise the algebra of pseudo-differential operators on homogeneous vector bundles over compact homogeneous manifolds by using the notion of global symbol.

Now, we record the formulation of H\"ormander's pseudo-differential operators on closed manifolds (and so on compact Lie groups) by using local coordinate systems  (see \cite{Hormander1985III} and  \cite{Taylorbook1981} for more details). 

If $U$ is an open topological subset of $\mathbb{R}^n,$ we say that a symbol  $a:U\times \mathbb{R}^n\rightarrow \mathbb{C},$ belongs to the H\"ormander symbol class $S^m_{\rho,\delta}(U\times \mathbb{R}^n),$ $0\leqslant \rho,\delta\leqslant 1,$ if for every compact subset $K\subset U,$ the inequality,
\begin{equation*}
    |\partial_{x}^\beta\partial_{\xi}^\alpha a(x,\xi)|\leqslant C_{\alpha,\beta,K}(1+|\xi|)^{m-\rho|\alpha|+\delta|\beta|},
\end{equation*} holds true uniformly in $x\in K$ and $\xi\in \mathbb{R}^n.$ 
Then, a continuous linear operator $A:C^\infty_0(U) \rightarrow C^\infty(U)$ 
is a pseudo-differential operator of order $m$ and of  $(\rho,\delta)$-type if there exists
a function $a\in S^m_{\rho,\delta}(U\times \mathbb{R}^n)$ such that
\begin{equation*}
    Af(x)=\int\limits_{\mathbb{R}^n}e^{2\pi i x\cdot \xi}a(x,\xi)(\mathscr{F}_{\mathbb{R}^n}{f})(\xi)d\xi,
\end{equation*} for all $f\in C^\infty_0(U),$ where
\begin{equation*}
  (\mathscr{F}_{\mathbb{R}^n}{f})(\xi):=\int\limits_Ue^{-i2\pi x\cdot \xi}f(x)dx, \,\,\,\,\,\,\,\,\,\,\,\,\,\,\,\,\,\,\,\,\,\,\,\,\,\,\,\,\,\,\,\,\,\,\,\,\,\,\,\,\,
\end{equation*} is the  Euclidean Fourier transform of $f$ at $\xi\in \mathbb{R}^n.$ The class $S^m_{\rho,\delta}(U\times \mathbb{R}^n)$ on the phase space $U\times \mathbb{R}^n,$ is invariant under coordinate changes only if $\rho\geqslant   1-\delta,$ while a symbolic calculus (which means that symbol classes are closed by taking products, adjoints, parametrices, etc.) is only possible for $\delta<\rho$ and $\rho\geqslant   1-\delta.$  Now, we state the following definition.
\begin{definition}[Local H\"ormander classes on compact manifolds]\label{HormanderManifolds} Let $M$ be a $C^\infty$-closed manifold. A continuous linear operator $A:C^\infty(M)\rightarrow C^\infty(M) $ is a pseudo-differential operator of order $m$ and of $(\rho,\delta)$-type, $ \rho\geqslant   1-\delta, $ if for every local  coordinate patch $\omega: M_{\omega}\subset M\rightarrow U\subset \mathbb{R}^n,$
and for every $\phi,\psi\in C^\infty_0(U),$ the operator
\begin{equation*}
    Tu:=\psi(\omega^{-1})^*A\omega^{*}(\phi u),\,\,u\in C^\infty(U),\footnote{As usually, $\omega^{*}$ and $(\omega^{-1})^*$ are the pullbacks induced by the maps $\omega$ and $\omega^{-1},$ respectively.}
\end{equation*} is a pseudo-differential operator with symbol $a_T\in S^m_{\rho,\delta}(U\times \mathbb{R}^n).$ In this case we write that $A\in \Psi^m_{\rho,\delta}(M,\textnormal{loc}).$
It was proved in \cite{Ruz,RuzhanskyTurunenIMRN} by using the Beals characterisation \cite{Beals} that the H\"ormander classes defined via localisations and defined by global symbols are equivalent, this means that
\begin{equation}\label{equivalenceofclasses}
  \Psi^m_{\rho,\delta}(G\times \widehat{G})= \Psi^m_{\rho,\delta}(G,\textnormal{loc}), 
\end{equation}for $\delta<\rho$ and $\rho\geq 1-\delta.$

\begin{remark}\label{remrklocalvecspaces}
If $u\in C^\infty_0(\mathbb{R}^n,\mathbb{C}^{\ell}),$ the Fourier transform of $u$ is defined by taking the Fourier transform of its components. So, $A:\textnormal{Dom}(A)\subset C^\infty(M,\mathbb{C}^{\ell})\rightarrow C^\infty(M,\mathbb{C}^{r}), $ is a pseudo-differential operator in the class $\Psi^{m}_{\rho,\delta}(M;\mathbb{C}^\ell,\mathbb{C}^r,\textnormal{loc})$, if locally (with the notation in Definition \ref{HormanderManifolds}), we can write 
\begin{equation*}
    Tu:=\psi(\omega^{-1})^*A\omega^{*}(\phi u),\,\,u\in C^\infty(U, \mathbb{C}^{\ell}),
\end{equation*}and for some matrix of symbols $\sigma \in \mathbb{C}^{r\times \ell }(S^{m}_{\rho,\delta}(U\times \mathbb{R}^n)),$ we have 
\begin{equation}
  Tu(x)= \int\limits_{\mathbb{R}^n}e^{2\pi i y\cdot \xi}\sigma(y,\xi)(\mathscr{F}_{\mathbb{R}^n}u)(\xi)d\xi,\,\quad u\in C^{\infty}_0({U}, \mathbb{C}^{\ell}).
\end{equation}
\end{remark}
\end{definition}

Now, Definition \ref{HormanderManifolds} can be extended to vector bundles as follows.
\begin{definition}[Local H\"ormander classes on vector bundles]Let $E$ and $F$ be  vector bundles on $M$. A continuous linear operator  $\tilde{A}:\Gamma^\infty(E)\rightarrow\Gamma^\infty(F)$ belongs to the H\"ormander class of order $m,$ and of type $(\rho,\delta)$  (in the sense of H\"ormander), if for every $x\in M,$ there exists a neighborhood $U$ of $x\in M,$ a diffeomorphism of $U$ with an open subset  $\tilde{U}\subset \mathbb{R}^n,$ via  $\phi:\tilde{U}\rightarrow {U },$ and $C^\infty$-trivialisations $\Psi:E_{U}\rightarrow {U}\times \mathbb{C}^{\ell}$ and $\Phi:F_{U}\rightarrow {U}\times \mathbb{C}^{r},$ such that for every section $s\in \Gamma^\infty(E),$ and for every $\psi_1,\psi_2\in C^\infty_0(\tilde{U}),$ there exist $H\in C^{\infty}(U, \mathbb{C}^\ell)$ and  $G\in C^{\infty}(U, \mathbb{C}^r)$ with
\begin{equation}
    \Psi s(\phi(y))=(\phi(y), H(\phi(y))),\,\,\Phi As(\phi(y))=(\phi(y), G(\phi(y)) ),\,y\in \tilde{U},
\end{equation}and  some $\sigma \in \mathbb{C}^{r\times \ell }(S^{m}_{\rho,\delta}(\tilde{U}\times \mathbb{R}^n)),$ such that
\begin{equation*}
   G(\phi(y))=\int\limits_{\mathbb{R}^n}e^{2\pi i y\cdot \xi}\psi_2(y)\sigma(y,\xi)\mathscr{F}_{\mathbb{R}^n}[ \psi_1\cdot H\circ \phi](\xi)d\xi,\,\quad y\in \tilde{U}. 
\end{equation*}
For $\delta<\rho$ and $\rho\geq 1-\delta,$ we write that $A\in \Psi^m_{\rho,\delta}(E,F;\textnormal{loc}).$

\end{definition}

\begin{definition}[Vector-valued elliptic H\"ormander classes]\label{symbolclass:L}
   Let $G$ be a compact Lie group and let  $0\leq \delta,\rho\leq 1.$ Let $E_0$ and $F_0$ be complex vector spaces with finite dimensions $d_\tau$ and $d_\omega,$ respectively.  If $m\in \mathbb{R},$ the H\"ormander symbol class $$S^{m}_{\rho,\delta}((G\times \widehat{G})\otimes \textnormal{End}(E_0,F_0)):=S^{m,\mathcal{L}_G}_{\rho,\delta}((G\times \widehat{G})\otimes \textnormal{End}(E_0,F_0))$$ of  order $m$ and of type $(\rho,\delta),$ consists of those functions $\sigma \in \Sigma(I_{d_\omega} \times I_{d_\tau} \times G\times \widehat{G}),$ satisfying the symbol inequalities
   \begin{equation}\label{InI:L:L}
      \tilde{p}_{\alpha,\beta,\rho,\delta,m, \gamma } (\sigma):= \sup_{\substack{1\leq i\leq d_\tau\\ 1\leq r\leq d_\omega \\ (x, [\xi])\in G\times \widehat{G} }
      } \Vert \langle\xi\rangle ^{(\rho|\alpha|-\delta|\beta|-m)}\partial_{X}^{(\beta)} \Delta_{\xi}^{\alpha}\sigma(i,r, x,\xi) \Vert_{\textnormal{op}} <\infty.
   \end{equation}
   We denote by $\Psi^{m}_{\rho,\delta}((G\times \widehat{G})\otimes \textnormal{End}(E_0,F_0))$ the class of continuous linear operators $A:C^\infty(G,E_0)\rightarrow C^\infty(G,F_0)$ with symbols in the class $S^{m}_{\rho,\delta}((G\times \widehat{G})\otimes \textnormal{End}(E_0,F_0)).$
  \end{definition}
  \begin{remark}\label{Remark:L:G} Observe that any operator $A\in \Psi^{m}_{\rho,\delta}((G\times \widehat{G})\otimes \textnormal{End}(E_0,F_0))$ mapping $C^{\infty}(G,E_0)^{\tau}$ into  $C^{\infty}(G,F_0)^{\omega}$ induces an operator $\tilde{A}$ in the class
   \begin{equation}
     \Psi^{m}_{\rho,\delta}(E,F):=\{\tilde{A}:\Gamma^\infty(E)\rightarrow\Gamma^\infty(F)|\,\,\,  \sigma_A\in {S}^{m,\mathcal{L}}_{\rho,\delta}((G\times \widehat{G})\otimes \textnormal{End}(E_0, F_0))  \} .
  \end{equation} For $E=F$ let us denote $\Psi^{m}_{\rho,\delta}(E):= \Psi^{m}_{\rho,\delta}(E,E),$ and  $\Psi^{m}_{\rho,\delta}((G\times \widehat{G})\otimes \textnormal{End}(E_0)).$ 
  \end{remark}
  \begin{remark}
  Observe that Definition \ref{symbolclass:L} is obtained from Definition \ref{symbolclass} if a sub-Laplacian $\mathcal{L}$ is replaced by the Laplacian $\mathcal{L}_G$ on $G.$
  \end{remark}
  The following theorem follows from the vector-valued subelliptic calculus and from the subelliptic calculus on homogeneous vector-bundles developed above.  We do not require the standard restriction $\rho\geq 1-\delta$ for the invariance under changes of coordinates that is usual in the H\"ormander-calculus via localisations.

\begin{theorem}[Calculus on homogeneous  vector-bundles associated with $\mathcal{L}_G$]\label{calculus:LG} Let $0\leqslant \delta<\rho\leqslant 1.$  Then we have the following properties.
\begin{itemize}
    \item [-]The mapping $A\mapsto A^{*}:\Psi^{m}_{\rho,\delta}((G\times \widehat{G})\otimes \textnormal{End}(E_0,F_0))\rightarrow \Psi^{m}_{\rho,\delta}((G\times \widehat{G})\otimes \textnormal{End}(F_0^*,E_0^*))$ is a continuous linear mapping between Fr\'echet spaces and  the  symbol of $A^*,$ $\sigma_{A^*}$ satisfies the asymptotic expansion,
 \begin{equation}
    \sigma_{A^*}(i,s, x,\xi)\sim  \sum_{|\alpha|= 0}^\infty\Delta_{\xi}^\alpha\partial_{X}^{(\alpha)} \sigma_{A}(i,s, x, \xi)^*.
 \end{equation} 
 This means that, for every $N \in \mathbb{N}$,
 \begin{eqnarray*}
    & \Delta_{\xi}^{\alpha_\ell} \partial_X^{(\beta)} \left(\sigma_{A^*}( i,s,x,\xi) -  \sum_{|\alpha|\leq N }\Delta_{\xi}^\alpha\partial_{X}^{(\alpha)} \sigma_{A}(i,s, x, \xi)^* \right)\\
     &\in {S}^{m-(\rho-\delta)(N+1)-\rho\ell+\delta|\beta|}_{\rho,\delta}((G\times \widehat{G}) \otimes \textnormal{End}(F_0^*,E_0^*)).
 \end{eqnarray*}   In particular, under the identifications $E_{0}^{*}\cong E_0$ and $F_{0}^{*}\cong F_0,$ the same asymptotic expansion also applies if additionally  $$A:C^{\infty}(G,E_{0})^\tau\rightarrow C^{\infty}(G,F_{0})^\omega$$ is a continuous linear operator. This means that \eqref{asymptoticofA*} remains valid for the continuous linear operator $A^*:C^{\infty}(G,F_{0})^\omega\rightarrow C^{\infty}(G,E_{0})^\tau.$ In this case, if for any operator $A,$ we define the operator $\tilde{A}^*$ by making commutative the following diagram 
 \begin{eqnarray}\label{Diagram:adjoint}
    \begin{tikzpicture}[every node/.style={midway}]
  \matrix[column sep={10em,between origins}, row sep={4em}] at (0,0) {
    \node(R) {$\Gamma^\infty(F)$}  ; & \node(S) {$\Gamma^\infty(E)$}; \\
    \node(R/I) {$C^\infty(G,F_0)^\tau$}; & \node (T) {$C^\infty(G,E_0)^\omega$};\\
  };
  \draw[<-] (R/I) -- (R) node[anchor=east]  {$\varkappa_{\tau}$};
  \draw[->] (R) -- (S) node[anchor=south] {$\tilde{A}^{*}$};
  \draw[->] (S) -- (T) node[anchor=west] {$\varkappa_\omega$};
  \draw[->] (R/I) -- (T) node[anchor=north] {$A^{*}$};
\end{tikzpicture}
 \end{eqnarray} then the mapping 
  $\tilde{A}\mapsto \tilde{A}^* :\Psi^{m}_{\rho,\delta}(F,E)\rightarrow \Psi^{m}_{\rho,\delta}(F,E)$ is a continuous linear mapping between Fr\'echet spaces.
\item [-] The mapping $(A,B)\mapsto A\circ B: \Psi^{m_1}_{\rho,\delta}((G\times \widehat{G})\otimes \textnormal{End}(E_0,F_0))\times \Psi^{m_1}_{\rho,\delta}((G\times \widehat{G})\otimes \textnormal{End}(F_0,H_0))\rightarrow \Psi^{m_1+m_2}_{\rho,\delta}((G\times \widehat{G})\otimes \textnormal{End}(E_0,H_0))$ is a continuous bilinear mapping between Fr\'echet spaces, and the  symbol of $BA,$ $\sigma_{BA},$ satisfies the asymptotic expansion,
\begin{equation}\label{compositionAandB2}
    \sigma_{BA}(i, s, x,\xi) \sim \sum_{r=1}^{d_\omega}\sum_{|\alpha|= 0}^\infty(\Delta_{\xi}^\alpha\sigma_{B}(r, s, x,\xi))(\partial_{X}^{(\alpha)} \sigma_{A}(i, r, x,\xi)), 
\end{equation}this means that, for every $N\in \mathbb{N},$
\begin{align*}
    &\Delta_{\xi}^{\alpha_\ell}\partial_{X}^{(\beta)}\left(\sigma_{BA}(i, s, x,\xi) - \sum_{r=1}^{d_\omega}\sum_{|\alpha| \leq N }(\Delta_{\xi}^\alpha\sigma_{B}(r, s, x,\xi))(\partial_{X}^{(\alpha)} \sigma_{A}(i, r, x,\xi))  \right)\\
    &\hspace{2cm}\in {S}^{m_1+m_2-(\rho-\delta)(N+1)-\rho\ell+\delta|\beta|}_{\rho,\delta}((G\times \widehat{G})\otimes \textnormal{End}(E_0,H_0)),
 \end{align*}for every difference operator $\Delta_{\xi}^{\alpha_\ell}$ of order $\ell\in\mathbb{N}_0.$ In particular, the same asymptotic expansion applies if additionally  $$A:C^{\infty}(G,E_{0})^\tau\rightarrow C^{\infty}(G,F_{0})^\omega,$$ and $$B:C^{\infty}(G,F_{0})^\tau\rightarrow C^{\infty}(G,H_{0})^\nu$$ are continuous linear operators. This means that \eqref{compositionAandB2} remains valid for the continuous linear operator $B\circ A:C^{\infty}(G,E_{0})^\tau\rightarrow C^{\infty}(G,H_{0})^\nu.$ In this case, $\tilde{A} \in \Psi^{m_1}_{\rho,\delta}(E,F)$ and $\tilde{B}\in \Psi^{m_2}_{\rho,\delta}(F,H),$ and the composition operator $\tilde{B}\tilde{A}:=\tilde{B}\circ\tilde{ A}:\Gamma^\infty(E)\rightarrow \Gamma^\infty(H)$ belongs to the  class $\Psi^{m_1+m_2}_{\rho,\delta}(E,H).$ 
\end{itemize}
\end{theorem}
\begin{remark}[For a complex bundle $E\rightarrow M$ we do not have the isomorphism $E\cong E^{*}$]\label{non:validity:of:iso}
      If the vector bundles $E\rightarrow M$ and $F\rightarrow M$ are real-vector bundles, under the identifications $E\cong E^{*},$ and $F\cong F^{*},$ the vector bundles $E\rightarrow M$ and $E^{*}\rightarrow M,$ and $F\rightarrow M$ and $F\rightarrow M,$ are isomorphic vector bundles. Indeed,  since the base space $M$  is a (para)compact and Hausdorff topological space, then the real, finite-rank vector bundles $E$ (respectively $F$) and its dual $E^{*}$ (respectively $F^{*}$) are isomorphic as vector bundles. Indeed, we pick a metric $E$ on the bundle $(\cdot,\cdot)$ (i.e. a continuously varying metric on the fibers, which exists because the base space $M$ is compact), and map a vector $v$ to $(v,\cdot)$, which is an isomorphism.  However, in our case the spaces $E_{0}\cong \mathbb{C}^{d_\tau}$ and $F_{0}\cong \mathbb{C}^{d_\omega}$ are complex-vector spaces making the fibrations  $E\rightarrow M$ and $F\rightarrow M$ complex-vector bundles. In the setting of complex vector bundles, choosing a hermitian metric on a complex vector bundle $E$ yields an isomorphism $E\cong \overline{E}^{*}$, where $\overline{E}^{*}$ is the vector bundle whose fibers are the complex conjugates of the fibers of the dual $E^{*}.$ In consequence, for complex vector bundles the property $E\cong E^{*}$ does not hold. For example, the tautological line bundle over the Riemann sphere is not isomorphic to its dual. For details we refer to Milnor and Stacheff \cite{Milnor:and:Stacheff}.  
\end{remark}
\begin{remark}[About the definition of the adjoint on a vector-bundle] We note that in order to define the adjoint $\tilde{A}^{*}$ of an operator $\tilde{A}:\Gamma^\infty(E)\rightarrow \Gamma^{\infty}(F),$ we consider the corresponding operator $A: C^{\infty}(G, E_0)^{\tau}\rightarrow C^{\infty}(G, F_0)^{\omega}$ making the following 
diagram \begin{eqnarray}
    \begin{tikzpicture}[every node/.style={midway}]
  \matrix[column sep={10em,between origins}, row sep={4em}] at (0,0) {
    \node(R) {$\Gamma^\infty(E)$}  ; & \node(S) {$\Gamma^\infty(F)$}; \\
    \node(R/I) {$C^\infty(G,E_0)^\tau$}; & \node (T) {$C^\infty(G,F_0)^\omega$};\\
  };
  \draw[<-] (R/I) -- (R) node[anchor=east]  {$\varkappa_{\tau}$};
  \draw[->] (R) -- (S) node[anchor=south] {$\tilde{A}$};
  \draw[->] (S) -- (T) node[anchor=west] {$\varkappa_\omega$};
  \draw[->] (R/I) -- (T) node[anchor=north] {$A$};
\end{tikzpicture}
\end{eqnarray} commutative. Then, we define the operator $A^{*}: C^{\infty}(G, F_0)^{\omega}\rightarrow C^{\infty}(G, E_0)^{\tau}$ under the identification of the vector-spaces $E_{0}\cong E_0^{*}$ and $F_0\cong F_0^{*}.$ And finally, the operator $\tilde{A}^{*}$ is defined by the commutative diagram \eqref{Diagram:adjoint}. The reason for following this construction of the adjoint operator instead of considering it as a continuous operator from $\Gamma^{\infty}(F^*)\rightarrow \Gamma^{\infty}(E^*)$ is that according to Remark \ref{non:validity:of:iso} we do not have the isomorphism of bundles  $E\cong E^{*},$ and  $F\cong F^{*},$ respectively.    
\end{remark}

To characterise the H\"ormander classes defined by localisations it will be useful to localise operators on homogeneous vector-bundles. We present it as follows.

\begin{definition}\label{definition321}
If $E\cong G\times_\tau E_0 $ and $F\cong G\times_\omega F_0$ are homogeneous vector bundles on $M,$ then we define the class
\begin{align*}
   & \Psi^{m}_{\rho,\delta}(C^{\infty}(G,E_0)^{\tau},C^{\infty}(G,F_0)^\omega;\textnormal{loc}), 
\end{align*}by those continuous linear operators $A:C^{\infty}(G,E_0)^{\tau}\rightarrow C^{\infty}(G,F_0)^\omega,$ such that for every $g\in G,$ and for every   coordinate patch $\omega: U_{\omega}\subset G\rightarrow U\subset \mathbb{R}^n,$ $g\in U_{\omega},$
and for every $\phi,\psi\in C^\infty_0(U),$ the operator $T:C^\infty(U,E_0)\rightarrow C^\infty(F,E_0),$ given by
\begin{equation*}
    Tu:=\psi(\omega^{-1})^*A\omega^{*}(\phi u),\,\,u\in C^\infty(U,E_0),
\end{equation*} is a pseudo-differential operator with symbol $a_T\in \mathbb{C}^{d_\omega\times d_\tau } (S^m_{\rho,\delta}(U\times \mathbb{R}^n)).$ Summarising, $$ A\in \Psi^{m}_{\rho,\delta}(C^{\infty}(G,E_0)^{\tau},C^{\infty}(G,F_0)^\omega;\textnormal{loc}), \textnormal{   if   } A\in \Psi^{m}_{\rho,\delta}(G;E_0,F_0,\textnormal{loc}).$$ Note that with the notation of Remark \ref{remrklocalvecspaces}, $ C^\infty(G, E_0)^{\tau} \subset \textnormal{Dom}(A) .$
\end{definition}
Now, we are ready to prove the main theorem of this subsection.

\begin{theorem}\label{classificationoflocalclasses} Let $\tilde{A}:\Gamma^\infty(E)\rightarrow\Gamma^\infty(F)$ be a continuous linear  operator and let $A: C^\infty(G,E_{0})^\tau\rightarrow C^\infty(G,F_{0})^\omega$ be the induced operator by the identifications $\Gamma^\infty(E)\cong C^\infty(G,E_{0})^\tau$ and $\Gamma^\infty(F)\cong C^\infty(G,F_{0})^\omega.$ If $0\leq \delta<\rho\leq 1,$ and $\rho\geq 1-\delta,$ the following statements are equivalent:
\begin{itemize}
    \item[(A)] $\tilde{A}\in \Psi^{m}_{\rho,\delta}(E,F;\textnormal{loc}).$
     \item[(B)] $\tilde{A}\in \Psi^{m}_{\rho,\delta}(E,F).$
      \item[(C)] ${A}\in \Psi^{m}_{\rho,\delta}((G\times \widehat{G})\otimes \textnormal{End}(E_0,F_0))\cap \mathscr{L}(C^\infty(G,E_0)^\tau,C^\infty(G,F_0)^\omega).$
      \item[(D)] ${A}\in \Psi^{m}_{\rho,\delta}(C^{\infty}(G,E_0)^{\tau},C^{\infty}(G,F_0)^\omega;\textnormal{loc}).$
\end{itemize}
\end{theorem}
\begin{proof} First, let us prove that $(C)$ and $(D)$ are equivalent. So, let us choose $A\in \Psi^{m}_{\rho,\delta}((G\times \widehat{G})\otimes \textnormal{End}(E_0,F_0))\cap \mathscr{L}(C^\infty(G,E_0)^\tau,C^\infty(G,F_0)^\omega).  $ From \eqref{sumofoperators}, we can decompose $A$ as 
\begin{align}\label{sumofoperators22}
   Af= \sum_{i=1}^{d_\tau} \sum_{r=1}^{d_\omega} A_{ir}f_i e_{r, F_0},
\end{align}where every $A_{ir}\in \Psi^{m}_{\rho,\delta}(G\times \widehat{G}).$ Using the notation in Definition  \ref{definition321}, we have that the operator $T$ defined by
\begin{equation*}
    Tu:=\psi(\omega^{-1})^*A\omega^{*}(\phi u)= \sum_{i=1}^{d_\tau} \sum_{r=1}^{d_\omega}  \psi(\omega^{-1})^*A_{ir}\omega^{*}(\phi u_i)e_{r, F_0},
\end{equation*} can be decomposed in the form
\begin{align*}
     Tu:=\psi(\omega^{-1})^*A\omega^{*}(\phi u)= \sum_{i=1}^{d_\tau} \sum_{r=1}^{d_\omega}  T_{ir}u_{i}e_{r, F_0},
\end{align*} where $T_{ir}:C^{\infty}(U)\rightarrow C^{\infty}(U)$ is defined by $T_{ir}f:=\psi(\omega^{-1})^*A_{ir}\omega^{*}(\phi f),$ $f\in C^{\infty}(U).$ In view of the equivalence of classes \eqref{equivalenceofclasses},  $ A_{ir}\in \Psi^{m}_{\rho,\delta}(G\times \widehat{G})=\Psi^{m}_{\rho,\delta}(G;\textnormal{loc})$ for $\delta<\rho$ and $\rho\geq 1-\delta,$ and consequently, we conclude that $T$ is a pseudo-differential operator with a matrix symbol in the class $S^{m}_{\rho,\delta}(U\times \mathbb{R}^n).$ This analysis in components shows that  
\begin{align*}
  &{A}\in \Psi^{m}_{\rho,\delta}((G\times \widehat{G})\otimes \textnormal{End}(E_0,F_0))\cap \mathscr{L}(C^\infty(G,E_0)^\tau,C^\infty(G,F_0)^\omega)\\
  &\Longleftrightarrow \forall i,r,\,  {A}_{ir}\in \Psi^{m}_{\rho,\delta}(G\times \widehat{G})\\
  &\Longleftrightarrow \forall i,r,\,  {A}_{ir}\in \Psi^{m}_{\rho,\delta}(G;\textnormal{loc})\\
  &\Longleftrightarrow{A}\in \Psi^{m}_{\rho,\delta}(C^{\infty}(G,E_0)^{\tau},C^{\infty}(G,F_0)^\omega;\textnormal{loc}). 
\end{align*} Observe that  (B) and (C) are equivalent by Definition \ref{Hormanderclassesonvectorbunbdles}. Under that identifications $\Gamma^\infty(E)\cong C^{\infty}(G,E_0)^\tau$ and $\Gamma^\infty(F)\cong C^{\infty}(G,F_0)^\omega,$ the diagram  \eqref{maindiagram} shows that (A) and (D) are equivalent. So, we conclude the proof.
\end{proof}

\section{Global functional calculus on homogeneous and applications}\label{GFCSECT}
In this section we establish the global functional calculus for the subelliptic H\"ormander classes of operators $\Psi^{m,\mathcal{L}}_{\rho,\delta}(E):=\Psi^{m,\mathcal{L}}_{\rho,\delta}(E,E),$ where $E\cong G\times_{\tau}E_0$ is a homogeneous vector bundle. 

We will start by establishing the global functional calculus for the vector-valued classes and later we will extend the functional calculus to the setting of vector bundles.

\subsection{Vector-valued $\mathcal{L}$-elliptic pseudo-differential operators} 

Here, we will discuss the vector-valued $\mathcal{L}$-elliptic pseudo-differential operators. The main aim of this section is  to show the existence of parametrices for $\mathcal{L}$-elliptic operator in the vector valued setting.   

 \begin{remark}[Leibniz rule for endomorphism-valued symbols] Let us consider the endomorphisms-valued symbols
 \begin{equation}
    \sigma_{A_1}:G\times \widehat{G}\rightarrow \bigcup_{[\xi]\in \widehat{G}}\mathbb{C}^{d_\xi\times d_\xi}( \textnormal{End}(E_0,F_0)),
\end{equation} and
\begin{equation}
    \sigma_{A_2}:G\times \widehat{G}\rightarrow \bigcup_{[\xi]\in \widehat{G}}\mathbb{C}^{d_\xi\times d_\xi}( \textnormal{End}(F_0,H_0)).
\end{equation}Observe that 
\begin{equation*}
    \sigma_{A_1}\sigma_{A_2}:G\times \widehat{G}\rightarrow \bigcup_{[\xi]\in \widehat{G}}\mathbb{C}^{d_\xi\times d_\xi}( \textnormal{End}(E_0,H_0)),
\end{equation*}where, for any $(x,[\xi]),$ the $(u,v)$-entry with $1\leq u,v\leq d_\xi,$ is given by
\begin{equation*}
    [\sigma_{A_1}\sigma_{A_2}]_{uv}(x,[\xi]):=\sum_{w=1}^{d_\xi}\sigma_{A_1}(x,[\xi])_{uw}\sigma_{A_2}(x,[\xi])_{wv}.
\end{equation*}The difference operators $\Delta_\xi^{\alpha}$ for matrix-valued symbols induce difference operators for endomorphisms-valued symbols and we will check the consistency of the Leibniz rule. First step is to define
\begin{equation}
    \Delta_\xi^\alpha[\sigma_1\sigma_2](x, \xi):=[ \Delta_\xi^\alpha(\sigma_1(x,\xi)\sigma_2(x,\xi))_{u,v}]_{u,v=1}^{d_\xi},
\end{equation}where  the family of endomorphisms $ \Delta_\xi^\alpha(\sigma_1(x, \xi)\sigma_2(x,\xi))_{u,v}$ is defined via
\begin{equation}
    ( \Delta_\xi^\alpha(\sigma_1(x, \xi)\sigma_2(x,\xi))_{u,v})e_{i,E_0},e_{r,H_0})=\Delta_\xi^\alpha(\sigma_1(i,r,x, \xi)\sigma_2(i,r,x,\xi))_{u,v},
    \end{equation}on the standard orthonormal bases that we fixed for $E_0$ and $H_0.$ Note that
    \begin{align*}
       \Delta_\xi^\alpha(\sigma_1(i,r,x,\xi)\sigma_2(i,r,x,\xi))_{u,v}=\sum_{w=1}^{d_\xi} \Delta_\xi^\alpha(\sigma_1(i,r,x, \xi)_{uw}\sigma_2(i,r,x,\xi)_{wv}),
    \end{align*}and then
    \begin{align*}
      \Delta_\xi^\alpha(\sigma_1(i,r,x, \xi)\sigma_2(i,r,x,\xi)) &=[\Delta_\xi^\alpha(\sigma_1(i,r,x, \xi)\sigma_2(i,r,x,\xi))_{u,v}]_{u,v=1}^{d_\xi}\\
      &=\left( \sum_{w=1}^{d_\xi} \Delta_\xi^\alpha(\sigma_1(i,r,x, \xi)_{uw}\sigma_2(i,r,x,\xi))_{wv}\right)_{u,v=1}^{d_\xi}.
    \end{align*}
 Using the Leibniz rule in Remark \ref{Leibnizrule} for the matrix-valued symbols, we obtain
 \begin{align*}
    &\Delta_\xi^\alpha(\sigma_1(i,r,x, \xi)\sigma_2(i,r,x,\xi))\\
    &=\sum_{  |\alpha_1|,|\alpha_2|\leq |\alpha|\leq |\alpha_1|+|\alpha_2|  } \Delta_\xi^{\alpha_1}(\sigma_1(i,r,x, \xi))(\Delta_\xi^{\alpha_2}(\sigma_2(i,r,x, \xi)).
 \end{align*}Consequently,
 \begin{align*}
    & \Delta_\xi^\alpha(\sigma_1(i,r,x, \xi)\sigma_2(i,r,x,\xi))_{uv}\\
     &=\sum_{  |\alpha_1|,|\alpha_2|\leq |\alpha|\leq |\alpha_1|+|\alpha_2|  }\sum_{w=1}^{d_\xi} \Delta_\xi^{\alpha_1}(\sigma_1(i,r,x, \xi))_{uw}(\Delta_\xi^{\alpha_2}(\sigma_2(i,r,x, \xi))_{wv}.
 \end{align*}This analysis leads to the following identities
 \begin{align*}
     & ( \Delta_\xi^\alpha(\sigma_1(x, \xi)\sigma_2(x,\xi))_{u,v})e_{i,E_0},e_{r,H_0})=\Delta_\xi^\alpha(\sigma_1(i,r,x, \xi)\sigma_2(i,r,x,\xi))_{u,v}\\
     &=\sum_{w=1}^{d_\xi} \Delta_\xi^{\alpha}(\sigma_1(i,r,x, \xi)_{uw}\sigma_2(i,r,x,[\xi])_{wv})\\
      &=\sum_{  |\alpha_1|,|\alpha_2|\leq |\alpha|\leq |\alpha_1|+|\alpha_2|  }\sum_{w=1}^{d_\xi}
      \Delta_\xi^{\alpha_1}(\sigma_1(i,r,x, \xi))_{uw}(\Delta_\xi^{\alpha_2}(\sigma_2(i,r,x, \xi))_{wv}\\
      &=\sum_{  |\alpha_1|,|\alpha_2|\leq |\alpha|\leq |\alpha_1|+|\alpha_2|  } [\Delta_\xi^{\alpha_1}(\sigma_1(i,r,x, \xi))(\Delta_\xi^{\alpha_2}(\sigma_2(i,r,x, \xi))]_{uv}\\
      &=\sum_{  |\alpha_1|,|\alpha_2|\leq |\alpha|\leq |\alpha_1|+|\alpha_2|  }( \Delta_\xi^{\alpha_1}(\sigma_1(x, \xi))(\Delta_\xi^{\alpha_2}\sigma_2(x,\xi))_{u,v}))e_{i,E_0},e_{r,H_0}),
 \end{align*}showing the Leibniz rule
 \begin{align}\label{leibniz:End:symbol} \boxed{
      \Delta_\xi^\alpha(\sigma_1(x,\xi)\sigma_2(x,\xi)=\sum_{  |\alpha_1|,|\alpha_2|\leq |\alpha|\leq |\alpha_1|+|\alpha_2|  } \Delta_\xi^{\alpha_1}(\sigma_1(x, \xi))(\Delta_\xi^{\alpha_2}\sigma_2(x,\xi))}
 \end{align}
 \end{remark}

\begin{theorem}\label{IesT} Let $m\in \mathbb{R},$ and let $0\leqslant \delta<\rho\leqslant 1.$  Let  $a=a(x,\xi)\in {S}^{m,\mathcal{L}}_{\rho,\delta}((G\times \widehat{G})\otimes \textnormal{End}(E_0)) .$  Assume also that $a(x,\xi)\in \mathbb{C}^{d_\xi \times d_\xi}(\textnormal{End}(E_0))$ is invertible for every $(x,[\xi])\in G\times\widehat{G},$\footnote{which means that there exists  $b_0(x,\xi)\in \mathbb{C}^{d_\xi\times d_\xi}(\textnormal{End}(E_0)) ,$ such that $b_{0}(x,\xi)a(x,\xi)=a(x,\xi)b_0(x,\xi)=I,$ where $I=(\textnormal{id}_{E_0})_{i,j=1}^{d_\xi}$ is the identity element in  $\mathbb{C}^{d_\xi\times d_\xi}(\textnormal{End}(E_0)).$} and satisfies
\begin{equation}\label{Ies}
   \sup_{(x,[\xi])\in G\times \widehat{G}}\Vert\widehat{\mathcal{M}}(\xi)^m\otimes a(x,\xi)^{-1} \Vert_{\textnormal{op}}<\infty.
\end{equation}
Then, $a^{-1}:=a(x,\xi)^{-1}\in {S}^{-m,\mathcal{L}}_{\rho,\delta}((G\times \widehat{G})\otimes \textnormal{End}(E_0) ).$ 
\end{theorem}
\begin{proof}
Let us estimate $\partial^{(\beta)}_{X_i}a^{-1}$ first. Suppose we have proved that 
\begin{equation*}
 I:= \sup_{|\beta|\leqslant \ell}\sup_{(x,[\xi])\in  G\times \widehat{G} }\Vert \widehat{ \mathcal{M}}(\xi)^{(-\delta|\beta|+m)}\otimes \partial_{X}^{(\beta)} (a(x,\xi)^{-1})\Vert_{\textnormal{op}} <\infty,
   \end{equation*} 
   for some  $ \ell\in \mathbb{N}.$ We proceed by mathematical induction. Let us analyse the cases $|\tilde\beta|=\ell+1.$ If we write $\partial_{X}^{(\tilde\beta)}= \partial_{X}\partial_{X}^{(\beta)}$ where $|\beta|\leqslant \ell,$ then  $\partial^{\tilde\beta}_{X_i}a^{-1}=\partial_{X}\partial_{X}^{(\beta)}a^{-1}.$ From the identity $$ a(x,\xi) \circ a(x,\xi)^{-1}=I_{ \mathbb{C}^{d_\xi\times d_\xi}( \textnormal{End}(E_0)),      }$$ we have
   \begin{equation*}
       a(x,\xi) \circ \partial_{X}\partial_{X}^{(\beta)}a^{-1}(x,\xi)=-\sum_{\beta_1+\beta_2=\beta+e_{j},|\beta_2|\leqslant  |\beta|}C_{\beta_1,\beta_2}(\partial_{X}^{(\beta_1)}a(x,\xi))\circ (\partial_{X}^{(\beta_2)}a^{-1}(x,\xi)).
   \end{equation*}
Consequently,
\begin{equation*}
       \partial_{X}\partial_{X}^{(\beta)}a^{-1}(x,\xi)=-a(x,\xi)^{-1} \circ\left(\sum_{\beta_1+\beta_2=\beta+e_{j},|\beta_2|\leqslant  |\beta|}C_{\beta_1,\beta_2}(\partial_{X}^{(\beta_1)}a(x,\xi))\circ(\partial_{X}^{(\beta_2)}a^{-1}(x,\xi)) \right).
   \end{equation*} We want to prove the estimate
   \begin{equation*}
 \sup_{(x,[\xi])\in G\times \widehat{G} }\Vert \widehat{ \mathcal{M}}(\xi)^{(-\delta(|\beta|+1)+m)}\otimes \partial_{X}\partial_{X}^{(\beta)} a(x,\xi)^{-1}\Vert_{\textnormal{op}} <\infty.
   \end{equation*}
   For this, we only need to show that for every $\beta_1$ and $\beta_2$ such that $\beta_1+\beta_2=\beta+e_{j},|\beta_2|\leqslant  |\beta|,$
\begin{equation*}
    \sup_{(x,[\xi])\in G\times \widehat{G} }\Vert \widehat{ \mathcal{M}}(\xi)^{(-\delta(|\beta|+1)+m)} \otimes a(x,\xi)^{-1}\circ(\partial_{X}^{(\beta_1)}a(x,\xi))\circ(\partial_{X}^{(\beta_2)}a^{-1}(x,\xi))   \Vert_{\textnormal{op}} <\infty.
\end{equation*} 
Observe that
\begin{align*}
 &\widehat{ \mathcal{M}}(\xi)^{(-\delta(|\beta|+1)+m)} \otimes a(x,\xi)^{-1}\circ(\partial_{X}^{(\beta_1)}a(x,\xi))\circ(\partial_{X}^{(\beta_2)}a^{-1}(x,\xi))   \\
   &=\widehat{ \mathcal{M}}(\xi)^{(-\delta(|\beta_1|+|\beta_2|)+m)} \otimes a(x,\xi)^{-1}\circ(\partial_{X}^{(\beta_1)}a(x,\xi))\circ(\partial_{X}^{(\beta_2)}a^{-1}(x,\xi)).
\end{align*} So, by Theorem \ref{characterisations} and \eqref{Ies} we have
\begin{align*}
   & \Vert \widehat{ \mathcal{M}}(\xi)^{(-\delta(|\beta|+1)+m)} \otimes a(x,\xi)^{-1}\circ(\partial_{X}^{(\beta_1)} a(x,\xi))\circ(\partial_{X}^{(\beta_2)}a^{-1}(x,\xi)) \Vert_{\textnormal{op}}\\
    &= \Vert \widehat{ \mathcal{M}}(\xi)^{(-\delta(|\beta_1|+|\beta_2|)} \widehat{ \mathcal{M}}(\xi)^{ m } \otimes  a(x,\xi)^{-1}\circ(\partial_{X}^{(\beta_1)}a(x,\xi))\circ(\partial_{X}^{(\beta_2)}a^{-1}(x,\xi)) \Vert_{\textnormal{op}}\\
    &\leqslant \Vert \widehat{ \mathcal{M}}(\xi)^{ m }\otimes a(x,\xi)^{-1} \Vert_{\textnormal{op}}\Vert(\partial_{X}^{(\beta_1)}a(x,\xi))\circ(\partial_{X}^{(\beta_2)}a^{-1}(x,\xi)) \otimes \widehat{ \mathcal{M}}(\xi)^{(-\delta(|\beta_1|+|\beta_2|)}\Vert_{\textnormal{op}}\\
    &\lesssim  \Vert(\partial_{X}^{(\beta_1)} a(x,\xi))\circ(\partial_{X}^{(\beta_2)}a^{-1}(x,\xi)) \otimes \widehat{ \mathcal{M}}(\xi)^{(-\delta(|\beta_1|+|\beta_2|)}\Vert_{\textnormal{op}}.
    \end{align*}
Again, by using  Theorem \ref{characterisations}, we have

\begin{align*}
   & \Vert(\partial_{X}^{(\beta_1)}a(x,\xi))\circ(\partial_{X}^{(\beta_2)}a^{-1}(x,\xi)) \otimes  \widehat{ \mathcal{M}}(\xi)^{(-\delta(|\beta_1|+|\beta_2|)}\Vert_{\textnormal{op}}\\
   &\leqslant \Vert(\partial_{X}^{(\beta_1)}a(x,\xi)) \otimes  \widehat{ \mathcal{M}}(\xi)^{(-m-\delta|\beta_1|)}\Vert_{\textnormal{op}} \\
   &\hspace{4cm}\times\Vert\widehat{ \mathcal{M}}(\xi)^{(m+\delta|\beta_1|)} \otimes   (\partial_{X}^{(\beta_2)}a^{-1}(x,\xi)) \otimes \widehat{ \mathcal{M}}(\xi)^{(-\delta(|\beta_1|+|\beta_2|))}  \Vert_{\textnormal{op}}\\
   &\lesssim  \Vert\widehat{ \mathcal{M}}(\xi)^{(m+\delta|\beta_1|)} \otimes   (\partial_{X}^{(\beta_2)}a^{-1}(x,\xi)) \otimes  \widehat{ \mathcal{M}}(\xi)^{(-\delta(|\beta_1|+|\beta_2|))}  \Vert_{\textnormal{op}}\\
   &=\Vert\widehat{ \mathcal{M}}(\xi)^{(m+\delta|\beta_1|)} \otimes   (\partial_{X}^{(\beta_2)}a^{-1}(x,\xi)) \otimes \widehat{ \mathcal{M}}(\xi)^{(-\delta(|\beta_2|+m)}\widehat{ \mathcal{M}}(\xi)^{(-\delta|\beta_1|-m))}\Vert_{\textnormal{op}}\\
   &\lesssim \Vert   (\partial_{X}^{(\beta_2)}a^{-1}(x,\xi)) \otimes \widehat{ \mathcal{M}}(\xi)^{(-\delta(|\beta_2|+m))}\Vert_{\textnormal{op}}\\
  & \lesssim 1.
\end{align*}A similar analysis using the Leibniz rule for difference operators can be used in order to estimate the differences $\Delta_\xi^{\alpha}a^{-1}.$ For this, we need the following two estimates,
\begin{equation}
     \Vert \widehat{\mathcal{M}}(\xi)^{(\rho|\gamma|-\delta|\beta|)} \otimes a(x,\xi)^{-1}\circ [\Delta_{\xi}^{\gamma}\partial_{X}^{(\beta)}a(x,\xi)]\Vert_{\textnormal{op}}=O(1),
\end{equation}
\begin{equation}
     \Vert  [\Delta_{\xi}^{\gamma}\partial_{X}^{(\beta)}a(x,\xi)]\circ a(x,\xi)^{-1} \otimes\widehat{\mathcal{M}}(\xi)^{(\rho|\gamma|-\delta|\beta|)}\Vert_{\textnormal{op}}\\ =O(1).
\end{equation}
For the proof, let us use  \eqref{Ies} as follows
\begin{align*}
     \Vert \widehat{\mathcal{M}}(\xi)^{(\rho|\gamma|-\delta|\beta|)} \otimes &a(x,\xi)^{-1}\circ[\Delta_{\xi}^{\gamma}\partial_{X}^{(\beta)}a(x,\xi)]\Vert_{\textnormal{op}}\\
    &= \Vert \widehat{\mathcal{M}}(\xi)^{(\rho|\gamma|-\delta|\beta|-m)}\widehat{\mathcal{M}}(\xi)^{ m } \otimes a(x,\xi)^{-1}\circ[\Delta_{\xi}^{\gamma}\partial_{X}^{(\beta)}a(x,\xi)]\Vert_{\textnormal{op}}\\
     &\lesssim \Vert \widehat{\mathcal{M}}(\xi)^{ m } \otimes a(x,\xi)^{-1}\circ[\Delta_{\xi}^{\gamma}\partial_{X}^{(\beta)}a(x,\xi)] \otimes \widehat{\mathcal{M}}(\xi)^{(\rho|\gamma|-\delta|\beta|-m)}\Vert_{\textnormal{op}}\\ 
     &\lesssim \Vert \widehat{\mathcal{M}}(\xi)^{ m } \otimes a(x,\xi)^{-1}\Vert_{\textnormal{op}}\Vert[\Delta_{\xi}^{\gamma}\partial_{X}^{(\beta)}a(x,\xi)] \otimes \widehat{\mathcal{M}}(\xi)^{(\rho|\gamma|-\delta|\beta|-m)}\Vert_{\textnormal{op}}\\ 
     &= O(1).
\end{align*} On the other hand,
\begin{align*}
     \Vert  [\Delta_{\xi}^{\gamma}\partial_{X}^{(\beta)}a(x,\xi)]& \circ a(x,\xi)^{-1} \otimes \widehat{\mathcal{M}}(\xi)^{(\rho|\gamma|-\delta|\beta|)}\Vert_{\textnormal{op}}\\
    &= \Vert  [\Delta_{\xi}^{\gamma}\partial_{X}^{(\beta)}a(x,\xi)] \otimes   \widehat{\mathcal{M}}(\xi)^{- m }\widehat{\mathcal{M}}(\xi)^{ m } \otimes a(x,\xi)^{-1} \otimes \widehat{\mathcal{M}}(\xi)^{(\rho|\gamma|-\delta|\beta|)}\Vert_{\textnormal{op}}\\
     &\lesssim \Vert \widehat{\mathcal{M}}(\xi)^{(\rho|\gamma|-\delta|\beta|)} \otimes [\Delta_{\xi}^{\gamma}\partial_{X}^{(\beta)}a(x,\xi)] \otimes \widehat{\mathcal{M}}(\xi)^{- m } \widehat{\mathcal{M}}(\xi)^{ m } \otimes a(x,\xi)^{-1}\Vert_{\textnormal{op}}\\ 
     &\lesssim \Vert \widehat{\mathcal{M}}(\xi)^{(\rho|\gamma|-\delta|\beta|)} \otimes [\Delta_{\xi}^{\gamma}\partial_{X}^{(\beta)}a(x,\xi)] \otimes \widehat{\mathcal{M}}(\xi)^{- m }\Vert_{\textnormal{op}} \Vert\widehat{\mathcal{M}}(\xi)^{ m } \otimes a(x,\xi)^{-1}\Vert_{\textnormal{op}}\\ 
     &\lesssim \Vert \widehat{\mathcal{M}}(\xi)^{(-m+\rho|\gamma|-\delta|\beta|)} \otimes [\Delta_{\xi}^{\gamma}\partial_{X}^{(\beta)}a(x,\xi)]\Vert_{\textnormal{op}} \\ 
     &= O(1).
\end{align*}
Now, we will estimate in $(m+\rho\ell)$ the subelliptic order for the differences  $\Delta_{q_{\ell}}a^{-1}$ where $\Delta_{q_{\ell}}$ denotes a difference operator of order $\ell$. To do so, we will use mathematical induction. The case $\ell=0$ holds true from the hypothesis of Theorem \ref{IesT}. To study the differences of higher order we will use the Leibniz rule \eqref{leibniz:End:symbol},
\begin{align*}
   \Delta_{q_\ell}[ a_{1} \circ a_{2}](x,\xi) =\sum_{ |\gamma|,|\varepsilon|\leqslant \ell\leqslant |\gamma|+|\varepsilon| }C_{\varepsilon,\gamma}(\Delta_{q_\gamma}a_{1})(x,\xi) \circ (\Delta_{q_\varepsilon}a_{2})(x,\xi).
\end{align*} 
From the identity, $$a(x,\xi)\circ a(x,\xi)^{-1}=I_{ \mathbb{C}^{d_\xi\times d_\xi}( \textnormal{End}(E_0))},$$ we deduce that
\begin{align*}
  (\Delta_{q_1}a)(x,\xi) \circ & a(x,\xi)^{-1}+a(x,\xi) \circ (\Delta_{q_1}a^{-1})(x,\xi)\\
   &=-\sum_{ 1= |\nu|,|\nu'| }C_{\nu,\nu'}(\Delta_{q_{(\nu)}}a)(x,\xi) \circ (\Delta_{q_{(\nu')}}a^{-1})(x,\xi),
\end{align*}where $\Delta_{q_{(\nu)}}$ are differences operators of order $|\nu|=1$. Consequently
\begin{align*}
   &  (\Delta_{q_1}a^{-1})(x,\xi)=-a(x,\xi)^{-1}\circ(\Delta_{q_1}a)(x,\xi)\circ a^{-1}(x,\xi)\\
   &\hspace{2cm}-\sum_{ 1= |\nu|,|\nu'| }C_{\nu,\nu'}a(x,\xi)^{-1}\circ (\Delta_{q_{(\nu)}}a)(x,\xi) \circ (\Delta_{q_{(\nu')}}a^{-1})(x,\xi).
\end{align*} The differences of higher order  $\Delta_{q_{\ell+1}}a^{-1},$ can be estimated e.g. using the difference operators of the form $$\Delta_{q_{\ell+1}}=\Delta_{q_{\ell}}\Delta_{q_{1}}.$$ 
Then, by applying the difference operator $\Delta_{q_{\ell}}$ to $(\Delta_{q_1}a^{-1})(x,\xi)$ we essentially obtain linear combinations of terms of the following kind as a consequence of the Leibniz rule:
\begin{itemize}
    \item $\textnormal{IV}_{(1)}(x,\xi):=\Delta_{q_{\ell}}a(x,\xi)^{-1}\circ (\Delta_{q_1}a)(x,\xi)\circ a^{-1}(x,\xi)$\\
    \item $\textnormal{IV}_{(2)}(x,\xi):=a(x,\xi)^{-1}\circ \Delta_{q_{\ell+1}}a(x,\xi)\circ a^{-1}(x,\xi)$\\
     \item $\textnormal{IV}_{(3)}(x,\xi):=a(x,\xi)^{-1}\circ \Delta_{q_{1}}a(x,\xi)\circ \Delta_{q_{\ell}}a^{-1}(x,\xi)$\\
      \item $\textnormal{IV}_{(4)}(x,\xi):=\Delta_{q_{\ell_1}} a(x,\xi)^{-1}\circ(\Delta_{q_{\ell_2+1}}a)(x,\xi)\circ \Delta_{q_{\ell_3+1}}a^{-1}(x,\xi),$ $1\leqslant \ell_{1},\ell_2,\ell_3\leqslant \ell\leqslant \ell_1+\ell_2+\ell_3,$\\
      \item $\textnormal{V}_{(1)}(x,\xi):=\Delta_{q_{\ell}}a(x,\xi)^{-1}\circ (\Delta_{q_{(\nu) } }a)(x,\xi)\circ \Delta_{q_{(\nu')}}a^{-1}(x,\xi)$\\
    \item $\textnormal{V}_{(2)}(x,\xi):=a(x,\xi)^{-1}\circ \Delta_{q_{\ell+\nu}}a(x,\xi)\circ \Delta_{q_{\nu'}}a^{-1}(x,\xi)$\\
     \item $\textnormal{V}_{(3)}(x,\xi):=a(x,\xi)^{-1}\circ \Delta_{q_{(\nu)}}a(x,\xi)\circ \Delta_{q_{\ell+\nu'}}a^{-1}(x,\xi)$\\
      \item $\textnormal{V}_{(4)}(x,\xi):=\Delta_{q_{\ell_1}} a(x,\xi)^{-1}\circ(\Delta_{q_{\ell_2+\nu}}a)(x,\xi)\circ \Delta_{q_{\ell_3+\nu'}}a^{-1}(x,\xi),$ $1\leqslant \ell_{1},\ell_2,\ell_3\leqslant \ell\leqslant \ell_1+\ell_2+\ell_3,$\\
\end{itemize}which we can estimate as follows. First, assume that for every $\ell\geqslant  1,$ $\ell\in \mathbb{N},$ we have
$$\sup_{(x,[\xi])}  \Vert\widehat{\mathcal{M}}(\xi)^{(m+\rho\ell)} \otimes (\Delta_{q_{\ell}}a^{-1})(x,\xi)\Vert_{\textnormal{op}}  <\infty. $$ We need to prove that
$$\sup_{(x,[\xi])}  \Vert\widehat{\mathcal{M}}(\xi)^{(m+\rho(\ell+1))} \otimes (\Delta_{q_{\ell+1}}a^{-1})(x,\xi)\Vert_{\textnormal{op}}  <\infty. $$
For this, it is enough to prove that 
$$ \sup_{(x,[\xi])}  \Vert \widehat{\mathcal{M}}(\xi)^{(m+\rho(\ell+1))} \otimes \textnormal{IV}_{(i)}(x,\xi)\Vert_{\textnormal{op}},\, \sup_{(x,[\xi])}  \Vert \widehat{\mathcal{M}}(\xi)^{ (m+\rho(\ell+1))} \otimes  \textnormal{V}_{(j)}(x,\xi)\Vert_{\textnormal{op}} <\infty, $$ for all $1\leqslant i,j\leqslant 4.$ Next, let us omit the argument $(x,\xi)$ in order to simplify the notation.
So, we have
\begin{align*}
    &\Vert \widehat{\mathcal{M}}(\xi)^{(m+\rho(\ell+1))} \otimes  \textnormal{IV}_{(1)}\Vert_{\textnormal{op}}+     \Vert \widehat{\mathcal{M}}(\xi)^{(m+\rho(\ell+1))} \otimes \textnormal{IV}_{(2)}(x,\xi)\Vert_{\textnormal{op}}\\
   &+   \Vert \widehat{\mathcal{M}}(\xi)^{(m+\rho(\ell+1))} \otimes \textnormal{IV}_{(3)}(x,\xi)\Vert_{\textnormal{op}}+\Vert \widehat{\mathcal{M}}(\xi)^{(m+\rho(\ell+1))} \otimes \textnormal{V}_{(1)}(x,\xi)\Vert_{\textnormal{op}}\\
    &\hspace{4cm}+\Vert \widehat{\mathcal{M}}(\xi)^{(m+\rho(\ell+1))} \otimes \textnormal{V}_{(2)}(x,\xi)\Vert_{\textnormal{op}}\\
    & =\Vert \widehat{\mathcal{M}}(\xi)^{(m+\rho(\ell+1))} \otimes\Delta_{q_{\ell}}a^{-1}\circ (\Delta_{q_1}a)\circ a^{-1}\Vert_{\textnormal{op}}\\&\hspace{2cm}+     \Vert \widehat{\mathcal{M}}(\xi)^{(m+\rho(\ell+1))} \otimes a^{-1}\circ \Delta_{q_{\ell+1}}a\circ a^{-1}\Vert_{\textnormal{op}}\\
   &\hspace{2cm}+   \Vert \widehat{\mathcal{M}}(\xi)^{(m+\rho(\ell+1))} \otimes a^{-1}\circ \Delta_{q_{1}}a\circ \Delta_{q_{\ell}}a^{-1}\Vert_{\textnormal{op}}\\
    &\hspace{2cm}+\Vert \widehat{\mathcal{M}}(\xi)^{(m+\rho(\ell+1))} \otimes \Delta_{q_{\ell}}a^{-1}\circ (\Delta_{q_{(\nu) } }a)\circ \Delta_{q_{(\nu')}}a^{-1}\Vert_{\textnormal{op}}\\
    &\hspace{2cm}+\Vert \widehat{\mathcal{M}}(\xi)^{(m+\rho(\ell+1))} \otimes a^{-1}\circ \Delta_{q_{\ell+\nu}}a\circ \Delta_{q_{\nu'}}a^{-1}\Vert_{\textnormal{op}}\\
    \end{align*}
    
    \begin{align*}
     \lesssim\Vert &\widehat{\mathcal{M}}(\xi)^{(m+\rho\ell)} \otimes \Delta_{q_{\ell}}a^{-1}\circ (\Delta_{q_1}a)\circ a^{-1} \otimes\widehat{\mathcal{M}}(\xi)^{{\rho}}\Vert_{\textnormal{op}}\\&\hspace{2cm}+     \Vert \widehat{\mathcal{M}}(\xi)^{ m } \otimes a^{-1}\circ \Delta_{q_{\ell+1}}a\circ a^{-1} \otimes \widehat{\mathcal{M}}(\xi)^{(\rho(\ell+1))} \Vert_{\textnormal{op}}\\
   &\hspace{2cm}+   \Vert \widehat{\mathcal{M}}(\xi)^{{\rho}} \otimes a^{-1}\circ \Delta_{q_{1}}a\circ \Delta_{q_{\ell}}a^{-1} \otimes \widehat{\mathcal{M}}(\xi)^{(m+\rho\ell)}\Vert_{\textnormal{op}}\\
    &\hspace{2cm}+\Vert \widehat{\mathcal{M}}(\xi)^{(m+\rho\ell)} \otimes \Delta_{q_{\ell}}a^{-1}\Vert_{\textnormal{op}}\Vert  (\Delta_{q_{(\nu) } }a)\circ \Delta_{q_{(\nu')}}a^{-1} \otimes \widehat{\mathcal{M}}(\xi)^{{\rho}}\Vert_{\textnormal{op}}\\
    &\hspace{2cm}+\Vert \widehat{\mathcal{M}}(\xi)^{(m+\rho(\ell+1))} \otimes a^{-1}\circ \Delta_{q_{\ell+\nu}}a\circ \Delta_{q_{\nu'}}a^{-1}\Vert_{\textnormal{op}}\\
    \\
    & \lesssim\Vert (\Delta_{q_1}a)\circ a^{-1} \otimes \widehat{\mathcal{M}}(\xi)^{{\rho}}\Vert_{\textnormal{op}}\\&\hspace{2cm}+      \Vert  \Delta_{q_{\ell+1}}a\circ a^{-1} \otimes \widehat{\mathcal{M}}(\xi)^{(\rho(\ell+1))} \Vert_{\textnormal{op}}\\
   &\hspace{2cm}+   \Vert \widehat{\mathcal{M}}(\xi)^{{\rho}} \otimes a^{-1}\circ \Delta_{q_{1}}a\Vert_{\textnormal{op}}\\
    &\hspace{2cm}+\Vert  (\Delta_{q_{(\nu) } }a)\circ a^{-1} \otimes \widehat{\mathcal{M}}(\xi)^{{\rho}}\Vert_{\textnormal{op}}\\
    &\hspace{2cm}+\Vert \widehat{\mathcal{M}}(\xi)^{(\rho(\ell+1))} \otimes a^{-1}\circ \Delta_{q_{\ell+\nu}}a\Vert_{\textnormal{op}}\\
     &\hspace{2cm}\lesssim 1,
\end{align*}
where we have used, and  the estimates
\begin{equation}\label{casobacano}
    \Vert  (\Delta_{q_{(\nu) } }a)\circ \Delta_{q_{(\nu')}}a^{-1} \otimes \widehat{\mathcal{M}}(\xi)^{{\rho}}\Vert_{\textnormal{op}}\lesssim 1,
\end{equation}
and 
\begin{equation}\label{casobacano2}
    \Vert \widehat{\mathcal{M}}(\xi)^{(m+\rho(\ell+1))} \otimes a^{-1} \circ \Delta_{q_{\ell+\nu}}a \circ \Delta_{q_{\nu'}}a^{-1}\Vert_{\textnormal{op}}\lesssim \Vert \widehat{\mathcal{M}}(\xi)^{(\rho(\ell+1))} \otimes a^{-1} \circ \Delta_{q_{\ell+\nu}}a\Vert_{\textnormal{op}}. 
\end{equation}
Indeed, for the proof of \eqref{casobacano}  observe that from the induction hypothesis, we have
\begin{align*}
  \Vert  (\Delta_{q_{(\nu) } }a)\circ &\Delta_{q_{(\nu')}}a^{-1} \otimes \widehat{\mathcal{M}}(\xi)^{{\rho}}\Vert_{\textnormal{op}}\\
  &\lesssim   \Vert \widehat{\mathcal{M}}(\xi)^{(-m+\rho)} \otimes (\Delta_{q_{(\nu) } }a)\circ \Delta_{q_{(\nu')}}a^{-1} \otimes \widehat{\mathcal{M}}(\xi)^{(m+\rho)} \widehat{\mathcal{M}}(\xi)^{-{\rho}} \Vert_{\textnormal{op}}   \\
  &=   \Vert \widehat{\mathcal{M}}(\xi)^{(-m+\rho)} \otimes (\Delta_{q_{(\nu) } }a)\Vert_{\textnormal{op}} \Vert \Delta_{q_{(\nu')}}a^{-1} \otimes \widehat{\mathcal{M}}(\xi)^{(m+\rho)}\Vert_{\textnormal{op}} \Vert \widehat{\mathcal{M}}(\xi)^{-{\rho} } \Vert_{\textnormal{op}}   \\
  &\lesssim 1.
\end{align*}In order to prove \eqref{casobacano2}, observe that
\begin{align*}
   & \Vert \widehat{\mathcal{M}}(\xi)^{(m+\rho(\ell+1))} \otimes a^{-1}\circ \Delta_{q_{\ell+\nu}}a\circ \Delta_{q_{\nu'}}a^{-1}\Vert_{\textnormal{op}}\\
    &\lesssim \Vert \widehat{\mathcal{M}}(\xi)^{(\rho(\ell+1))} \otimes a^{-1}\circ \Delta_{q_{\ell+\nu}}a\circ \Delta_{q_{\nu'}}a^{-1} \otimes\widehat{\mathcal{M}}(\xi)^{ m }\Vert_{\textnormal{op}}\\
     &\lesssim \Vert \widehat{\mathcal{M}}(\xi)^{(\rho(\ell+1))} \otimes a^{-1}\circ \Delta_{q_{\ell+\nu}}a\Vert_{\textnormal{op}}\Vert \Delta_{q_{\nu'}}a^{-1} \otimes\widehat{\mathcal{M}}(\xi)^{(m+\rho)}\Vert_{\textnormal{op}}\Vert \widehat{\mathcal{M}}(\xi)^{-{\rho}}\Vert_{\textnormal{op}} \\
     &\lesssim \Vert \widehat{\mathcal{M}}(\xi)^{(\rho(\ell+1))} \otimes a^{-1}\circ \Delta_{q_{\ell+\nu}}a\Vert_{\textnormal{op}}.
\end{align*}
A similar analysis can be used to study $\textnormal{IV}_{(4)}(x,\xi),\textnormal{V}_{(3)}(x,\xi)$ and $\textnormal{V}_{(4)}(x,\xi).$ Thus, we end the proof.
\end{proof}


\begin{lemma}\label{lemma:asymp:exp}
    Let $\{m_j\}_{j=0}^{\infty}$ be a strictly decreasing sequence of real numbers such that $\lim_{j\rightarrow \infty}m_j=-\infty.$ Let $b_j=b_j(x,\xi)\in {S}^{m_j,\mathcal{L}}_{\rho,\delta}((G\times \widehat{G}) \otimes \textnormal{End}(E_0))$ be a sequence of symbols. Then, there exists $b\in {S}^{m_0,\mathcal{L}}_{\rho,\delta}((G\times \widehat{G}) \otimes \textnormal{End}(E_0))$ such that,
    \begin{equation}
        b\sim \sum_{j=0}^\infty b_j,
    \end{equation}
  in the sense that  for any $N\in \mathbb{N},$
    \begin{equation}
        b-\sum_{j=0}^{N}b_j\in {S}^{m_{N+1},\mathcal{L}}_{\rho,\delta}((G\times \widehat{G}) \otimes \textnormal{End}(E_0)). 
    \end{equation}
\end{lemma} 
\begin{proof}
    Let $\phi\in C^\infty(\mathbb{R}_0^+)$ be a smooth function such that $\phi\equiv 0,$ on the interval $[0,1/2],$ and such that $\phi\equiv 1,$ on  $[1,\infty).$ Note that $1-\phi$ is smooth and compactly supported. Let us consider the family of seminorms $ p_{\alpha,\beta,\rho,\delta,m,\textnormal{left}}(\cdot)$ as in Theorem \ref{characterisations}.   In view of Lemma 6.4 of \cite{RuzhanskyCardona2020}, if $$\forall 1\leq i\leq d_\tau,\,\forall 1\leq r\leq d_\omega,\,  (i,r,x,[\xi])\mapsto b_{j}(i,r,x,[\xi])$$ denotes the matrix-valued symbol associated to $b_j,$ then we have the estimate 
    \begin{equation}
       \exists C>0,\forall t\in (0,1),\,\, 
    \end{equation}
$$ p_{\alpha,\beta,\rho,\delta,m_0,\textnormal{left}}((i,r,x,[\xi])\mapsto b_{j}(i,r,x,[\xi])\phi (t\widehat{\mathcal{M}}(\xi)))$$ 
$$\lesssim_{m_0}  p_{\alpha,\beta,\rho,\delta,m_j,\textnormal{left}}\left((i,r,x,[\xi])\mapsto b_{j}(i,r,x,[\xi])\right)t^{m_0-m_j}.$$
In particular, if $t_j\rightarrow 0^+,$ is such that $$p_{\alpha,\beta,\rho,\delta,m_j,\textnormal{left}}((i,r,x,[\xi])\mapsto b_{j}(i,r,x,[\xi]))t^{m_0-m_j}\leq 2^{-j},$$ then we define the symbol $$\tilde{b}_j(i,r,x,[\xi])= {b}_j(i,r,x,[\xi]) \phi (t_j\widehat{\mathcal{M}}(\xi)).$$ Since
$$\sum_{j=0}^{\infty}p_{\alpha,\beta,\rho,\delta,m_0,\textnormal{left}}((i,r,x,[\xi])\mapsto \tilde b_{j}(i,r,x,[\xi]))\leq \sum_{j=0}^\infty 2^{-j}<\infty,$$ and taking into account that the class $S^{m_0,\mathcal{L}}_{\rho,\delta}(G\times \widehat{G})$ is a Fr\'echet space endowed with  the family of seminorms $ p_{\alpha,\beta,\rho,\delta,m_0,\textnormal{left}},$ there exists $ b_{i,r}\in S^{m_0,\mathcal{L}}_{\rho,\delta}(G\times \widehat{G})$ such that $ b_{i,r}=\sum_{j=0}^\infty\tilde b_{j}(i,r,\cdot,\cdot)$ with the sum converging in the topology induced by the Fr\'echet structure in $S^{m_0,\mathcal{L}}_{\rho,\delta}(G\times \widehat{G}).$ Starting the summation at $j=N,$ the same proof above shows that $\sum_{j=N}\tilde b_{j}(i,r,\cdot,\cdot)\in S^{m_N,\mathcal{L}}_{\rho,\delta}(G\times \widehat{G}).$ Hence, the symbol defined by
$$b(i,r,x,[\xi])=\sum_{j=0}^\infty b_{ir}(x,\xi) ,$$ satisfies the condition of the theorem. Indeed, for any $N,$ we have that 
\begin{align*}
  b(i,r,x,[\xi])-&\sum_{j=0}^N b_j(i,r,x,\xi) \\
  &= \sum_{j=0}^N b_j(i,r,x,\xi)(1-\phi)(t\widehat{\mathcal{M}}(\xi))+ \sum_{j=N+1}\tilde b_{j}(i,r,\cdot,\cdot).
\end{align*} Note that the symbols $ b_j(i,r,x,\xi)(1-\phi)(t\widehat{\mathcal{M}}(\xi))$ are smoothing terms, namely these are symbols in $S^{-\infty,\mathcal{L}}_{\rho,\delta}(G\times \widehat{G}),$  since  $1-\phi$ is smooth and compactly supported. Because $\sum_{j=N+1}\tilde b_{j}(i,r,\cdot,\cdot)\in S^{m_{N+1},\mathcal{L}}_{\rho,\delta}(G\times \widehat{G})$ we have that the difference $ b(i,r,x,[\xi])-\sum_{j=0}^N b_j(i,r,x,\xi) $ belongs to $ S^{m_{N+1},\mathcal{L}}_{\rho,\delta}(G\times \widehat{G}).$ At the level of the vector-valued symbol $b:=b(x,\xi)$ induced by the matrix-valued symbol $(i,r,x,[\xi])\mapsto b(i,r,x,[\xi]),$ we have proved that
$$   b-\sum_{j=0}^{N}b_j\in {S}^{m_{N+1},\mathcal{L}}_{\rho,\delta}((G\times \widehat{G}) \otimes \textnormal{End}(E_0)).$$ The proof is complete.    
\end{proof}

Now, we  present a technical result concerning   the existence of parametrices for $\mathcal{L}$-elliptic operators in the subelliptic calculus. We denote $$ {\Psi}^{-\infty,\mathcal{L}}((G\times \widehat{G})\otimes \textnormal{End}(E_0)):=\bigcap_{m\in \mathbb{R}}{\Psi}^{m,\mathcal{L}}_{\rho,\delta}((G\times \widehat{G}) \otimes \textnormal{End}(E_0)).$$

\begin{proposition}\label{IesTParametrix} Let $m\in \mathbb{R},$  $0\leqslant \delta<\rho\leqslant 1.$  Let  $a=a(x,\xi)\in {S}^{m,\mathcal{L}}_{\rho,\delta}((G\times \widehat{G}) \otimes \textnormal{End}(E_0)).$  Assume also that $a(x,\xi)\in \mathbb{C}^{d_\xi\times d_\xi}(\textnormal{End}(E_0))$ is invertible for every $(x,[\xi])\in G\times\widehat{G},$  and satisfies
\begin{equation}\label{Iesparametrix}
   \sup_{(x,[\xi])\in G\times \widehat{G}} \Vert \widehat{\mathcal{M}}(\xi)^{m} \otimes a(x,\xi)^{-1}\Vert_{\textnormal{op}}<\infty.
\end{equation}Then, there exists $B\in {\Psi}^{-m,\mathcal{L}}_{\rho,\delta}((G\times \widehat{G}) \otimes \textnormal{End}(E_0)),$ such that $AB-I\in {\Psi}^{-\infty,\mathcal{L}}((G\times \widehat{G}) \otimes \textnormal{End}(E_0)). $  Also, there exists  $B'\in {\Psi}^{-m,\mathcal{L}}_{\rho,\delta}((G\times \widehat{G}) \otimes \textnormal{End}(E_0)),$ such that $B'A-I\in {\Psi}^{-\infty,\mathcal{L}}((G\times \widehat{G}) \otimes \textnormal{End}(E_0)). $ 
\end{proposition} 
\begin{proof}
Let us follow the argument in H\"ormander \cite[Page 73]{Hormander1985III}.  Let $b_0(x,\xi)=a(x,\xi)^{-1},$ and let us consider its associated pseudo-differential operator $b_0(x,D).$ By the subelliptic calculus, we have that
$$ -r(x,D):=a(x,D)b_0(x,D)-I\in \textnormal{Op}( S^{-(\rho-\delta),\mathcal{L}}_{\rho,\delta}(G\times \widehat{G})). $$
Define $b_k(x,D)=b_0(x,D)r(x,D)^{k} \in \textnormal{Op}(S^{-m-k(\rho-\delta),\mathcal{L}}_{\rho,\delta}(G\times \widehat{G})),$ for any $k\in \mathbb{N}.$ By applying Lemma \ref{lemma:asymp:exp}, consider a symbol $ b(x,\xi)$ satisfying the asymptotic expansion $b(x,\xi)\sim \sum_{k=0}^\infty b_k(x,\xi).$
Define $S_N(x,\xi)=\sum_{k=0}^{N-1} b_k(x,\xi).$ Note that
\begin{align*}
    a(x,D)S_N(x,D)-I  &=a(x,D)\left(\sum_{k=0}^{N-1} b_0(x,D)r(x,D)^{k}\right)-I\\
    & =a(x,D)b_0(x,D)\left(\sum_{k=0}^{N-1} r(x,D)^{k}\right)-I\\
    &  =(I-r(x,D))\left(\sum_{k=0}^{N-1} r(x,D)^{k}\right)-I\\
    &  =I-r(x,D)^{N}-I\\
     &  =-r(x,D)^{N}\in \textnormal{Op}(S^{-N(\rho-\delta),\mathcal{L}}_{\rho,\delta}(G\times \widehat{G})).
\end{align*}Since $b(x,D)-S_{N}(x,D)\in  \textnormal{Op}(S^{-m-N(\rho-\delta),\mathcal{L}}_{\rho,\delta}(G\times \widehat{G})),$ we have that
\begin{align*}
 & a(x,D)b(x,D) -I\\
 &= a(x,D)(b(x,D)-S_N(x,D))+  a(x,D)S_N(x,D)-I \\
  &=a(x,D)(b(x,D)-S_N(x,D))-r(x,D)^{N}\in \textnormal{Op}(S^{-N(\rho-\delta),\mathcal{L}}_{\rho,\delta}(G\times \widehat{G})),
\end{align*} for all $N\in \mathbb{N}.$ Then $B=b(x,D)$ is such that $a(x,D)b(x,D) -I\in \textnormal{Op}(S^{-\infty,\mathcal{L}}_{\rho,\delta}(G\times \widehat{G})) .$ The same argument above can be applied to construct $B'=b'(x,D)\sim \sum_{k=0}^\infty b_k(x,D),$ by defining $$b_k(x,D)=r(x,D)^{k}b_0(x,D) \in \textnormal{Op}(S^{-m-k(\rho-\delta),\mathcal{L}}_{\rho,\delta}(G\times \widehat{G})).$$ 
 The proof of Proposition \ref{IesTParametrix} is complete.
\end{proof}

\subsection{$\mathcal{L}$-elliptic operators on homogeneous vector bundles} 
Now, we will prove the existence of parametrices for $\mathcal{L}$-elliptic operators  in the subelliptic calculus. We denote $$ {\Psi}^{-\infty,\mathcal{L}}(E)=\bigcap_{m\in \mathbb{R}}{\Psi}^{m,\mathcal{L}}_{\rho,\delta}(E).$$

\begin{proposition}\label{IesTParametrix2} Let $m\in \mathbb{R},$ and let $0\leqslant \delta<\rho\leqslant 1.$  Let  $a=a(x,\xi)\in {S}^{m,\mathcal{L}}_{\rho,\delta}((G\times \widehat{G}) \otimes \textnormal{End}(E_0)).$  Assume also that $a(x,\xi)\in \mathbb{C}^{d_\xi\times d_\xi}(\textnormal{End}(E_0))$ is invertible for every $(x,[\xi])\in G\times\widehat{G},$ and satisfies
\begin{equation}\label{Iesparametrix2}
   \sup_{(x,[\xi])\in G\times \widehat{G}} \Vert \widehat{\mathcal{M}}(\xi)^{m} \otimes a(x,\xi)^{-1}\Vert_{\textnormal{op}}<\infty.
\end{equation}Let $\tilde{A}\in \Psi^{m,\mathcal{L}}_{\rho,\delta}(E)$ be the operator associated to the symbol  $a$ via \eqref{quantizationonhomogeneous2}. Then, there exists $\tilde{B}\in {\Psi}^{-m,\mathcal{L}}_{\rho,\delta}(E),$ such that $\tilde{A}\tilde{B}-I\in {\Psi}^{-\infty,\mathcal{L}}(E). $ Also there exists  $\tilde{B}'\in {\Psi}^{-m,\mathcal{L}}_{\rho,\delta}(E),$ such that $\tilde{B}'\tilde{A}-I\in {\Psi}^{-\infty,\mathcal{L}}(E). $ 
\end{proposition}
\begin{proof}
     If $A$ and $B$ are the vector-valued pseudo-differential operators associated with $\tilde{A}$ and $\tilde{B},$ Proposition  \ref{IesTParametrix} shows that $AB-I,B'A-I\in \textnormal{Op}( {S}^{-\infty,\mathcal{L}}((G\times \widehat{G}) \otimes \textnormal{End}(E_0))) $ which immediately implies that $\tilde{A}\tilde{B}-I,\tilde{B}'\tilde{A}-I\in {\Psi}^{-\infty,\mathcal{L}}(E). $ So, the proof is complete.
\end{proof}

\subsection{Fredholmness on homogeneous vector bundles} In this subsection, we discuss the Fredholmness of $\mathcal{L}$-elliptic pseudo-differential operators. 

First, let us briefly recall the definition of Fredholm operators and index of an operator. 
Let $X$ and $Y$ be two Banach spaces. A continuous linear operator $A:X \rightarrow Y$ is a Fredholm operator if its kernel $\textnormal{Ker}(A)$ is finite-dimensional and the range (the image) $\textnormal{Im}(A)$ is closed and of finite codimension. Then its index is defined as 
$$\textnormal{Ind(A)}:=\textnormal{dim Ker}(A)-\textnormal{codim Im}(A),$$
or equivalently 
$$\textnormal{Ind(A)}:=\textnormal{dim Ker}(A)-\textnormal{dim Ker}(A^*),$$ where $A^*:Y^{*}\rightarrow X^{*}$ is the adjoint of $A.$

The following criterion for a closed linear operator to be Fredholm is 
usually attributed to Atkinson \cite{At}.
\begin{theorem} \label{Atki}
Let $A$ be a closed linear operator from a complex Banach space $X$ into a 
complex Banach space $Y$ with dense domain $\mathcal{D}(A).$ Then $A$ is 
Fredholm if and only if we can find a bounded linear operator 
$B:Y\rightarrow X,$ a compact operator $K_{1}:X\rightarrow X$ and a 
compact operator $K_2:Y\rightarrow Y$ such that $BA=I+K_1$ on $\mathcal{D}(A)$ and $AB=I+K_2$ on $Y$.
\end{theorem}

In the following theorem we will show that $\mathcal{L}$-ellipticity implies the Fredholmness. 

\begin{theorem} \label{Vector:Fred} Let us consider the identifications $E_{0}^{*}\cong E_0.$  Let $m\in \mathbb{R},$ and let $0\leqslant \delta<\rho\leqslant 1.$  Let  $a=a(x,\xi)\in {S}^{m,\mathcal{L}}_{\rho,\delta}((G\times \widehat{G}) \otimes \textnormal{End}(E_0)).$  Assume also that $a(x,\xi)\in \mathbb{C}^{d_\xi\times d_\xi}(\textnormal{End}(E_0))$ is invertible for every $(x,[\xi])\in G\times\widehat{G},$ and satisfies
\begin{equation*}
   \sup_{(x,[\xi])\in G\times \widehat{G}} \Vert \widehat{\mathcal{M}}(\xi)^{m} \otimes a(x,\xi)^{-1}\Vert_{\textnormal{op}}<\infty.
\end{equation*} 

Let $A\in \Psi^{m,\mathcal{L}}_{\rho,\delta}((G \times \widehat{G}) \otimes \textnormal{End}(E_0))$ be the operator associated to the symbol  $a$. Then $A$ is a Fredholm operator from $L^{2, \mathcal{L}}_s(G, E_0)$ to $L^{2, \mathcal{L}}_{s-m}(G, E_0)$ for all $s \in \mathbb{R}.$ 
Moreover,  the kernels $\textnormal{Ker}(A),$ $\textnormal{Ker}(A^*)$ are in $C^\infty(G, E_0),$ so that $\textnormal{Ker}(A),$ $\textnormal{Ker}(A^*)$ and $\textnormal{Ind}(A)$ are independent on $s.$ 

In particular, $A$ is a Fredholm operator from $L^{2, \mathcal{L}}_s(G, E_0)^\tau$ to $L^{2, \mathcal{L}}_{s-m}(G, E_0)^\tau$ for all $s \in \mathbb{R}.$ 
Moreover, the kernels $\textnormal{Ker}(A),$ $\textnormal{Ker}(A^*)$ are in $C^\infty(G, E_0)^\tau,$ so that $\textnormal{Ker}(A),$ $\textnormal{Ker}(A^*)$ and $\textnormal{Ind}(A)$ are independent on $s.$
\end{theorem}
\begin{proof}
First, let us remark that any vector-valued operator $T$,  with symbol in the class ${S}^{-\infty,\mathcal{L}}_{\rho,\delta}((G\times \widehat{G}) \otimes \textnormal{End}(E_0)),$ is a compact operator from  $L^2(G)$ to itself,  and  from any subelliptic Sobolev space $L^{2,\mathcal{L}}_{s_1}(G,E_0)$ into $L^{2,\mathcal{L}}_{s_2}(G,E_0),$ for any pair $s_1,s_2\in \mathbb{R}.$ 

Indeed, its compactness on $L^2(G,E_0)$ can be deduced from Theorem    \ref{HS:cHara:VB} and the fact that any Hilbert-Schmidt operator is compact. Observe that the compactness of $T$ from $L^{2,\mathcal{L}}_{s_1}(G,E_0)$ into $L^{2,\mathcal{L}}_{s_2}(G,E_0),$ follows from the compactness of the operator $\mathcal{M}_{s_2,E_0} T \mathcal{M}_{-s_1,E_0} $ on $L^2(G,E_0).$

Now, we will show that  $A$ is a Fredholm operator from $L^{2, \mathcal{L}}_{s}(G, E_0)$ to $L^{2, \mathcal{L}}_{s-m}(G, E_0)$ for all $s \in \mathbb{R}.$ To prove this suppose that $B \in {S}^{-m,\mathcal{L}}_{\rho,\delta}((G\times \widehat{G}) \otimes \textnormal{End}(E_0))$ be a parametrix of the operator $A$ as in Proposition \ref{IesTParametrix}, that means, there exist operators $T$, $ T' \in {S}^{-\infty,\mathcal{L}}_{\rho,\delta}((G\times \widehat{G}) \otimes \textnormal{End}(E_0))$  such that $BA=I-T$ and $AB=I-T'.$ 

Since the smoothing operators $$T:L^{2, \mathcal{L}}_s(G, E_0) \rightarrow L^{2, \mathcal{L}}_s(G, E_0)$$ and $$T':L^{2, \mathcal{L}}_{s-m}(G, E_0) \rightarrow L^{2, \mathcal{L}}_{s-m}(G, E_0)$$ are compact operators and $B: L^{2, \mathcal{L}}_{s-m}(G, E_0) \rightarrow L^{2, \mathcal{L}}_{s}(G, E_0)$ is a bounded linear operator (Theorem \ref{VectSobosub}) we deduce, from Theorem \ref{Atki}, that $A$ is a Fredholm operator. This means, $\textnormal{Ker}(A)$ and   $ \textnormal{Ker}(A^*) $ are   subspaces  of finite dimension.

Now, we will prove that the index of $A$
$$ \textnormal{Ind}_s(A):=\textnormal{dim Ker}_s(A)-\textnormal{dim Ker}_s(A^*)  $$
is independent of $s\in \mathbb{R}.$ 
To show this, let us consider the space 
$$\textnormal{Ker}_s(A)= \left\{ f \in L^{2, \mathcal{L}}_s(G, E_0): Af=0  \right\}.$$
We note that 
$$ \textnormal{Ker}_s(A) \subset  \textnormal{Ker}_s(BA) =  \textnormal{Ker}_s(I-T) \subset C^\infty(G, E_0).$$ 
Indeed,  $T\in   {S}^{r,\mathcal{L}}_{\rho,\delta}((G\times \widehat{G}) \otimes \textnormal{End}(E_0))$ for any $r\in \mathbb{R},$ and then $$T:L^{2, \mathcal{L}}_s(G, E_0) \rightarrow L^{2, \mathcal{L}}_{s-r}(G, E_0)$$ continuously for any $r\in \mathbb{R}.$ Now, if we  take $f\in \textnormal{Ker}_s(I-T),$  then it follows that $f\in L^{2, \mathcal{L}}_s(G, E_0),$ and $f=Tf\in L^{2, \mathcal{L}}_{s-r}(G, E_0), $ for any $r\in \mathbb{R}.$ This implies (see \eqref{CinftyL})
\begin{equation*}
    f\in  \bigcap_{t \in \mathbb{R}} L^{2, \mathcal{L}}_t(G, E_0)=C^\infty(G, E_0).
\end{equation*}
Similarly, in view of the identification $E_{0}^{*}\cong E_0$ we can see that $\textnormal{Ker}_{m-s}(A^*) \subset C^\infty(G, E_0)$ with the use of Theorem \ref{VectorAdjoint}. Since these kernels are in $C^\infty(G, E_0)$, they do not depend on $s$, and thus the Fredholm index is independent of $s.$






Finally, using the fact $L^{2, \mathcal{L}}_s(G, E_0)^\tau$ is a closed subspace of $L^{2, \mathcal{L}}_s(G, E_0)$ with $\bigcap_{t \in \mathbb{R}} L^{2, \mathcal{L}}_t(G, E_0)^\tau=C^\infty(G, E_0)^\tau$ and Theorem \ref{VectSobosub}, a similar argument can be made to establish that $A$ is a Fredholm operator from $L^{2, \mathcal{L}}_s(G, E_0)^\tau$ to $L^{2, \mathcal{L}}_{s-m}(G, E_0)^\tau$ for all $s \in \mathbb{R},$ and the kernels $\textnormal{Ker}(A),$ $\textnormal{Ker}(A^*)$ are in $C^\infty(G, E_0)^\tau,$ so that $\textnormal{Ker}(A),$ $\textnormal{Ker}(A^*)$ and $\textnormal{Ind}(A)$ are independent on $s.$
\end{proof}

As a consequence of Theorem \ref{Vector:Fred} we have the following result in the context of homogeneous vector bundles.
\begin{theorem}  Let $m\in \mathbb{R},$ and let $0\leqslant \delta<\rho\leqslant 1.$  Let  $a=a(x,\xi)\in {S}^{m,\mathcal{L}}_{\rho,\delta}((G\times \widehat{G}) \otimes \textnormal{End}(E_0)).$  Assume also that $a(x,\xi)\in \mathbb{C}^{d_\xi\times d_\xi}(\textnormal{End}(E_0))$ is invertible for every $(x,[\xi])\in G\times\widehat{G},$ and satisfies
\begin{equation}
   \sup_{(x,[\xi])\in G\times \widehat{G}} \Vert \widehat{\mathcal{M}}(\xi)^{m} \otimes a(x,\xi)^{-1}\Vert_{\textnormal{op}}<\infty.
\end{equation}Let $\tilde{A}\in \Psi^{m,\mathcal{L}}_{\rho,\delta}(E)$ be the operator associated to the symbol  $a$ via \eqref{quantizationonhomogeneous2}. Then $\tilde{A}$ is a Fredholm operator from $L^{2, \mathcal{L}}_s(E)$ to $L^{2, \mathcal{L}}_{s-m}(E)$ for all $s \in \mathbb{R}.$ 

Moreover, the kernels $\textnormal{Ker}(\tilde{A}),$ $\textnormal{Ker}(\tilde{A}^*)$ are in $\Gamma^\infty(E),$ so that $\textnormal{Ker}(\tilde{A}),$ $\textnormal{Ker}(\tilde{A}^*)$ and $\textnormal{Ind}(\tilde{A})$ are independent on $s.$
\end{theorem}
\begin{proof} In view of Remark \ref{soboiso},  the following  diagram \begin{eqnarray}
    \begin{tikzpicture}[every node/.style={midway}]
  \matrix[column sep={10em,between origins}, row sep={4em}] at (0,0) {
    \node(R) {$L^{2, \mathcal{L}}_s(E)$}  ; & \node(S) {$L^{2, \mathcal{L}}_{s-m}(E)$}; \\
    \node(R/I) {$L^{2, \mathcal{L}}_s(G, E_0)^\tau$}; & \node (T) {$L^{2, \mathcal{L}}_{s-m}(G, E_0)^\tau$};\\
  };
  \draw[<-] (R/I) -- (R) node[anchor=east]  {$\varkappa_{\tau}$};
  \draw[->] (R) -- (S) node[anchor=south] {$\tilde{A}$};
  \draw[->] (S) -- (T) node[anchor=west] {$\varkappa_\tau$};
  \draw[->] (R/I) -- (T) node[anchor=north] {$A$};
\end{tikzpicture}
 \end{eqnarray} commutes. Now, observe that  $\tilde{A}:L^{2, \mathcal{L}}_s(E) \rightarrow L^{2, \mathcal{L}}_{s-m}(E)$ is a Fredholm operator if and only $A:L^{2, \mathcal{L}}_s(G, E_0) \rightarrow L^{2, \mathcal{L}}_{s-m}(G, E_0)$ is a Fredholm operator. Indeed, using the relation,
 $$AB-I= T \Leftrightarrow \varkappa_\tau A \varkappa_\tau^{-1} \varkappa_\tau B \varkappa_\tau^{-1}=I-\varkappa_\tau T \varkappa_\tau^{-1} \Leftrightarrow \tilde{A}\tilde{B}=I-\tilde{T}$$
 along with Theorem \ref{Atki}, Theorem \ref{Vector:Fred} and the fact that $T$ is compact if and only if $\tilde{T}$ is compact,  the proof of theorem follows. \end{proof}

\subsection{Parameter $\mathcal{L}$-ellipticity} \label{Parameterellipticity}

To develop the subelliptic functional calculus we need a more wide notion of ellipticity, namely, the parameter ellipticity. Here, we will consider the case where $E_0=F_0,$ and every continuous linear operator $A$  on $C^\infty(G,E_0)$ will be considered with its {\it$\textnormal{End}(E_0)$-valued symbol} as in Definition \ref{equivalenceXXX}.

\begin{definition}\label{parameterellipticityonvundles} Let $m>0,$ and let $0\leqslant \delta<\rho\leqslant 1.$
Let $\Lambda=\{\gamma(t):t\in J\footnote{where $J=[a,b],$ $-\infty<a\leqslant b<\infty,$ $J=[a,\infty),$ $J=(-\infty,b]$ or $J=(-\infty,\infty).$}\}$  be an analytic curve in the complex plane $\mathbb{C}.$ If $J$ is a finite interval we assume that $\Lambda$ is a closed curve. For  simplicity, if $J$ is an infinite interval we assume that $\Lambda$ is homotopy equivalent to the line $\Lambda_{i\mathbb{R}}:=\{iy:-\infty<y<\infty\}.$  Let  $a=a(x,\xi)\in {S}^{m,\mathcal{L}}_{\rho,\delta}((G\times \widehat{G})\otimes \textnormal{End}(E_0)).$  Assume also that $$R_{\lambda}(x,\xi)^{-1}:=a(x,\xi)-\lambda \in \mathbb{C}^{d_\xi\times d_\xi}( \textnormal{End}(E_0)), \footnote{We have denoted $a(x,\xi)-\lambda:=a(x,\xi)-\lambda I_{ \mathbb{C}^{d_\xi\times d_\xi}( \textnormal{End}(E_0))}$ to simplify the notation.}$$ is invertible   for every $(x,[\xi])\in G\times\widehat{G},$ and for all $\lambda \in  \Lambda.$ We say that $a$ is parameter $\mathcal{L}$-elliptic with respect to $\Lambda,$ if
\begin{equation}
    \sup_{\lambda\in \Lambda}\sup_{(x,[\xi])\in G\times \widehat{G}}\Vert (|\lambda|^{\frac{1}{m}}+\widehat{\mathcal{M}}(\xi))^{m}\otimes R_{\lambda}(x,\xi)\Vert_{\textnormal{op}}<\infty.
\end{equation}
\end{definition}
\begin{remark}
Observe that for $a=b=0,$ $J=\{0\},$ and for the trivial curve $\gamma(t)=0,$ that $a$ is parameter $\mathcal{L}$-elliptic with respect to $\Lambda=\{0\},$ is equivalent to say that $a$ is  $\mathcal{L}$-elliptic.
\end{remark}
The following theorem classifies the matrix resolvent $R_{\lambda}(x,\xi)$ of a parameter $\mathcal{L}$-elliptic symbol $a.$
\begin{theorem}\label{lambdalambdita} Let $m>0,$ and let $0\leqslant \delta<\rho\leqslant 1.$ If $a$ is parameter $\mathcal{L}$-elliptic with respect to $\Lambda,$ the following estimate 
\begin{equation*}
   \sup_{\lambda\in \Lambda}\sup_{(x,[\xi])\in G\times \widehat{G}}\Vert (|\lambda|^{\frac{1}{m}}+\widehat{\mathcal{M}}(\xi))^{m(k+1)}\widehat{\mathcal{M}}(\xi)^{\rho|\alpha|-\delta|\beta|}\otimes  \partial_{\lambda}^k\partial_{X}^{(\beta)}\Delta_{\xi}^{\alpha}R_{\lambda}(x,\xi)\Vert_{\textnormal{op}}<\infty,
\end{equation*}holds true for all $\alpha,\beta\in \mathbb{N}_0^n$ and $k\in \mathbb{N}_0.$

\end{theorem}
\begin{proof}
    We will split the proof in the cases $|\lambda|\leqslant 1,$  and $|\lambda|> 1,$ where $\lambda\in \Lambda.$ It is possible however that one of these two cases could be trivial in the sense that $\Lambda_{1}:=\{\lambda\in \Lambda:|\lambda|\leqslant 1\}$ or $\Lambda_{1}^{c}:=\{\lambda\in \Lambda:|\lambda|> 1\}$ could be empty sets. In such a case the proof is self-contained in the situation that we will consider where we assume that $\Lambda_{1}$ and $\Lambda_{1}^c$ are not trivial sets.   For $|\lambda|\leqslant 1,$ observe that
\begin{align*}
    &\Vert (|\lambda|^{\frac{1}{m}}+\widehat{\mathcal{M}}(\xi))^{m(k+1)}\widehat{\mathcal{M}}(\xi)^{\rho|\alpha|-\delta|\beta|} \otimes \partial_{\lambda}^k\partial_{X}^{(\beta)}\Delta_{\xi}^{\alpha}R_{\lambda}(x,\xi)\Vert_{\textnormal{op}}\\
    &=\Vert (|\lambda|^{\frac{1}{m}}+\widehat{\mathcal{M}}(\xi))^{m(k+1)}\widehat{\mathcal{M}}(\xi)^{-m(k+1)}\widehat{\mathcal{M}}(\xi)^{m(k+1)+\rho|\alpha|-\delta|\beta|} \otimes \partial_{\lambda}^k\partial_{X}^{(\beta)}\Delta_{\xi}^{\alpha}R_{\lambda}(x,\xi)\Vert_{\textnormal{op}}\\
    &\leqslant \Vert (|\lambda|^{\frac{1}{m}}+\widehat{\mathcal{M}}(\xi))^{m(k+1)}\widehat{\mathcal{M}}(\xi)^{-m(k+1)}\Vert_{\textnormal{op}}\\
    &\hspace{5cm}\times\Vert\widehat{\mathcal{M}}(\xi)^{m(k+1)+\rho|\alpha|-\delta|\beta|} \otimes\partial_{\lambda}^k\partial_{X}^{(\beta)}\Delta_{\xi}^{\alpha}R_{\lambda}(x,\xi)\Vert_{\textnormal{op}}.
\end{align*} We have 
\begin{align*}
    &\Vert (|\lambda|^{\frac{1}{m}}+\widehat{\mathcal{M}}(\xi))^{m(k+1)}\widehat{\mathcal{M}}(\xi)^{-m(k+1)}\Vert_{\textnormal{op}}\\
    &= \Vert (|\lambda|^{\frac{1}{m}}\widehat{\mathcal{M}}(\xi)^{-1}+I_{d_\xi})^{m(k+1)}\Vert_{\textnormal{op}}\leqslant \Vert |\lambda|^{\frac{1}{m}}\widehat{\mathcal{M}}(\xi)^{-1}+I_{d_\xi}\Vert^{m(k+1)}_{\textnormal{op}}\\
    &\leqslant \sup_{|\lambda|\in [0,1]}\sup_{1\leqslant j\leqslant d_\xi}(|\lambda|^{\frac{1}{m}}(1+\nu_{jj}(\xi)^2)^{-\frac{1}{2}}+1))^{k(m+1)}\\
    &=O(1).
\end{align*} On the other hand, we can prove that
\begin{align*}
   \Vert\widehat{\mathcal{M}}(\xi)^{m(k+1)+\rho|\alpha|-\delta|\beta|} \otimes \partial_{\lambda}^k\partial_{X}^{(\beta)}\Delta_{\xi}^{\alpha}R_{\lambda}(x,\xi)\Vert_{\textnormal{op}}=O(1).
\end{align*} For $k=1,$ $\partial_{\lambda}R_{\lambda}(x,\xi)= R_{\lambda}(x,\xi)^{2}.$ This can be deduced from the Leibniz rule, indeed,
\begin{align*}
 0=\partial_{\lambda}(R_{\lambda}(x,\xi) \circ (a(x,\xi)-\lambda))=(\partial_{\lambda}R_{\lambda}(x,\xi)) \circ (a(x,\xi)-\lambda)+ R_{\lambda}(x,\xi)(-1) 
\end{align*}implies that
    \begin{align*}
 \partial_{\lambda}(R_{\lambda}(x,\xi))\circ(a(x,\xi)-\lambda)= R_{\lambda}(x,\xi). 
\end{align*} Because $(a(x,\xi)-\lambda)=R_{\lambda}(x,\xi)^{-1}$ the identity for the first derivative of $R_\lambda,$ $\partial_{\lambda}R_{\lambda}$ it follows. So, from the chain rule we obtain that the term of higher order expanding the derivative   $ \partial_{\lambda}^kR_{\lambda} $ is $ R_{\lambda}^{k+1}.$ From Theorem   \ref{IesT} we deduce that $R_{\lambda}\in S^{-m,\mathcal{L}}_{\rho,\delta}((G\times \widehat{G})\otimes \textnormal{End}(E_0)).$ The subelliptic calculus implies that $R_{\lambda}^{k+1}\in S^{-m(k+1),\mathcal{L}}_{\rho,\delta}((G\times \widehat{G})\otimes \textnormal{End}(E_0)).$ This fact, and the compactness of $\Lambda_1\subset \mathbb{C},$ provide us the uniform estimate 
    \begin{equation*}
   \sup_{\lambda\in \Lambda_1}\sup_{(x,[\xi])\in G\times \widehat{G}}\Vert\widehat{\mathcal{M}}(\xi)^{m(k+1)+\rho|\alpha|-\delta|\beta|} \otimes \partial_{\lambda}^k\partial_{X}^{(\beta)}\Delta_{\xi}^{\alpha}R_{\lambda}(x,\xi)\Vert_{\textnormal{op}}<\infty.
\end{equation*}Now, we will analyse the situation for $\lambda\in \Lambda_1^c.$ We will use induction over $k$ in order to prove that
\begin{equation*}
   \sup_{\lambda\in \Lambda_1^c}\sup_{(x,[\xi])\in G\times \widehat{G}}\Vert (|\lambda|^{\frac{1}{m}}+\widehat{\mathcal{M}}(\xi))^{m(k+1)}\widehat{\mathcal{M}}(\xi)^{\rho|\alpha|-\delta|\beta|} \otimes \partial_{\lambda}^k\partial_{X}^{(\beta)}\Delta_{\xi}^{\alpha}R_{\lambda}(x,\xi)\Vert_{\textnormal{op}}<\infty.
\end{equation*}
For $k=0$ notice that
\begin{align*}
     &\Vert (|\lambda|^{\frac{1}{m}}+\widehat{\mathcal{M}}(\xi))^{m(k+1)}\widehat{\mathcal{M}}(\xi)^{\rho|\alpha|-\delta|\beta|} \otimes \partial_{\lambda}^k\partial_{X}^{(\beta)}\Delta_{\xi}^{\alpha}R_{\lambda}(x,\xi)\Vert_{\textnormal{op}}\\
     &=\Vert (|\lambda|^{\frac{1}{m}}+\widehat{\mathcal{M}}(\xi))^{m}\widehat{\mathcal{M}}(\xi)^{\rho|\alpha|-\delta|\beta|}\otimes \partial_{X}^{(\beta)}  \Delta_{\xi}^{\alpha}(a(x,\xi)-\lambda)^{-1}\Vert_{\textnormal{op}},
\end{align*}and denoting $\theta=\frac{1}{|\lambda|},$ $\omega=\frac{\lambda}{|\lambda|},$ we have
 \begin{align*}
     &\Vert (|\lambda|^{\frac{1}{m}}+\widehat{\mathcal{M}}(\xi))^{m(k+1)}\widehat{\mathcal{M}}(\xi)^{\rho|\alpha|-\delta|\beta|} \otimes \partial_{\lambda}^k\partial_{X}^{(\beta)}\Delta_{\xi}^{\alpha}R_{\lambda}(x,\xi)\Vert_{\textnormal{op}}\\
     &=\Vert (|\lambda|^{\frac{1}{m}}+\widehat{\mathcal{M}}(\xi))^{m}|\lambda|^{-1}\widehat{\mathcal{M}}(\xi)^{\rho|\alpha|-\delta|\beta|} \otimes \partial_{X}^{(\beta)}\Delta_{\xi}^{\alpha}(\theta\times  a(x,\xi)-\omega)^{-1}\Vert_{\textnormal{op}}\\
     &=\Vert (1+|\lambda|^{-\frac{1}{m}}\widehat{\mathcal{M}}(\xi))^{m}\widehat{\mathcal{M}}(\xi)^{\rho|\alpha|-\delta|\beta|} \otimes \partial_{X}^{(\beta)}\Delta_{\xi}^{\alpha}(\theta\times  a(x,\xi)-\omega)^{-1}\Vert_{\textnormal{op}}\\
     &=\Vert (1+\theta^{\frac{1}{m}}\widehat{\mathcal{M}}(\xi))^{m}\widehat{\mathcal{M}}(\xi)^{\rho|\alpha|-\delta|\beta|} \otimes \partial_{X}^{(\beta)}\Delta_{\xi}^{\alpha}(\theta\times  a(x,\xi)-\omega)^{-1}\Vert_{\textnormal{op}}\\
     &=\Vert (1+\theta^{\frac{1}{m}}\widehat{\mathcal{M}}(\xi))^{m}\widehat{\mathcal{M}}(\xi)^{-m}\widehat{\mathcal{M}}(\xi)^{m+\rho|\alpha|-\delta|\beta|} \otimes \partial_{X}^{(\beta)}  \Delta_{\xi}^{\alpha}(\theta\times  a(x,\xi)-\omega)^{-1}\Vert_{\textnormal{op}}\\
     &\leqslant \Vert (1+\theta^{\frac{1}{m}}\widehat{\mathcal{M}}(\xi))^{m}\widehat{\mathcal{M}}(\xi)^{-m}\Vert_{\textnormal{op}}\Vert\widehat{\mathcal{M}}(\xi)^{m+\rho|\alpha|-\delta|\beta|} \otimes \partial_{X}^{(\beta)}\Delta_{\xi}^{\alpha}(\theta\times  a(x,\xi)-\omega)^{-1}\Vert_{\textnormal{op}}.
\end{align*}   Because $ (1+\theta^{\frac{1}{m}}\widehat{\mathcal{M}}(\xi))^{m}\widehat{\mathcal{M}}(\xi)^{-m} \in S^{0,\mathcal{L}}_{\rho,\delta}(G\times \widehat{G}) ,$ we have that the operator norm $ \Vert (1+\theta^{\frac{1}{m}}\widehat{\mathcal{M}}(\xi))^{m}\widehat{\mathcal{M}}(\xi)^{-m}\Vert_{\textnormal{op}}$ is uniformly bounded in $\theta\in [\varepsilon_0
    ,1]$ where $\varepsilon_{0}>0   $ is small enough and to be chosen later.  The same argument can be applied to the operator norm $$ \Vert\widehat{\mathcal{M}}(\xi)^{m+\rho|\alpha|-\delta|\beta|}\otimes \partial_{X}^{(\beta)}\Delta_{\xi}^{\alpha}(\theta\times  a(x,\xi)-\omega)^{-1}\Vert_{\textnormal{op}}, $$
    by using that $(\theta\times  a(x,\xi)-\omega)^{-1}\in S^{-m,\mathcal{L}}_{\rho,\delta}((G\times \widehat{G})\otimes \textnormal{End}(E_0)), $ with $\theta\in [\varepsilon_0
    ,1]$  and with $\omega$ being an element of the complex circle. Now, let us analyse the case where $\theta\rightarrow 0^{+}.$ Note that in this case $|\omega|= 1.$ Consider the operator norm
    $$\Vert (1+\theta^{\frac{1}{m}}\widehat{\mathcal{M}}(\xi))^{m}\widehat{\mathcal{M}}(\xi)^{-m}\widehat{\mathcal{M}}(\xi)^{m+\rho|\alpha|-\delta|\beta|} \otimes \partial_{X}^{(\beta)}  \Delta_{\xi}^{\alpha}(\theta\times  a(x,\xi)-\omega)^{-1}\Vert_{\textnormal{op}}$$
    when $|\alpha|=|\beta|=0.$ Then, we have that
    \begin{align*}
       &\lim_{\theta\rightarrow0^{+}} \Vert (1+\theta^{\frac{1}{m}}\widehat{\mathcal{M}}(\xi))^{m}\widehat{\mathcal{M}}(\xi)^{-m}\widehat{\mathcal{M}}(\xi)^{m} \otimes(\theta\times  a(x,\xi)-\omega)^{-1}\Vert_{\textnormal{op}}\\
       &\lesssim\lim_{\theta\rightarrow0^{+}} \Vert (\theta\times  a(x,\xi)-\omega)^{-1}\Vert_{\textnormal{op}}=|\omega|^{-1}=1.
    \end{align*}On the other hand, if $|\beta|\neq 0,$ when  $\theta\rightarrow 0^{+},\,|\omega|=1,$ the symbol $(\theta\times  a(x,\xi)-\omega)^{-1}$ tends to the constant operator $\omega^{-1}I,$ and the derivative   $ \partial_{X}^{(\beta)}  \Delta_{\xi}^{\alpha}(\theta\times  a(x,\xi)-\omega)^{-1}$ tends to the null operator with respect to the operator norm of matrices and uniformly in $[\xi]\in \widehat{G}$. Consequently, the term
    $$\Vert (1+\theta^{\frac{1}{m}}\widehat{\mathcal{M}}(\xi))^{m}\widehat{\mathcal{M}}(\xi)^{-m}\widehat{\mathcal{M}}(\xi)^{m+\rho|\alpha|-\delta|\beta|} \otimes \partial_{X}^{(\beta)}  \Delta_{\xi}^{\alpha}(\theta\times  a(x,\xi)-\omega)^{-1}\Vert_{\textnormal{op}}$$
    can be estimated from above by $O(1).$
    
    The case  $k\geqslant 1$ for $\lambda\in \Lambda_1^c$ can be proved in an analogous way.
\end{proof} Combining Proposition \ref{IesTParametrix} and Theorem \ref{lambdalambdita} we obtain the following corollaries.
\begin{corollary}\label{parameterparametrix}
Let $m>0,$ and let $0\leqslant \delta<\rho\leqslant 1.$ Let  $a$ be a parameter $\mathcal{L}$-elliptic symbol with respect to $\Lambda.$ Then  there exists a parameter-dependent parametrix of $A-\lambda I,$ with symbol $a^{-\#}(x,\xi,\lambda)$ satisfying the estimates
\begin{equation*}
   \sup_{\lambda\in \Lambda}\sup_{(x,[\xi])\in G\times \widehat{G}}\Vert (|\lambda|^{\frac{1}{m}}+\widehat{\mathcal{M}}(\xi))^{m(k+1)}\widehat{\mathcal{M}}(\xi)^{\rho|\alpha|-\delta|\beta|}\otimes \partial_{\lambda}^k\partial_{X}^{(\beta)}\Delta_{\xi}^{\alpha}a^{-\#}(x,\xi,\lambda)\Vert_{\textnormal{op}}<\infty,
\end{equation*}for all $\alpha,\beta\in \mathbb{N}_0^n$ and $k\in \mathbb{N}_0.$
\end{corollary}
\begin{corollary}\label{resolv}
Let $m>0,$ and let $a\in S^{m,\mathcal{L}}_{\rho,\delta}((G\times \widehat{G}) \otimes \textnormal{End}(E_0)) $ where  $0\leqslant \delta<\rho\leqslant 1.$ Let us assume that $\Lambda$ is a subset of the $L^2$-resolvent set of $A,$ $\textnormal{Resolv}(A):=\mathbb{C}\setminus \textnormal{Spec}(A).$ Then $A-\lambda I$ is invertible on $\mathscr{D}'(G)$ and the symbol of the resolvent operator $\mathcal{R}_{\lambda}:=(A-\lambda I)^{-1},$ $\widehat{\mathcal{R}}_{\lambda}(x,\xi)$ belongs to $S^{-m,\mathcal{L}}_{\rho,\delta}((G\times \widehat{G}) \otimes \textnormal{End}(E_0)).$ 
\end{corollary}

\subsection{Vector-valued subelliptic global functional calculus}\label{S8}
Let $a\in S^{m,\mathcal{L}}_{\rho,\delta}((G\times \widehat{G}) \otimes \textnormal{End}(E_0))$ be a parameter $\mathcal{L}$-elliptic symbol  of order $m>0$ with respect to the sector $\Lambda\subset\mathbb{C}.$ For a pseudo-differential operator $A$ with symbol $a,$ let us define the operator $F(A)$  by the (Dunford-Riesz) complex functional calculus
\begin{equation}\label{F(A)}
    F(A)=-\frac{1}{2\pi i}\oint\limits_{\partial \Lambda_\varepsilon}F(z)(A-zI)^{-1}dz,
\end{equation}where
\begin{itemize}
    \item[(CI).] $\Lambda_{\varepsilon}:=\Lambda\cup \{z:|z|\leqslant \varepsilon\},$ $\varepsilon>0,$ and $\Gamma=\partial \Lambda_\varepsilon\subset\textnormal{Resolv}(A)$ is a positively oriented curve in the complex plane $\mathbb{C}$.
    \item[(CII).] $F$ is a holomorphic function in $\mathbb{C}\setminus \Lambda_{\varepsilon},$ and continuous on its closure. 
    \item[(CIII).] We will assume  decay of $F$ along $\partial \Lambda_\varepsilon$ in order that the operator \eqref{F(A)} will be densely defined on $C^\infty(G, E_0)$ in the strong sense of the topology on $L^2(G, E_0)$ (for instance, $|F(\lambda)|\leqslant C|\lambda|^s$ uniformly in $\lambda,$ for some $s<0$ will be enough).
\end{itemize} Now, we will compute the  $\textnormal{End}(E_0)$-valued symbols for operators defined by this complex functional calculus.
\begin{lemma}\label{LemmaFC}
Let $a\in S^{m,\mathcal{L}}_{\rho,\delta}((G\times \widehat{G}) \otimes \textnormal{End}(E_0))$ be a parameter $\mathcal{L}$-elliptic symbol  of order $m>0$ with respect to the sector $\Lambda\subset\mathbb{C}.$ Let $F(A):C^\infty(G, E_0)\rightarrow \mathscr{D}'(G, E_0)$ be the operator defined by the analytical functional calculus as in \eqref{F(A)}. Under the assumptions $\textnormal{(CI)}$, $\textnormal{(CII)}$, and $\textnormal{(CIII)}$, the $\textnormal{End}(E_0)$-valued symbol of $F(A),$ $\sigma_{F(A)}(x,\xi)$ is given by,
\begin{equation*}
    \sigma_{F(A)}(x,\xi)=-\frac{1}{2\pi i}\oint\limits_{\partial \Lambda_\varepsilon}F(z)\widehat{\mathcal{R}}_z(x,\xi)dz,
\end{equation*}where $\mathcal{R}_z=(A-zI)^{-1}$ denotes the resolvent of $A,$ and $\widehat{\mathcal{R}}_z=\widehat{\mathcal{R}}_z(x,\xi)\in S^{-m,\mathcal{L}}_{\rho,\delta}((G\times \widehat{G})\otimes \textnormal{End}(E_0) ) $ is its symbol.
\end{lemma}
\begin{proof}
 From Corollary \ref{resolv}, we have that  $\widehat{\mathcal{R}}_z\in S^{-m,\mathcal{L}}_{\rho,\delta}((G\times \widehat{G})\otimes \textnormal{End}(E_0) ) .$ Now, observe from \eqref{Homogeneoussymbol} that, for $i,r \in I_{d_\omega},$  we have
 \begin{align*}
  \sigma_{F(A)}(i, r, x,\xi)&=\xi(x)^* e_{r,E_0}^*[F(A)(\xi \otimes e_{i, E_0})(x)]\\&=-\frac{1}{2\pi i}\oint\limits_{\partial \Lambda_\varepsilon}F(z) \xi(x)^* e_{r,E_0}^*[(A-zI)^{-1}(\xi \otimes e_{i, E_0})(x)]dz.  \end{align*} We finish the proof by observing that $\widehat{\mathcal{R}}_z(i,r, x,\xi)=\xi(x)^* e_{r,E_0}^*[(A-zI)^{-1}(\xi \otimes e_{i, E_0})(x)],$  for every $z\in \textnormal{Resolv}(A).$
\end{proof}
Assumption (CIII) will be clarified in the following theorem where we show that the subelliptic calculus on homogeneous vector bundles is stable under the action of the complex functional calculus.
\begin{theorem}\label{DunforRiesz}
Let $m>0,$ and let $0\leqslant \delta<\rho\leqslant 1.$ Let  $a\in S^{m,\mathcal{L}}_{\rho,\delta}((G\times \widehat{G})\otimes \textnormal{End}(E_0))$ be a parameter $\mathcal{L}$-elliptic symbol with respect to $\Lambda.$ Let us assume that $F$ satisfies the  estimate $|F(\lambda)|\leqslant C|\lambda|^s$ uniformly in $\lambda,$ for some $s<0.$  Then  the symbol of $F(A),$  $\sigma_{F(A)}\in S^{ms,\mathcal{L}}_{\rho,\delta}((G\times \widehat{G})\otimes \textnormal{End}(E_0)) $ admits an asymptotic expansion of the form
\begin{equation}\label{asymcomplex}
    \sigma_{F(A)}(x,\xi)\sim 
     \sum_{N=0}^\infty\sigma_{{B}_{N}}(x,\xi),\,\,\,(x,[\xi])\in G\times \widehat{G},
\end{equation}where $\sigma_{{B}_{N}}(x,\xi)\in {S}^{ms-(\rho-\delta)N,\mathcal{L}}_{\rho,\delta}((G\times \widehat{G})\otimes \textnormal{End}(E_0))$ and 
\begin{equation*}
    \sigma_{{B}_{0}}(x,\xi)=-\frac{1}{2\pi i}\oint\limits_{\partial \Lambda_\varepsilon}F(z)(a(x,\xi)-z)^{-1}dz\in {S}^{ms,\mathcal{L}}_{\rho,\delta}((G\times \widehat{G})\otimes \textnormal{End}(E_0)).
\end{equation*}Moreover, 
\begin{equation*}
     \sigma_{F(A)}(x,\xi)\equiv -\frac{1}{2\pi i}\oint\limits_{\partial \Lambda_\varepsilon}F(z)a^{-\#}(x,\xi,\lambda)dz \textnormal{  mod  } {S}^{-\infty,\mathcal{L}}((G\times \widehat{G})\otimes \textnormal{End}(E_0)),
\end{equation*}where $a^{-\#}(x,\xi,\lambda)$ is the symbol of the parametrix to $A-\lambda I,$   in Corollary \ref{parameterparametrix}.
\end{theorem}
\begin{proof}
    First, we need to prove that the condition $|F(\lambda)|\leqslant C|\lambda|^s$ uniformly in $\lambda,$ for some $s<0,$ is enough in order to guarantee that \begin{equation*}
    \sigma_{{B}_{0}}(x,\xi):=-\frac{1}{2\pi i}\oint\limits_{\partial \Lambda_\varepsilon}F(z)(a(x,\xi)-z)^{-1}dz,
\end{equation*} is a well defined matrix-symbol.
From Theorem \ref{lambdalambdita} we deduce that $(a(x,\xi)-z)^{-1}$ satisfies the estimate
\begin{equation*}
   \Vert (|z|^{\frac{1}{m}}+\widehat{\mathcal{M}}(\xi))^{m(k+1)}\widehat{\mathcal{M}}(\xi)^{\rho|\alpha|-\delta|\beta|} \otimes \partial_{z}^k\partial_{X}^{(\beta)}\Delta_{\xi}^{\alpha}(a(x,\xi)-z)^{-1}\Vert_{\textnormal{op}}<\infty.
\end{equation*}
Observe that 
\begin{align*}
    &\Vert(a(x,\xi)-z)^{-1}\Vert_{\textnormal{op}}\\
    & =\Vert (|z|^{\frac{1}{m}}+\widehat{\mathcal{M}}(\xi))^{-m}(|z|^{\frac{1}{m}}+\widehat{\mathcal{M}}(\xi))^{m} \otimes (a(x,\xi)-z)^{-1}\Vert_{\textnormal{op}} \\
    &\lesssim  \sup_{1\leqslant j\leqslant d_\xi}(|z|^{\frac{1}{m}}+(1+\nu_{jj}(\xi)^2)^{\frac{1}{2}})^{-m}\\&\leqslant |z|^{-1},
\end{align*} and the condition $s<0$ implies
\begin{align*}
    \left|\frac{1}{2\pi i}\oint\limits_{\partial \Lambda_\varepsilon}F(z)(a(x,\xi)-z)^{-1}dz\right|\lesssim \oint\limits_{\partial \Lambda_\varepsilon}|z|^{-1+s}|dz|<\infty,
\end{align*}uniformly in $(x,[\xi])\in G\times \widehat{G}.$ In order to check that $\sigma_{B_0}\in {S}^{ms,\mathcal{L}}_{\rho,\delta}((G\times \widehat{G})\otimes \textnormal{End}(E_0))$ let us analyse the cases $-1<s<0$ and $s\leqslant -1$ separately. So, let us analyse first the situation of $-1<s<0.$ We observe that
\begin{align*}
   &\Vert \widehat{\mathcal{M}}(\xi)^{-ms+\rho|\alpha|-\delta|\beta|} \otimes \partial_{X}^{(\beta)}\Delta_{\xi}^{\alpha}\sigma_{B_0}(x,\xi)\Vert_{\textnormal{op}}\\
   &\leqslant \frac{C}{2\pi }\oint\limits_{\partial \Lambda_\varepsilon} |z|^{s}\Vert      \widehat{\mathcal{M}}(\xi)^{-ms+\rho|\alpha|-\delta|\beta|} \otimes \partial_{X}^{(\beta)}\Delta_{\xi}^{\alpha}(a(x,\xi)-z)^{-1}\Vert_{\textnormal{op}} |dz|.
\end{align*}Now, we will estimate the operator norm inside of the integral. Indeed, the identity
\begin{align*}
    &\Vert      \widehat{\mathcal{M}}(\xi)^{-ms+\rho|\alpha|-\delta|\beta|} \otimes \partial_{X}^{(\beta)}\Delta_{\xi}^{\alpha}(a(x,\xi)-z)^{-1}\Vert_{\textnormal{op}}=\\
    &\Vert (|z|^{\frac{1}{m}}+\widehat{\mathcal{M}}(\xi))^{-m}(|z|^{\frac{1}{m}}+\widehat{\mathcal{M}}(\xi))^{m}\widehat{\mathcal{M}}(\xi)^{-ms+\rho|\alpha|-\delta|\beta|} \otimes \partial_{X}^{(\beta)}\Delta_{\xi}^{\alpha}(a(x,\xi)-z)^{-1}\Vert_{\textnormal{op}}
\end{align*}implies that
\begin{align*}
    &\Vert      \widehat{\mathcal{M}}(\xi)^{-ms+\rho|\alpha|-\delta|\beta|}\otimes \partial_{X}^{(\beta)}\Delta_{\xi}^{\alpha}(a(x,\xi)-z)^{-1}\Vert_{\textnormal{op}} \lesssim  \Vert (|z|^{\frac{1}{m}}+\widehat{\mathcal{M}}(\xi))^{-m}\widehat{\mathcal{M}}(\xi)^{-ms}\Vert_{\textnormal{op}}
\end{align*}where we have used that
\begin{align*}
\sup_{z\in \partial \Lambda_\varepsilon}  \sup_{(x,\xi)} \Vert (|z|^{\frac{1}{m}}+\widehat{\mathcal{M}}(\xi))^{m}\widehat{\mathcal{M}}(\xi)^{\rho|\alpha|-\delta|\beta|} \otimes \partial_{X}^{(\beta)}\Delta_{\xi}^{\alpha}(a(x,\xi)-z)^{-1}\Vert_{\textnormal{op}} <\infty.
\end{align*}
Consequently, by using that  $s<0,$ we deduce
\begin{align*}
   & \frac{C}{2\pi }\oint\limits_{\partial \Lambda_\varepsilon} |z|^{s}\Vert      \widehat{\mathcal{M}}(\xi)^{ms+\rho|\alpha|-\delta|\beta|}\otimes \partial_{X}^{(\beta)}\Delta_{\xi}^{\alpha}(a(x,\xi)-z)^{-1}\Vert_{\textnormal{op}} |dz|\\
    &\lesssim \frac{C}{2\pi }\oint\limits_{\partial \Lambda_\varepsilon} |z|^{s}\Vert (|z|^{\frac{1}{m}}+\widehat{\mathcal{M}}(\xi))^{-m} \widehat{\mathcal{M}}(\xi)^{-ms}\Vert_{\textnormal{op}} |dz|\\
    &= \frac{C}{2\pi }\oint\limits_{\partial \Lambda_\varepsilon} |z|^{s}\sup_{1\leqslant j\leqslant d_\xi} (|z|^{\frac{1}{m}}+(1+\nu_{jj}(\xi)^{2})^{\frac{1}{2}})^{-m}(1+\nu_{jj}(\xi)^{2})^{-\frac{ms}{2}}|dz|.
\end{align*}To study the convergence of the last contour integral we only need to check the convergence of $\int_{1}^{\infty}r^s(r^{\frac{1}{m}}+\varkappa)^{-m}\varkappa^{-ms}dr,$ where $\varkappa>1$ in a parameter. The change of variable $r=\varkappa^{m}t$ implies that
\begin{align*}
   \int\limits_{1}^{\infty}r^s(r^{\frac{1}{m}}+\varkappa)^{-m}\varkappa^{-ms}dr&=\int\limits_{\varkappa^{-m}}^{\infty}\varkappa^{ms}t^s(\varkappa t^{\frac{1}{m}}+\varkappa)^{-m}\varkappa^{-ms}\varkappa^mdt=\int\limits_{\varkappa^{-m}}^{\infty}t^s(t^{\frac{1}{m}}+1)^{-m}dt\\
   &\lesssim \int\limits_{\varkappa^{-m}}^{1}t^sdt+\int\limits_{1}^{\infty}t^{-1+s}<\infty.
\end{align*}Indeed, for $t\rightarrow\infty,$ $t^s(t^{\frac{1}{m}}+1)^{-m}\lesssim t^{-1+s}$ and we conclude the estimate because $\int\limits_{1}^{\infty} t^{-1+s'}dt<\infty,$ for all $s'<0.$ On the other hand, the condition $-1<s<0$ implies that
\begin{align*}
 \int\limits_{\varkappa^{-m}}^{1}t^sdt=\frac{1}{1+s}-\frac{\varkappa^{-m(1+s)}}{1+s}=   O(1).
\end{align*} In the case where $s\leqslant -1,$ we can find an analytic function $J(z)$ such that it is a holomorphic function in $\mathbb{C}\setminus \Lambda_{\varepsilon},$ and continuous on its closure and additionally satisfying that $F(\lambda)=J(\lambda)^{1+[-s]}.$\footnote{ $[-s]$ denotes  the integer part of $-s.$} In this case,  $J(A)$ defined by the complex functional calculus 
\begin{equation}\label{G(A)}
    J(A)=-\frac{1}{2\pi i}\oint\limits_{\partial \Lambda_\varepsilon}J(z)(A-zI)^{-1}dz,
\end{equation}
has symbol belonging to ${S}^{\frac{sm}{1+[-s]},\mathcal{L}}_{\rho,\delta}((G\times \widehat{G})\otimes \textnormal{End}(E_0))$ because $J$ satisfies the estimate $|J(\lambda)|\leqslant C|\lambda|^{\frac{s}{1+[-s]}},$ with $-1<\frac{s}{1+[-s]}<0.$ 
By observing that
\begin{align*}
    \sigma_{F(A)}(x,\xi)&=-\frac{1}{2\pi i}\oint\limits_{\partial \Lambda_\varepsilon}F(z)\widehat{\mathcal{R}}_z(x,\xi)dz=-\frac{1}{2\pi i}\oint\limits_{\partial \Lambda_\varepsilon}J(z)^{1+[-s]}\widehat{\mathcal{R}}_z(x,\xi)dz\\
    &=\sigma_{J(A)^{1+[-s]}}(x,\xi),
\end{align*}and computing the symbol $\sigma_{J(A)^{1+[-s]}}(x,\xi)$ by iterating $1+[-s]$-times  the asymptotic formula for the composition in the vector-valued subelliptic calculus (see Theorem \ref{VectorcompositionC}), we can see that the term with higher order in such expansion is $\sigma_{J(A)}(x,\xi)^{1+[-s]}\in {S}^{ms,\mathcal{L}}_{\rho,\delta}((G\times \widehat{G})\otimes \textnormal{End}(E_0)).$ Consequently we have proved that $\sigma_{F(A)}(x,\xi)\in {S}^{ms,\mathcal{L}}_{\rho,\delta}((G\times \widehat{G})\otimes \textnormal{End}(E_0)).$
This completes the proof for the first part of the theorem.
For the second part of the proof, let us denote by $a^{-\#}(x,\xi,\lambda)$  the symbol of the parametrix to $A-\lambda I,$   in Corollary \ref{parameterparametrix}. Let $P_{\lambda}=\textnormal{Op}(a^{-\#}(\cdot,\cdot,\lambda)).$ Because $\lambda\in \textnormal{Resolv}(A)$ for $\lambda\in \partial \Lambda_\varepsilon,$ $(A-\lambda)^{-1}-P_{\lambda}$ is a smoothing operator. Consequently, from Lemma \ref{LemmaFC} we deduce that
\begin{align*}
   & \sigma_{F(A)}(x,\xi)=-\frac{1}{2\pi i}\oint\limits_{\partial \Lambda_\varepsilon}F(z)\widehat{\mathcal{R}}_z(x,\xi)dz\\
    &=-\frac{1}{2\pi i}\oint\limits_{\partial \Lambda_\varepsilon}F(z)a^{-\#}(x,\xi,z)dz-\frac{1}{2\pi i}\oint\limits_{\partial \Lambda_\varepsilon}F(z)(\widehat{\mathcal{R}}_z(x,\xi)-a^{-\#}(x,\xi,z))dz\\
    &\equiv -\frac{1}{2\pi i}\oint\limits_{\partial \Lambda_\varepsilon}F(z)a^{-\#}(x,\xi,z)dz  \textnormal{  mod  } {S}^{-\infty,\mathcal{L}}((G\times \widehat{G})\otimes \textnormal{End}(E_0)).
\end{align*}The asymptotic expansion \eqref{asymcomplex} comes from the construction of the parametrix in the vector-valued subelliptic calculus (see Proposition \ref{IesTParametrix}).
\end{proof}

\subsection{Global functional calculus  on homogeneous vector bundles}
Let $a\in S^{m,\mathcal{L}}_{\rho,\delta}((G\times \widehat{G}) \otimes \textnormal{End}(E_0))$ be a parameter $\mathcal{L}$-elliptic symbol  of order $m>0$ with respect to the sector $\Lambda\subset\mathbb{C}.$ For a pseudo-differential operator $\tilde{A}$ with symbol $a,$ defined via \eqref{quantizationonhomogeneous2}, let us define, again, as in the vector-valued case, the operator $F(\tilde{A})$  by using the (Dunford-Riesz) complex functional calculus:
\begin{equation}\label{F(A)'}
    F(\tilde{A}):=-\frac{1}{2\pi i}\oint\limits_{\partial \Lambda_\varepsilon}F(z)(\tilde{A}-zI)^{-1}dz,
\end{equation}where
\begin{itemize}
    \item[(CI)':] $\Lambda_{\varepsilon}:=\Lambda\cup \{z:|z|\leqslant \varepsilon\},$ $\varepsilon>0,$ and $\Gamma=\partial \Lambda_\varepsilon\subset\textnormal{Resolv}(\tilde{A})$ is a positively oriented curve in the complex plane $\mathbb{C}$.
    \item[(CII)':] $F$ is a holomorphic function in $\mathbb{C}\setminus \Lambda_{\varepsilon},$ and continuous on its closure. 
    \item[(CIII)':] We will assume  decay of $F$ along $\partial \Lambda_\varepsilon$ in order that the operator ${F(\tilde{A})}$ will be densely defined on $\Gamma^\infty(E)$ in the strong sense of the topology on $L^2(E).$  For instance, we require  $|F(\lambda)|\leqslant C|\lambda|^s$ uniformly in $\lambda,$ for some $s<0.$ 
\end{itemize}
\begin{remark}\label{equivalenceofresolv}
Let $\tilde{A}\in \Psi^{m,\mathcal{L}}_{\rho,\delta}(E),$ and let $A$ be its associated vector-valued operator. Because $A$ and $\tilde{A}$ are unitarily equivalent, (because of  \eqref{unitarilyequivaklent} for $\tau=\omega$), and using the fact that the spectrum of densely defined linear operators is invariant under unitary transformations, we have that
\begin{equation*}
    \textnormal{Resolv}(A)=\textnormal{Resolv}(\tilde{A}),\quad \textnormal{Spectrum}(A)= \textnormal{Spectrum}(\tilde{A}).
\end{equation*}These equalities of real/complex sets imply that
\begin{equation*}
   \widetilde{ [(A-\lambda I)^{-1}]}=(\tilde{A}-\lambda I)^{-1},
\end{equation*}for any $\lambda \in  \textnormal{Resolv}(A).$ For the proof of the last statement observe that the identity
\begin{align}
    \widetilde{ [(A-\lambda I)^{-1}]}=\varkappa_\tau^{-1} [(A-\lambda I)^{-1}]\varkappa_\tau.
\end{align}implies that
\begin{align*}
  \varkappa_\tau^{-1} [(A-\lambda I)^{-1}]\varkappa_\tau \circ  [(\tilde{  A} -\lambda I)]&= \varkappa_\tau^{-1} [(A-\lambda I)^{-1}] \varkappa_\tau \circ[(\varkappa_\tau^{-1}  A\varkappa_\tau -\lambda I)] \\
  &=\varkappa_\tau^{-1} [(A-\lambda I)^{-1}]\varkappa_\tau \circ \varkappa_\tau^{-1} [(A-\lambda I)]\varkappa_\tau\\
  &=\varkappa_\tau^{-1} [(A-\lambda I)^{-1}] \circ (A-\lambda I)\varkappa_\tau \\
  &=I. 
\end{align*}
\end{remark}
So, in view of Theorem \ref{DunforRiesz} and Remark \ref{equivalenceofresolv} we have:
\begin{theorem}\label{DunforRiesz222}
Let $m>0,$ and let $0\leqslant \delta<\rho\leqslant 1.$ Let  $a\in S^{m,\mathcal{L}}_{\rho,\delta}((G\times \widehat{G})\otimes \textnormal{End}(E_0))$ be a parameter $\mathcal{L}$-elliptic symbol with respect to $\Lambda.$ Let us assume that $F$ satisfies the  estimate $|F(\lambda)|\leqslant C|\lambda|^s$ uniformly in $\lambda,$ for some $s<0.$  Then  the symbol of $F(\tilde{A}),$  $\sigma_{F(\tilde{A})}\in S^{ms,\mathcal{L}}_{\rho,\delta}((G\times \widehat{G})\otimes \textnormal{End}(E_0)) $ admits an asymptotic expansion of the form
\begin{equation}\label{asymcomplex2}
     \sigma_{F(\tilde{A})}(x,\xi)=\sigma_{\widetilde{F(A)}}(x,\xi)\sim 
     \sum_{N=0}^\infty\sigma_{{B}_{N}}(x,\xi),\,\,\,(x,[\xi])\in G\times \widehat{G},
\end{equation}where $\sigma_{{B}_{N}}\in {S}^{ms-(\rho-\delta)N,\mathcal{L}}_{\rho,\delta}((G\times \widehat{G})\otimes \textnormal{End}(E_0))$ and 
\begin{equation*}
    \sigma_{{B}_{0}}(x,\xi)=-\frac{1}{2\pi i}\oint\limits_{\partial \Lambda_\varepsilon}F(z)(a(x,\xi)-z)^{-1}dz\in {S}^{ms,\mathcal{L}}_{\rho,\delta}((G\times \widehat{G})\otimes \textnormal{End}(E_0)),
\end{equation*}where $a$ is the $\textnormal{End}(E_0)$-valued symbol of the vector-valued operator $A$ associated with $\tilde{A}.$
\end{theorem}
\begin{proof}
Let us prove that $\widetilde{F(A)}=F(\tilde{A}).$ Indeed, Remark \ref{equivalenceofresolv} implies that
\begin{align*}
    F(\tilde{A}):&=-\frac{1}{2\pi i}\oint\limits_{\partial \Lambda_\varepsilon}F(z)(\tilde{A}-zI)^{-1}dz=-\frac{1}{2\pi i}\oint\limits_{\partial \Lambda_\varepsilon}F(z)\varkappa_\tau^{-1} [(A-z I)^{-1}]\varkappa_\tau dz\\
    &=\varkappa_\tau^{-1}\left(-\frac{1}{2\pi i}\oint\limits_{\partial \Lambda_\varepsilon}F(z)(A-z I)^{-1} dz\right)\varkappa_\tau=\varkappa_\tau^{-1}F(A)\varkappa_\tau=\widetilde{F(A)}.
\end{align*}So, we can conclude that   $\sigma_{F(\tilde{A})}=\sigma_{\widetilde{F(A)}}:=\sigma_{{F(A)}}.$ Observe that the asymptotic expansion \eqref{asymcomplex2} for the symbol $\sigma_{{F(A)}}$ was proved in Theorem \ref{DunforRiesz}. So, we end the proof.
\end{proof}

\subsection{Global functional calculus of the Laplacian} Let us consider the Laplacian $\mathcal{L}_G$ on $G.$  If we choose   $\mathbb{X}=\{X_1,\cdots X_n\}$ to be  an orthonormal basis of the Lie algebra $\mathfrak{g}$ of $G,$ then $$\mathcal{L}_G:=-\sum_{j=1}^{n}X_j^2.$$
The Laplacian $\mathcal{L}_G$ induces a vector-valued Laplacian $\mathcal{L}_{G,E_0}$ (see \eqref{InI:L:L}) that in principle maps $C^{\infty}(G,E_0)$ into $C^{\infty}(G,E_0).$ However, it was proved in \cite[Page 124]{Wallach1973}   that $\mathcal{L}_{G,E_0}$ maps  $C^{\infty}(G,E_0)^{\tau}$ into $C^{\infty}(G,E_0)^{\tau}.$ So, in view of the diagram
\begin{center}
    \begin{tikzpicture}[every node/.style={midway}]
  \matrix[column sep={10em,between origins}, row sep={4em}] at (0,0) {
    \node(R) {$\Gamma^\infty(E)$}  ; & \node(S) {$\Gamma^\infty(E)$}; \\
    \node(R/I) {$C^\infty(G,E_0)^\tau$}; & \node (T) {$C^\infty(G,E_0)^\tau$};\\
  };
  \draw[<-] (R/I) -- (R) node[anchor=east]  {$\varkappa_{\tau}$};
  \draw[->] (R) -- (S) node[anchor=south] {$\tilde{\mathcal{L}}_{G,E}$};
  \draw[->] (S) -- (T) node[anchor=west] {$\varkappa_\tau$};
  \draw[->] (R/I) -- (T) node[anchor=north] {$\mathcal{L}_{G,E_0}$};
\end{tikzpicture} 
\end{center}
$\mathcal{L}_{G,E_0}$ induces  a continuous linear operator $\tilde{\mathcal{L}}_{G,E}$ defined on the sections of the homogeneous vector-bundle $E.$ The operator $\tilde{\mathcal{L}}_{G,E}$  is $G$-invariant and it agrees with the Laplacian on the vector-bundle $E\rightarrow M$ (see \cite[Page 124]{Wallach1973} for details).

In this subsection we discuss the membership of the  operators defined by the functional calculus of ${\mathcal{L}}_{G,E}$  and $\tilde{\mathcal{L}}_{G,E}$ to the global (elliptic) Hörmander classes (see Definition \ref{symbolclass:L} and Remark \ref{Remark:L:G}).

Recall that  the {\it vector-valued Laplacian} is the operator defined via
 \begin{equation}\label{Vec:sub:2}
     \mathcal{L}_{G,E_0}:=\textnormal{Op}\left(\widehat{\mathcal{L}_G}(\xi) \otimes I_{ \textnormal{End}(E_0)}\right):C^{\infty}(G,E_0)\rightarrow C^{\infty}(G,E_0),
 \end{equation} In \eqref{Vec:sub}, $\widehat{\mathcal{L}}_G$ denotes the matrix-valued symbol of the  Laplacian $\mathcal{L}_G.$
 For $m\in \mathbb{R},$ let us consider the Euclidean class of symbols $S^{m}(\mathbb{R}^{+}_0),$ endowed with the standard countable family of seminorms 
\begin{equation}
    \Vert f \Vert_{d,\,S^{m}(\mathbb{R}^{+}_0)}:=\sup_{\ell\leq d}\sup_{\lambda\geq 0}(1+\lambda)^{-m+\ell}|\partial^\ell_{\lambda}f(\lambda)|,\,d\in \mathbb{N}_0.
\end{equation}
For any $f\in S^{m}(\mathbb{R}^{+}_0), $ we define
 \begin{equation}\label{Vec:sub:f:2}
    f( \mathcal{L}_{G, E_0}):=\textnormal{Op}\left(f(\widehat{\mathcal{L}}_G(\xi)) \otimes I_{ \textnormal{End}(E_0)}\right):C^{\infty}(G,E_0)\rightarrow C^{\infty}(G,E_0),
 \end{equation} with $f(\widehat{\mathcal{L}}_G(\xi)),$ $[\xi]\in \widehat{G},$ defined by the spectral calculus of matrices.
As a consequence of Theorem \ref{orders:theorems:VB} we have the following result.
\begin{corollary} For $f\in S^{\frac{m}{2}}(\mathbb{R}^{+}_0), $ $m\in \mathbb{R},$ we have  
\begin{equation}
   f(\mathcal{L}_{G,E_0})\in \Psi^{m}_{1,0}((G\times \widehat{G})\otimes \textnormal{End}(E_0)). 
\end{equation}
 In particular for $f_s(\lambda):=(1+\lambda)^{\frac{s}{2}},$ the global symbol of ${M}_{s,E_0}=f_s(\mathcal{L}_{G,E_0})$ belongs to $ S^{s}_{1,0}( (G\times \widehat{G})\otimes \textnormal{End}(E_0)) $ for all $s\in \mathbb{R}.$ In other words, we have 
\begin{equation}\label{thepowersof:LE0:2}
    {M}_{s,E_0}:=(1+\mathcal{L}_{G,E_0})^{\frac{s}{2}}\in \Psi^{s}_{1,0}((G\times \widehat{G})\otimes\textnormal{End}(E_0))
\end{equation}
for all $s\in \mathbb{R}.$
\end{corollary} As a consequence of Theorem \ref{DunforRiesz222} and the previous corollary, we get the following result that guarantees the existence of matrix-valued symbols for the operators defined by the real functional calculus of the Laplacian on $E$.
\begin{corollary} For $f\in S^{\frac{m}{2}}(\mathbb{R}^{+}_0), $ $m\in \mathbb{R},$ we have  
\begin{equation}
   f(\tilde{\mathcal{L}}_{G,E})\in \Psi^{m}_{1,0}(E). 
\end{equation}
 In particular for $f_s(\lambda):=(1+\lambda)^{\frac{s}{2}},$ the global symbol of $\tilde{{M}}_{s,E}=f_s(\tilde{\mathcal{L}}_{G,E})$ belongs to $ S^{s}_{1,0}( E) $ for all $s\in \mathbb{R}.$ In other words, we have 
\begin{equation}\label{thepowersof:LE0:3}
    \tilde{M}_{s,E}:=(1+\tilde{\mathcal{L}}_{G,E})^{\frac{s}{2}}\in \Psi^{s}_{1,0}(E)
\end{equation}
for all $s\in \mathbb{R}.$
\end{corollary}

\subsection{Vector-valued G\r{a}rding inequality} \label{Gardingsection'}
In this section we will prove the vector-valued G\r{a}rding inequality for the vector-valued subelliptic calculus. 
We start our analysis with the following proposition.
\begin{proposition}
Let $0\leqslant \delta<\rho\leqslant 1.$  Let  $a\in S^{m,\mathcal{L}}_{\rho,\delta}((G\times \widehat{G})\otimes \textnormal{End}(E_0))$ be an $\mathcal{L}$-elliptic symbol where $m\geqslant 0$ and let us assume that $a$ is positive definite,  namely, that the symbol
$$ a(x,\xi):=(\sigma_A(x,\xi)_{u,v})_{1\leq v,u\leq d_\xi}\in \mathbb{C}^{d_\xi\times d_\xi}(\textnormal{End}(E_0)),$$ is self-adjoint and its spectrum is contained in $\mathbb{R}^{+}_0:=[0,\infty).$  Then $a$ is parameter-elliptic with respect to $\mathbb{R}_{-}:=\{z=x+i0:x<0\}\subset\mathbb{C}.$ Furthermore, for any number $s\in \mathbb{C},$ 
\begin{equation*}
    \widehat{B}(x,\xi)\equiv a(x,\xi)^s:=\exp(s\log(a(x,\xi))),\,\,(x,[\xi])\in G\times \widehat{G},
\end{equation*}defines a symbol $\widehat{B}\in S^{m\times\textnormal{Re}(s),\mathcal{L}}_{\rho,\delta}((G\times \widehat{G})\otimes \textnormal{End}(E_0)).$
\end{proposition}
\begin{proof}Note that for any $(x,\xi)\in G\times \widehat{G},$ $a(x,\xi)$ is invertible. Then its spectrum is strictly contained in $\mathbb{R}^{+}:=(0,\infty).$ From the estimates 
\begin{equation*}
     \sup_{(x,[\xi])}\Vert\widehat{\mathcal{M}}(\xi)^{-m} \otimes a(x,\xi) \Vert_{\textnormal{op}}<\infty,\quad \sup_{(x,[\xi])}\Vert \widehat{\mathcal{M}}^m(\xi)\otimes a(x,\xi)^{-1} \Vert_{\textnormal{op}}<\infty,
\end{equation*}
   we deduce that
    \begin{equation*}
       \sup_{x\in G}\Vert a(x,\xi) \Vert_{\textnormal{op}}\asymp\Vert  \widehat{\mathcal{M}}(\xi)^{m}\Vert_{\textnormal{op}}, \quad \sup_{x\in G}\Vert a(x,\xi)^{-1} \Vert_{\textnormal{op}}\asymp\Vert  \widehat{\mathcal{M}}(\xi)^{-m}\Vert_{\textnormal{op}}. 
    \end{equation*} Consequently we have
    \begin{equation*}
      \Vert  \widehat{\mathcal{M}}(\xi)^{m}\Vert_{\textnormal{op}}^{-1} \sup_{x\in G}\Vert a(x,\xi) \Vert_{\textnormal{op}}\asymp 1, \quad \Vert  \widehat{\mathcal{M}}(\xi)^{-m}\Vert_{\textnormal{op}}^{-1} \sup_{x\in G}\Vert a(x,\xi)^{-1} \Vert_{\textnormal{op}}\asymp 1, 
    \end{equation*} from which we deduce that
    \begin{equation*}
        \Vert  \widehat{\mathcal{M}}(\xi)^{-m}\Vert_{\textnormal{op}}^{-1}\textnormal{Spectrum}(a(x,\xi))\subset [c,C],
    \end{equation*}where $c,C>0$ are positive real numbers. For every $\lambda\in \mathbb{R}_{-}$ we have
    \begin{align*}
      &   \Vert (|\lambda|^{\frac{1}{m}}+\widehat{\mathcal{M}}(\xi))^m \otimes (a(x,\xi)-\lambda)^{-1}\Vert_{\textnormal{op}}\\
         &\asymp \Vert (|\lambda|^{\frac{1}{m}}+\widehat{\mathcal{M}}(\xi))^m (\widehat{\mathcal{M}}(\xi)^m-\lambda)^{-1}(\widehat{\mathcal{M}}(\xi)^m-\lambda) \otimes (a(x,\xi)-\lambda)^{-1}\Vert_{\textnormal{op}} \\
         &\lesssim \Vert (|\lambda|^{\frac{1}{m}}+\widehat{\mathcal{M}}(\xi))^m (\widehat{\mathcal{M}}(\xi)^m-\lambda)^{-1}\Vert_{\textnormal{op}} \Vert (\widehat{\mathcal{M}}(\xi)^m-\lambda)\otimes(a(x,\xi)-\lambda)^{-1}\Vert_{\textnormal{op}}\\
         &\lesssim \Vert (|\lambda|^{\frac{1}{m}}+\widehat{\mathcal{M}}(\xi))^m (\widehat{\mathcal{M}}(\xi)^m-\lambda)^{-1}\Vert_{\textnormal{op}},
    \end{align*}
where we have used that $\Vert (\widehat{\mathcal{M}}(\xi)^m-\lambda)\otimes(a(x,\xi)-\lambda)^{-1}\Vert_{\textnormal{op}}$ is uniformly bounded in $\lambda.$ Indeed, with $\theta=1/|\lambda|=-1/\lambda,$ and $\omega=-\lambda/|\lambda|=1,$ the previous operator norm can be written as 
$$\Vert (\theta\widehat{\mathcal{M}}(\xi)^m+1)\otimes(\theta a(x,\xi)+1)^{-1}\Vert_{\textnormal{op}},\,\theta\in [0,1],$$
and the compactness of the unit interval allows us to estimate this operator norm uniformly in $[0,1].$ Indeed, note that the functional calculus, the positivity of the symbol $a(x,\xi),$  and its $\mathcal{L}$-ellipticity imply that for $\theta\in [0,1],$   $(\theta a(x,\xi)+1)^{-1}\in S^{-m,\mathcal{L}}_{\rho,\delta}((G\times \widehat{G})\otimes \textnormal{End}(E_0)).$ Because $(\theta \widehat{\mathcal{M}}(\xi)^m+1)\in S^{m,\mathcal{L}}_{1,0}(G\times \widehat{G})), $ then $$  (\theta\widehat{\mathcal{M}}(\xi)^m+1)\otimes(\theta a(x,\xi)+1)^{-1}\in S^{0,\mathcal{L}}_{\rho,\delta}((G\times \widehat{G})\otimes \textnormal{End}(E_0),  $$
    and then $\sup_{\theta\in [0,1]}\Vert (\theta\widehat{\mathcal{M}}(\xi)^m+1)\otimes(\theta a(x,\xi)+1)^{-1}\Vert_{\textnormal{op}}=O(1).$
    
    Now, fixing again $\lambda\in \mathbb{R}_{-}$ observe that from the compactness of $[0,1/2]$ we deduce that
    \begin{align*}
        \sup_{0\leqslant |\lambda| \leqslant 1/2}\Vert (|\lambda|^{\frac{1}{m}}+\widehat{\mathcal{M}}(\xi))^m (\widehat{\mathcal{M}}(\xi)^m-\lambda)^{-1}\Vert_{\textnormal{op}}&\asymp \sup_{0\leqslant |\lambda| \leqslant 1/2}\Vert \widehat{\mathcal{M}}(\xi)^m (\widehat{\mathcal{M}}(\xi)^m-\lambda)^{-1}\Vert_{\textnormal{op}} \\
         &\asymp \sup_{0\leqslant |\lambda| \leqslant 1/2}\Vert  (I_{d_\xi}-\lambda\widehat{\mathcal{M}}(\xi)^{-m})^{-1}\Vert_{\textnormal{op}} \lesssim 1,
    \end{align*} 
    where (using the notation in \eqref{eigenvalues:hatM}) we have used the estimate 
    $$  0\leq |\lambda|(1+\nu_{ii}(\xi)^2)^{-\frac{m}{2}}\leq \frac{1}{2},  $$ for $\lambda<0$ with
    $0\leq |\lambda|\leq 1/2,$ which implies
    \begin{align*}
       \Vert  (I_{d_\xi}-\lambda\widehat{\mathcal{M}}(\xi)^{-m})^{-1}\Vert_{\textnormal{op}}=\sup_{1\leq i\leq d_\xi}\left( 1+ |\lambda|(1+\nu_{ii}(\xi)^2)^{-\frac{m}{2}} \right)^{-1}\in [2/3,1]. 
    \end{align*}
On the other hand, by writing $\theta=1/|\lambda|$ and $\omega=\lambda/|\lambda|,$ we have that 
    \begin{align*}
    &  \sup_{ |\lambda|\geqslant 1/2 }\Vert (|\lambda|^{\frac{1}{m}}+\widehat{\mathcal{M}}(\xi))^m (\widehat{\mathcal{M}}(\xi)^m-\lambda)^{-1}\Vert_{\textnormal{op}}  \\
      &=\sup_{ |\lambda|\geqslant 1/2 }\Vert (|\lambda|^{\frac{1}{m}}\widehat{\mathcal{M}}(\xi)^{-1}+I_{d_\xi})^m (I_{d_\xi}-\widehat{\mathcal{M}}(\xi)^{-m}\lambda)^{-1}\Vert_{\textnormal{op}}\\
      &=\sup_{ |\lambda|\geqslant 1/2 }\Vert (\widehat{\mathcal{M}}(\xi)^{-1}+|\lambda|^{-\frac{1}{m}}I_{d_\xi})^m |\lambda|(I_{d_\xi}-\widehat{\mathcal{M}}(\xi)^{-m}\lambda)^{-1}\Vert_{\textnormal{op}}\\
      &\leq \sup_{ |\theta|\leqslant 2 }\Vert (\widehat{\mathcal{M}}(\xi)^{-1}+\theta^{\frac{1}{m}}I_{d_\xi})^m (\theta-\omega\widehat{\mathcal{M}}(\xi)^{-m})^{-1}\Vert_{\textnormal{op}}.
    \end{align*}The previous operator norm is $O(1),$ by considering the compatness of the interval $[0,1/2],$ and the fact that $(\widehat{\mathcal{M}}(\xi)^{-1}+\theta^{\frac{1}{m}}I_{d_\xi})^m (\theta-\omega\widehat{\mathcal{M}}(\xi)^{-m})^{-1}\in S^{0,\mathcal{L}}_{1,0}(G\times \widehat{G}).$

    So, we have proved that $a$ is parameter-elliptic with respect to $\mathbb{R}_{-}.$ To prove that $\widehat{B}=\widehat{B}(x,\xi)\in S^{m\times\textnormal{Re}(s),\mathcal{L}}_{\rho,\delta}((G\times \widehat{G})\otimes \textnormal{End}(E_0)),$ we can observe that for $\textnormal{Re}(s)<0,$ we can apply Theorem \ref{DunforRiesz}. If  $\textnormal{Re}(s)\geqslant 0,$ we can find $k\in \mathbb{N}$ such that $\textnormal{Re}(s)-k<0$ and consequently from the spectral calculus of matrices we deduce that $a(x,\xi)^{\textnormal{Re}(s)-k}\in S^{m\times(\textnormal{Re}(s)-k),\mathcal{L}}_{\rho,\delta}((G\times \widehat{G})\otimes \textnormal{End}(E_0)).$ So, from the calculus we conclude that   $$a(x,\xi)^{s}=a(x,\xi)^{s-k}\circ a(x,\xi)^{k}\in S^{m\times\textnormal{Re}(s),\mathcal{L}}_{\rho,\delta}((G\times \widehat{G})\otimes \textnormal{End}(E_0)).$$ Thus the proof is complete.
\end{proof}
\begin{corollary}\label{1/2}
Let $0\leqslant \delta<\rho\leqslant 1.$  Let  $a\in S^{m,\mathcal{L}}_{\rho,\delta}((G\times \widehat{G})\otimes \textnormal{End}(E_0)),$  be an $\mathcal{L}$-elliptic symbol  where $m\geqslant 0$ and let us assume that $a$ is positive definite. Then 
$\widehat{B}(x,\xi)\equiv a(x,\xi)^\frac{1}{2}:=\exp(\frac{1}{2}\log(a(x,\xi)))\in S^{\frac{m}{2},\mathcal{L}}_{\rho,\delta}((G\times \widehat{G})\otimes \textnormal{End}(E_0)).$
\end{corollary}
Now, let us assume that 
\begin{equation*}
    A(x,\xi):=\frac{1}{2}(a(x,\xi)+a(x,\xi)^{*}),\,(x,[\xi])\in G\times \widehat{G},\,\,a\in S^{m,\mathcal{L}}_{\rho,\delta}((G\times \widehat{G})\otimes \textnormal{End}(E_0)), 
\end{equation*}satisfies
\begin{align}\label{eqi}
    \Vert\widehat{\mathcal{M}}(\xi)^{m} \otimes A(x,\xi)^{-1} \Vert_{\textnormal{op}}\leqslant C_{0}.
\end{align} Observe that \eqref{eqi} implies that
\begin{equation*}
    \lambda(x,\xi):=\inf\{\tilde{\lambda}(x, \xi)^{-1}\}\leqslant C_{0},\,\,\,(x,[\xi])\in G\times \widehat{G},
\end{equation*}
where $\tilde{\lambda}(x,\xi)$ runs over the eigenvalues of the operator $\widehat{\mathcal{M}}(\xi)^{-m}\otimes A(x,\xi).$ So, $\lambda(x,\xi)^{-1}\geqslant \frac{1}{C_0}$ and consequently, in the sense of matrix of operators, we have
\begin{align*}
  \widehat{\mathcal{M}}(\xi)^{-m} \otimes A(x,\xi)\geqslant\frac{1}{C_0}I_{\mathbb{C}^{d_\xi\times d_\xi}( \textnormal{End}(E_0))}.
\end{align*}This implies that
\begin{align*}
  A(x,\xi)\geqslant\frac{1}{C_0}\widehat{\mathcal{M}}(\xi)^{m} \otimes I_{\mathbb{C}^{d_\xi\times d_\xi}( \textnormal{End}(E_0))},
\end{align*}and for $C_1\in(0, \frac{1}{C_0})$ we have that
\begin{align*}
 A(x,\xi)-C_{1}  \widehat{\mathcal{M}}(\xi)^{m} \otimes I_{\mathbb{C}^{d_\xi\times d_\xi}( \textnormal{End}(E_0))}\geqslant \left(\frac{1}{C_0}-C_1\right) \widehat{\mathcal{M}}(\xi)^{m} \otimes I_{\mathbb{C}^{d_\xi\times d_\xi}( \textnormal{End}(E_0))}>0.
\end{align*}If  $0\leqslant \delta<\rho\leqslant 1,$   from Corollary \ref{1/2}, we have that
\begin{align*}
    q(x,\xi):=(A(x,\xi)-C_{1}  \widehat{\mathcal{M}}(\xi)^{m} \otimes I_{\mathbb{C}^{d_\xi\times d_\xi}( \textnormal{End}(E_0))})^{\frac{1}{2}}\in  S^{\frac{m}{2},\mathcal{L}}_{\rho,\delta}((G\times \widehat{G})\otimes \textnormal{End}(E_0)).
\end{align*}From the symbolic calculus we obtain
\begin{align*}
  q(x,\xi)q(x,\xi)^*= A(x,\xi)-C_{1}  \widehat{\mathcal{M}}(\xi)^{m} \otimes I_{\mathbb{C}^{d_\xi\times d_\xi}( \textnormal{End}(E_0))}+r(x,\xi),\,\, 
\end{align*}
where $  r= r(x,\xi)\in  S^{m-(\rho-\delta),\mathcal{L}}_{\rho,\delta}((G\times \widehat{G})\otimes \textnormal{End}(E_0)).$
Now, let us assume that $u\in C^\infty(G, E_0).$ Then we have, by setting $$\mathcal{M}_{m,E_0}:=\textnormal{Op}(\widehat{\mathcal{M}}(\xi)^{m} \otimes I_{\mathbb{C}^{d_\xi\times d_\xi}( \textnormal{End}(E_0))}),$$ that 
\begin{align*}
    \textnormal{Re}(a(x,D)u,u)&=\frac{1}{2}((a(x,D)+a(x, \xi)^*))u,u)=(A(x,D)u,u)\\
    &=C_{1}(\mathcal{M}_{m,E_0} u,u)+(q(x,D)q(x,D)^{*}u,u)-(r(x,D)u,u)\\
    &=C_{1}(\mathcal{M}_{m,E_0}u,u)+(q(x,D)^{*}u,q(x,D)^*u)-(r(x,D)u,u)\\
    &\geqslant C_{1}\Vert u\Vert^2_{{L}^{2,\mathcal{L}}_{\frac{m}{2}}(G, E_0)}-(r(x,D)u,u)\\
     &= C_{1}\Vert u\Vert_{{L}^{2,\mathcal{L}}_{\frac{m}{2}}(G, E_0)}^2-(\mathcal{M}_{-\frac{m-(\rho-\delta)}{2},E_0}r(x,D)u, \mathcal{M}_{\frac{m-(\rho-\delta)}{2},E_0}u).
\end{align*}By using Remark \ref{characterisationSobolevspaces}, we have
\begin{align*}
    (\mathcal{M}_{-\frac{m-(\rho-\delta)}{2},E_0}r(x,D)u, \mathcal{M}_{\frac{m-(\rho-\delta)}{2},E_0}u)&\leqslant \Vert \mathcal{M}_{-\frac{m-(\rho-\delta)}{2},E_0}r(x,D)u \Vert_{L^2(G, E_0)}\Vert u\Vert_{L^{2,\mathcal{L}  }_{\frac{m-(\rho-\delta)}{2}}(G, E_0)}\\
    &= \Vert r(x,D)u \Vert_{L^{2,\mathcal{L}  }_{-\frac{m-(\rho-\delta)}{2}}(G, E_0)}\Vert u\Vert_{L^{2,\mathcal{L}  }_{\frac{m-(\rho-\delta)}{2}}(G, E_0)}\\
     &\leqslant C\Vert u \Vert_{L^{2,\mathcal{L}  }_{\frac{m-(\rho-\delta)}{2}}(G, E_0)}\Vert u\Vert_{L^{2,\mathcal{L}  }_{\frac{m-(\rho-\delta)}{2}}(G, E_0)},
\end{align*}where in the last line we have used the subelliptic Sobolev boundedness of $r(x,D)$ from $L^{2,\mathcal{L}  }_{\frac{m-(\rho-\delta)}{2}}(G, E_0)$ into $L^{2,\mathcal{L}  }_{-\frac{m-(\rho-\delta)}{2}}(G, E_0).$ Consequently, we deduce the lower bound
\begin{align*}
    \textnormal{Re}(a(x,D)u,u) \geqslant C_{1}\Vert u\Vert^2_{{L}^{2,\mathcal{L}}_{\frac{m}{2}}(G, E_0)}-C\Vert u\Vert_{L^{2,\mathcal{L}  }_{\frac{m-(\rho-\delta)}{2}}(G, E_0)}^2.
\end{align*} If we assume for a moment that for every $\varepsilon>0,$ there exists $C_{\varepsilon}>0,$ such that
\begin{equation}\label{lemararo}
    \Vert u\Vert_{{L}^{2,\mathcal{L}}_{\frac{m-(\rho-\delta)}{2}}(G, E_0)}^2\leqslant \varepsilon\Vert u\Vert_{{L}^{2,\mathcal{L}}_{\frac{m}{2}}(G, E_0)}^2+C_{\varepsilon}\Vert u \Vert_{L^2(G, E_0)}^2,
\end{equation} for $0<\varepsilon<C_{1}$ we have
\begin{align*}
    \textnormal{Re}(a(x,D)u,u) \geqslant (C_{1}-\varepsilon)\Vert u\Vert_{{L}^{2,\mathcal{L}}_{\frac{m}{2}}(G, E_0)}^2-C_{\varepsilon}\Vert u\Vert_{L^2(G, E_0)}^2.
\end{align*}So, with the exception of the proof of \eqref{lemararo} we have deduced the following estimate which is the main result of this subsection.

\begin{theorem}[Subelliptic vector-valued G\r{a}rding inequality]\label{GardinTheorem} Let $G$ be a compact Lie group. For $0\leqslant \delta<\rho\leqslant 1$, let $a(x,D):C^\infty(G, E_0)\rightarrow\mathscr{D}'(G, E_0)$ be an operator with symbol  $a\in {S}^{m,\mathcal{L}}_{\rho,\delta}(( G\times \widehat{G})\otimes \textnormal{End}(E_0))$, $m\in \mathbb{R}$. Let us assume that 
\begin{equation*}
    A(x,\xi):=\frac{1}{2}(a(x,\xi)+a(x,\xi)^{*}),\,(x,[\xi])\in G\times \widehat{G},\,\,a\in S^{m,\mathcal{L}}_{\rho,\delta}((G\times \widehat{G})\otimes \textnormal{End}(E_0)), 
\end{equation*}satisfies
\begin{align*}\label{garding}
    \Vert\widehat{\mathcal{M}}(\xi)^{m} \otimes A(x,\xi)^{-1} \Vert_{\textnormal{op}}\leqslant C_{0}.
\end{align*}Then, there exist $C_{1},C_{2}>0,$ such that the lower bound
\begin{align}
    \textnormal{Re}(a(x,D)u,u) \geqslant C_1\Vert u\Vert^2_{{L}^{2,\mathcal{L}}_{\frac{m}{2}}(G, E_0)}-C_2\Vert u\Vert_{L^2(G, E_0)}^2,
\end{align}holds true for every $u\in C^\infty(G, E_0).$
\end{theorem}
In view of the analysis above, for the proof of Theorem \ref{GardinTheorem} we only need to prove \eqref{lemararo}. However we will deduce it from the following more general lemma.
\begin{lemma}Let us assume that $s\geqslant t\geqslant 0$ or that $s,t<0.$ Then, for every $\varepsilon>0,$ there exists $C_\varepsilon>0$ such that 
\begin{equation}\label{lemararo2}
    \Vert u\Vert_{{L}^{2,\mathcal{L}}_{t}(G, E_0)}^2\leqslant \varepsilon\Vert u\Vert_{{L}^{2,\mathcal{L}}_{s}(G, E_0)}^2+C_{\varepsilon}\Vert u \Vert_{L^2(G, E_0)}^2,
\end{equation}holds true for every $u\in C^\infty(G, E_0).$ 
\end{lemma}
\begin{proof}
    Let $\varepsilon>0.$ Then, there exists $C_{\varepsilon}>0$ such that
    \begin{equation*}
        (1+\nu_{ii}(\xi)^2)^{t}-\varepsilon (1+\nu_{ii}(\xi)^2)^s\leqslant C_{\varepsilon},
    \end{equation*}uniformly in $[\xi]\in \widehat{G},$ where $\nu_{ii}(\xi)^2$ are eigenvalues of  the symbol of the sub-Laplacian $\mathcal{L}$ at $[\xi].$ Then \eqref{lemararo2}  would follow from the Plancherel theorem. Indeed,
    \begin{align*}
     \Vert u\Vert_{{L}^{2,\mathcal{L}}_{t}(G, E_0)}^2& =\sum_{[\xi]\in \widehat{G};1\leq i_0\leq d_\tau}d_\xi\sum_{i,j=1}^{d_\xi} (1+\nu_{ii}(\xi)^2)^{t}|\widehat{u}_{ij}(i_0,\xi)|^{2}  \\
     &\leqslant \sum_{[\xi]\in \widehat{G};1\leq i_0\leq d_\tau}d_\xi\sum_{i,j=1}^{d_\xi} (\varepsilon(1+\nu_{ii}(\xi)^2)^{s}+C_\varepsilon)|\widehat{u}_{ij}(i_0,\xi)|^{2}  \\
    &= \varepsilon\Vert u\Vert_{{L}^{2,\mathcal{L}}_{s}(G, E_0)}^2+C_{\varepsilon}\Vert u \Vert_{L^2(G, E_0)}^2,
    \end{align*}
    completing the proof.
\end{proof}
\begin{corollary}\label{GardinTheorem2} Let $G$ be a compact Lie group. Let $a(x,D):C^\infty(G, E_0)\rightarrow\mathscr{D}'(G, E_0)$ be an operator with symbol  $a\in {S}^{m,\mathcal{L}}_{\rho,\delta}(( G\times \widehat{G})\otimes \textnormal{End}(E_0))$, $m\in \mathbb{R}$ and let $0\leqslant \delta<\rho\leqslant 1$. Let us assume that 
\begin{equation*}
    a(x,\xi)\geqslant 0,\,(x,[\xi])\in G\times \widehat{G}, 
\end{equation*}satisfies
\begin{align*}\label{garding22}
    \Vert\widehat{\mathcal{M}}(\xi)^{m} \otimes a(x,\xi)^{-1} \Vert_{\textnormal{op}}\leqslant C_{0}.
\end{align*}Then, there exist $C_{1},C_{2}>0,$ such that the lower bound
\begin{align}
    \textnormal{Re}(a(x,D)u,u) \geqslant C_1\Vert u\Vert^2_{{L}^{2,\mathcal{L}}_{\frac{m}{2}}(G, E_0)}-C_2\Vert u\Vert_{L^2(G, E_0)}^2,
\end{align}holds true for every $u\in C^\infty(G, E_0).$
\end{corollary}

\subsection{G\r{a}rding inequality on homogeneous vector bundles}\label{Gardingsection} Now, the analysis above allows us to transfer the G\r{a}rding inequality from the vector-valued context to the setting of vector-bundles.

\begin{theorem}[Subelliptic  G\r{a}rding inequality on vector bundles]\label{GardinTheorem22} Let $E$ be a $G$-invariant vector bundle over $M=G/K,$ where $G$ is a compact Lie group. For $0\leqslant \delta<\rho\leqslant 1,$   let $\tilde{A}\in \Psi^{m,\mathcal{L}}_{\rho,\delta}(E)$ be an operator with symbol  $a\in {S}^{m,\mathcal{L}}_{\rho,\delta}(( G\times \widehat{G})\otimes \textnormal{End}(E_0))$, $m\in \mathbb{R}$. Let us assume that 
\begin{equation*}
    A(x,\xi):=\frac{1}{2}(a(x,\xi)+a(x,\xi)^{*}),\,(x,[\xi])\in G\times \widehat{G},\,\,a\in S^{m,\mathcal{L}}_{\rho,\delta}((G\times \widehat{G})\otimes \textnormal{End}(E_0)), 
\end{equation*}satisfies
\begin{align*}\label{garding2}
    \Vert\widehat{\mathcal{M}}(\xi)^{m} \otimes A(x,\xi)^{-1} \Vert_{\textnormal{op}}\leqslant C_{0}.
\end{align*}Then, there exist $C_{1},C_{2}>0,$ such that the lower bound
\begin{align}
    \textnormal{Re}(\tilde{A}s,s) \geqslant C_1\Vert s\Vert^2_{{L}^{2,\mathcal{L}}_{\frac{m}{2}}(E)}-C_2\Vert s\Vert_{L^2(E)}^2,
\end{align}holds true for every $s\in \Gamma^\infty(E).$
\end{theorem}
\begin{proof}
    Observe that, if $A:C^{\infty}(G,E_0)^{\tau}\rightarrow C^{\infty}(G,G_0)^{\omega} $ is the vector-valued operator associated to $\tilde{A},$ and $u=\varkappa_\tau s,$ $s\in \Gamma^\infty(E),$ then Theorem \ref{GardinTheorem} implies that
    \begin{align*}
    \textnormal{Re}(Au,u) &\geqslant C_1\Vert u\Vert^2_{{L}^{2,\mathcal{L}}_{\frac{m}{2}}(G, E_0)}-C_2\Vert u\Vert_{L^2(G, E_0)}^2\\
    &= C_1\Vert s\Vert^2_{{L}^{2,\mathcal{L}}_{\frac{m}{2}}(E)}-C_2\Vert s\Vert_{L^2(E)}^2.
\end{align*}Now, the proof is completed by observing that
\begin{align*}
  \textnormal{Re}(\tilde{A}s,s)=\textnormal{Re}(\varkappa_{\tau }^{-1}{A}\varkappa_\tau \varkappa_\tau^{-1}u,\varkappa_{\tau}^{-1} u) =\textnormal{Re}({A}\varkappa_\tau \varkappa_\tau^{-1}u, u)=  \textnormal{Re}({A}u, u),
\end{align*}where we have use that $\varkappa_\tau^{-1}$ is an isometry.
\end{proof}

\subsection{Global wellposedness  on homogeneous vector bundles}
In this subsection we present some applications of the subelliptic pseudo-differential calculus developed in previous sections to study global solvability, global wellposedness and uniqueness of the solution of hyperbolic/parabolic PDE associated to subelliptic pseudo-differential operators on homogeneous vector bundles. As before we consider $E\cong G\times_{\tau}E_0.$

Specifically, let us consider the following Cauchy problem 
\begin{equation}\label{PVI}(\textnormal{PVI}): \begin{cases}\frac{\partial u}{\partial t}=K(t,x,D)u+f ,& \text{ }u\in \mathscr{D}'((0,T)\times E),
\\u(0)=u_0, & \text{ } \end{cases}
\end{equation}where the initial data $u_0\in L^2(E),$ $K(t):=K(t,x,D)\in C([0,T], S^{m,\mathcal{L}}_{\rho,\delta}((G\times \widehat{G}) \otimes \textnormal{End}(E_0)),$ $f\in L^2([0,T],L^2(E)) ,$ $m>0.$ In addition, we also assume that the associated operator $\textnormal{Re}(K(t))$ with operator $K(t)$ defined by 
\begin{equation*}
    \textnormal{Re}(K(t)):=\frac{1}{2}(K(t)+K(t)^*),\,\,0\leqslant t\leqslant T,\footnote{ This means that $A=K(t)$ is  strongly $\mathcal{L}$-elliptic.}
\end{equation*}is $\mathcal{L}$-elliptic.

Under such assumptions we will prove the existence and uniqueness of a solution $u\in C^1([0,T],L^2(E))\cap C([0,T], L^{2, \mathcal{L}}_{m}(E)),$ $E\cong G\times_{\tau}E_0.$ We will start with the following energy estimate.
\begin{theorem}\label{energyestimate}
Let $K(t)=K(t,x,D)\in C([0,T], S^{m,\mathcal{L}}_{\rho,\delta}((G\times \widehat{G}) \otimes \textnormal{End}(E_0)),$ $0\leqslant \delta<\rho\leqslant  1,$   be a subelliptic pseudo-differential operator of order $m>0,$ and let us assume that $\textnormal{Re}(K(t))$ is $\mathcal{L}$-elliptic, for every $t\in[0,T]$ with $T>0.$ If  $u\in C^1([0,T], L^2(E) )  \cap C([0,T],L^{2,\mathcal{L}}_m (E))$ is solution of the problem \eqref{PVI} then there exist $C,C'>0,$ such that
\begin{equation} \label{GI-03}
\Vert u(t) \Vert_{L^2(E)}\leqslant   \left(C\Vert u(0) \Vert^2_{L^2(E)}+C'\int\limits_{0}^T \Vert (\partial_{t}-K(t_1))u(t_1) \Vert^2_{L^2(E)}dt_1 \right),  
\end{equation}holds for every $0\leqslant t\leqslant T.$ 

Moreover, we also have the estimate
\begin{equation}\label{Q*}
\Vert u(t) \Vert_{L^2(E)}\leqslant   \left(C\Vert u(T) \Vert^2_{L^2(E)}+C'\int\limits_{0}^T \Vert (\partial_{t}-K(\tau)^{*})u(t_1) \Vert^2_{L^2(E)}dt_1 \right).  
\end{equation}
\end{theorem}
\begin{proof} 
Let $u\in C^1([0,T], L^2(E) )  \cap C([0,T], L^{2,\mathcal{L}}_m (E)).$  We note  for further use that, by using the  embedding $L^{2, \mathcal{L}}_{m} \hookrightarrow L^{2, \mathcal{L}}_{\frac{m}{2}} $ we have   $u\in  C([0,T],L^{2, \mathcal{L}}_{\frac{m}{2}} (E)).$   So, 
$u\in \textnormal{Dom}(\partial_{t_1}-K(t_1))$ for every $0\leqslant t_1 \leqslant T.$ In view of the embedding $L^{2, \mathcal{L}}_{m} \hookrightarrow L^2(E),$ we also have that  $u\in C([0,T], L^2(E) ).$  Let us define $f(t_1):=Q(t_1)u(t_1),$ $Q(t_1):=(\partial_{t_1}-K(t_1)),$ for every $0\leqslant t_1\leqslant T.$ Observe that

\begin{align*}
   \frac{d}{dt}\Vert u(t) \Vert^2_{L^2(E)}&= \frac{d}{dt}\left(u(t),u(t)\right)_{L^2(E)}\\&=\left(\frac{d u(t)}{dt},u(t)\right)_{L^2(E)}+\left(u(t),\frac{d u(t)}{dt}\right)_{L^2(E)}\\
   &=\left(K(t)u(t)+f(t),u(t)\right)_{L^2(E)}+\left(u(t),K(t)u(t)+f(t)\right)_{L^2(E)}\\
    &=\left((K(t)+K(t)^{*})u(t),u(t)\right)_{L^2(E)}+2\textnormal{Re}(f(t), u(t))_{L^2(E)}\\
     &=(2\textnormal{Re}K(t)u(t),u(t))_{L^2(E)}+2\textnormal{Re}(f(t), u(t))_{L^2(E)}.
\end{align*}Now, from the subelliptic G\r{a}rding inequality on homogeneous vector bundles, 
\begin{align}
    \textnormal{Re}(-K(t)u(t),u(t)) \geqslant C_1\Vert u(t)\Vert_{L^{2, \mathcal{L}}_{\frac{m}{2}} (E)}-C_2\Vert u(t)\Vert_{L^2(E)}^2,
\end{align}and  the parallelogram law, we have
\begin{align*}
 2\textnormal{Re}(f(t), u(t))_{L^2(E)}&\leqslant 2\textnormal{Re}(f(t), u(t))_{L^2(E)}+\Vert f(t)\Vert_{L^2(E)}^2+\Vert u(t)\Vert_{L^2(E)}^2 \\
 &= \Vert f(t)+u(t)\Vert^2_{L^2(E)}\leqslant \Vert f(t)+u(t)\Vert^2_{L^2(E)}+\Vert f(t)-u(t)\Vert^2_{L^2(E)} \\
&= 2\Vert f(t)\Vert^2_{L^2(E)}+2\Vert u(t)\Vert^2_{L^2(E)},
\end{align*}
and consequently
\begin{align*}
   & \frac{d}{dt}\Vert u(t) \Vert^2_{L^2(E)}\\
   &\leqslant 2\left(C_2\Vert u(t)\Vert_{L^2(E)}^2-C_1\Vert u(t)\Vert_{L^{2, \mathcal{L}}_{\frac{m}{2}} (E)}\right)+2\Vert f(t)\Vert^2_{L^2(E)}+2\Vert u(t)\Vert^2_{L^2(E)}.
\end{align*}  So, we have proved that
\begin{align*}
   \frac{d}{dt}\Vert u(t) \Vert^2_{L^2(E)}\lesssim  \Vert f(t)\Vert^2_{L^2(E)}+\Vert u(t)\Vert^2_{L^2(E)}.
\end{align*} By using Gronwall's Lemma we obtain the energy estimate
\begin{equation}
\Vert u(t) \Vert^2_{L^2(E)}\leqslant  \left(C\Vert u(0) \Vert_{L^2(E)}^2+C'\int\limits_{0}^T \Vert f(t_1) \Vert_{L^2(E)}^2dt_1 \right),   
\end{equation}for every $0\leqslant t\leqslant T,$ and $T>0.$ 

For the second part of the proof, we change the analysis above with $v(T-\cdot)$ instead of $v(\cdot),$ $f(T-\cdot)$ instead of $f(\cdot)$ and $Q^{*}=-\partial_{t}-K(t)^{*},$ (or equivalently $Q=\partial_{t}-K(t)$ ) instead of $Q^{*}=-\partial_{t}+K(t)^{*}$ (or equivalently $Q=\partial_{t}-K(t)$) using that $\textnormal{Re}(K(T-t)^*)=\textnormal{Re}(K(T-t))$ to deduce that
\begin{align*}
&\Vert u(T-t) \Vert^2_{L^2(E)}\\
&\leqslant \left(C\Vert u(T) \Vert^2_{L^2(E)}+C'\int\limits_{0}^{T} \Vert (-\partial_{t}+K(T-t_1)^{*})u(T-t_1) \Vert^2_{L^2(E)}dt_1 \right)\\
&=   \left(C\Vert u(T) \Vert^2_{L^2(E)}+C'\int\limits_{0}^T \Vert (-\partial_{t}-K(t_1')^{*})u(t_1') \Vert^2_{L^2(E)}dt_1' \right).
\end{align*}So, we conclude the proof.
\end{proof}
In the following theorem we prove the uniqueness and existence of the solution of the problem \eqref{PVI} on the vector bundle $E\cong G\times_{\tau}E_0\rightarrow M$
\begin{theorem}
Let $K(t)=K(t,x,D)\in C([0,T], S^{m,\mathcal{L}}_{\rho,\delta}((G\times \widehat{G})\otimes \textnormal{End}(E_0)),$ $0\leqslant \delta<\rho\leqslant  1,$   be a subelliptic pseudo-differential operator of order $m>0,$ and let us assume that $\textnormal{Re}(K(t))$ is $\mathcal{L}$-elliptic, for every $t\in[0,T]$ with $T>0.$ Let   $f\in  L^2([0,T],L^2(E))$, and let $u_0\in L^2([0,T],L^2(E)).$ Then there exists a unique $u\in C^1([0,T], L^2(E) )  \cap C([0,T],L^{2,\mathcal{L}}_m(E)) $ solving \eqref{PVI}. Moreover, $u$ satisfies the energy estimate
\begin{equation}
\Vert u(t) \Vert_{L^2(E)}\leqslant C\Vert u_0 \Vert_{L^2(E)}+C'\Vert f \Vert_{L^2([0,T],L^2(E))} ,
\end{equation}for every $0\leqslant t\leqslant T.$
\end{theorem}
\begin{proof}
The energy estimate \eqref{GI-03} and the classical Picard iteration theorem implies the existence a solution $u$ of \eqref{PVI}. Now, in order to show the uniqueness of $u,$ let us assume that $v\in  C^1([0,T], L^2(E) )  \bigcap C([0,T],L^{2, \mathcal{L}}_m(E))$ is  another solution of the problem
\begin{equation*} \begin{cases}\frac{\partial v}{\partial t}=K(t,x,D)u+f ,& \text{ }v\in \mathscr{D}'((0,T)\times E),
\\v(0)=v_0 .& \text{ } \end{cases}
\end{equation*} Then $\omega:=v-u\in C^1([0,T], L^2(E) )  \cap C([0,T],L^{2, \mathcal{L}}_m(E))$ solves the differential problem
\begin{equation*} 
\begin{cases}
\frac{\partial \omega}{\partial t}=K(t,x,D)\omega,& \text{ }\omega\in \mathscr{D}'((0,T)\times E),
\\\omega(0)=0 .& \text{ } 
\end{cases}
\end{equation*}
From Theorem \ref{energyestimate} it follows that $\Vert \omega(t)\Vert_{L^2(E)}=0,$ for all $0\leqslant t\leqslant T.$ Hence, from the continuity in $t$ of the functions we have that $u(t,x)=v(t,x)$ for all $t\in [0,T]$ and a.e. $x\in E.$ 
\end{proof}

\section{ Pseudo-differential operators on compact homogeneous manifolds}\label{examplecompactmanifolds} 
The aim of this section is to use the general construction of the subelliptic pseudo-differential calculus developed in the previous sections to establish the properties of the subelliptic pseudo-differential calculus on the compact homogeneous manifolds $M=G/K$.  First, we  note the following observations:

\begin{itemize}
    \item The space of smooth functions $ C^{\infty}(M)$ can be identified with
    \begin{equation}\label{tauinvariantfunction1}
    C^{\infty}(G)^K=\{\dot{f}\in C^{\infty}(G)\,\, |\forall g\in G,\,\forall k\in K,\,\,\dot{f}(gk)=\dot{f}(g) \}.
\end{equation}Indeed, any function $f\in C^{\infty}(M),$ induces a unique function $\dot{f}\in C^{\infty}(G)^{K},$ via $\dot{f}(g):=f(gK),$ and viceversa.
\item Let us take the trivial representation $\tau=\omega=1_{\widehat{K}}$ of $K,$ and $E_0=F_0=\mathbb{C}.$ Then, an element $[g,z]\in G\times_{1_{\widehat{K}} }\mathbb{C}$ has the form
$$ [g,z]:=(g,z)\cdot K=\{(gk,1_{\widehat{K}}(k)^{-1}z):k\in K\}\cong (gK,z)\in M\times \mathbb{C}.   $$
As a consequence, a section $s\in \Gamma(G\times_{1_{\widehat{K}} }\mathbb{C})=\Gamma(M\times \mathbb{C}),$ is (a section on a trivial vector bundle $M\times \mathbb{C}\rightarrow M$) given by $s(x)=(x,{f}(x)),$ $x\in M,$ for some $f\in C^{\infty}(M).$ So, it is natural to identify $C^{\infty}(M)$ with $ \Gamma(G\times_{1_{\widehat{K}} }\mathbb{C} )$. Moreover, we also have the following isomorphisms $$\Gamma(G\times_{1_{\widehat{K}} }\mathbb{C})\cong C^\infty(G,\mathbb{C})^{ 1_{\widehat{K}}  }=:C^{\infty}(G)^{K}\cong  C^{\infty}(M).$$
Therefore, it makes sense to identify an operator $$A:C^{\infty}(M)\rightarrow C^{\infty}(M),$$ with another operator $$ \dot{A}: C^{\infty}(G)^{K}\rightarrow C^{\infty}(G)^{K}, $$ defined via $\dot{A}\dot{f}:=Af,$ $f\in C^{\infty}(M),$ and also
with another continuous and linear operator  $$\tilde{A}:\Gamma(G\times_{1_{\widehat{K}} } \mathbb{C})\rightarrow \Gamma(G\times_{1_{\widehat{K}} } \mathbb{C}),$$  defined by
$$   \tilde{A}s(x):=(x,Af(x)),\,s(x):=(x,f(x)),\,\,x\in M. $$
\end{itemize}
Based on these identifications we will construct the subelliptic pseudo-differential calculus on $M$ as a special case of the theory developed in previous sections.
\subsection{The quantisation formula on compact homogeneous manifolds}
It is easy to see  that every continuous linear operator $\mathbb{A}:C^\infty(G)\rightarrow C^\infty(G)$ admits a restriction $\mathbb{A}|_{C^\infty(G)^{K}}:C^\infty(G)^{K}\rightarrow C^\infty(G).$ However, the next theorem gives us a criterion in order that $\mathbb{A}({C^\infty(G)^{K}})\subset {C^\infty(G)^{K}}, $ or, in other words,  the restriction operator
$$\dot{A}\equiv \mathbb{A}|_{C^\infty(G)^{K}}:C^\infty(G)^{K}\rightarrow C^\infty(G)^{K} , $$ is well defined.  Indeed, we have the following characterisation obtained as a consequence of Theorem  \ref{TheoremCharK}.

\begin{theorem}\label{Th:invariantK}
Let $\mathbb{A}:C^\infty(G)\rightarrow  C^\infty(G),$ be a continuous linear operator. Then, for every $f\in C^\infty(G)^K$ we have $\mathbb{A}f\in  C^\infty(G)^{K}, $ if and only if, the Schwartz kernel of $\mathbb{A}$ satisfies the following identity
\begin{equation}
 K_{\mathbb{A}}(g_1,g_2)=  K_{\mathbb{A}}(g_1k_1,g_2k_2),
\end{equation}for all $k_1,k_2\in K,$ and all $g_1,g_2\in G.$
\end{theorem}
Let $\Omega^{k}(M):=\bigsqcup_{x\in M}\bigwedge^{k}(T^{*}_{x}M)$ be the space of differential $k$-forms on $M$ and  let us fix a $G$-invariant volume element $\Omega_M\in \Omega^{\textnormal{dim}(M)}(M),$ $\textnormal{dim}(M)=\textnormal{dim}(G)-\textnormal{dim}(K),$  such that, for $f \in C(G/K),$ (see \cite[Page 116]{Wallach1973})
$$\int\limits_{M} f \Omega_M= \int\limits_G f(gK)\, dg .$$
Let us recall that for local coordinates $y=(y_1,\cdots,y_{\dim(M)})$ of $M,$ we have $$\Omega_M(y)=\psi_{M}(y)dy_1\wedge\cdots \wedge dy_{\dim(M)},$$ for some smooth function $\psi_{M}\in C^{\infty}(M).$ We refer the reader to  \cite{Spivak} for the general aspects of the geometry of differential forms.

If $A:C^{\infty}(M)\rightarrow C^{\infty}(M),$ $C^{\infty}(M)\cong C^{\infty}(G)^K,$ is a continuous linear operator, the Schwartz kernel theorem provides $K_{A}\in \mathscr{D}'(M\times M)$ such that
\begin{equation}
    Af(x)=\int\limits_MK_{A}(x,y)f(y)\Omega_M(y)=\int\limits_GK_{A}(gK,hK)f(hK)dh,\,x=gK.
\end{equation}In view of Theorem \ref{Th:invariantK}, the kernel $K_{\mathbb{A}}(g,h):=K_{A}(gK,hK)$ is a well defined distribution in $\mathscr{D}'(G\times G),$ it is $K$-invariant in the sense that
\begin{equation}\label{Kinvarienceforkernel}
    \forall k_1,\,k_2\in K,\,\, K_{\mathbb{A}}(g,h)=K_{\mathbb{A}}(gk_1,hk_2),
\end{equation}   
and it defines the continuous linear operator 
\begin{equation}
    \mathbb{A}v(g):=\int\limits_{G}K_{\mathbb{A}}(g,h)v(h)dh,\,v\in C^{\infty}(G),
\end{equation}with the property that  $\mathbb{A}|_{C^{\infty}(G)^K}=\dot{A}.$ By keeping the notation $\dot{f}(h):=f(hK),$ we have
\begin{equation}
    Af(gK)=\int\limits_GK_{\mathbb{A}}(g,h)\dot{f}(h)dh=\int\limits_Gk_{g}(h^{-1}g)\dot{f}(h)dh=(\dot{f}\ast k_{g})(g),\,f\in C^{\infty}(M),
\end{equation}where $k_{g}(z)\equiv R_{A}(g,z):=K_{\mathbb{A}}(g,gz^{-1})$ is the right-convolution kernel associated with $\mathbb{A}.$ So, the Fourier inversion formula gives
\begin{equation}
 (\dot{f}\ast k_{g})(g)=\sum_{[\xi]\in \widehat{G}}d_\xi\textnormal{Tr}[\xi(g)\mathscr{F}_{G}[\dot{f}* k_{g}](\xi)].   
\end{equation}Now, using  the matrix-valued symbol of $\mathbb{A}$ we define 
\begin{equation}
    \sigma_{A}(g,\xi):= \sigma_{\mathbb{A}}(g,\xi)=\widehat{k}_{g}(\xi),\,\,[\xi]\in \widehat{G},
\end{equation}the Fourier transform of the convolution leads to the following quantisation formula
\begin{equation}
 Af(gK)= \sum_{[\xi]\in \widehat{G}}d_\xi\textnormal{Tr}[\xi(g)\sigma_A(g,\xi)\widehat{\dot{f}}(\xi)].  
\end{equation}

\begin{proof}Note that for any $k\in K,$ $1_{\widehat{K}}(k)=1.$ In view of Theorem \ref{classifiucation:tau:f}, we have that
\begin{align*}
    f\in  C^{\infty}(G)^K=C^{\infty}(G,\mathbb{C})^{1_{\widehat{K}}}
\end{align*}if and only if \begin{equation}
    \widehat{f}(i,\xi)=\xi(k)^{*}\widehat{f}(i,\xi),
\end{equation}for all $[\xi]\in \widehat{G},$ and all $k\in K.$ The proof is complete.    
\end{proof}
So, we have enough motivation to introduce the following definition.
\begin{definition}[Global symbols on $M=G/K$] Let $\Omega_M$ be a $G$-invariant volume form on $M.$ Let $A: C^\infty(M)\rightarrow C^\infty(M)$ be a continuous linear operator and let $K_{A}\in \mathscr{D}'(M\times M)$ be its Schwartz kernel, so that
\begin{equation}
    Af(x)=\int\limits_MK_{A}(x,y)f(y)\Omega_M(y),\,\,f\in C^{\infty}(M).
\end{equation}
Let $K_{\mathbb{A}}\in \mathscr{D}'(G\times G)$ be the distribution defined via $K_{\mathbb{A}}(g_1,g_2):=K_A(g_1K,g_2K)$ and let us consider the operator
\begin{equation}
    \mathbb{A}v(g_1):=\int\limits_{G}K_{\mathbb{A}}(g_1,g_2)v(g_2)dg_2,\,v\in C^{\infty}(G).
\end{equation}The symbol of $A$ is defined to be the distribution
\begin{equation}\label{thesymbolofA:cM:CM}
    \sigma_A(g,\xi)\equiv \sigma_{\mathbb{A}}(g,\xi):=\widehat{k}_{g}(\xi)=\int\limits_{G}k_{g}(z)\xi(z)^{*}dz,\,\,[\xi]\in \widehat{G},
\end{equation}where $g\mapsto k_{g}:G\rightarrow \mathscr{D}'(G)$ is the right-convolution kernel of $\mathbb{A}.$ Because the Schwartz kernel $K_A$ depends on  the $G$-invariant volume form $\Omega_M,$ so are the kernel  $K_{\mathbb{A}}$ and the symbol $\sigma_{A}.$ 
\end{definition}

The analysis above allows us to formulate the following version of  Theorem \ref{mainqhomovec} in the case of the trivial representation $\tau=\omega:=1_{\widehat{K}}.$
 \begin{theorem}\label{mainqhomovec:trivial:rep}  Let $A:  C^{\infty}(M)\rightarrow  C^{\infty}(M)$  be a continuous linear operator and let $\tilde{A}:\Gamma(G\times_{1_{\widehat{K}} }\mathbb{C})\rightarrow\Gamma(G\times_{1_{\widehat{K}} }\mathbb{C})$ be the induced operator by the identifications $C^{\infty}(M)\cong \Gamma(G\times_{1_{\widehat{K}} }\mathbb{C}).$  Then,  we have
\begin{equation}\label{quantizationonhomogeneous4}
  \boxed{     \tilde{A}s(gK)=\left(gK, \,Af(gK)\right)}
\end{equation}for any section $s=(x,f(x))\in \Gamma(G\times_{1_{\widehat{K}} }\mathbb{C})=\Gamma(M\times \mathbb{C}),$  and we also have
\begin{equation}\label{quantisation2manifolds}
   \boxed{    Af(gK)=\sum_{[\xi]\in \widehat{G}}\textnormal{Tr}[\xi(g)\sigma_A(g,[\xi])  \widehat{\dot{f}}([\xi])] ,\, f\in C^{\infty}(G)^{K}   }
\end{equation} with $\dot{f}(h)=f(hK),$ $h\in G.$
\end{theorem}

\subsection{The calculus on compact homogeneous manifolds}
In this section we summarise the main properties of the subelliptic pseudo-differential calculus on $M=G/K.$ We start with the definition of the subelliptic H\"ormander classes on $M$. Here, we consider a sub-Laplacian $\mathcal{L}=-(X_1^2+\cdots +X_k^2)$ on $G,$ where the system of vector fields $X=\{X_i\}_{i=1}^{k}$ satisfies the H\"ormander condition. 

\begin{definition}[Subelliptic H\"ormander classes of symbols on $M$] Let $\Omega_M$ be a $G$-invariant volume form on $M.$ We say that  $\sigma_A\in S^{m,\mathcal{L}}_{\rho,\delta}(M,\Omega_M), $ $0\leq \rho,\delta\leq 1,$ if the symbol $\sigma_{\mathbb{A}}$ in \eqref{thesymbolofA:cM:CM} belongs to the subelliptic H\"ormander class $S^{m,\mathcal{L}}_{\rho,\delta}(G\times \widehat{G}).$ We define
$$ \Psi^{m,\mathcal{L}}_{\rho,\delta}(M,\Omega_M):=\{A:C^{\infty}(M)\rightarrow C^{\infty}(M):\sigma_A\in S^{m,\mathcal{L}}_{\rho,\delta}(M,\Omega_M)\}.   $$
 
\end{definition}

\begin{remark}
Observe that 
\begin{equation}\label{secondaruequivalence}
    A\in  \Psi^{m,\mathcal{L}}_{\rho,\delta}(M,\Omega_M)\,\iff\,\tilde{A}\in  \Psi^{m,\mathcal{L}}_{\rho,\delta}(G\times_{1_{\widehat{K}} }\mathbb{C},\,G\times_{1_{\widehat{K}}}\mathbb{C}),
\end{equation} with the notation of homogeneous vector-bundles in \eqref{Subelliptichormanderclassesonvectorbundles}.
\end{remark}
The following theorem follows from the subelliptic calculus on homogeneous vector bundles developed in Section  \ref{vectcal}.

\begin{theorem}\label{calculus} Let $0\leqslant \delta<\rho\leqslant 1,$ let $\Omega_M$ be a $G$-invariant volume form on $M,$ and for any $m\in \mathbb{R},$ let $ \Psi^{m,\mathcal{L}}_{\rho,\delta}(M,\Omega_M)$ be the corresponding subelliptic H\"ormander classes associated with $\Omega_M.$ Then we have the following properties.
\begin{itemize}
    \item [-] The mapping $A\mapsto A^{*}:\Psi^{m,\mathcal{L}}_{\rho,\delta}(M,\Omega_M)\rightarrow \Psi^{m,\mathcal{L}}_{\rho,\delta}(M,\Omega_M)$ is a continuous linear mapping between Fr\'echet spaces and  the  symbol of $A^*,$ $\sigma_{A^*}(x,\xi)$ satisfies the asymptotic expansion,
 \begin{equation*}
   \sigma_{A^*}(g,\xi)\sim \sum_{|\alpha|= 0}^\infty\Delta_{\xi}^\alpha\partial_{X}^{(\alpha)} (\sigma_{A}(g,\xi)^{*}).
 \end{equation*} This means that, for every $N\in \mathbb{N},$ and all $\ell\in \mathbb{N},$
\begin{equation*}
   \Small{ \Delta_{\xi}^{\alpha_\ell}\partial_{X}^{(\beta)}\left(\sigma_{A^*}(g,\xi)-\sum_{|\alpha|\leqslant N}\Delta_{\xi}^\alpha\partial_{X}^{(\alpha)} (\sigma_{A}(g,\xi)^{*}) \right)\in {S}^{m-(\rho-\delta)(N+1)-\rho\ell+\delta|\beta|,\mathcal{L}}_{\rho,\delta}(G\times\widehat{G}) },
\end{equation*} where $|\alpha_\ell|=\ell.$
\item [-] The mapping $(A_1,A_2)\mapsto A_1\circ A_2: \Psi^{m_1,\mathcal{L}}_{\rho,\delta}(M,\Omega_M)\times \Psi^{m_2,\mathcal{L}}_{\rho,\delta}(M,\Omega_M)\rightarrow \Psi^{m_1+m_2,\mathcal{L}}_{\rho,\delta}(M,\Omega_M)$ is a continuous bilinear mapping between Fr\'echet spaces. Moreover,  the symbol of $A=A_{1}\circ A_2,$ where each $A_j$ has symbol $\widehat{A}_j,$ with $j=1,2,$ satisfies the asymptotic expansion,
\begin{equation*}
    \sigma_A(x,\xi)\sim \sum_{|\alpha|= 0}^\infty(\Delta_{\xi}^\alpha\widehat{A}_{1}(x,\xi))(\partial_{X}^{(\alpha)} \widehat{A}_2(x,\xi)),
\end{equation*}this means that, for every $N\in \mathbb{N},$ and all $\ell \in\mathbb{N},$
\begin{align*}
    &\Delta_{\xi}^{\alpha_\ell}\partial_{X}^{(\beta)}\left(\sigma_A(g,\xi)-\sum_{|\alpha|\leqslant N}  (\Delta_{\xi}^\alpha\widehat{A}_{1}(g,\xi))(\partial_{X}^{(\alpha)} \widehat{A}_2(g,\xi))  \right)\\
    &\hspace{2cm}\in {S}^{m_1+m_2-(\rho-\delta)(N+1)-\rho\ell+\delta|\beta|,\mathcal{L}}_{\rho,\delta}(G\times \widehat{G}),
\end{align*}for every  $\alpha_\ell \in \mathbb{N}_0^n$ with $|\alpha_\ell|=\ell.$
\item [-] For  $0\leqslant \delta< \rho\leqslant    1,$ (or for $0\leq \delta\leq \rho\leq 1,$ $\delta<1/\kappa$) let us consider a continuous linear operator $A:C^\infty(M)\rightarrow C^\infty(M)$ with symbol  $\sigma\in {S}^{0,\mathcal{L}}_{\rho,\delta}(M,\Omega_M)$. Then $A$ extends to a bounded operator from $L^2(M,\Omega_M)$ to  $L^2(M,\Omega_M).$ 
\end{itemize}
\end{theorem}
\begin{proof}
The asymptotic expansions for the product and the adjoint of subelliptic pseudo-differential operators follow from \cite{RuzhanskyCardona2020}.  So, to finish the proof let us prove the $L^2$-boundedness for operators with symbols in the class ${S}^{0,\mathcal{L}}_{\rho,\delta}(M,\Omega_M).$ Observe that for $A\in \Psi^{0,\mathcal{L}}_{\rho,\delta}(M,\Omega_M),$ we have $\mathbb{A}\in \Psi^{0,\mathcal{L}}_{\rho,\delta}(G\times \widehat{G}),$ and then, for all $f\in C^\infty(M),$ we have
\begin{align*}
    \Vert Af\Vert_{L^2(M,\Omega_M)}^{2}&=\int\limits_{M}|Af|^2\Omega_M=\int\limits_{G}|\dot{A}\dot{f}(g)|^2dg=\int\limits_{G}|\mathbb{A}\dot{f}(g)|^2dg\leq \Vert \mathbb{A}\Vert^2_{\mathscr{B}(L^2(G))}\Vert \dot{f}\Vert^2_{L^2(G)}\\
    &=\Vert \mathbb{A}\Vert^2_{\mathscr{B}(L^2(G))}\Vert f\Vert^2_{L^2(M,\Omega_M)},
\end{align*}having used the boundedness of $\mathbb{A}$ on $L^2(G)$ in view of the subelliptic Calder\'on-Vaillancourt Theorem \ref{CVT}. 
\end{proof}

We finish this section with the following characterisation of the H\"ormander classes on $M$ defined by local coordinate systems.
We denote $\Psi^{m}_{\rho,\delta}(M,\Omega_M):=\Psi^{m,\mathcal{L}_G}_{\rho,\delta}(M,\Omega_M)$ the elliptic H\"ormander classes defined by the Laplacian on $G,$ this means, $$ A\in \Psi^{m}_{\rho,\delta}(M,\Omega_M)\,\iff\, \mathbb{A}\in \Psi^{m}_{\rho,\delta}(G\times \widehat{G}).$$
\begin{corollary}
Let $m\in \mathbb{R},$ and let $0\leq \delta<\rho\leq 1,$ and $\rho\geq 1-\delta.$ Then,  the following facts are equivalent.
\begin{itemize}
    \item[(A)] $A\in \Psi^{m}_{\rho,\delta}(M,\Omega_M). $
    \item[(B)] $A\in \Psi^{m}_{\rho,\delta}(M,\textnormal{loc}).$ 
\end{itemize}
\end{corollary}
\begin{proof} Indeed, in view of Theorem \ref{classificationoflocalclasses} and using the equivalence \eqref{secondaruequivalence} we have
\begin{align*}
    A\in \Psi^{m}_{\rho,\delta}(M,\Omega_M)&\,\iff\,\tilde{A}\in  \Psi^{m}_{\rho,\delta}(G\times_{1_{\widehat{K}} }\mathbb{C},\,G\times_{1_{\widehat{K}}}\mathbb{C})\\
    &\,\iff\,\tilde{A}\in  \Psi^{m}_{\rho,\delta}(G\times_{1_{\widehat{K}} }\mathbb{C},\,G\times_{1_{\widehat{K}}}\mathbb{C};\textnormal{loc})\\
    &\,\iff\,{A}\in  \Psi^{m}_{\rho,\delta}(M;\textnormal{loc}),
\end{align*}provided that $0\leq \delta<\rho\leq 1,$ and $\rho\geq 1-\delta.$ Indeed, the equivalence $$\tilde{A}\in  \Psi^{m}_{\rho,\delta}(G\times_{1_{\widehat{K}} }\mathbb{C},\,G\times_{1_{\widehat{K}}}\mathbb{C};\textnormal{loc})\,\iff\,{A}\in  \Psi^{m}_{\rho,\delta}(M;\textnormal{loc})$$ is valid in view of the commutativity of the following diagram
\begin{center}
    \begin{tikzpicture}[every node/.style={midway}]
  \matrix[column sep={10em,between origins}, row sep={4em}] at (0,0) {
    \node(R) {$\Gamma(G\times_{1_{\widehat{K}} }\mathbb{C})$}  ; & \node(S) {$\Gamma(G\times_{1_{\widehat{K}} }\mathbb{C})$}; \\
    \node(R/I) {$C^\infty(M)$}; & \node (T) {$C^\infty(M)$};\\
  };
  \draw[<-] (R/I) -- (R) node[anchor=east]  {$\varkappa_{1_{\widehat{K}}}$};
  \draw[->] (R) -- (S) node[anchor=south] {$\tilde{A}$};
  \draw[->] (S) -- (T) node[anchor=west] {$\varkappa_{1_{\widehat{K}}}$};
  \draw[->] (R/I) -- (T) node[anchor=north] {$A$};
\end{tikzpicture} 
\end{center} provided that $0\leq \delta<\rho\leq 1,$ and $\rho\geq 1-\delta.$ Thus, the proof is complete.
\end{proof}

\section{Global symbols   on  vector bundles of differential forms}\label{Section:Forms}

To illustrate the quantisation formula on homogeneous vector-bundles, in this section we compute the global symbol of the exterior derivative, its adjoint, the Hodge Laplacian and the Dirac operator. These are the operators of relevant interest in the theory of differential forms on $M.$ Notations and preliminaries are taken from Wallach \cite{Wallach1973} and from  \cite[Section 2]{Ikeda}. 

\subsection{Differential forms on a homogeneous space}
Let us consider $G$ to be connected and simply connected, and the quotient manifold $M$ to be orientable. We denote by $\mathfrak{g}$ and $\mathfrak{k}$ the Lie algebras of $G$ and $K,$ respectively. Because a $K$-invariant inner product on $\mathfrak{g}/\mathfrak{k}$ induces a $G$-invariant Riemannian metric on $G$ (see \cite[Section 2]{Ikeda} for this and further details),  let us fix a $G$-invariant metric on $G,$ and let us extend it in the canonical way, to an Hermitian metric $\langle\cdot ,\cdot \rangle_{\Omega^p},$ on $$\Omega^p(M):=\bigwedge^p T^{*}M\otimes \mathbb{C}, $$ that is, the  $p$-th exterior  power  of  the  complexified  cotangent bundle  of  $M$.
The space of (smooth) sections  of the homogeneous ($G$-invariant) vector bundle 
\begin{equation}
   \Omega^p(M)\rightarrow M,\,\,\,1\leq p\leq m:=\dim(M),
\end{equation}is the set of $p$-differential forms on $M.$ So, a $p$-form $\omega\in \Gamma( \Omega^p(M))$ assigns to any point $x\in M,$ an element $\omega_x\in \bigwedge^{p}T_{x}^{*}M,$ of $C^{\infty}$-class  with respect to $x.$ In an arbitrary coordinate neighborhood $U$ of $M$, and $x_1,x_2,\cdots, x_m$ coordinate functions defined on $U,$ the family of vectors
$$  \frac{\partial}{\partial x_1}|_{x},\cdots,  \frac{\partial}{\partial x_m}|_{x} , $$
provides a basis of the tangent space $T_{x}M.$ The functions $x_i:U\rightarrow\mathbb{R}$ are smooth, and one can consider the differential $(dx_i)|_{x}: T_{x}U\rightarrow \mathbb{R},$ at $x,$ can be considered as an element of $T_{x}^{*}M,$ via
$$  dx_i\left(\frac{\partial}{\partial x_j}\right)=\delta_{ij}.  $$ In consequence,
$$ (dx_1)|_{x},\cdots,  (dx_m)|_{x},  $$ becomes a basis of the cotangent space $T_{x}^{*}M,$ and 
\begin{equation}\label{defi:difffrom}
    \omega_x=\sum_{I=(i_1,\cdots, i_p)}\omega_{I}(x)dx^{I}|_x,\,\,dx^{I}|_x:=dx_{i_1}|_x\wedge \cdots\wedge dx_{i_p}|_x,\,\omega_{I}\in C^{\infty}(M),
\end{equation}where $1\leq i_1<\cdots <i_p\leq m.$ So, $\dim( \bigwedge^{p}T_{x}^{*}M)={{m}\choose{p}}=\textnormal{rank}(\Omega^p(M)).$

By considering the adjoint action of $K$ on $\mathfrak{g},$ we have the representation $\textnormal{Ad}|_{K}:K\rightarrow \textnormal{GL}(\mathfrak{g}).$ The fiber at $e_M=e_G K$ of the tangent bundle
\begin{equation}\label{tangentbundle}
    TM\rightarrow M,
\end{equation} that is $T_{e_M},$ can be identified with $\mathfrak{g}/\mathfrak{k}.$ There is a natural structure on \eqref{tangentbundle} making of it a homogeneous vector-bundle. 
Indeed, let us consider the right action of $K$ on $G\times \mathfrak{g}/\mathfrak{k},$
\begin{equation}\label{equationofadjoint}
    (G\times \mathfrak{g}/\mathfrak{k})\times K\rightarrow G\times \mathfrak{g}/\mathfrak{k},\,\quad(g,Y+\mathfrak{k})\cdot k=(gk,p_{\mathfrak{k}}\circ\textnormal{Ad}|_{K}(k^{-1})(Y+\mathfrak{k})),
\end{equation}where $p_{\mathfrak{k}}:\mathfrak{g}\rightarrow \mathfrak{g}/\mathfrak{k}$ is the quotient mapping. Let us record that the adjoint representation  $\textnormal{Ad}:G\rightarrow \textnormal{GL}(\mathfrak{g})$ is defined via
$$  \textnormal{Ad}(g):=(dC_{g})_{e_G}:\mathfrak{g}\rightarrow \mathfrak{g},\,\,\, C_{g}:G\rightarrow G,\,C_{g_1}(g_2):=g^{-1}_1g_2g_1,\,\,g_1,g_2\in G.  $$ In particular, in \eqref{equationofadjoint} we have considered the restriction operator $\textnormal{Ad}|_{K}:K\rightarrow \textnormal{GL}(\mathfrak{g}).$
So, the fact that the fiber at $e_{G}K\in M,$ of the tangent bundle $TM\rightarrow M$ can be identified with $\mathfrak{g}/\mathfrak{k},$   implies that
\begin{equation}
  TM\cong   G\times_{p_{\mathfrak{k}}\circ \textnormal{Ad}|_{K}}  \mathfrak{g}/\mathfrak{k}.
\end{equation}Similarly, we have
\begin{equation}
  T^*M\cong   G\times_{(p_{\mathfrak{k}}\circ\textnormal{Ad}|_{K})^{*} }  ( \mathfrak{g}/\mathfrak{k})^{*},
\end{equation}
with the representation $(p_{\mathfrak{k}}\circ\textnormal{Ad}|_{K})^{*}$ of $K$ given by
\begin{equation}
  (\textnormal{Ad}|_{K})^{*}(k)(\eta)(Y+\mathfrak{k})=\eta(p_{\mathfrak{k}}\circ\textnormal{Ad}|_{K}(k^{-1})(Y+\mathfrak{k})  ).
\end{equation}
By considering the complexification $TM\otimes \mathbb{C}$ and $T^*M\otimes \mathbb{C},$ of the spaces $TM$ and $T^*M,$ respectively, we also have the following identifications for the space of $p$-forms,
\begin{equation}
    \Omega^{p}(M)\cong G\times_{\bigwedge^{p}\textnormal{Ad}|_{K}}   \bigwedge^{p}(\mathfrak{g}/\mathfrak{k})^{*}\otimes\mathbb{C}.
\end{equation}Let us remark that the representation $\bigwedge^{p}\textnormal{Ad}|_{K}$ is defined via
\begin{align*}
   & \bigwedge^{p}\textnormal{Ad}|_{K}(k)(e_{i_1}\wedge \cdots \wedge e_{i_p} )[Y_1+\mathfrak{k},\cdots,Y_p+\mathfrak{k}]\\
    &:=e_{i_1}\wedge \cdots \wedge e_{i_p} [p_{\mathfrak{k}}\circ \textnormal{Ad}|_{K}(k^{-1})(Y_1+\mathfrak{k}),\cdots, p_{\mathfrak{k}}\circ \textnormal{Ad}|_{K}(k^{-1})(Y_1+\mathfrak{k})],\,\,k\in K,    
\end{align*} on canonical elements $e_{i_1}\wedge \cdots \wedge e_{i_p}$ of $\bigwedge^{p}(\mathfrak{g}/\mathfrak{k})^{*}\otimes\mathbb{C}.$
Because the previous identifications can be defined in terms of the adjoint action of $G$ into $\mathfrak{g},$ we allow the following notations
\begin{itemize}
    \item $ TM\cong  G\times_{K}  \mathfrak{g}/\mathfrak{k}:= G\times_{p_{\mathfrak{k}}\circ \textnormal{Ad}|_{K}}  \mathfrak{g}/\mathfrak{k},$\\
    \item $ T^*M\cong  G\times_{K }  (\mathfrak{g}/\mathfrak{k})^{*}:= G\times_{(p_{\mathfrak{k}}\circ \textnormal{Ad}|_{K})^{*} }  (\mathfrak{g}/\mathfrak{k})^{*},$\\
    \item $ \Omega^{p}(M)\cong G\times_{K}   \bigwedge^{p}(\mathfrak{g}/\mathfrak{k})^{*}\otimes\mathbb{C}:=    G\times_{\bigwedge^{p}\textnormal{Ad}|_{K}}   \bigwedge^{p}(\mathfrak{g}/\mathfrak{k})^{*}\otimes\mathbb{C}.$
\end{itemize}Finally, the family of sections $\Gamma(\Omega^{p}(M))$ of the homogeneous ($G$-invariant) vector-bundle $\Omega^{p}(M)\rightarrow M$ is the family of $p$-differential forms on $M,$ which admits the following identification

\begin{equation}
    \Gamma(\Omega^{p}(M))\cong C^{\infty}(G,\bigwedge^{p}(\mathfrak{g}/\mathfrak{k})^{*}\otimes\mathbb{C})^{\bigwedge^{p}\textnormal{Ad}|_{K}}=:C^{\infty}(G,\bigwedge^{p}(\mathfrak{g}/\mathfrak{k})^{*}\otimes\mathbb{C})^{K}.
\end{equation}In this terminology, a $p$-form $\omega\in \Gamma(\Omega^{p}(M)), $ is identified with a function from $G$ into $\bigwedge^{p}(\mathfrak{g}/\mathfrak{k})^{*}\otimes\mathbb{C}$ satisfying
\begin{equation}
    \omega(gk)=\bigwedge^{p}\textnormal{Ad}|_{K}(k^{-1}) \omega(g).
\end{equation}We also have the isomorphism of Hilbert spaces
\begin{equation}
   L^2(\Omega^{p}(M))\cong L^2(G,\bigwedge^{p}(\mathfrak{g}/\mathfrak{k})^{*}\otimes\mathbb{C})^{\bigwedge^{p}\textnormal{Ad}|_{K}}=:L^2(G,\bigwedge^{p}(\mathfrak{g}/\mathfrak{k})^{*}\otimes\mathbb{C})^{K}.
\end{equation}

\subsection{Fourier analysis associated to differential forms} Starting with a $K$-invariant inner product on $\mathfrak{g}/\mathfrak{k},$ we can identify $(\mathfrak{g}/\mathfrak{k})^{*},$ with the set of linear forms on $\mathfrak{g},$ vanishing in $\mathfrak{k}.$  Let $\mathfrak{m}$ be the orthogonal complement of $\mathfrak{k}$ in $\mathfrak{g}$ with respect to the the Killing form
\begin{equation*}
   B(Y_1,Y_2):=\textnormal{Tr}(\textnormal{ad}(Y_1)\circ \textnormal{ad}(Y_2)).
\end{equation*}Restricting    $-B$  to  $\mathfrak{m}$, we  can  define  a  $G$-invariant Riemannian  metric  on $M=G/K.$ Now, we choose an orthonormal basis
\begin{equation}\label{basisofg}
    \{ Y_1,\cdots, Y_m, Y_{m+1}\cdots,Y_{n}    \},\,m=\dim(M),\,n=\dim(G),
\end{equation}of $\mathfrak{g}$ with respect to the Killing form, in such a way that $\{ Y_1,\cdots, Y_m\}$ is a basis of $\mathfrak{m},$ and $\{ Y_{m+1}\cdots,Y_{n}    \}$ a basis of $\mathfrak{k}.$ A $p$-form $\omega \in  \Gamma(\Omega^{p}(M))\cong C^{\infty}(G,\bigwedge^{p}(\mathfrak{g}/\mathfrak{k})^{*}\otimes\mathbb{C})^{K}, $ is determined by the system of $C^{\infty}$-functions
\begin{equation}
    \omega(Y_{i_1},\cdots, Y_{i_p}),\,\,\,1\leq i_1<\cdots<i_p\leq m,
\end{equation}
satisfying
\begin{equation}
    \omega(g)=\sum_{I_p}\omega_{I_p}(g) \omega(Y_{i_1},\cdots, Y_{i_p}).
\end{equation}
Indeed, denoting $I_p:=(i_1,\cdots,i_p),$ and by considering the dual basis of \eqref{basisofg},
\begin{equation*}
    \{ dY_1,\cdots, dY_m, dY_{m+1}\cdots,dY_{n}    \},\,dY_i(Y_j):=\delta_{i,j},
\end{equation*}
we have the unique representation 
\begin{equation}\label{extensionintipicallemtnts}
    \omega(g)=\sum_{I_p}\omega_{I_p}(g)dY_{i_1}\wedge\cdots\wedge dY_{i_p},\,\, 
\end{equation}from which one deduces
\begin{equation}
    \omega(g)(Y_{i_1},\cdots, Y_{i_p})=\omega_{I_p}(g).
\end{equation}
\begin{definition}
   Let  us consider a $p$-form $\omega \in  \Gamma(\Omega^{p}(M))\cong C^{\infty}(G,\bigwedge^{p}(\mathfrak{g}/\mathfrak{k})^{*}\otimes\mathbb{C})^{K}. $ The Fourier transform of $\omega$ at $(I_p,[\xi])$ is given by
   \begin{equation}
       \widehat{\omega}(I_p,[\xi]):=\int\limits_G \omega_{I_p}(g)\xi(g)^{*}dg.
   \end{equation}
 
 \end{definition}

\subsection{Global symbol of the exterior derivative $d$}

 The exterior derivative $d_p:\Gamma(\Omega^{p}(M))\rightarrow \Gamma(\Omega^{p+1}(M)) $ acting on $p$-forms is defined in local coordinates on a $p$-form $\omega$ (see \eqref{defi:difffrom}) by
\begin{equation}\label{defi:difffrom2}
    d_p\omega_x=\sum_{i=1}^{m}\sum_{I=(i_1,\cdots, i_p)}\frac{\omega_{I}(x)}{\partial x_i}dx_{i}|_{x}\wedge dx^{I}|_{x},\,x\in M.
\end{equation}As a section of the vector bundle $\Omega^{p+1}(M)\rightarrow M,$ $d_p\omega$ is defined via
\begin{equation}\label{defi:difffrom22}
    d_p\omega=\sum_{i=1}^{m}\sum_{I=(i_1,\cdots, i_p)}\frac{\omega_{I}(x)}{\partial x_i}dx_{i}\wedge dx^{I}.
\end{equation}
\begin{remark}[The identification $\Gamma(\Omega^{p}(M))\cong C^{\infty}(G,\bigwedge^{p}(\mathfrak{g}/\mathfrak{k})^{*}\otimes\mathbb{C})^{K}$]\label{remarkaboutidentification}
 As usual, let  $p_{M}:G\rightarrow M,$ $g\mapsto gK,$ $g\in G,$ be the natural projection. Let us define the function $\omega\mapsto \tilde{\omega},$
\begin{eqnarray}
 \Gamma(\Omega^{p}(M))  \ni \omega\mapsto \tilde{\omega},\,&\tilde{\omega}(Y_1,\cdots, Y_p)=p_{M}^{*}\omega(Y_1,\cdots, Y_p)(g)\label{identificationof forms}\\
 &:=\omega_{gK}(dp_{M}\circ dL_{g}|_{e_G}Y_1,\cdots,dp_{M}\circ dL_{g}|_{e_G}Y_p )\label{identificationof forms2}.
\end{eqnarray}Observe that we have used the isomorphism between the set of left-invariant vector fields and the Lie algebra $\mathfrak{g}.$\footnote{Indeed, the mapping $TG\rightarrow\mathfrak{g},$ $Y\mapsto Y_{e_G},$ is an isomorphism, and the mapping 
$$  Y_{g}:=dL_{g}|_{e_G}Y,\,\,Y\in \mathfrak{g},\,\,L_{g}:G\rightarrow G,\,L_{g}(h):=gh,\,g, h\in G, $$ associates to any vector $Y\in \mathfrak{g},$ a vector-field $g\mapsto Y_g$ in $TG$.} Observe that we also have done the  identification of  $(\mathfrak{g}/\mathfrak{k})^{*},$ with the set of linear forms on $\mathfrak{g},$ vanishing in $\mathfrak{k}$ which makes of $\tilde{\omega}$ in  \eqref{identificationof forms} defined via  \eqref{identificationof forms2} a well defined function  in  $C^{\infty}(G,\bigwedge^{p}(\mathfrak{g}/\mathfrak{k})^{*}\otimes\mathbb{C})^{K}.$
\end{remark}
The exterior derivative $d_{p}$ induces an operator $\mathfrak{d}_{p}$ defined by the identification $\Gamma(\Omega^{p}(M))\cong C^{\infty}(G,\bigwedge^{p}(\mathfrak{g}/\mathfrak{k})^{*}\otimes\mathbb{C})^{K}$ in Remark \ref{remarkaboutidentification}. It makes commutative the following diagram

\begin{eqnarray}
    \begin{tikzpicture}[every node/.style={midway}]
  \matrix[column sep={14em,between origins}, row sep={4em}] at (0,0) {
    \node(R) {$\Gamma(\Omega^{p}(M))$}  ; & \node(S) {$\Gamma(\Omega^{p+1}(M))$}; \\
    \node(R/I) {$C^{\infty}(G,\bigwedge^{p}(\mathfrak{g}/\mathfrak{k})^{*}\otimes\mathbb{C})^{K}$}; & \node (T) {$C^{\infty}(G,\bigwedge^{p+1}(\mathfrak{g}/\mathfrak{k})^{*}\otimes\mathbb{C})^{K}.$};\\
  };
  \draw[<-] (R/I) -- (R) node[anchor=east]  {$\cong$};
  \draw[->] (R) -- (S) node[anchor=south] {$d_{p}$};
  \draw[->] (S) -- (T) node[anchor=west] {$\cong$};
  \draw[->] (R/I) -- (T) node[anchor=north] {$\mathfrak{d}_{p}$};
\end{tikzpicture}
\end{eqnarray} More precisely,
\begin{equation}
   \boxed{ \mathfrak{d}_{p}\tilde{\omega}=\widetilde{d_{p}\omega}  ,\,\,\omega\in \Omega^{p}(M) }
\end{equation}
Using the basis \begin{equation}
    \{ Y_1,\cdots, Y_m, Y_{m+1}\cdots,Y_{n}    \},\,m=\dim(M),\,n=\dim(G),
\end{equation}of $\mathfrak{g}$ in \eqref{basisofg}, one can express the operator $\mathfrak{d}_{p}$ by  (see \cite[Page 520]{Ikeda})
\begin{equation}
   \mathfrak{d}_{p}\tilde{\omega}(Y_{i_1},Y_{i_2},\cdots ,Y_{i_{p+1}}):=\sum_{u=1}^{p+1}(-1)^{u-1}Y_{i_u}\tilde{\omega}(Y_{i_1},Y_{i_2},\cdots,Y_{i_u}^{\bullet} ,\cdots ,Y_{i_{p+1}}).
\end{equation}The notation $Y_{i_u}^{\bullet}$ means that the term  $Y_{i_{u}}$ is omitted for any $1\leq i_{1}<\cdots < i_{p+1}\leq n.$ Let us consider the basis 
\begin{equation}\label{basisiofthefiberat0}
    \{dY^{I_p}:=dY_{i_1}\wedge\cdots\wedge dY_{i_{p+1}} \},\,\,\,I_{p}=(i_1,\cdots,i_p), \,1\leq i_1<\cdots < i_p\leq n,
\end{equation} of $\bigwedge^{p}(\mathfrak{g}/\mathfrak{k})^{*}\otimes\mathbb{C}.$ We compute the symbol of the exterior derivative in the following theorem. We are going to use the notation $J_{p}=\{j_1,\cdots, j_p\},$ where $1\leq j_1<\cdots<j_{n}\leq n,$ for any $p\in \mathbb{N}.$
\begin{theorem}\label{symbolofD}Let us consider  the Cartan decomposition $\mathfrak{g}=\mathfrak{m}\oplus \mathfrak{k}.$ Under the  identification of  $(\mathfrak{g}/\mathfrak{k})^{*},$ with the set of linear forms on $\mathfrak{g},$ vanishing in $\mathfrak{k},$\footnote{and in terms of the basis in \eqref{basisiofthefiberat0}, where 
$
    \{ Y_1,\cdots, Y_m, Y_{m+1}\cdots,Y_{n}    \},\,m=\dim(M),\,n=\dim(G),
$ is  an orthonormal basis of $\mathfrak{g}$ with respect to the Killing form, in such a way that $\{ Y_1,\cdots, Y_m\}$ is a basis of $\mathfrak{m},$ and $\{ Y_{m+1}\cdots,Y_{n}    \}$ a basis of $\mathfrak{k}.$} the symbol of the exterior derivative $\mathfrak{d}_p$ is given by
\begin{equation}\label{symbol:of:d_p}
    \sigma_{\mathfrak{d}_{p}}(I_{p},J_{p+1},[\xi])=\sum_{u=1}^{p+1}(-1)^{u-1}\textnormal{det}\left[\delta_{i_k,j_v}\right]_{ \begin{subarray}{l}  1\leq k\leq p\\
  1\leq v\leq p+1,\,v\neq u
  \end{subarray} }\sigma_{Y_{j_u}}(\xi),
\end{equation}for any triple $(I_p,J_{p+1},[\xi]),$ with $[\xi]\in \widehat{G}.$

\end{theorem}
\begin{proof}
For any $\tilde{\omega}\in C^{\infty}(G,\bigwedge^{p}(\mathfrak{g}/\mathfrak{k})^{*}\otimes\mathbb{C})^{K} $ we have
\begin{equation}
      \tilde{\omega}(g)=\sum_{I_p}\tilde{\omega}_{I_p}(g)dY_{i_1}\wedge\cdots\wedge Y_{i_p}=:\sum_{I_p}\omega_{I_p}(g)dY^{I_p},\,\,
\end{equation}with components $\tilde{\omega}_{I_p}\in C^{\infty}(G).$ Let us choose a $(p+1)-$multi-index  $J_{p+1}=(j_{1},\cdots, j_{p+1}),$ $1\leq j_1<\cdots<j_{p+1}\leq n.$  Observing that \begin{align*}
   &  \mathfrak{d}_{p}\tilde{\omega}(g)(Y_{j_1},Y_{j_2},\cdots ,Y_{j_{p+1}})\\
   &=\sum_{u=1}^{p+1}(-1)^{u-1}Y_{j_u}\tilde{\omega}(Y_{j_1},Y_{j_2},\cdots,Y_{j_u}^{\bullet} ,\cdots ,Y_{j_{p+1}})\\
     &=\sum_{I_p}\sum_{u=1}^{p+1}(-1)^{u-1}Y_{j_u}\omega_{I_p}(g)dY^{I_p}(Y_{j_1},Y_{j_2},\cdots,Y_{j_u}^{\bullet} ,\cdots ,Y_{j_{p+1}})\\
     &=\sum_{I_p}\sum_{u=1}^{p+1}(-1)^{u-1}Y_{j_u}\omega_{I_p}(g)\textnormal{det}[dY_{i_k}(Y_{j_v})]_{ \begin{subarray}{l}  1\leq k\leq p\\
  1\leq v\leq p+1,\,v\neq u
  \end{subarray} },
\end{align*}
and making use of the Fourier inversion formula, we have
\begin{align*}
& \mathfrak{d}_{p}\tilde{\omega}(g)(Y_{j_1},Y_{j_2},\cdots ,Y_{j_{p+1}})=\sum_{I_p}\sum_{u=1}^{p+1}(-1)^{u-1}Y_{j_u}\omega_{I_p}(g)\textnormal{det}[dY_{i_k}(Y_{j_v})]_{ \begin{subarray}{l}  1\leq k\leq p\\
  1\leq v\leq p+1,\,v\neq u
  \end{subarray} }\\
  &=\sum_{[\xi]\in \widehat{G}}\sum_{I_p}d_\xi  \textnormal{Tr}[ \xi(g) \left(\sum_{u=1}^{p+1}(-1)^{u-1}\textnormal{det}[dY_{i_k}(Y_{j_v})]_{ \begin{subarray}{l}  1\leq k\leq p\\
  1\leq v\leq p+1,\,v\neq u
  \end{subarray} }Y_{j_u}\xi(e_G)\right)\widehat{\tilde{\omega}}_{I_p}(\xi)].
\end{align*}
Note that the terms $  \mathfrak{d}_{p}\tilde{\omega}(g)(Y_{j_1},Y_{j_2},\cdots ,Y_{j_{p+1}})$ are the coefficients of the function $\mathfrak{d}_{p}\tilde{\omega}(g)$ with respect to its expansion as a linear combination of elements of the basis $\{dY_{J_{p+1}}\},$ so that
\begin{align*}
& \mathfrak{d}_{p}\tilde{\omega}(g)=\\
&\sum_{[\xi]\in \widehat{G}}\sum_{I_p,J_{p+1}} d_\xi  \textnormal{Tr}[ \xi(g) \left(\sum_{u=1}^{p+1}(-1)^{u-1}\textnormal{det}[dY_{i_k}(Y_{j_v})]_{ \begin{subarray}{l}  1\leq k\leq p\\
  1\leq v\leq p+1,\,v\neq u
  \end{subarray} }Y_{j_u}\xi(e_G)\right)\widehat{\tilde{\omega}}_{I_p}(\xi)]dY_{J_{p+1}}.
\end{align*}
Since we want to  find the unique matrix-valued function $$(I_{p},J_{p+1},[\xi])\mapsto \sigma_{\mathfrak{d}_{p}}(I_{p},J_{p+1},[\xi])$$ satisfying
\begin{equation*}
   \mathfrak{d}_p\tilde{\omega}(g)=\sum_{I_{p},J_{p+1},[\xi]\in \widehat{G}}d_{\xi}\textnormal{Tr}[\xi(g)\sigma_{\mathfrak{d}_p}(I_p,J_{p+1},\xi)    \widehat{\tilde{\omega}}(I_p,\xi)   ]dY^{J_{p+1}},
\end{equation*}
 we deduce that
\begin{equation*}
    \sigma_{\mathfrak{d}_{p}}(I_{p},J_{p+1},[\xi])=\sum_{u=1}^{p+1}(-1)^{u-1}\textnormal{det}[dY_{i_k}(Y_{j_v})]_{ \begin{subarray}{l}  1\leq k\leq p\\
  1\leq v\leq p+1,\,v\neq u
  \end{subarray} }Y_{j_u}\xi(e_G),
\end{equation*} We conclude the proof by using  the identity $\sigma_{Y_{j_u}}(\xi)=Y_{j_u}\xi(e_G).$
\end{proof}

\subsection{The global symbol of $d^{*}$} The adjoint operator $d_{p}^{*}$ of the exterior derivative $d_{p}$ induces an operator $\mathfrak{d}_{p}^{*}$ defined by the identification $\Gamma(\Omega^{p}(M))\cong C^{\infty}(G,\bigwedge^{p}(\mathfrak{g}/\mathfrak{k})^{*}\otimes\mathbb{C})^{K}$ in Remark \ref{remarkaboutidentification}. It makes commutative the following diagram

\begin{eqnarray}
    \begin{tikzpicture}[every node/.style={midway}]
  \matrix[column sep={14em,between origins}, row sep={4em}] at (0,0) {
    \node(R) {$\Gamma(\Omega^{p+1}(M))$}  ; & \node(S) {$\Gamma(\Omega^{p}(M))$}; \\
    \node(R/I) {$C^{\infty}(G,\bigwedge^{p+1}(\mathfrak{g}/\mathfrak{k})^{*}\otimes\mathbb{C})^{K}$}; & \node (T) {$C^{\infty}(G,\bigwedge^{p}(\mathfrak{g}/\mathfrak{k})^{*}\otimes\mathbb{C})^{K}.$};\\
  };
  \draw[<-] (R/I) -- (R) node[anchor=east]  {$\cong$};
  \draw[->] (R) -- (S) node[anchor=south] {$d_{p}^*$};
  \draw[->] (S) -- (T) node[anchor=west] {$\cong$};
  \draw[->] (R/I) -- (T) node[anchor=north] {$\mathfrak{d}_{p}^*$};
\end{tikzpicture}
\end{eqnarray} More precisely,
\begin{equation}
   \boxed{ \mathfrak{d}_{p}^*\tilde{\omega}=\widetilde{d_{p}^*\omega}  ,\,\,\omega\in\Gamma( \Omega^{p+1}(M)) }
\end{equation}
Using the basis \begin{equation}
    \{ Y_1,\cdots, Y_m, Y_{m+1}\cdots,Y_{n}    \},\,m=\dim(M),\,n=\dim(G),
\end{equation}of $\mathfrak{g}$ in \eqref{basisofg}, one can express the operator $\mathfrak{d}_{p}^*$ by (see \cite[Page 520]{Ikeda})
\begin{equation}
   \mathfrak{d}_{p}^*\tilde{\omega}(Y_{i_1},Y_{i_2},\cdots ,Y_{i_{p}}):=-\sum_{u=1}^{n}Y_{u}\tilde{\omega}(Y_u,Y_{i_1},Y_{i_2},\cdots ,Y_{i_{p}}).
\end{equation}
\begin{theorem}\label{symbolofadjoint} Let us consider the  Cartan decomposition $\mathfrak{g}=\mathfrak{m}\oplus \mathfrak{k}.$ Under the  identification of  $(\mathfrak{g}/\mathfrak{k})^{*},$ with the set of linear forms on $\mathfrak{g},$ vanishing in $\mathfrak{k},$ the symbol of $\mathfrak{d}_p^{*}$ is given by
\begin{equation}\label{symbol:of:d_p1}
    \sigma_{\mathfrak{d}_{p}^*}(I_{p+1},J_{p},[\xi])=-\sum_{u=1}^{n}\sigma_{Y_{u}}(\xi) \det[\delta_{i_k,j_v}]_{ \begin{subarray}{l}  1\leq k\leq p+1\\
  0\leq v\leq p,\,j_0= u
  \end{subarray} },
\end{equation}for any triple $(I_p,J_{p+1},[\xi]),$ with $[\xi]\in \widehat{G}.$

\end{theorem}
\begin{proof} We want to  find the unique matrix-valued function $$(I_{p+1},J_{p},[\xi])\mapsto \sigma_{\mathfrak{d}_{p}^*}(I_{p+1},J_{p},[\xi])$$ satisfying the quantisation formula
\begin{equation*}
   \mathfrak{d}_p^*\tilde{\omega}(g)=\sum_{I_{p+1},J_{p},[\xi]\in \widehat{G}}d_{\xi}\textnormal{Tr}[\xi(g)\sigma_{\mathfrak{d}_p^*}(I_{p+1},J_{p},\xi)    \widehat{\tilde{\omega}}(I_{p+1},\xi)   ]dY^{J_{p}}.
\end{equation*}By considering the expansion of the $p+1$-form
 $\tilde{\omega}\in C^{\infty}(G,\bigwedge^{p+1}(\mathfrak{g}/\mathfrak{k})^{*}\otimes\mathbb{C})^{K} $ we have
\begin{equation}
      \tilde{\omega}(g)=\sum_{I_{p+1}}\omega_{I_{p+1}}(g)dY^{I_{p+1}}.\,\,
\end{equation}By evaluating $\mathfrak{d}_p^*\tilde{\omega}(g)$ at $(Y_{j_1},Y_{j_2},\cdots ,Y_{j_{p}}),$ we obtain
\begin{align*}
\mathfrak{d}_p^*\tilde{\omega}(g)(Y_{j_1},Y_{j_2},\cdots ,Y_{j_{p}})
&=-\sum_{u=1}^{n}Y_{u}\tilde{\omega}(Y_u,Y_{j_1},Y_{j_2},\cdots ,Y_{j_{p}})\\
&=-\sum_{I_{p+1}}\sum_{u=1}^{n}Y_{u}\omega_{I_{p+1}}(g)dY^{I_{p+1}}(Y_u,Y_{j_1},Y_{j_2},\cdots ,Y_{j_{p}})\\
&=-\sum_{I_{p+1}}\sum_{u=1}^{n}Y_{u}{\omega}_{I_{p+1}}(g)\det[dY_{i_k}(Y_{j_v})]_{ \begin{subarray}{l}  1\leq k\leq p+1\\
  0\leq v\leq p,\,j_0= u
  \end{subarray} }.
\end{align*}
The Fourier inversion formula implies
\begin{align*}
 &\mathfrak{d}_p^*\tilde{\omega}(g)(Y_{j_1},Y_{j_2},\cdots ,Y_{j_{p}})\\
 &=\sum_{[\xi]\in \widehat{G}} \sum_{I_{p+1}}d_\xi\textnormal{Tr}[\xi(g)\left(-\sum_{u=1}^{n}Y_{u}\xi(e_G) \det[dY_{i_k}(Y_{j_v})]_{ \begin{subarray}{l}  1\leq k\leq p+1\\
  0\leq v\leq p,\,j_0= u
  \end{subarray} } \right)\widehat{{\omega}}(I_{p+1},[\xi])].
\end{align*}Consequently,
\begin{align*}
 &\mathfrak{d}_p^*\tilde{\omega}(g)\\
 &=\sum_{[\xi]\in \widehat{G}} \sum_{I_{p+1},J_p}d_\xi\textnormal{Tr}[\xi(g)\left(-\sum_{u=1}^{n}Y_{u}\xi(e_G) \det[dY_{i_k}(Y_{j_v})]_{ \begin{subarray}{l}  1\leq k\leq p+1\\
  0\leq v\leq p,\,j_0= u
  \end{subarray} } \right)\widehat{{\omega}}(I_{p+1},[\xi])]dY_{J_p},
\end{align*}and we deduce that
\begin{equation*}
      \sigma_{\mathfrak{d}_{p}}(I_{p+1},J_{p},[\xi])=-\sum_{u=1}^{n}Y_{u}\xi(e_G) \det[dY_{i_k}(Y_{j_v})]_{ \begin{subarray}{l}  1\leq k\leq p+1\\
  0\leq v\leq p,\,j_0= u
  \end{subarray} }. 
\end{equation*}Thus, we conclude the proof.
\end{proof}

\subsection{The Dirac operator and the global symbol of the Hodge-Laplacian}

Let us consider the total space $\Omega^{\bullet}(M)=\bigoplus_{p=0}^{m}\Omega^{p}(M)$ and the corresponding vector bundle
\begin{equation}
     \Omega^{\bullet}(M)=\bigoplus_{p=0}^{m}\Omega^{p}(M)\rightarrow M,
\end{equation}
over $M.$ A section of the vector bundle $\Omega^{\bullet}(M)\rightarrow M$ is an element of the form
\begin{equation}
    \omega=\bigoplus \omega^{p},\,\,\omega^{p}\in \Gamma(\Omega^p(M)),
\end{equation}that is, a direct sum of $p$-differential forms over $M.$ The Dirac operator
\begin{equation}
    \slashed{D}:=d+d^{*}:\Gamma(\Omega^\bullet(M))\rightarrow \Gamma(\Omega^\bullet(M))
\end{equation}operates on a graded form $ \omega=\bigoplus_{p=0}^{m} \omega^{p},$ by
\begin{equation}
  \slashed{D}\omega\equiv (d+d^{*})\omega:= d^{*}_0\omega^{1}\oplus(d_0\omega^{0}+d^{*}_1\omega^2)\oplus \cdots \oplus (d_{k-1}\omega^{k-1}+d^{*}_k\omega^{k+1})\oplus\cdots   
\end{equation}
So, to compute the action of the Dirac operator on a graded form $\omega$, we may compute the components $d_{k-1}\omega^{k-1}+d^{*}_k\omega^{k+1}.$
The Fourier analysis above and Theorems  \ref{symbolofD} and \ref{symbolofadjoint} imply 
\begin{align*}
 &\widetilde{\slashed{D}\omega}|_{\Omega^{k}(M)}:=  \widetilde{d_{k-1}\omega^{k-1}}+\widetilde{d^{*}_k\omega^{k+1}}\\ &=\sum_{I_{k-1},J_{k},[\xi]\in \widehat{G}}d_{\xi}\textnormal{Tr}[\xi(g)\sigma_{\mathfrak{d}_p}(I_{k-1},J_{k},\xi)    \widehat{\tilde{\omega}}^{k-1}(I_{k-1},\xi)   ]dY^{J_{k}}\\
 &\,\,\,+\sum_{I_{k+1},J_{k},[\xi]\in \widehat{G}}d_{\xi}\textnormal{Tr}[\xi(g)\sigma_{\mathfrak{d}_{k}^*}(I_{k+1},J_{k},\xi)    \widehat{\tilde{\omega}}^{k+1}(I_{k+1},\xi)   ]dY^{J_{k}}\\
 &=\sum_{\begin{subarray}{l} I_{k-1},I_{k+1}\\
 J_{k},\,[\xi]\in \widehat{G}
  \end{subarray}}d_{\xi}\textnormal{Tr}[\xi(g)\left(\sigma_{\mathfrak{d}_p}(I_{k-1},J_{k},\xi)    \widehat{\tilde{\omega}}^{k-1}(I_{k-1},\xi)+\sigma_{\mathfrak{d}_{k}^*}(I_{k+1},J_{k},\xi)    \widehat{\tilde{\omega}}^{k+1}(I_{k+1},\xi)    \right)  ]dY^{J_{k}}.
\end{align*}
Finally, if one considers the Hodge-Laplacian $$ \slashed{D}^2_{k+1}=d_{k}d^{*}_k+d^{*}_{k+1}d_{k+1}:\Gamma(\Omega^{k+1}(M))\rightarrow \Gamma(\Omega^{k+1}(M)),$$ it was shown in \cite[Page 519]{Ikeda} that
\begin{equation}
   \boxed{ -\mathcal{L}_G\tilde\omega^{k+1}=\widetilde{\slashed{D}^2_{k+1}\omega^{k+1}}  ,\,\,\omega^{k+1}\in \Gamma( \Omega^{k}(M))  }
\end{equation}
restricting  the  Killing  form  sign  changed  to  $\mathfrak{m}$ in the Cartan decomposition $\mathfrak{g}=\mathfrak{m}\oplus\mathfrak{k}.$
By using the argument in the proof of Theorem \ref{symbolofD}, let us find the unique matrix-valued function $$(I_{k+1},J_{k+1},[\xi])\mapsto \sigma_{\slashed{D}^2_{k+1}}(I_{k+1},J_{k+1},[\xi])$$ satisfying the quantisation formula
\begin{equation*}
    -\mathcal{L}_G\tilde\omega^{k+1}(g)=\sum_{I_{k+1},J_{k+1},[\xi]\in \widehat{G}}d_{\xi}\textnormal{Tr}[\xi(g)\sigma_{\slashed{D}^2_{k+1}}(I_{k+1},J_{k+1},[\xi])   \widehat{\tilde{\omega}}^{k+1}(I_{k+1},\xi)   ]dY^{J_{k+1}}.
\end{equation*}By considering the expansion of the $k+1$-form
 $\tilde{\omega}\in C^{\infty}(G,\bigwedge^{k+1}(\mathfrak{g}/\mathfrak{k})^{*}\otimes\mathbb{C})^{K} $ we have
\begin{equation}
      \tilde{\omega}(g)=\sum_{I_{k+1}}\omega_{I_{k+1}}(g)dY^{I_{k+1}}.\,\,
\end{equation}By evaluating $-\mathcal{L}_G\tilde\omega^{k+1}(g)$ at $(Y_{j_1},Y_{j_2},\cdots ,Y_{j_{k+1}}),$ and using the Fourier inversion formula, we obtain
\begin{align*}
-\mathcal{L}_G\tilde\omega^{k+1}(g)(Y_{j_1},Y_{j_2},\cdots ,Y_{j_{k+1}})
&=\sum_{I_{k+1}}-\mathcal{L}_G\omega_{I_{k+1}}(g)\\
 &=\sum_{[\xi]\in \widehat{G}} \sum_{I_{p+1}}d_\xi\textnormal{Tr}[\xi(g)\left(-\lambda_{[\xi]}I_{d_\xi}\right)\widehat{{\omega}}(I_{k+1},[\xi])].
\end{align*}Consequently,
\begin{align*}
 -\mathcal{L}_G\tilde\omega^{k+1}(g)
 &=\sum_{[\xi]\in \widehat{G}} \sum_{I_{k+1},J_{k+1}}d_\xi\textnormal{Tr}[\xi(g)\left(-\lambda_{[\xi]}I_{d_\xi}\right)\widehat{{\omega}}^{k+1}(I_{k+1},[\xi])]dY_{J_{k+1}},
\end{align*}and we deduce that
\begin{equation*}
      \sigma_{\slashed{D}^2_{k+1}}(I_{k+1},J_{k+1},[\xi])   =-\lambda_{[\xi]}I_{d_\xi}, 
\end{equation*}as expected.

\subsubsection*{Acknowledgements} The authors thank Julio Delgado and  David Rottensteiner for discussions.

\bibliographystyle{amsplain}

\end{document}